\documentclass[a4paper,11pt]{article}
\usepackage{amsthm,amsfonts,amssymb,euscript}
\usepackage{latexsym, multicol, fancybox}
\usepackage{graphicx}
\usepackage{color}
\usepackage{amsmath, amsthm, amssymb, bm}
\usepackage{epstopdf}
\usepackage{caption}
\usepackage{psfrag}
\newtheorem{theorem}{Theorem}[section]
\newtheorem{lemma}[theorem]{Lemma}
\newtheorem{proposition}[theorem]{Proposition}
\newtheorem{corollary}[theorem]{Corollary}
\newtheorem{definition}[theorem]{Definition}
\newtheorem{remark}[theorem]{Remark}

\newcommand{\bea}{\begin{eqnarray}}
\newcommand{\eea}{\end{eqnarray}}
\def\beaa{\begin{eqnarray*}}
\def\eeaa{\end{eqnarray*}}
\def\ba{\begin{array}}
\def\ea{\end{array}}
\def\be#1{\begin{equation} \label{#1}}
\def \eeq{\end{equation}}

\newcommand{\nn}{\nonumber}

\def\a{{\alpha}}
\def\al{{\alpha}}
\def\b{{\beta}}
\def\be{{\beta}}
\def\ga{\gamma}
\def\Ga{\Gamma}
\def\de{\delta}
\def\De{\Delta}
\def\ep{\epsilon}

\def\la{\lambda}
\def\La{\Lambda}
\def\si{\sigma}

\def\om{\omega}

\def\vphi{\varphi}
\def\varo{{\varrho}}
\def\th{\theta}

\def\ze{\zeta}

\def\nab{\nabla}
\def\qf{\frak{q}}

\def\pr{{\partial}}

\def\c{\cdot}

\def\AA{{math\cal A}}

\def\CC{{\mathcal C}}
\def\MM{{\mathcal M}}

\def\LL{{\mathcal L}}

\def\FF{{\mathcal F}}
\def\EE{{\mathcal E}}

\def\SS{{\mathcal S}}

\def\DD{{\mathcal D}}
\def\PP{{\mathcal P}}

\def\QQ{{\mathcal Q}}
\def\AA{{\mathcal A}}

\def\DDco{\overline{\DD}}

\def\dk{\frak{d}}

\def\C{{\bf C}}
\def\D{{\bf D}}

\def\M{{\bf M}}

\def\O{{\bf O}}

\def\R{{\bf R}}

\def\T{{\bf T}}
\def\Z{{\bf Z}}
\def\g{{\bf g}}

\def\CCC{{\Bbb C}}
\def\f12{{\frac 1 2}}

\def\dual{{\,\,^*}}
\def\div{{\mbox div\,}}
\def\curl{{\mbox curl\,}}
\def\lot{\mbox{ l.o.t.}}
\def\Lb{{\,\underline{l}}}
\def\Hb{\,\underline{H}}
\def\Lb{{\,\underline{L}}}

\def\Xho{\,^{(h)}X}
\def\Yh{\,^{(h)}Y}
\def\trch{{\mbox tr}\, \chi}
\def\chih{{\hat \chi}}
\def\chib{{\underline \chi}}
\def\chibh{{\underline{\chih}}}

\def\etab{{\underline \eta}}
\def\omb{{\underline{\om}}}
\def\bb{{\underline{\b}}}
\def\aa{{\underline{\a}}}
\def\xib{{\underline \xi}}

\def\Xib{\underline{\Xi}}

\def\Ab{\underline{A}}
\def\Bb{\underline{B}}
\def\Xb{\underline{X}}

\def\Xh{\widehat{X}}
\def\Xbh{\widehat{\Xb}}

\def\varoc{\check{\varrho}}

\def\lap{{\Delta}}

\def\tr{\mbox{tr}}
\def\atr{\,^{(a)}\mbox{tr}}

\def\trchb{{\tr \,\chib}}

\def\Div{\mbox{Div}}

\def\chia{{\,^{(a)}\chi}}

\def\chiba{{\,^{(a)}\chib}}

\def\Us{{\,^{(s)}U}}
\def\Ua{{\,^{(a)}U}}
\def\Uh{{\widehat U}}

\def\Vh{{\widehat V}}

\def\atrch{\atr\chi}
\def\atrchb{\atr\chib}

\def\hot{\widehat{\otimes}}
\def\rhod{\,\dual\rho}

\def\piX{\, ^{(X)}\pi}

\def\err{\mbox{Err}}
\def\ov{\overline}
\def\squared{\dot{\square}}

\newcommand{\hch}{\widehat{\chi}}

\def\f12{\frac 1 2}
\def\lab{\label}
\def\nabc{\,^{(c)}\nab}

\def\bsplit{\begin{split}}

\def\divc{\,^{(c)}\div}
\def\curlc{\,^{(c)}\curl}
\def\DDc{\,^{(c)} \DD}

\def\DDb{\ov{\DD}}
\def\DDbc{\ov{\DDc}}

\def\DDov{\ov{\DD}}

\def\DDs{ \, \DD \hspace{-2.4pt}\dual    \mkern-20mu /}
\def\DDd{ \, \DD \hspace{-2.4pt}    \mkern-8mu /}

\def\Ddot{\dot{\D}}

\DeclareFontFamily{U}{mathx}{\hyphenchar\font45}
\DeclareFontShape{U}{mathx}{m}{n}{
      <5> <6> <7> <8> <9> <10>
      <10.95> <12> <14.4> <17.28> <20.74> <24.88>
      mathx10
      }{}
\DeclareSymbolFont{mathx}{U}{mathx}{m}{n}
\DeclareFontSubstitution{U}{mathx}{m}{n}
\DeclareMathAccent{\widecheck}{0}{mathx}{"71}

\begin{document}

\title{A general formalism for the stability of Kerr}
\author{Elena Giorgi, Sergiu Klainerman, J\'er\'emie Szeftel}

\maketitle

\begin{abstract}
The goal of this  paper is to  provide a  geometric   framework  for analyzing the uniform  decay properties of  solutions to the 
  Teukolsky equation in  the fully nonlinear setting of  perturbations of Kerr. 
 It contains  the first nonlinear version of the Chandrasekhar   transformation introduced  in the linearized setting in  \cite{D-H-R-Kerr} and \cite{Ma} with the intent to use it in our ongoing project to prove the full nonlinear stability of slowly rotating Kerr as solution to the  Einstein vacuum equations.
\end{abstract}

\tableofcontents


\section{Introduction}


The goal of this  paper is to  provide the geometric framework  for analyzing the uniform  decay properties of  solutions to the 
  Teukolsky equation in  the fully nonlinear setting of  perturbations of Kerr, as solutions to the  Einstein vacuum  equations 
  \bea
  \operatorname{Ric}(g)=0.
  \eea
    In the linearized setting, i.e. the  Teukolsky equation on a fixed  Kerr background with small angular momentum,  this  was achieved 
  independently  in  \cite{D-H-R-Kerr}  and \cite{Ma} by  a suitable version of the  so called  Chandrasekhar transformation. 
  We recall  that this transformation was  first  introduced   by Chandrasekhar  to pass from the Teukolsky  equation to  a  Regge-Wheeler 
equation  in  the context of   mode stability for the linearized gravity system  around Schwarzschild, see \cite{Chand}.   A  first physical version of that transformation  appeared in \cite{D-H-R}   and was used by the authors to derive  robust uniform decay  estimates 
for the Teukolsky equation in Schwarzschild. Its   nonlinear version   appears  in \cite{KS}  where it plays  an  important  role  in the proof of the  full nonlinear  stability of Schwarzschild  for axially symmetric polarized perturbations.  In this paper we derive  the first nonlinear version of the    transformations introduced   in  \cite{D-H-R-Kerr} and \cite{Ma} with the intent to use it in our ongoing project to prove the full nonlinear stability of slowly rotating Kerr as solution to the  Einstein vacuum equations.  Note also that the  result of  \cite{Ma}  was used in
\cite{Kerr-lin1} to provide a first proof of  linearized stability of Kerr. A different proof, based on generalized wave coordinates, was also given in \cite{Kerr-lin2}.


\subsection{General  comments on perturbations of Kerr}


A  proof of the nonlinear stability\footnote{We refer  here  to  a geometric  proof   similar in spirit 
to   that of the  recent  proof of the    nonlinear stability of Schwarzschild   under polarized perturbations \cite{KS}. Other approaches are also being pursued, most notably those based on a generalized notion of  wave coordinates.}  of Kerr, at least for  small  angular   momentum,      requires the following   major    ingredients.

\begin{enumerate}
\item   A  formalism  to derive   the  Teukolsky  and Regge-Wheeler type  equations  in the  nonlinear setting.
\item An analytic mechanism to derive  estimates  for these.

\item A dynamical mechanism  to  identify the  final mass and final angular momentum.

\item  A   dynamical mechanism for    finding the right gauge conditions   in which  the very notion of convergence to the final state  can be   formulated.

\item  A precisely formulated continuity argument   based on a grand bootstrap   scheme  which assigns to all geometric quantities  involved in the process   specific  decay rates,  which can   be   dynamically recovered  from the initial  conditions   by a long series  of estimates,   and thus ensure  convergence to a   final Kerr  state. 
\end{enumerate}

In \cite{KS-Kerr1} and \cite{KS-Kerr2},   the last two authors    have provided a   framework for  
 dealing  with  the  issue  4, by  constructing generalized notions of  generally covariant modulated (GCM)  spheres\footnote{Generalizing those   used in the  nonlinear stability of  Schwarzschild  in the polarized case, see \cite{KS}.} in the asymptotic region of a general perturbation of Kerr. In connection  to  the issue 3, it  is expected, as in \cite{KS},   that  the  Hawking mass  associated to  the  $2$-spheres of  a well adapted  spacetime foliation  will converge to the final mass. In  \cite{KS-Kerr2},   a definition of   angular momentum  for GCM spheres was also given with the expectation that  it will play a similar role in detecting the final angular momentum of a general perturbation of Kerr.  
 
 In this paper  we   deal with   the first issue, i.e. we  provide a geometric framework  in which 
 the Teukolsky and Regge-Wheeler  equations  can be   derived  for general, nonlinear,  \textit{realistic} perturbations  of Kerr.  We emphasize  the word  realistic  by which we mean 
 that the perturbations   have  decay properties consistent  with the bootstrap assumptions mentioned in   5.   We make these assumptions  here  by  adapting    those made in \cite{KS}.

 
\subsection{Teukolsky  equation  in linearized theory}


It has been known since the work of Teukolsky  \cite{Teuk} \cite{P-T}   that,  for  linearized  perturbations   of Kerr,  the   extreme components  of the curvature tensor  verify  decoupled wave equations in a fixed Kerr background.  As mentioned above,  a first  robust  derivation of decay estimates  for solutions of the  Teukolsky  equation in Schwarzschild was given in  \cite{D-H-R},  based on a physical space   variant of the   Chandrasekhar transformation\footnote{Which takes solutions of the Teukolsky equation into solutions of a Regge-Wheeler type     equation,  that can then be analyzed using a  variant of the vectorfield method based on  Morawetz  
 and $r^p$-weighted estimates.}. The method was recently  extended to  the Teukolsky equation in Kerr spacetimes $Kerr(a,m) $  with $|a|\ll m$  independently in \cite{Ma} and \cite{D-H-R-Kerr}.

Both   \cite{D-H-R-Kerr}  and \cite{Ma}    start from the    well known second order    Teukolsky equation
$\LL(\alpha^{[\pm2]}) =0$ 
in  a Kerr spacetime $Kerr(a, m)$. The complex  scalar $\alpha^{[\pm2]}$ (for spin $+2$ or spin $-2$)  is         defined,  in the context of the Newman-Penrose  (NP) formalism,  to be     one of the extreme curvature  components, relative  to   a principal    null  frame\footnote{We recall 
that principal null frames have the very  important property  of diagonalizing the   curvature tensor.}. The Teukolsky equation for $\alpha^{[\pm2]}$ is a decoupled wave equation  which is obtained as consequence of    the linearized   Einstein vacuum  equations. 
Both     \cite{D-H-R-Kerr}  and \cite{Ma}   transform $\alpha^{[\pm2]}$,    by an appropriate  generalized   Chandrasekhar  transformation,   into a    new complex scalar  $\Psi=Q( \alpha^{[\pm2]}) $, with $Q$ a second order operator (in the direction of the ingoing null direction for spin $+2$, outgoing for spin $-2$),  expressed relative to the standard Kerr coordinates.   The complex scalar $\Psi$  verifies  a       generalized Regge-Wheeler equation of the schematic form
\bea
\lab{Intro:Regge-Wheeler}
\square_{a,m} \Psi + i a\,  c(r, \th) \partial_t \Psi +V(r, \th) \Psi = a\,  L_{Q}(\alpha^{[\pm2]}) 
\eea
where $\square_{a,m}$ is the D'Alembertian operator associated to the Kerr metric $Kerr(a, m)$, $c(r, \th)$ is a function of $r$ and $\th$, $V(r, \th)$ is a    favorable potential and $L_{Q}(\alpha^{[\pm2]}) $ are lower order terms in the derivatives of $\alpha^{[\pm2]}$.  
Though  the complex scalar $\alpha^{[\pm2]}$   is defined using  a principal null  tetrad via the NP formalism, all calculations which take the  Teukolsky equation to \eqref{Intro:Regge-Wheeler} are done  in coordinates.   This approach  leads to serious difficulties   when one  tries to  extend the  calculations in the nonlinear setting where the precise structure of the nonlinear terms is essential. 

 In the  setting of polarized perturbations of Schwarzschild \cite{KS},         this calculation  was performed  using null   frames, which are both   adapted to   an integrable   foliation and   almost diagonalize the curvature tensor\footnote{i.e. such  frames diagonalize   the curvature tensor up to terms which are small with respect to the perturbation.   These frames are  thus small perturbations of standard null frames  in Schwarzschild   which are both principal  and  integrable.}.    One could thus  rely on  the geometric  formalism developed 
 in the context of the proof of the nonlinear  stability of Minkowski space \cite{Ch-Kl}.
  This latter  cannot  be straightforwardly applied from Minkowski to perturbations of Kerr.    To capture the almost  diagonalizable properties of  principal null frames  one has to give up on integrability 
     and thus, seemingly,  forced to work with a general Newman-Penrose (NP)   formalism, as introduced in  \cite{NP}.
      The NP formalism, however, presents its own set of difficulties.    Indeed  complex calculations, such those needed  to derive the nonlinear analogue of    \eqref{Intro:Regge-Wheeler},    depend    on higher  derivatives  of  all   connection coefficients of the NP  frame rather than only  those which are geometrically significant.  This could  seriously affect the structure of non-linear corrections.

  In our work we rely instead on a  tensorial approach, based on horizontal structures which closely mimics the calculations  done  in  integrable  settings while maintaining the important diagonalizable  properties of the  principal  directions.  
  This allows us to maintain, with minimal  changes, the  geometric formalism  of \cite{Ch-Kl}  widely used today  in mathematical GR. 
       We give a full account of it in this paper\footnote{Elements of the  horizontal structure  framework    were  also  shortly discussed   in  the appendix  to \cite{I-Kl}  as an alternative to the  NP formalism.    Relations between the two   formalisms  are   addressed  there as well.}  and  use it   to
  derive the proper nonlinear corrections to  \eqref{Intro:Regge-Wheeler}. We will show in fact that these corrections  admit a structure very similar to the one  found in \cite{KS}, see  Theorem 2.4.7 in that paper.

 
  \subsection{Short summary of the formalism}
  
  
  We give here a short summary of the formalism, which is extensively described in  section \ref{section-general-formalism}.

  
  \subsubsection{Basic definitions}
  
  
Let $(\MM, \g)$ a Lorentzian space-time. Consider a fixed    null pair  of vectorfields  $(e_3, e_4)$, i.e. 
\beaa
\g(e_3, e_3)=\g(e_4, e_4)=0, \qquad  \g(e_3, e_4)=-2, 
\eeaa
and denote  by  $\O(\MM)$ the vector space  of horizontal vectorfields $X$  on $\MM$, i.e.  $\g(e_3, X)= \g(e_4, X)=0$.
 A null  frame $(e_3, e_4, e_1, e_2)$ on $\MM$ consists, in addition to the null pair $(e_3, e_4)$, of a choice of horizontal vectorfields  $(e_1, e_2)$, such that
 \beaa
 \g(e_a, e_b)=\de_{ab}\qquad  a, b=1,2.
 \eeaa  
 The commutator $[X,Y]$ of two horizontal vectorfields
may fail however to be horizontal. We say that the pair $(e_3, e_4 )$ is integrable if   $\O(\MM)$  forms an integrable distribution,
i.e. $X, Y\in\O(\MM) $ implies that $[X,Y]\in\O(\MM)$. As  it is well-known,  the  principal null pair in Kerr fails to be integrable, 
see also Remark \ref{rem:nonintegrabilityandatrchatrchb}. Given an arbitrary vectorfield $X$ we denote by $^{(h)}X$
its  horizontal projection, 
$$^{(h)}X=X+ \frac 1 2 \g(X,e_3)e_4+ \frac 1 2   \g(X,e_4) e_3.$$ 
For any $ X, Y\in \O(\MM)$ we define  $\ga(X, Y)=\g(X, Y)$ and\footnote{In the particular case where the horizontal structure is integrable, $\ga$ is the induced metric, and $\chi$ and $\chib$ are the  null second fundamental forms.}  
\beaa
\chi(X,Y)=\g(\D_Xe_4 ,Y), \qquad \chib(X,Y)=\g(\D_Xe_3,Y).
\eeaa
Observe that    $\chi$
 and $\chib$  are  symmetric if and 
 only if   the horizontal structure is 
 integrable. Indeed this follows easily from
 the formulas\footnote{Note  that  we  can view  $\chi$ and $\chib$ as horizontal 2-covariant tensor-fields
 by extending their definition to arbitrary  $X, Y$, by  $\chi(X, Y)= \chi( ^{(h)}X, ^{(h)}Y)$,  $\chib(X, Y)= \chib( ^{(h)}X, ^{(h)}Y)$.},
 \beaa
 \chi(X,Y)-\chi(Y,X)&=&\g(\D_X e_4, Y)-\g(\D_Ye_4,X)=-\g(e_4, [X,Y]),\\
 \chib(X,Y)-\chib(Y,X)&=&\g(\D_X e_3, Y)-\g(\D_Ye_3,X)=-\g(e_3, [X,Y]).
\eeaa
   We define their trace $\trch$, $\trchb$,  and anti-trace $\atrch$, $\atrchb$ as follows
\beaa
\trch:=\de^{ab}\chi_{ab}, \qquad \trchb:=\de^{ab}\chib_{ab}, \qquad \atrch:=\in^{ab}\chi_{ab}, \qquad \atrchb:=\in^{ab}\chib_{ab}.
\eeaa
Accordingly, we  decompose $\chi, \chib$ as follows
\beaa
\chi_{ab}&=&\chih_{ab} +\frac 1 2 \de_{ab} \trch+\frac 1 2 \in_{ab}\atrch,\\
\chib_{ab}&=&\chibh_{ab} +\frac 1 2 \de_{ab} \trchb+\frac 1 2 \in_{ab}\atrchb.
\eeaa

\begin{remark}\lab{rem:nonintegrabilityandatrchatrchb}
The non integrability of $(e_3, e_4)$  corresponds to non vanishing $\atrch$ and $\atrchb$. A celebrated example of  a non integrable null frame  is the principal null frame of Kerr for which $\atrch$ and $\atrchb$ are indeed non trivial.
\end{remark}

We define the horizontal covariant operator $\nab$ as follows:
 \bea
 \nab_X Y&:=&^{(h)}(\D_XY)=\D_XY- \frac 1 2 \chib(X,Y)e_4 -  \frac 1 2 \chi(X,Y) e_3, \quad X, Y\in \O(\MM).
 \eea
Note that,
\beaa
 \nab_X Y-\nab_Y X   &=&[X, Y]-\frac 1 2 \left(\atrchb\,  e_4+\atrch \, e_3\right)\in(X, Y).
\eeaa
 Note that  $\nab$  acts like a Levi-Civita connection  i.e.,  for  all  $X,Y, Z\in \O(\MM)$,
 \beaa
 Z \ga (X,Y)=\ga(\nab_Z X, Y)+ \ga(X, \nab_ZY).
 \eeaa
 We can then
  define connection  and curvature coefficients   similar to   those  in the integrable case, as  in  \cite{Ch-Kl}, 
  \beaa
 \begin{split}
\chib_{ab}&=\g(\D_ae_3, e_b),\qquad \,\,\chi_{ab}=\g(\D_ae_4, e_b),\qquad\,\,\,\,
\xib_a=\frac 1 2 \g(\D_3 e_3 , e_a),\qquad \,\,\,\xi_a=\frac 1 2 \g(\D_4 e_4, e_a),\\
\omb&=\frac 1 4 \g(\D_3e_3 , e_4),\qquad\,\, \om=\frac 1 4 \g(\D_4 e_4, e_3),\qquad 
\etab_a=\frac 1 2 \g(\D_4 e_3, e_a),\qquad \quad \eta_a=\frac 1 2 \g(\D_3 e_4, e_a),\qquad\\
 \ze_a&=\frac 1 2 \g(\D_{e_a}e_4,  e_3),
 \end{split}
\eeaa
\beaa
\a_{ab}=\R_{a4b4},\quad \b_a=\frac 12 \R_{a434}, \quad   \bb_a=\frac 1 2 \R_{a334},  \quad \aa_{ab}=\R_{a3b3},\quad \rho=\frac 1 4 \R_{3434} , \quad \rhod=\frac 1 4\dual \R_{3434},
\eeaa
and derive  the corresponding   null structure and  null Bianchi equations, see  Propositions 
\ref{prop-nullstr} and  \ref{prop:bianchi}. 

\begin{remark} The main advantage of  the formalism presented above    is that, with the exception of the presence of  the  integrable defect  scalars $\atrch, \atrchb$,  the  null structure and  null Bianchi equations  look very similar to the  familiar equations in  \cite{Ch-Kl}.
\end{remark}


\subsubsection{Conformally invariant derivative operators}


 The Ricci and  curvature coefficients   depend, of course, on the particular   null frame we choose.  Of  particular importance in our work here are    the conformal frame  transformations
  $e_3'=\la^{-1} e_3, \,  e'_4 = \la e_4 , \, e_a'= e_a$. 
 Note, in particular,  that under such   a  frame  transformation  we have
\beaa
 \trchb'&=&\la^{-1} \trchb, \quad \atrchb'=\la^{-1} \atrchb,  \quad   \trch'=\la\trch, \quad \atrch'=\la \atrch, \\
  \xi'&=& \la \xi, \quad   \eta'=\eta, \quad \etab'=\etab,  \quad \xib'=\la^{-1}\xib,\\
   \a'&=&\la^2 \a,\quad \b'=\la \b,    \quad \rho'=\rho,    \quad \rhod'=\rhod,  \quad \bb'=\la^{-1} \bb,\quad  \aa'=\la^{-2} \aa,
   \eeaa
   and
   \beaa
 \omb'&=& \la^{-1}\left(\omb +\frac{1}{2} e_3(\log \la)\right), \quad \om'= \la\left(\om -\frac{1}{2} e_4(\log \la)\right), \quad
 \ze'= \ze - \nab (\log \la).\\
\eeaa
We say that  a horizontal tensor $f$ is $s$-conformal invariant  if, under the  conformal  frame transformation above it changes as\footnote{Note that $s$ is precisely what   is called    in \cite{Ch-Kl}    the  signature of the tensor.} $f'=\la^s f $.  According to this definition $\chi  $ is $1$-conformal,  $\a$  is  $2$-conformal, \ldots, while  $\om, \, \omb, \, \ze$ fail to be conformal.

\begin{remark}
\lab{remark-Conf-derivaties-Intro}
If $f$ verifies $f'= \lambda^s f$, then   the derivatives  $\nab_3 f, \nab_4 f, \nab_a f$ are not conformal invariant.
 Note however that 
  \begin{itemize}
 \item $\nabc_3 f:= \nab_3f-2 s \omb f$ is $(s-1)$-conformally invariant.
 \item $\nabc_4 f:= \nab_4f+2 s \om f$ is $(s+1)$-conformally invariant.
 \item $\nabc_Af:= \nab_Af+ s \ze_A f$ is $s$-conformally invariant. 
  \end{itemize}
  The conformal invariant operator  $\nabc_3 $ played an important role in the derivation of the Regge-Wheeler equation in \cite{KS}. Here they  play an  equally important role.
 \end{remark}


\subsubsection{Complexification and values in Kerr}


The   null structure and null Bianchi equations   simplify considerably  by introducing  complex notations such as 
\beaa
&& A:=\a+i\dual\a, \quad B:=\b+i\dual\b, \quad P:=\rho+i\dual\rho,\quad \Bb:=\bb+i\dual\bb, \quad \Ab:=\aa+i\dual\aa,
\\
&& X:=\chi+i\dual\chi, \quad \Xb:=\chib+i\dual\chib, \quad H:=\eta+i\dual \eta, \quad \Hb:=\etab+i\dual \etab, \quad Z:=\ze+i\dual\ze.
\eeaa    
In particular, note that  $\tr X = \trch-i\atrch, \,   \tr\Xb = \trchb -i\atrchb$.

 These  complexified  tensors   take a particularly simple form  in Kerr,  relative to a principal null frame\footnote{There is an indeterminacy in the principal null frame as one may replace the null pair $(e_3, e_4)$ with $(\la^{-1}e_3, \la e_4)$ for any $\la>0$. The formulas provided here correspond to the arbitrary choice of $\la>0$ ensuring $\nab_4e_4=0$ and thus $\om=0$. Note that our main result, stated in Theorem \ref{main-theorem-intro},  is independent of this particular choice.}, see  section \ref{sec:recallbasicthingsinKerr}, 
 \beaa
   \Xh=\Xbh=\Xi=\Xib=\om=0,\qquad A=B=\Bb=\Ab=0,
 \eeaa
and
\beaa
\tr X=\frac{2}{q}, \qquad \tr\Xb=-\frac{2\Delta q}{|q|^4}, \qquad P=-\frac{2m}{q^3},
\eeaa
where $q= r+i a \cos \th$ and $\De= r^2+a^2-2 m r$ relative to  the Boyer-Lindquist coordinates $(r, \th)$.  We note also that $q$ verifies the equations, see \eqref{equations:forq},
\bea
\nab_4 q= \frac 1 2 \tr X q,  \qquad  \nab_3 q=\frac 1 2 \ov{\tr \Xb }  q, \qquad 
\DD q= q  \Hb , \qquad   \DDb q=q \ov{H}.
\eea


\subsection{General perturbations of Kerr} 


To state  our main theorem we need to have a sufficiently   general candidate for spacetime structures which are  small perturbations of Kerr.

\begin{definition}
We say that a spacetime $\MM$ is  an $O(\ep)$ perturbation of   $Kerr(a, m)$, $|a|<m$,   if  $\MM$ comes equipped with two coordinate  functions $r:\MM\longrightarrow (0, \infty), \,\,\th: \MM\longrightarrow [0, \pi] $    and       a null frame $(e_3, e_4, e_1, e_2)$ such that  all Ricci and  curvature coefficients which  vanish    in Kerr are $O(\ep)$ and all  other quantities differ by $O(\ep)$  from their corresponding values in Kerr, expressed with respect to  $(r, \th)$.
 \end{definition}
 
We introduce a schematic notation, similar  to the one used in \cite{KS}, see  Definition 2.3.8 in that paper,  to keep track of the  specific  structure of   error terms. We divide the connection coefficient terms into
\bea\label{definition-Gammas-intro}
\bsplit
\Ga_g^{(0)}&=\left\{ r \Xi, \quad \Xh, \quad \widecheck{Z}, \quad \widecheck{\Hb}, \quad \frac 1 r\nab_4 q- \frac {1}{ 2r} \tr X q, \quad  \frac 1 r  (\DD q- q  \Hb),  \quad  \frac 1 r ( \DDb q-q \ov{H})\right\}, \\
\Ga_b^{(0)}&=\left\{ \widecheck{H},\quad  \Xbh,\quad  \Xib, \quad  \frac 1 r \nab_3 q-\frac {1}{ 2r} \ov{\tr \Xb }  q\right\},
\end{split}
\eea
where, given a quantity $Q$,   we have  denoted by $\widecheck{Q}  $   the renormalized  quantity $\widecheck{Q}= Q-Q_{Kerr} $, with $Q_{Kerr} $   the corresponding   value of $Q $  in Kerr,  expressed in terms of the variables  $r$ and $\th$. Thus, for example,
\beaa
\widecheck{P}=P+\frac{2m}{q^3}, \qquad q=r+i a \cos\th. 
\eeaa

For higher derivatives, we denote 
\beaa
\Ga_g^{(s)}&=& \dk^{\leq s} \Ga_g, \qquad \Ga_b^{(s)}= \dk^{\leq s} \Ga_b,
\eeaa
where
\beaa
\dk&=& \left\{\nab_3,  r\nab_4, r\DD \right\}.
\eeaa

\begin{remark}
The notations above  embody the expected decay properties of the Ricci coefficients,  that are better for $\Ga_g$ than for $\Ga_b$.  The notation  $\dk$ corresponds to the fact that the decay properties of derivatives  behave differently in different directions.
\end{remark}


\subsection{Statement of the main theorems}



\subsubsection{Teukolsky equation} 


Making use of the null Bianchi equations for $A$ and $B$, expressed using  our conformal  derivatives,  we derive the following version of the Teukolsky equation.

 \begin{proposition}[Teukolsky equation]\label{Teukolsky-proposition-intro} The complex tensor $A$ satisfies the following equation:
 \bea\label{Teukolsky-equation-tens}
 \LL(A)&=& \err[\LL(A)]
 \eea
 where\footnote{The operators  $\nabc_4, \, \nabc_3, \, \nabc, \,  \DDc$ are  extensions    of the conformal operators,  introduced in Remark \ref{remark-Conf-derivaties-Intro},  for complex tensors.}
\bea\label{Teukolsky-operator-intro}
\begin{split}
\LL(A) &=-\nabc_4\nabc_3A+ \frac{1}{2}\DDc\hot (\DDbc \c A)+\left(- \frac 1 2 \tr X -2\ov{\tr X} \right)\nabc_3A-\frac{1}{2}\tr\Xb \nabc_4A\\
&+\left( 4H+\Hb +\ov{\Hb} \right)\c \nabc A+ \left(-\ov{\tr X} \tr \Xb +2\ov{P}\right) A+  2H   \hot (\ov{\Hb} \c A)
\end{split}
\eea
with error term  expressed schematically\footnote{The error terms  denoted  $\lot$ are 
quadratic   in  the perturbation and  enjoy  better  decay  properties,  or are higher order  and decay at least as good.}
\bea
\err[\LL(A)]&=& r^{-1}\frak{d}^{\leq 1} \left( \Ga_g B\right)+\Xi \nab_3 B+\lot
\eea
\end{proposition}
 A comparison between  this version of the Teukolsky  equation  and the traditional version\footnote{Note that $A$ here is a complex horizontal $2$-tensor, while the standard Teukolsky equation is expressed relative to the complex scalar $\alpha^{[\pm2]}$.}  derived   in   the NP formalism  is done in   Appendix  \ref{appendix:comparisonofTeuk}.


\subsubsection{Generalized Regge-Wheeler   equation}


We look for a quantity  which generalizes  the quantity $\qf$  used in \cite{KS}. In analogy   with   \cite{KS},  it is natural  to look   for a renormalized version of the  conformally invariant  tensor
\bea\lab{definition-Q(A)-intro}
 Q(A)&=& \nabc_3\nabc_3 A + \C \  \nabc_3A +\D \  A
 \eea
 for some well chosen  complex scalar functions $\C, \D$.  Note  that any such expression is   $O(\ep^2)$  invariant and vanishes in Kerr. 
\begin{definition} 
Given a general null frame $(e_3, e_4, e_1, e_2)$ and given scalar functions $r$ and $\th$ satisfying the assumptions in section \ref{perturbations-section}, we define our main quantity $\qf$ as 
\bea
\qf&=& q \ov{q}^{3} Q(A)=q \ov{q}^{3} \left( \nabc_3\nabc_3 A + \C \  \nabc_3A +\D \  A\right)
\eea
where $q=r+ i a \cos\th$, and where  the complex scalar functions $\C, \D$  are to be suitably  chosen.
\end{definition}

We are  now  ready to state the main result of the paper,   concerning  the wave equation satisfied by $\qf$. 
 
 \begin{theorem}[The generalized Regge-Wheeler equation]\label{main-theorem-intro} 
 There  exist choices of  complex scalar functions $\C, \D$  such that $\qf$ defined above  verifies the equation
 \bea\lab{wave-equation-qf-intro}
 \square_2 \qf   - \frac{4i a\cos\th}{|q|^2} \T( \qf )   - V  \qf &=& a \,  L_{\qf}[A] + \err[\square_2 \qf]
 \eea
 where
 \begin{itemize}
 \item $\T$ is a vectorfield defined by \eqref{definition-T-vectorfield} in the nonlinear case, which reduces to $\partial_t$ in Kerr, 
 
 \item the potential $V$ is a complex scalar function  whose real part coincides with the potential of the Regge-Wheeler equation  in \cite{KS}, i.e. $V=-\trch\trchb+ O\left(\frac{|a|}{r^3}\right)$, 
 
 \item $ L_{\qf}[A ]$ is  a  linear second order   operator  in $A$,   given by
  \beaa
 L_{\qf}[A] &=& c_1\nabc_2 ( \nabc_3 A) + c_2\nabc_3 A + c_3 \nabc (A) + c_4 A
\eeaa
with $c_1,\ldots, c_4$  smooth  functions  of $(r, \th)$ having the following fall-off in $r$ 
\beaa
c_1=O\left(r\right), \qquad c_2=O\left(1\right), \qquad c_3=O\left(1\right), \qquad c_4=O\left(\frac{1}{r}\right),
\eeaa

 \item $\err[\square_2 \qf]$ is  the nonlinear correction term, which is given schematically by  the expression\footnote{The error terms  denoted  $\lot$ are 
quadratic   in  the perturbation and  enjoy  better  decay  properties,  or are higher order  and decay at least as good.} 
 \beaa
 \err[\square_2 \qf]&=& r^2 \frak{d}^{\leq 2} (\Ga_g \c (A, B))+ \nab_3 (r^3 \frak{d}^{\leq 2}( \Ga_g \c (A, B)))\\
 &&+\frak{d}^{\leq 1} (\Ga_g \qf) + r^2 \widecheck{H} \nab^{\leq 1} \nab_3^{\leq 1} A+\lot
 \eeaa
 \end{itemize}
 \end{theorem}
 
\begin{remark}
Note the  similarity between the structure  of the nonlinear terms here with that of  Theorem  2.4.7 in \cite{KS}.
\end{remark}

\begin{remark}
Observe that the generalized Regge-Wheeler equation \eqref{wave-equation-qf-intro} for the complexified symmetric traceless 2-tensor $\qf$ reduces to the equation \eqref{Intro:Regge-Wheeler} obtained in \cite{Ma} and \cite{D-H-R-Kerr} for complex scalar functions once the equation is projected to the component $\qf(e_1, e_1)$, as shown in section \ref{section-projection-equation}.    A   general discussion of   such projections can be found in section     \ref{section:projwave}.
\end{remark}


\subsection{Structure of the paper}


The structure of the paper is as follows.
\begin{itemize}
\item In section 2,  we introduce our general  formalism. 

\item In section 3,  we  introduce our complex notations and rewrite the equations of section 2 that considerably simplify in this setting. 

\item In section 4,  we express the covariant wave operator for complexified symmetric traceless 2-tensors using our complex notations.

\item  In section 5,  we specify our general formalism in the particular case of Kerr.
 
\item In section 6, we  derive the Teukolsky equation in the nonlinear setting.

\item In section 7,  we derive the generalized Regge-Wheeler equation in the nonlinear setting.

\item In section 8,  we derive useful identities involving the complexified symmetric traceless 2-tensor $\qf$.
\end{itemize}


\subsection{Acknowledgments}


The second author is supported   by  the  NSF grant  DMS 180841 as well as by  the Simons grant  10011738. He would like to thank the Laboratoire Jacques-Louis Lions  of Sorbonne Universit\'e  and   IHES   for their  hospitality during  his many visits.  The third author is supported by ERC grant  ERC-2016 CoG 725589 EPGR.


\section{A general formalism}\label{section-general-formalism}



\subsection{Null pairs and horizontal structures}


Let $(\MM, \g)$ a Lorentzian space-time. Consider an arbitrary null pair $e_3=\Lb$, $e_4=L$, i.e.
\beaa
\g(e_3, e_3)=\g(e_4, e_4)=0, \qquad  \g(e_3, e_4)=-2.
\eeaa

\begin{definition}
We say that a vectorfield $X$ is $(L,\Lb)$-horizontal, or
simply horizontal,  if
\beaa
\g(L,X)=\g(\Lb, X)=0. 
\eeaa
We denote by $\O(\MM)$ the set of horizontal vectorfield on $\MM$.  Given a fixed 
 orientation  on $\MM$,  with corresponding  volume form  $\in$,  we define  the induced 
 volume form on   $\O(\MM)$ by,
 \bea
 \in(X, Y):=\frac 1 2\in(X, Y, \Lb, L).  \label{vol.form}
 \eea
\end{definition}

Clearly, any linear combination   of horizontal vectorfields is again
horizontal. The commutator $[X,Y]$ of two horizontal vectorfields
may fail however to be horizontal. We say that the pair $(L,\Lb)$ is integrable if the set of horizontal vectorfields  forms an integrable distribution,
i.e. $X, Y\in\O(\MM) $ implies that $[X,Y]\in\O(\MM)$.
 In this work we will work  with general,
 not necessarily integrable null pairs.

Given an arbitrary vectorfield $X$ we denote by $\Xho$
its  horizontal projection,
\beaa
\Xho&=&X+ \frac 1 2 \g(X,\Lb)L+ \frac 1 2   \g(X,L)\Lb.
\eeaa

\begin{definition}
A  $k$-covariant tensor-field $U$ is said to be horizontal,  $U\in \O_k(\MM)$,
if  for any $X_1,\ldots X_k$ we have,
\beaa
U(X_1,\ldots X_k)=U(\Xho_1,\ldots \Xho_k).
\eeaa
\end{definition}

We can  define the projection operator,
\beaa
\Pi^{\mu\nu}&=& \g^{\mu\nu}+ \frac 1 2 (\Lb^\mu L^\nu+L^\mu\Lb^\nu).
\eeaa
Clearly $ \Pi_{\a}^{\mu}\Pi_{\mu}^\b=\Pi_{\a}^\b$.  An arbitrary tensor 
$U_{\a_1\ldots\a_m}$ is horizontal,  if
\beaa
\Pi_{\a_1}^{\b_1}\ldots \Pi_{\a_m}^{\b_m}\, U_{\b_1\ldots\b_m}=U_{\a_1\ldots\a_m}.
\eeaa

 \begin{definition}\lab{def:proxyfirstandsecondfundameentalform} 
 For any horizontal $X,Y$ we define\footnote{In the particular case where the horizontal structure is integrable, $\ga$ is the induced metric, and $\chi$ and $\chib$ are the  null second fundamental forms.} 
  \bea
 \ga(X,Y) &=& \g(X,Y)
 \eea
 and   
\begin{equation}\label{fo1}
\begin{cases}
&\chi(X,Y)=\g(\D_XL,Y),\\ 
&\chib(X,Y)=\g(\D_X\Lb,Y).
\end{cases}
\end{equation}
\end{definition}

Observe that    $\chi$
 and $\chib$  are  symmetric if and 
 only if   the horizontal structure is 
 integrable. Indeed this follows easily from
 the formulas,
 \beaa
 \chi(X,Y)-\chi(Y,X)&=&\g(\D_X L, Y)-\g(\D_YL,X)=-\g(L, [X,Y]),\\
 \chib(X,Y)-\chib(Y,X)&=&\g(\D_X \Lb, Y)-\g(\D_Y\Lb,X)=-\g(\Lb, [X,Y]).
\eeaa
 We can view $\ga$,  $\chi$ and $\chib$ as horizontal 2-covariant tensor-fields
 by extending their definition to arbitrary vectorfields  $X, Y$  according to,
 \beaa
  \ga(X,Y)&=&\ga(\Xho,\Yh)
  \eeaa
  and
  \beaa
  \chi(X,Y)&=&\chi(\Xho, \Yh),\qquad \chib(X,Y)=\chib(\Xho, \Yh).
 \eeaa

  Given a  general 2- covariant horizontal  tensor $U$
  we decompose it in its symmetric and antisymmetric part as follows,
  \bea
  \Us(X,Y)&=&\frac 1 2 \big(U(X, Y)+U(Y, X)\big),\qquad\\
   \Ua(X,Y)&=&\frac 1 2 \big(U(X, Y)-U(Y, X)\big).\nn
  \eea
  Given an  horizontal structure  defined by $e_3=\Lb$,  $e_4=L$ 
  we  associate a null frame by choosing   orthonormal  horizontal  vectorfields
  $e_1, e_2$  such that $\ga(e_a, e_b)=\de_{ab}$.  By convention, we   say that 
  $(e_1, e_2)$ is positively oriented  on $\O(\MM)$  if,
  \bea
  \in(e_1, e_2)=\in(e_1, e_2,  e_3, e_4)=1.
  \eea

  Given a covariant, horizontal,  2-tensor $U$   and an arbitrary
 orthonormal  horizontal frame $(e_a)_{a=1,2}$ we have,
\beaa
\Us_{ab}=\frac 12 (U_{ab}+U_{ba}),\qquad \Ua_{ab}=\frac 12 (U_{ab}-U_{ba}).
\eeaa
We define the trace and anti-trace  according to,
\begin{definition}
The trace  of a  horizontal  2-tensor $U$ is defined by
\bea
\tr (U):=\de^{ab}U_{ab}=\de^{ab}\Us_{ab}.
\eea
We define the anti-trace of $U$ by,
\bea
\atr (U):=\in^{ab}U_{ab}=\in^{ab}\Ua_{ab}.
\eea
Observe that  the  first    trace  is   independent  of  the particular choice  of  the  frame $e_1, e_2$.
On the other hand, for fixed $e_3, e_4$,   $\atr$  depends on the orientation of $e_1, e_2$.
Also, by interchanging  $e_3, e_4$, $\atr$  changes sign.
\end{definition}

A general  horizontal,  2-tensor $U$  can be decomposed according to,
\bea
U_{ab}&=&\Us_{ab}+\Ua_{ab}=\widehat{U}_{ab}+\frac 1 2 \de_{ab}\, \tr( U)+\frac 12 \in_{ab}\atr (U). \label{decompU}
\eea

In what follows we fix  a null pair $e_3, e_4$ and an orientation  on $\O(\MM)$.
\begin{definition}\label{definition-hodge-duals}
We define the left and right duals of horizontal 1-forms $\om$ and $2$- covariant  tensor-fields $U$,
\bea
\dual \om_{a}=\in_{ab}\om_b,\qquad \om^*\, _{a}=\om_b\in_{ba},
\eea
\bea
(\dual U)_{ab}&=&\in_{ac} U_{cb},\qquad (U^*)_{ab}=
U_{ac}\in_{cb}.
\eea
\end{definition}

\begin{lemma}
Given a horizontal 1-form $\om$, we have
\beaa
\dual(\dual \om)=-\om, \qquad \dual\om=-\om^*.
\eeaa
\end{lemma}

\begin{lemma} 
Given an arbitrary covariant, horizontal 2-tensor $U$, we have
\begin{enumerate}
\item $\dual (\dual U)=-U$.

\item If $U$ is symmetric, then $\dual U=- U^*$. 

\item If  $U=\hat U$ is symmetric, traceless,
then, $\dual \hat U=-\hat U^*$ is also symmetric 
traceless. 
\item  In general, 
\beaa
\tr(\dual U)&=&\tr(U^*)=-\atr(U),\\
\atr(\dual U)&=&\atr(U^*)=\tr(U),\\
\widehat{\dual U}&=&\dual\hat{U}.
\eeaa
\end{enumerate}
\end{lemma}

Given a general horizontal 2-form $U$ we have, according to \eqref{decompU}
\beaa
U_{ab}&=&\hat{U}_{ab}+\frac 1 2 \de_{ab}\, \tr( U)+\frac 12 \in_{ab}\atr (U). 
\eeaa
Therefore,
\beaa
U^*_{ab}&=&\hat{U}^*_{ab}+\frac 1 2 \in_{ab}\, \tr( U)- \frac 12 \de_{ab}\atr (U), \\
\dual U_{ab}&=&\dual \hat{U}_{ab}+\frac 1 2 \in_{ab}\, \tr( U)- \frac 12 \de_{ab}\atr (U).
\eeaa
Hence,
\bea
U^*_{ab}&=&-\dual U_{ab}+\in_{ab}\, \tr( U)-  \de_{ab}\atr (U).
\eea

We note the following lemma,
\begin{lemma}
\label{le:duals}
Given two  1-forms $\xi, \eta$ we have,
\beaa
\dual\xi \c  \eta=-\xi\dual\c \eta=\xi\c\eta^*.
\eeaa
Given  a  1-form  $\xi$ and    2-tensor $U$ we have,
\beaa
\xi_a U^*_{ab}&=&\dual \xi_a  U_{ab}-\dual \xi_b( \tr U)-\xi_b (\atr U ),\\
\dual \xi_a U^*_{ab}&=&- \xi_a  U_{ab}+ \xi_b( \tr U)-\dual \xi_b (\atr U ).
\eeaa
Thus,
\beaa
\dual \xi_a U^*_{ab}+ \xi_a  U_{ab}&=& \xi_b( \tr U)-\dual \xi_b (\atr U ).
\eeaa
Also,
\beaa
\dual \xi_a U^*_{ab}- \xi_a  U_{ab}&=&-2\xi_a \hat{U}_{ab}.
\eeaa
\end{lemma}

\begin{proof} 
The last  formula  follows from,
\beaa
\dual \xi_a U^*_{ab}&=&\dual \xi_a\left(\hat{U}^*_{ab}+\frac 1 2 \in_{ab}\tr(U)-\frac 1 2 \de_{ab} \atr(U)\right)\\
&=&-\xi_a\hat{U}_{ab}+\frac 1 2 \xi_b\tr(U)-\dual \xi_b\atr(U),\\
 \xi_a  U_{ab}&=&\xi_a\left(\hat{U}_{ab}+\frac 1 2 \de_{ab} \tr(U)+\frac 1 2 \in_{ab} \atr(U)\right)\\
&=& \xi_a\hat{U}_{ab}+\frac 1 2 \xi_b \tr(U)-\xi_b \atr(U).
 \eeaa
\end{proof}

\begin{definition}\label{definition-SS-real}
We denote by $\SS_0=\SS_0(\MM)$ the set of pairs of scalar functions on $\MM$, $\SS_1=\SS_1(\MM)$ the  set of horizontal $1$-forms  on $\MM$, and by $\SS_2=\SS_2(\MM)$
  the set of symmetric traceless   horizontal $2$-forms on $\MM$.
  \end{definition}

\begin{definition} 
Given  $\xi, \eta\in\SS_1 $  we denote
\beaa
\xi\c \eta&:=&\de^{ab} \xi_a\eta_b,\\
\xi\wedge\eta&:=&\in^{ab} \xi_a\eta_b=\xi\c\dual \eta,\\
(\xi\hot \eta)_{ab}&:=&\frac 1 2 \big( \xi_a \eta_b +\xi_b \eta_a-\de_{ab} \xi\c \eta\big).
\eeaa
 Given   $\xi\in \SS_1 $,  $U\in \SS_2$ we denote
\beaa
(\xi\c U)_a&:=&\de^{bc} \xi_b U_{ac}.
\eeaa
Given     $U, V \in \SS_2$ we denote
\beaa
(U\wedge V)_{ab} &:=& \ep^{ab}U_{ac}V_{cb}.
\eeaa
\end{definition}

\begin{remark} 
Notice that the definition of $\xi\hot \eta$ differ by $\frac 1 2$ from the one given in \cite{Ch-Kl}
\end{remark}

The following two lemmas   are   immediate.
\begin{lemma}
Given a horizontal vector  $\xi$ and a symmetric 2-form $U$, we have
\beaa
U_{ab}\xi^b&=& \left(\hat{U}\c\xi+\frac{1}{2}\tr(U)\xi+\frac{1}{2}\atr(U)\dual\xi\right)_a.
\eeaa 
\end{lemma}

\begin{lemma}
\label{le:sym-product}
Given  two symmetric, traceless, horizontal 2-tensors $\hat U, \hat V$
we have,
\beaa
\hat U_{ac}\hat V_{cb}+\hat V_{ac}\hat U_{cb}=\de_{ab} \hat U\c\hat V
\eeaa
where,
 \beaa
 \hat U\c\hat V=\de^{ac}\de^{bd} \hat U_{ab} \hat V_{cd}.
 \eeaa
In particular,
\beaa
\hat V_{ac }\hat V_{cb}=\frac 1 2 \de_{ab}|\hat V|^2
\eeaa
with,
$
|\hat V|^2=\hat V\c\hat V.
$
\end{lemma}

\begin{remark}\lab{rmk:tracelesspartofproducttracelessis0}
The previous lemma implies in particular $\widehat{\hat{U}\c\hat{V}}=0$.
\end{remark}

We generalize  the lemma as follows,
\begin{lemma}
\lab{le:traces}
Given $U,V$  arbitrary 2-covariant horizontal tensor-fields, we  have,
\bea
\de^{ab}U_{ac}V_{cb}&=&\hat U\c\hat V+\frac 1 2 \big(\tr(U)\tr(V)-\atr(U)\atr(V)\big),\\
\in^{ab}U_{ac}V_{cb}&=&\hat U\wedge \hat V+\frac 12 \big(\atr(U)\tr (V)+\tr(U)\atr(V)\big),\\
\widehat{U_{ac} V_{cb}}&=&\frac 1 2 \big(\Uh_{ab}\tr(V)+\Vh_{ab}\tr(U)\big)
+\f12\big(-\dual\Uh_{ab}\atr(V)+\dual\Vh_{ab}\atr(U)\big),
\eea
where,
\beaa
\hat U\cdot\hat V&=&\de^{ac}\de^{bd}\hat U_{ab}\hat V_{cd},\\
\hat U\wedge \hat V&=&\hat U\cdot \dual \hat V=\in^{ab}\hat U_{ac}\hat V_{cb}.
\eeaa
\end{lemma}

\begin{proof} 
In view of the decomposition \eqref{decompU}, we have
\beaa
U_{ac} V_{cb}&=&\Uh_{ac} \Vh_{cb}+\frac 1 2 \big(\tr(V) \Uh_{ab}+\tr(U) \Vh{ab}\big)
+\frac 1 2\big(  \atr(U)\dual \Vh_{ab}+\atr(V) \Uh^*_{ab}\big)\\
&+&\frac 1 4 \big( \tr(U) \tr(V)-\atr(U)\atr(V)\big)\de_{ab}+\frac 1 4 \big( \tr(U) \atr(V)+\atr(U)\tr(V)\big)\in_{ab}
\eeaa
and the proof easily follows, using also the fact that $\widehat{\hat{U}\c\hat{V}}=0$ according to Remark \ref{rmk:tracelesspartofproducttracelessis0}.
\end{proof}

The following is an  immediate consequences of the lemma.
\begin{corollary}
\label{special.products}
In the particular case when $U=V$ the lemma becomes,
\beaa
\de^{ab}U_{ac}U_{cb}
&=&|\hat U|^2+\frac 1 2 \big((\tr(U))^2-(\atr(U))^2\big),\\
\in^{ab}U_{ac}U_{cb}&=&\tr(U)\atr(U),\\
\widehat{U_{ac}U_{cb}} &=& \tr(U)\Uh_{ab}.
\eeaa
\end{corollary}

As another corollary to Lemma \ref{le:traces} we have
\begin{lemma}
 \lab{le:nonsym-product}
Let   $u$ be an arbitrary $2$-horizontal tensor and $v\in \SS_2$. Then
\beaa
u_{ac} v_{cb}+ u_{bc} v_{ca}&=&\de_{ab} u\c v   +(\tr u )v_{ab}  +\frac 1 2\Big[\big( u_{ac}-u_{ca}\big) v_{cb}+ \big( u_{bc}-u_{cb}\big)v_{ca}\Big].
\eeaa
\end{lemma}

\begin{proof}
We give below a  direct proof based on Lemma \ref{le:sym-product}
 according to which,  given  $u, v\in\SS_2$,   we have
\beaa
u_{ac} v_{cb}+ u_{bc} v_{ca}=\de_{ab} u\c v. 
\eeaa
If $u$ is only symmetric and $v\in \SS_2$   we  can write,
\beaa
u_{ac} v_{cb}+ u_{bc} v_{ca}&=&\left( \hat{u}_{ac}+\frac 1 2 \de_{ac} tr (u)\right) v_{cb}+  \left( \hat{u}_{bc}+\frac 1 2 \de_{bc} tr (u)\right) v_{ca}\\
&=&\de_{ab}\hat{u} \c v   +(\tr u )v_{ab}.
\eeaa
If $u$ is an arbitrary $2$-tensor and $v\in \SS_2$,
\beaa
u_{ac} v_{cb}+ u_{bc} v_{ca}&=&\frac 1 2 \Big( u_{ac}+u_{ca} +\big( u_{ac}-u_{ca}\big)\Big) v_{ca}+\frac 1 2  \Big( u_{bc}+u_{cb} +\big( u_{bc}-u_{cb}\big)\Big) v_{ca}\\
&=&\frac  12 \de_{ab} \Big( u_{ac}+u_{ca}\Big) v_{ac} +\frac 1 2\Big(\big( u_{ac}-u_{ca}\big)\Big) v_{cb}+ \big( u_{bc}-u_{cb}\big)\Big) v_{ca}\Big)\\
&=&\de_{ab} u\c v   +(\tr u )v_{ab}  +\frac 1 2\Big(\big( u_{ac}-u_{ca}\big)\Big) v_{cb}+ \big( u_{bc}-u_{cb}\big)\Big) v_{ca}\Big).
\eeaa
\end{proof}

\begin{lemma}\label{lemma:usefulidentitiesforcomplexification}
The following hold true
\begin{itemize}
\item Given two  1-forms $\xi, \eta$ we have,
\beaa
\dual\xi \c  \eta &=& -\xi\c\dual\eta,\\
\dual\xi \c  \dual\eta &=& \xi\c\eta,\\
\dual\xi\wedge\eta &=& -\xi\wedge\dual\eta,\\
\dual\xi\wedge\dual\eta &=& \xi\wedge\eta,\\
\dual\xi \hot  \eta &=& \xi\hot\dual\eta,\\
\dual(\xi\hot\eta) &=& \dual\xi \hot  \eta,\\
\dual\xi \hot  \dual\eta &=& -\xi\hot\eta.
\eeaa

\item Given a horizontal vector  $\xi$ and a symmetric traceless 2-form $U$, we have
\beaa
\dual(\xi\c U) &=& \xi\c\dual U,\\
\dual\xi\c U &=& -\xi\c\dual U,\\
\dual\xi\c\dual U &=& \xi\c U.
\eeaa

\item Given 2 symmetric traceless 2-forms $U,V$, we  have,
\beaa
\dual U\c V &=& -U\c\dual V,\\
\dual U\c\dual V &=& U\c V,\\
\dual U\wedge V &=& -U\wedge\dual V,\\
\dual U\wedge\dual V &=& U\wedge V.
\eeaa
\end{itemize}
\end{lemma}

\begin{proof}
The statements follow from the above results except the ones involving $\hot$. To check those,  we write 
      \beaa
      \dual( \xi\hot \eta)_{11}&=&( \xi\hot \eta)_{21} =   \xi_2 \eta_1+\xi_1 \eta_2, \\
      ( \xi\hot \dual \eta)_{11}&=& \xi_1( \dual \eta)_1- \xi_2 (\dual \eta)_2 =\xi_1 \eta_2 + \xi_2 \eta_1, \\
      (\dual \xi \hot \eta)_{11} &=&(\dual \xi)_1 \eta_1-(\dual \xi)_2 \eta_2= \xi_2 \eta_1+ \xi_1 \eta_2.
      \eeaa
  Also,
      \beaa
       2 \dual(\xi\hot\eta)_{12}&=&\xi_2\eta_2-\xi_1 \eta_1, \\
       2 (\xi\hot \dual\eta)_{12}&=&\xi_1 \dual \eta _2+\xi_2\dual  \eta_1 = -\xi_1 \eta_1 + \xi_2 \eta_2,\\
       2 (\dual\xi \hot\eta)_{12} &=&\dual \xi_1 \eta_2+ \dual\xi_2 \eta_1  -\xi_1 \eta_1 + \xi_2 \eta_2.
      \eeaa
      Hence,
      \beaa
      \dual(\xi\hot\eta)=\dual \xi \hot \eta= \xi\hot \dual\eta.
      \eeaa  
\end{proof}

\begin{lemma}
\lab{dot-hot}
Given  $\xi , \eta \in \SS_1$, $  u\in \SS_2$  we have 
\beaa
\xi \hot ( \eta  \c u) + \eta  \hot ( \xi \c u)= (\xi\c \eta  ) u. 
\eeaa
\end{lemma}

\begin{proof}
We check first  the identities
\beaa
\Big(\xi \hot ( \eta  \c u) \Big)_{ab}&=& \frac 1 2 (\xi_a \eta_c u_{bc}+ \xi _b \eta _c u_{ac}- \delta_{ab} (\xi  \hot \eta ) \c u),\\
\Big(\eta  \hot ( \xi  \c u)  \Big)_{ab}&=& \frac 1 2 (\eta _a \xi_c u_{bc}+ \eta  _b \xi  _c u_{ac}- \delta_{ab} (\eta   \hot \xi) \c u),
\eeaa
which follows  easily from the definition of $\hot$ and 
\beaa
 (\xi  \hot \eta ) \c u&=& \xi \c ( \eta\c u)= \eta(\xi\c u). 
\eeaa
Note also that,
\beaa
(\xi\hot \eta)_{ac}  u_{cb}+ (\xi\hot \eta)_{bc}  u_{ca}&=& \frac 1 2 \big(\xi_a \eta_c+ \xi_c \eta _a 
- \delta_{ac} (\xi \c \eta\big)  u_{cb}+\frac 1 2\big (\xi_b \eta_c +\xi_c\eta_b -\delta_{bc} (\xi \c \eta )\big) u_{ca}\\
&=&\big(\xi \hot (\eta \c u)\big)_{ab}+\frac 1 2 \delta_{ab} (\xi \hot\eta) \c u+\big(\eta \hot (\xi \c u)\big)_{ab}+\frac 1 2\delta_{ab} (\xi \hot\eta) \c u \\
& -& (\xi \c\eta) u_{ab}\\
&=&\big(\xi \hot (\eta \c u)\big)_{ab}+\big(\eta  \hot (\xi \c u)\big)_{ab}-  (\xi \c\eta) u_{ab}+\delta_{ab} (\xi \hot\eta) \c u.  
\eeaa
We now apply  Lemma \ref{le:sym-product}   to  the  $\SS_2$ tensors $\xi\hot \eta$ and $u$ to deduce,
\beaa
(\xi\hot \eta)_{ac}  u_{cb}+ (\xi\hot \eta)_{bc}  u_{ca}= (\xi\hot \eta)\c  u \de_{ab}. 
\eeaa
Hence,
\beaa
  (\xi\hot \eta)\c  u \de_{ab} &=&  (\xi\hot \eta)_{ac}  u_{cb}+ (\xi\hot \eta)_{bc}  u_{ca}\\
  &=&\big(\xi \hot (\eta \c u)\big)_{ab}+\big(\eta  \hot (\xi \c u)\big)_{ab}-  (\xi \c\eta) u_{ab}+\delta_{ab} (\xi \hot\eta) \c u  
\eeaa
i.e.,
\beaa
\big(\xi \hot (\eta \c u)\big)_{ab}+\big(\eta  \hot (\xi \c u)\big)_{ab}-  (\xi \c\eta) u_{ab}=0
\eeaa
as desired.
\end{proof}

  
  \subsection{Horizontal covariant derivative}
  
  
 We are now ready to define the horizontal covariant operator $\nab$ as follows:
 Given $X, Y\in\O(\MM)$ the covariant derivative $\D_XY$ fails in general  to be horizontal.
 We thus define,
 \bea
 \nab_X Y&:=&^{(h)}(\D_XY)=\D_XY- \frac 1 2 \chib(X,Y)L -  \frac 1 2 \chi(X,Y) \Lb.
 \eea
 
 \begin{proposition}
 For all  $X,Y\in \O(\MM)$,
  \beaa
  \nab_X Y-\nab_Y X
  &=&[X,Y]-\chiba(X,Y)L -\chia(X,Y)\Lb\\
  &=&[X, Y]-\frac 1 2 \left(\atrchb\,  L+\atrch \, \Lb\right)\in(X, Y).
 \eeaa
 In particular,
 \bea\,[X, Y]^\perp&=&\frac 1 2\left(\atrchb\,  L+\atrch \, \Lb\right)\in(X, Y).
 \eea
 
 For  all  $X,Y, Z\in \O(\MM)$,
 \beaa
 Z\ga(X,Y)=\ga(\nab_Z X, Y)+\ga(X, \nab_ZY).
 \eeaa
\end{proposition}

\begin{remark} 
In the integrable case, $\nab$ coincides with the Levi-Civita connection
 of the metric induced on the integral surfaces of   $\O(\MM)$. 
\end{remark} 
 
 Given a general covariant, horizontal tensor-field  $U$
 we define its horizontal covariant derivative according to
 the formula,
 \bea
 \nab_Z U(X_1,\ldots X_k)=Z (U(X_1,\ldots X_k))&-&U(\nab_ZX_1,\ldots X_k)-\nn
 \\
\ldots   &-& U(X_1,\ldots \nab_ZX_k).
 \eea
 
 Given $X$ horizontal, $\D_lX$ and $\D_\Lb X$ are in general not horizontal.
 We define $\nab_L X$ and $\nab_\Lb X$  to be the horizontal projections
 of the former.  More precisely,
 \beaa
 \nab_L X&:=&^{(h)}(\D_L X)=\D_L X- g(X, \D_L\Lb) L- g(X, \D_L L) \Lb,\\
 \nab_\Lb X&:=&^{(h)}(\D_\Lb X)=\D_\Lb X- g(X, \D_\Lb\Lb) L - g(X, \D_\Lb L) \Lb. 
 \eeaa


 \subsection{Ricci coefficients}
 
 
 \begin{definition} 
 We define the horizontal  $1$-forms,
 \bea
 \etab(X)&:=& \frac 1 2 g(X, \D_L\Lb),\qquad \eta(X):= \frac 1 2  g(X, \D_\Lb L),\\
 \xib(X)&:=& \frac 1 2  g(X, \D_\Lb \Lb),\qquad \xi(X):= \frac 1 2 g(X, \D_L L).
 \eea
 With these definitions we have,
 \beaa
 \nab_L X&:=&^{(h)}(\D_L X)=\D_L X-\etab(X)L -\xi(X) \Lb,\\
 \nab_\Lb X&:=&^{(h)}(\D_\Lb X)=\D_\Lb X-\xib(X)L -\eta(X) \Lb. 
 \eeaa
 \end{definition}
 
  We can extend the operators $\nab_L$ and $\nab_\Lb$ to
 arbitrary  $k$-covariant, horizontal tensor-fields  $U$  as follows,
 \beaa
 \nab_LU(X_1,\ldots, X_k)=L(U(X_1,\ldots, X_k))&-&U(\nab_L X_1,\ldots, X_k)-\\
 \ldots&-&U( X_1,\ldots, \nab_LX_k),\\
  \nab_\Lb U(X_1,\ldots, X_k)=\Lb(U(X_1,\ldots, X_k))&-&U(\nab_\Lb X_1,\ldots, X_k)-\\
  \ldots&-&U( X_1,\ldots, \nab_\Lb X_k).
 \eeaa 
 The following  proposition follows easily from the definition.
 \begin{proposition}
 The operators $\nab$, $\nab_L$ and $\nab_\Lb$ take horizontal tensor-fields into
 horizontal tensor-fields. We have,
 \bea
 \nab \ga=\nab_L\ga=\nab_\Lb\ga=0.
 \eea
 \end{proposition}
 
 In addition to the horizontal tensor-fields $\chi,\chib,\etab, \eta, \xi,\xib$ introduced above 
 we also define the scalars,
 \bea
 \omb&:=& \frac 1  4 g(\D_\Lb\Lb, L),\qquad\quad  \om:=\frac 1 4 g(\D_L L, \Lb),
 \eea
 and the horizontal 1-form,
 \bea
 \ze(X)&=&\frac 1 2 g(\D_XL,\Lb).
 \eea
 We summarize below the definition of the the horizontal $1$-forms $\xi, \xib, \eta, \etab, \ze\in\mathbf{O}_1$:
\begin{equation}\label{fo2}
\begin{cases}
&\xi(X)=\frac 1 2 \g(\D_LL,X),\quad \xib(X) =\frac 1 2 \g(\D_{\Lb}\Lb,X),\\
&\eta(X)=\frac 1 2 \g(\D_{\Lb}L,X),\quad \etab(X)=\frac 1 2 \g(\D_L\Lb,X),\\
&\ze(X)=\frac 1 2 \g(\D_X L,\Lb),
\end{cases}
\end{equation}
and the real scalars
\begin{equation}\label{fo3}
\om=\frac 1 4 \g(\D_LL,\Lb),\qquad\omb=\frac 1 4 \g(\D_{\Lb}\Lb,L).
\end{equation}
 
 \begin{definition}
 The horizontal tensor-fields $\chi,\chib, \eta, \etab, \ze,  \xi,\xib,\om, \omb$ are  called
 the connection coefficients of the null pair $(L,\Lb)$. Given  an  arbitrary  basis of 
  horizontal vectorfields $e_1,  e_2$, we write using the short hand notation $\D_a=\D_{e_a}, a=1,2$, 
  \beaa
\chib_{ab}&=&g(\D_a\Lb, e_b),\qquad \chi_{ab}=g(\D_aL, e_b),\\
\xib_a&=&\frac 1 2 \g(\D_\Lb\Lb, e_a),\qquad \xi_a=\frac 1 2 \ g(\D_L L, e_a),\\
\omb&=&\frac 1 4 g(\D_\Lb\Lb, L),\qquad\quad  \om=\frac 1 4 g(\D_L L, \Lb),\qquad \\
\etab_a&=&\frac 1 2 (\D_L\Lb, e_a),\qquad \quad \eta_a=\frac 1 2 g(\D_\Lb L, e_a),\qquad\\
 \ze_a&=&\frac 1 2 g(\D_{e_a}L,\Lb).
\eeaa
 \end{definition}
 
We easily  derive the  Ricci formulae,
\bea
\D_a e_b&=&\nab_a e_b+\frac 1 2 \chi_{ab} e_3+\frac 1 2  \chib_{ab}e_4,\nn\\
\D_a e_4&=&\chi_{ab}e_b -\ze_a e_4,\nn\\
\D_a e_3&=&\chib_{ab} e_b +\ze_ae_3,\nn\\
\D_3 e_a&=&\nab_3 e_a +\eta_a e_3+\xib_a e_4,\nn\\
\D_3 e_3&=& -2\omb e_3+ 2 \xib_b e_b,\label{ricci}\\
\D_3 e_4&=&2\omb e_4+2\eta_b e_b,\nn\\
\D_4 e_a&=&\nab_4 e_a +\etab_a e_4 +\xi_a e_3,\nn\\
\D_4 e_4&=&-2 \om e_4 +2\xi_b e_b,\nn\\
\D_4 e_3&=&2 \om e_3+2\etab_b e_b.\nn
\eea 

\begin{definition}
We  introduce the  notation
\bea
\trch:=\tr(\chi), \quad \atr\chi:=\atr(\chi),\quad \trchb:=\tr(\chib), \quad \atr\chib:=\atr(\chib).
\eea
$\chih, \trch$ and $\chibh, \trchb$  are called, respectively,  the shear and expansion 
of the horizontal distribution $\O(\MM)$. The scalars $ \atr\chi$ and $\atr\chib$ measure the
integrability defects of the distribution.
\end{definition}

\begin{definition} 
For a given horizontal   $1$ -form $\om$,
we  define the frame dependent   operators,   
\beaa
\div\om&=&\de^{ab}\nab_b\om_a,\qquad 
\curl\om=\in^{ab}\nab_a\om_b,\\
(\nab\hot \om)_{ba}&=&\frac 1 2 \big(\nab_b\om_a+\nab_a  \om_b-\de_{ab}( \div \om)\big).
\eeaa
\end{definition}


\subsection{Curvature and Weyl fields} 


Assume that $W\in\mathbf{T}_4^0(\MM)$ is a Weyl field, i.e.
\begin{equation}\label{fo4}
\begin{cases}
&W_{\al\be\mu\nu}=-W_{\be\al\mu\nu}=-W_{\al\be\nu\mu}=W_{\mu\nu\al\be},\\
&W_{\al\be\mu\nu}+W_{\al\mu\nu\be}+W_{\al\nu\be\mu}=0,\\
&\g^{\be\nu}W_{\al\be\mu\nu}=0.
\end{cases}
\end{equation}
We define the null components of the Weyl field $W$, $\al(W),\aa(W),\varrho(W)\in \mathbf{O}_2(\MM)$ and $\be(W),\bb(W)\in\mathbf{O}_1(\MM)$ by the formulas
\begin{equation}\label{fo5}
\begin{cases}
\al(W)(X,Y)=W(L,X,L,Y),\\
\aa(W)(X,Y)=W(\Lb,X,\Lb,Y),\\
\b(W)(X)=\frac 1 2 W(X,L,\Lb,L),\\
\bb(W)(X)=\frac 1 2 W(X,\Lb,\Lb, L),\\
\varrho(W)(X,Y)= W(X,\Lb,Y,L).
\end{cases}
\end{equation}
Recall that if $W$ is a Weyl field its 
Hodge dual $\dual W$, defined by 
${}^{\ast}W_{\al\be\mu\nu}=\frac{1}{2}{\in_{\mu\nu}}^{\rho\si}W_{\al\be\rho\si}$,  is also a Weyl field. We easily
check the formulas,
\begin{equation}
\label{eq:dualW}
\begin{cases}
&\aa(\dual W)=\dual \aa(W),\qquad \a(\dual W)=-
\dual \a(W), \\
&\bb(\dual W)=\dual\bb(W),\qquad \b(\dual W)=-\dual \b(W),\\
&\varrho(\dual W)=\dual \varrho(W). 
\end{cases}
\end{equation}
It is easy to check that $\a,\aa$ are symmetric traceless  horizontal tensor-fields.
 On the other hand    the horizontal 2-tensorfield  $\varrho$ is   neither symmetric nor 
 traceless.  It is convenient to express it in terms
 of  the following  two scalar quantities,
 \bea
 \label{fo5'}
 \rho(W)=\frac 1 4 W(L,\Lb,L,\Lb),\qquad \dual\rho(W)=
\frac 1 4  \dual W(L,\Lb,L,\Lb)\label{rho-dualrho}.
 \eea
 Observe also that,
 \beaa
\rho(\dual W)=\rhod(W), \qquad \rhod(\dual W)=-\rho.
\eeaa  
Thus,
\bea
\varrho(X,Y)=\big(-\rho\,\ga(X,Y)+\rhod\, \in(X,Y)\big),\qquad 
\forall\, X,Y\in \O(\MM).
\eea
We have
 \beaa
 W_{a3b4}&=&\varrho_{ab}=(-\rho\de_{ab} +\dual \rho \in_{ab}),\\
 W_{ab34}&=& 2 \in_{ab}\dual\rho,\\
 W_{abcd}&=&-\in_{ab}\in_{cd}\rho,\\
 W_{abc3}&=&\in_{ab}\dual \bb_c,\\
 W_{abc4}&=&-\in_{ab}\dual \b_c.
 \eeaa

\begin{remark}
\label{rem:pairing}
In addition to the Hodge duality we  will need to take into account the
 the duality with respect to the interchange of $L, \Lb$, which we call a pairing transformation.  Clearly, under this transformation,   $\a\leftrightarrow \aa$, 
 $\b \leftrightarrow -\bb$,   $\rho   \leftrightarrow \rho$,  $\dual \rho   \leftrightarrow -\dual \rho$,
  $\varrho \leftrightarrow \check{\varrho}$ with 
  $\check{\varrho}_{ab}:=\varrho_{ba}$.    One has to be careful however when combining 
  the Hodge dual  and pairing  transformations. In that case we have,  
   $\dual \aa    \leftrightarrow  -\dual \a$,   $\dual\bb \leftrightarrow \dual \b$. 
   This is due to the  fact that under the pairing transformation 
$\in_{ab}\to-\in_{ab}$ (since $\in_{ab}=\in_{ab34}$).  Indeed, for example,
   \beaa
   \dual\aa_{ab}&=&  \aa(\dual W)_{ab}    = \dual W_{a3b3} =  - \in_{a3c4} W_{c3 b3}=\in_{ac34}W_{c3 b3}=\in_{ac}\aa_{cb},\\
      \dual\a_{ab}&=&  \a(\dual W)_{ab}    = \dual W_{a4b4} =  - \in_{a4c3} W_{c4 b4}=-\in_{cb34}W_{c4 b4}=
   -   \in_{ac}\a_{cb}.
   \eeaa
   \end{remark}

The decomposition above applies in particular to the Riemann curvature tensor  $\R$ of a vacuum spacetime.

\begin{definition}[Horizontal curvature tensor]
 We define the curvature tensor $R_{cdab}$ of the horizontal structure  by the usual  formula,
 \beaa
\nab_a \nab_b X_c -\nab_b \nab_a X_c&=&R_{cd ab} X^d. 
 \eeaa
\end{definition} 
 


\subsection{Connection to  the Newman-Penrose formalism} 
\lab{section:NPformalism}


In the Newman-Penrose NP  formalism,  one choses a   specific orthonormal  basis of  horizontal  vectors $(e_1, e_2)$ and defines  all connection coefficients relative   to the  complexified  frame $(e_3, e_4, m, \ov{m})$ where $m=e_1+i e_2$, $\ov{m}= e_1-i e_2.$ Thus, all quantities of interest are complex scalars instead of our horizontal  tensors such as  $\SS_1, \SS_2$.   The NP  formalism works well  for deriving  the basic equations, but has the disadvantage 
of  substantially  increasing  the number of  variables.    Moreover, the calculations become far  more    cumbersome when deriving equations involving higher   derivatives   of the main quantities, in perturbations of Kerr.
 Another  advantage of  the formalism used here is that all important equations look similar to the ones  in \cite{Ch-Kl}.  
 
   We refer  to \cite{NP}  for the original form of the NP  formalism and to  the appendix in \cite{I-Kl} for a  more  detailed comparison between the NP and our  formalism.


\subsection{Null  structure equations}


We state below   the  null structure equation in the general setting  discussed above.   We assume given a  vacuum spacetime endowed with  a general null frame $(e_3, e_4, e_1, e_2)$ relative to which we define our  connection and curvature coefficients.
 
\begin{proposition}[Null structure equations] 
The  connection coefficients  verify the following   equations
\label{prop-nullstr}
\beaa
\nab_3\trchb&=&-|\chibh|^2-\frac 1 2 \big( \trchb^2-\atrchb^2\big)+2\div\xib  - 2\omb \trchb +  2 \xib\c(\eta+\etab-2\ze),\\
\nab_3\atrchb&=&-\trchb\atrchb +2\curl \xib -2\omb\atrchb+ 2 \xib\wedge(-\eta+\etab+2\ze),\\
\nab_3\chibh&=&-\trchb\,  \chibh+ 2 \nab\hot \xib- 2 \omb \chibh+  2  \xib\hot(\eta+\etab-2\ze)-\aa,
\eeaa
\beaa
\nab_3\trch
&=& -\chibh\c\chih -\frac 1 2 \trchb\trch+\frac 1 2 \atrchb\atrch    +   2   \div \eta+ 2 \omb \trch + 2 \big(\xi\c \xib +|\eta|^2\big)+ 2\rho,\\
\nab_3\atrch
&=&-\chibh\wedge\chih-\frac 1 2(\atrchb \trch+\trchb\atrch)+ 2 \curl \eta + 2 \omb \atrch + 2 \xib\wedge\xi  -  2 \dual \rho,\\
\nab_3\chih
&=&-\frac 1 2 \big( \trch \chibh+\trchb \chih\big)-\frac 1 2 \big(-\dual \chibh \, \atrch+\dual \chih\,\atrchb\big)
+2 \nab\hot \eta +2 \omb \chih\\
&+&2 \xib\hot\xi +2\eta\hot\eta,
\eeaa
\beaa
\nab_4\trchb
&=& -\chih\c\chibh -\frac 1 2 \trch\trchb+\frac 1 2 \atrch\atrchb    +  2   \div \etab+ 2 \om \trchb + 2\big( \xi\c \xib +|\etab|^2\big)+2\rho,\\
\nab_4\atrchb
&=&-\chih\wedge\chibh-\frac 1 2(\atrch \trchb+\trch\atrchb)+ 2 \curl \etab + 2 \om \atrchb + 2 \xi\wedge\xib+2 \dual \rho,\\
\nab_4\chibh
&=&-\frac 1 2 \big( \trchb \chih+\trch \chibh\big)-\frac 1 2 \big(-\dual \chih \, \atrchb+\dual \chibh\,\atrch\big)
+\nab\hot \etab +2 \om \chibh\\
&+&2 \xi\hot\xib +2 \etab\hot\etab,
\eeaa
\beaa
\nab_4\trch&=&-|\chih|^2-\frac 1 2 \big( \trch^2-\atrch^2\big)+ 2 \div\xi  - 2 \om \trch + 2   \xi\c(\etab+\eta+2\ze),\\
\nab_4\atrch&=&-\trch\atrch + 2 \curl \xi - 2 \om\atrch+ 2 \xi\wedge(-\etab+\eta-2\ze),\\
\nab_4\chih&=&-\trch\,  \chih+2 \nab\hot \xi- 2 \om \chih+  2  \xi\hot(\etab+\eta+2\ze)-\a.
\eeaa
Also,
\beaa
\nab_3 \ze+2\nab\omb&=& -\chibh\c(\ze+\eta)-\frac{1}{2}\trchb(\ze+\eta)-\frac{1}{2}\atrchb(\dual\ze+\dual\eta)+ 2 \omb(\ze-\eta)\\
&&+\hch\c\xib+\frac{1}{2}\trch\,\xib+\frac{1}{2}\atrch\dual\xib +2 \om \xib -\bb,
\\
\nab_4 \ze -2\nab\om&=& \chih\c(-\ze+\etab)+\frac{1}{2}\trch(-\ze+\etab)+\frac{1}{2}\atrch(-\dual\ze+\dual\etab)+2 \om(\ze+\etab)\\
&& -\chibh\c\xi -\frac{1}{2}\trchb\,\xi-\frac{1}{2}\atrchb\dual\xi -2 \omb \xi -\b,
\\
\nab_3 \etab -\nab_4\xib &=& -\chibh\c(\etab-\eta) -\frac{1}{2}\trchb(\etab-\eta)+\frac{1}{2}\atrchb(\dual\etab-\dual\eta) -4 \om \xib  +\bb, \\
\nab_4 \eta    -    \nab_3\xi &=& -\chih\c(\eta-\etab) -\frac{1}{2}\trch(\eta-\etab)+\frac{1}{2}\atrch(\dual\eta-\dual\etab)-4\omb \xi -\b,\\
\eeaa
and
\beaa
\nab_3\om+\nab_4\omb -4\om\omb -\xi\c \xib -(\eta-\etab)\c\ze +\eta\c\etab&=&   \rho.
\eeaa
Also,
\beaa
\div\chih +\ze\c\chih &=& \frac{1}{2}\nab\trch+\frac{1}{2}\trch\ze -\frac{1}{2}\dual\nab\atrch-\frac{1}{2}\atrch\dual\ze -\atrch\dual\eta-\atrchb\dual\xi -\b,\\
\div\chibh -\ze\c\chibh &=& \frac{1}{2}\nab\trchb-\frac{1}{2}\trchb\ze -\frac{1}{2}\dual\nab\atrchb+\frac{1}{2}\atrchb\dual\ze -\atrchb\dual\etab-\atrch\dual\xib +\bb,
\eeaa
and\footnote{Note that this equation follows from  expanding  $\R_{34ab}$.}
\beaa
\curl\ze&=&-\frac 1 2 \chih\wedge\chibh   +\frac 1 4 \big(  \trch\atrchb-\trchb\atrch   \big)+\om \atrchb -\omb\atrch+\dual \rho.
\eeaa
\end{proposition}

\begin{proof}
 Except for  the fact that the order of indices  in $\chi, \chib$ is important, since they are no longer symmetric,    the  derivation is   exactly as in \cite{Ch-Kl}.  
 \end{proof}
 
 Note that  we  are missing the traditional Gauss equation which, in the integrable case,  connects the  Gauss curvature   of a  sphere   to  a Riemann curvature component. In what follows we state a result which is its non-integrable analogue. 
 
 \begin{proposition}
 \lab{prop-nullstr:0}
 The following identity holds true.
 \bea
\lab{Gauss-eq-horizontal}
\bsplit
 \nab_a (\nab_b X_c)- \nab_b (\nab_a X_c)&= \R_{c d   ab}X^d+\frac 1 2 \in_{ab}\Big(\atrch\nab_3+\atrchb \nab_4\Big) X_c
\\
 &-\frac 1 2 \Big(\chi_{ac}\chib_{bd} + \chib_{ac}\chi_{bd} - \chi_{bc}\chib_{ad}- \chib_{bc}\chi_{ad}\Big) X^d .
 \end{split}
\eea
 \end{proposition}  
 
 \begin{proof}
 Given  $X\in \SS_1$ we have,
\beaa
\D_b X_c&=&\nab_b X_c, \qquad 
\D_3X_c=\nab_3 X_c,\\
\D_4 X_c&=&\nab_4 X_c, \qquad 
\D_b X_3=-\chib_{bd} X_d, \qquad 
\D_b X_4=-\chi_{bd} X_d.
\eeaa
Also,
\beaa
\D_a \D_b X_c&=&\nab_a (\nab_b X_c)- \frac 1 2 \chi_{ab} \D_3 X_c -\frac 1 2 \chib_{ab}\D_4 X_c -\frac 1 2\chi_{ac} \D_b X_3-\frac 1 2 \chib_{ac} \D_b X_4\\
&=&\nab_a (\nab_b X_c)- \frac 1 2 \chi_{ab} \nab_3 X_c -\frac 1 2 \chib_{ab}\nab_4 X_c+\frac 1 2 \chi_{ac}\chib_{bd} X_d+\frac 1 2 \chib_{ac}\chi_{bd} X_d.
\eeaa
Hence,
\beaa
\D_a \D_b X_c&=&\nab_a (\nab_b X_c)- \frac 1 2 \chi_{ab} \nab_3 X_c -\frac 1 2 \chib_{ab}\nab_4 X_c+\frac 1 2 \chi_{ac}\chib_{bd} X_d+\frac 1 2 \chib_{ac}\chi_{bd} X_d,\\
\D_b\D_a X_c&=&\nab_b (\nab_a X_c)- \frac 1 2 \chi_{ba} \nab_3 X_c -\frac 1 2 \chib_{ba}\nab_4 X_c+\frac 1 2 \chi_{bc}\chib_{ad} X_d+\frac 1 2 \chib_{bc}\chi_{ad} X_d.
\eeaa
Subtracting we derive
\beaa
 \R_{c d   ab}X^d&=&\D_a \D_b X_c-\D_b\D_a X_c\\
 &=& \nab_a (\nab_b X_c)- \nab_b (\nab_a X_c)-\frac 1 2 (\chi_{ab}-\chi_{ba})\nab_3 X_c-\frac 1 2 (\chib_{ab}-\chib_{ba})\nab_4 X_c\\
&+&\frac 1 2 \Big(\chi_{ac}\chib_{bd} X_d+ \chib_{ac}\chi_{bd} - \chi_{bc}\chib_{ad}- \chib_{bc}\chi_{ad}\Big) X^d. 
\eeaa
Thus,
\beaa
 \nab_a (\nab_b X_c)- \nab_b (\nab_a X_c)&=&\frac 1 2 (\chi_{ab}-\chi_{ba})\nab_3 X_c
 +\frac 1 2 (\chib_{ab}-\chib_{ba})\nab_4 X_c\\
 &-&\frac 1 2 \Big(\chi_{ac}\chib_{bd} + \chib_{ac}\chi_{bd} - \chi_{bc}\chib_{ad}- \chib_{bc}\chi_{ad}\Big) X^d + \R_{c d   ab}X^d.
\eeaa
To end the proof we set
\beaa
\chi_{ab}-\chi_{ba}=\in_{ab} \atrch, \qquad \chib_{ab}-\chib_{ba}=\in_{ab} \atrchb.
\eeaa
\end{proof}

\begin{remark}
Recall that 
\beaa
\nab_a (\nab_b X_c)- \nab_b (\nab_a X_c)=R_{cdab} X^d. 
\eeaa
Thus, taking $X_d=e_d$, we can rewrite our Gauss formula
in the form
\beaa
R_{c d   ab}&=&\frac 1 2 \atrch\in_{ab} g(\nab_3  e_d, e_c)  +\frac 1 2 \atrchb\in_{ab} g(\nab_4 e_d, e_c)\\
&-&\frac 1 2 \Big(\chi_{ac}\chib_{bd} + \chib_{ac}\chi_{bd} - \chi_{bc}\chib_{ad}- \chib_{bc}\chi_{ad}\Big) + \R_{c d   ab}.
\eeaa
\end{remark}


\subsection{Null Bianchi identities}


We state  below the  equations verified by the null curvature components of an Einstein vacuum manifold.
 
    \begin{proposition}[Null Bianchi]\label{prop:bianchi} 
        We have,
    \beaa
    \nab_3\a-2  \nab\hot \b&=&-\frac 1 2 \big(\trchb\a+\atrchb\dual \a)+4\omb \a+
 2 (\ze+4\eta)\hot \b - 3 (\rho\chih +\rhod\dual\chih),\\
\nab_4\beta - \div\a &=&-2(\trch\beta-\atrch \dual \b) - 2  \om\b +\a\c  (2 \ze +\etab) + 3  (\xi\rho+\dual \xi\rhod),\\
     \nab_3\b+\div\varrho&=&-(\trchb \b+\atrchb \dual \b)+2 \omb\,\b+2\bb\c \chih+3 (\rho\eta+\rhod\dual \eta)+    \a\c\xib,\\
 \nab_4 \rho-\div \b&=&-\frac 3 2 (\trch \rho+\atrch \rhod)+(2\etab+\ze)\c\b-2\xi\c\bb-\frac 1 2 \chibh \c\a,\\
   \nab_4 \rhod+\curl\b&=&-\frac 3 2 (\trch \rhod-\atrch \rho)-(2\etab+\ze)\c\dual \b-2\xi\c\dual \bb+\frac 1 2 \chibh \c\dual \a, \\
     \nab_3 \rho+\div\bb&=&-\frac 3 2 (\trchb \rho -\atrchb \rhod) -(2\eta-\ze) \c\bb+2\xib\c\b-\frac{1}{2}\chih\c\aa,
 \\
   \nab_3 \rhod+\curl\bb&=&-\frac 3 2 (\trchb \rhod+\atrchb \rho)- (2\eta-\ze) \c\dual \bb-2\xib\c\dual\b-\frac 1 2 \chih\c\dual \aa,\\
     \nab_4\bb-\div\varoc&=&-(\trch \bb+ \atrch \dual \bb)+ 2\om\,\bb+2\b\c \chibh
    -3 (\rho\etab-\rhod\dual \etab)-    \aa\c\xi,\\
     \nab_3\bb +\div\aa &=&-2(\trchb\,\bb-\atrchb \dual \bb)- 2  \omb\bb-\aa\c(-2\ze+\eta) - 3  (\xib\rho-\dual \xib \rhod),\\
     \nab_4\aa+2 \nab\hot \bb&=&-\frac 1 2 \big(\trch\aa-\atrch\dual \aa)+4\om \aa+2
 (\ze-4\etab)\hot \bb - 3  (\rho\chibh -\rhod\dual\chibh).
\eeaa
Here,
\beaa
\div\varo&=&- (\nab\rho+\dual\nab\rhod),\\
\div\varoc&=&- (\nab\rho-\dual\nab\rhod).
\eeaa
    \end{proposition} 
    
    \begin{proof}
    The proof  follows line by line from the derivation  in \cite{Ch-Kl} except, once more, for keeping track of the lack of symmetry for $\chi, \chib$. Note also  that $\varoc_{ab}=\varo_{ba}$ and  that $ (\div\varo)_b=\nab^a\varo_{ab}$.
    \end{proof}

We also recall, see \cite{Ch-Kl}, the  basic Hodge operators
\begin{itemize}
\item  $\DDd_1 $ takes $\SS_1$ into\footnote{Recall that $\SS_0$ refers to pairs of scalar functions $(a,b)$.}  $\SS_0$
\beaa
\DDd_1 \xi =(\div\xi, \curl \xi),
\eeaa
\item  $\DDd_1 $ takes $\SS_2$ into $\SS_1$
\beaa
(\DDd_2 \xi)_a =\nab^b \xi_{ab},
\eeaa
\item $\DDs_1$ takes  $\SS_0 $ into  $\SS_1$
\beaa
\DDs_1( f, f_*) &=& -\nab_a f+\in_{ab}  \nab_b  f_*,
\eeaa
\item 
$\DDs_2 $ takes  $\SS_1$ into $\SS_2$ 
\beaa
\DDs_2 \xi&=& -\nab\hot \xi.
\eeaa
\end{itemize}

     
\subsection{Commutation formulas}


  \begin{lemma}
   \lab{lemma:comm-gen}
Let $U_{A}= U_{a_1\ldots a_k} $ be a general $k$-horizontal  tensorfield.
\begin{enumerate}
\item  We have
\bea
\bsplit
[\nab_3, \nab_b] U_A&= -\chib_{bc}\nab_c U_{A} +(\eta_b-\ze_b)  \nab_3 U_A  +\sum_{i=1}^k \big( \chib_{a_i  b}  \eta_c-\chib_{bc} \eta_{a_i} \big) U_{a_1\ldots }\,^ c \,_{\ldots a_k} \\
&+\err_{3bA}[U],\\
\err_{3bA}[U]&=\sum_{i=1}^k \Big( \chi_{a_i  c}  \xib_c-\chi_{bc}\, \xib_{a_i}-\in_{a_i c} \dual\bb_b  \Big) U_{a_1\ldots}\,^ c\,_{\ldots a_k} +\xib_b \nab_4 U_{A }.
\end{split} 
\eea

\item We have
\bea
\bsplit
[\nab_4, \nab_b] U_A&= -\chi_{bc}\nab_c U_{A}  +(\etab_b+\ze_b)  \nab_4 U_A+\sum_{i=1}^k \big( \chib_{a_i  b}  \eta_c-\chib_{bc} \eta_{a_i} \big) U_{a_1\ldots c\ldots a_k} \\
&+\err_{4bA}[U],\\
\err_{4bA}[U]&=\sum_{i=1}^k \Big( \chib_{a_i  c}  \xi_c-\chib_{bc}\, \xi_{a_i}+\in_{a_i c} \dual\b_b  \Big) U_{a_1\ldots}\,^ c\,_{\ldots a_k} +\xi_b \nab_3  U_{A }.
\end{split} 
\eea

\item We have,
\bea
\bsplit
\, [\nab_4, \nab_3] U_A&= 2(\etab_b-\eta_b ) \nab_b U_A + 2\sum_{i=1}^k\big( \eta_{a_i} \etab_b-\etab_{a_i} \eta_b- \in_{a_i b}\dual \rho) U_{a_1\ldots}\,^b\,_{\ldots a_k} \\
&+ 2 \om \nab_3 U_A -2\omb \nab_4 U_A +\err_{43A},\\
\err_{43A}&= 2\sum_{i=1}^k \big( \xib_{a_i}  \xi_b- \xi_{a_i}  \xib_b )U_{a_1\ldots} \,^b\,_{\ldots a_k}.
\end{split}
\eea
\end{enumerate}
\end{lemma}

\begin{proof}
It suffices to consider the case $k=1$. We write,
\beaa
\D^2_{3b} U_a&=& \nab_3\nab_bU_a +\big(\eta_a \chib_{bc}+\xib_a \chi_{bc} \big)U_c -\eta_b \nab_3 U_a -\xib_b \nab_4 U_a,\\
\D^2_{b3} U_a&=& \nab_b \nab_3 U_a +\big(\chi_{ab} \xib_c +\chib_{ab}\eta_c\big) U_c  -\chib_{bc} \nab_c U_a -\ze_b \nab_3 U_a,\\
\D^2_{3b} U_a-\D^2_{b3} U_a&=& \R_{ac3b} U_c= -\in_{ac}\dual \bb_b U_c.
\eeaa
Hence,
\beaa
\,[\nab_3, \nab_b] U_a &=&  \nab_3\nab_bU_a- \nab_b \nab_3 U_a\\
&=&-\big(\eta_a \chib_{bc}+\xib_a \chi_{bc} \big)U_c +\eta_b \nab_3 U_a +\xib \nab_4 U_a +\D^2_{3b} U_a\\
&+&\big(\chi_{ab} \xib_c +\chib_{ab}\eta_c\big) U_c  -\chib_{bc} \nab_c U_a -\ze_b \nab_3 U_a-\D^2_{b3} U_a\\
&=& -\chib_{bc} \nab_c U_a+\big(\chib_{ab}  \eta_c - \chib_{bc}  \eta_a\big) U_c+(\eta_b-\ze_b)\nab_3 U_a \\
&+&\big(\chi_{ab} \xib_c- \chi_{bc} \xib_a  \big) +\xib \nab_4 U_a  -\in_{ac}\dual \bb_b U_c
\eeaa
i.e.,
\beaa
\,[\nab_3, \nab_b] U_a &=&  -\chib_{bc} \nab_c U_a+\big(\chib_{ab}  \eta_c - \chib_{bc}  \eta_a\big) U_c+(\eta_b-\ze_b)\nab_3 U_a \\
&+&\big(\chi_{ab} \xib_c- \chi_{bc} \xib_a  \big) +\xib \nab_4 U_a  -\in_{ac}\dual \bb_b U_c
\eeaa
as stated. 
The commutator formula  for   $\,[\nab_4, \nab_b] U_a $ is derived easily by symmetry.
Also,
\beaa
\D^2_{43} U_a &=& \nab_4 \nab_3 U_a -2\om \nab_3 U_a-2\etab^b \nab_b U_a +2\etab_a \eta^b U_b  + 2 \xi_a \xib^c U_c,\\
\D^2_{34} U_a&=&\nab_3\nab_4 U_a   -2 \omb \nab_4 U_a-  2\eta^b \nab_b U_a + 2 \eta_a \etab_b U_b + 2 \xib_a \xi^c U_c,\\
\D^2_{43}U_a-\D^2_{34} U_a &=&\R_{ab43}  U^b=-2\rhod \in_{ab} U^b.
\eeaa
We deduce,
\beaa
\,[\nab_4, \nab_3]U_a &=&2\om \nab_3 U_a+2\etab^b \nab_b U_a -2\etab_a \eta^b U_b  - 2 \xi_a \xib^c U_c\\
  &-&2 \omb \nab_4 U_a-  2\eta^b \nab_b U_a + 2 \eta_a \etab_b U_b + 2 \xib_a \xi^c U_c -2\rhod \in_{ab} U^b\\
  &=& 2(\etab_b-\eta_b)\nab_b  U_a+ 2\om \nab_3 U_a-2\omb \nab_4 U + 2\big( \eta_a \etab_b-\etab_a\eta_b) U^b  -2\rhod \in_{ab} U^b\\
  &+& 2\big( \xib_a \xi_c- \xi_a \xib_c)U^c.
\eeaa
Thus,
\beaa
\,[\nab_4, \nab_3]U_a &=&2(\etab_b-\eta_b)\nab_b  U_a+ 2\om \nab_3 U_a-2\omb \nab_4 U + 2\big( \eta_a \etab_b-\etab_a\eta_b-\dual \rho \in_{ab}) U^b\\
 &+& 2\big( \xib_a \xi_c- \xi_a \xib_c)U^c
\eeaa
as stated. 
\end{proof} 

\begin{remark}
The formulas  of Lemma \ref{lemma:comm-gen} remain true if we   substitute  in the  main terms 
\beaa
\chi_{a b} \longrightarrow \frac 1 2\big( \trch \de_{ab}+ \atrch \in _{ab}\big)\\
\chib_{a b} \longrightarrow \frac 1 2\big( \trchb \de_{ab}+ \atrchb \in _{ab}\big)
\eeaa
 and add   in the error terms  the contributions due to $ \chih, \chibh$. 
 \end{remark} 
 
   Thus the  main terms in  the commutator formula  $[\nab_3, \nab_b] U_A$  take the form
 \beaa
 [\nab_3, \nab_b] U_A&=& -\chib_{bc}\nab_c U_{A} +(\eta_b-\ze_b)  \nab_3 U_a   +\sum_{i=1}^k \big( \chib_{a_i  b}  \eta_c-\chib_{bc} \eta_{a_i} \big) U_{a_1\ldots }\,^ c \,_{\ldots a_k}\\
 &=&-\frac 1 2 \big( \trchb \nab_b U_A +\atrchb \dual \nab_b U_A\big) +(\eta_b-\ze_b)  \nab_3 U_A  \\
 &+&\frac 1 2 \sum_{i=1}^k\big(\de_{a_i b} \trchb+\in_{a_i b} \atrchb\big) \eta_c  U_{a_1\ldots }\,^ c \,_{\ldots a_k}\\
 &-&\frac 1 2   \sum_{i=1}^k\big(\trchb  \de_{bc}+\atrchb\in_{bc} \big) \eta_{a_i}   U_{a_1\ldots }\,^ c \,_{\ldots a_k}.
 \eeaa

 If in addition $U$ is symmetric traceless in all indices, i.e. $U\in \SS_k$ we deduce,
 \beaa
  [\nab_3, \nab_b] U_A&= &-\frac 1 2 \big( \trchb \nab_b U_A +\atrchb \dual \nab_b U_A\big) +(\eta_b-\ze_b)  \nab_3 U_A\\
  &+&\frac 1 2 \sum_{i=1}^k\big(\de_{a_i b} \trchb+\in_{a_i b} \atrchb\big) \eta_c  U_{a_1\ldots }\,^ c \,_{\ldots a_k}\\
  &-&\frac 1 2 \sum_{i=1}^k \eta_{a_i}  \big(\trchb U_{a_1\ldots b\ldots a_k} +\atrchb \dual U_{a_1\ldots b\ldots a_k}  \big).
 \eeaa
 
   In the following Lemma we specialize  to the case of $\SS_0$, $\SS_1$ and $\SS_2$.
  
   \begin{lemma}
   \lab{lemma:comm}
   The following commutation formulas hold true.
   \begin{enumerate}
\item Given   $f \in \SS_0$ we have
       \beaa
        \,[\nab_3, \nab_a] f &=&-\frac 1 2 \left(\trchb \nab_a f+\atrchb \dual \nab f\right)+(\eta_a-\ze_a) \nab_3 f-\chibh_{ab}\nab_b f  +\xib_a \nab_4 f,\\
         \,[\nab_4, \nab_a] f &=&-\frac 1 2 \left(\trch \nab_a f+\atrch \dual \nab f\right)+(\etab_a+\ze_a) \nab_4 f-\chih_{ab}\nab_b f  +\xi_a \nab_3 f, \\
         \, [\nab_4, \nab_3] f&=& 2(\etab-\eta ) \c \nab f + 2 \om \nab_3 f -2\omb \nab_4 f. 
       \eeaa

  \item   Given  $u\in \SS_1$ we have
    \bea\label{commutator-3-a-u-b}\label{commutator-u-in-SS1}
         \bsplit            
\,  [\nab_3,\nab_a] u_b    &=-\frac 1 2 \trchb \big( \nab_a u_b+\eta_b u_a-\de_{ab} \eta \c u \big) -\frac 1 2 \atrchb \big( \dual \nab_a u_b+\eta_b \dual u_a-\in_{ab} \eta\c u\big)\\
&+(\eta-\ze)_a \nab_3 u_b+\err_{3ab}[u],\\
  \err_{3ab}[u] &=-\dual \bb_a\dual u_b+\xib_a\nab_4 u_b-\xib_b \chi_{ac} u_c+\chi_{ab} \,\xib\c u-\chibh_{ac}\nab_c u_b-\eta_b\chibh_{ac}u_c+\chibh_{ab}\eta\c u,
   \end{split}
   \eea
   \bea\label{commutator-4-a-u-b}
   \bsplit
\,  [\nab_4,\nab_a] u_b    &=-\frac 1 2 \trch \big( \nab_a u_b+\etab_b u_a-\de_{ab} \etab \c u \big) -\frac 1 2 \atrch \big( \dual \nab_a u_b+\etab_b \dual u_a-\in_{ab} \etab\c u\big)\\
&+(\etab +\ze)_a \nab_4 u_b +\err_{4ab}[u],\\
   \err_{4ab}[u]&=\dual \b_a\dual u_b+\xi_a\nab_3 u_b-\xi_b \chib_{ac} u_c+\chib_{ab} \,\xi\c u-\chih_{ac}\nab_c u_b-\etab_b\chih_{ac}u_c+\chih_{ab}\etab\c u, 
      \end{split}
   \eea
   \bea
   \bsplit
 \, [\nab_4, \nab_3] u_a&=2 \om \nab_3 u_a -2\omb \nab_4 u_a+ 2(\etab_b-\eta_b ) \nab_b u_a +2(\etab \c u ) \eta_{a} -2 (\eta \c u )\etab_{a} -2 \dual \rho \dual u_a \\
 & +\err_{43a}[u],\\
 \err_{43a}[u]&= 2 \big( \xib_{a}  \xi_b- \xi_{a}  \xib_b )u^b.
\end{split}
\eea
\item  Given  $u\in \SS_2$ we have 
    \bea\label{commutator-u-in-SS2}\label{commutator-3-a-u-bc}
         \bsplit            
\,  [\nab_3,\nab_a] u_{bc}    &=-\frac 1 2 \trchb\, (\nab_a u_{bc}+\eta_bu_{ac}+\eta_c u_{ab}-\de_{a b}(\eta \c u)_c-\de_{a c}(\eta \c u)_b )\\
&-\frac 1 2 \atrchb\, (\dual \nab_a u_{bc} +\eta_b\dual u_{ac}+\eta_c\dual u_{ab}- \in_{a b}(\eta \c u)_c- \in_{a c}(\eta \c u)_b )\\
&+(\eta_a-\ze_a)\nab_3 u_{bc}+\err_{3abc}[u],\\
\err_{3abc}[u]&= -2\dual \bb_a \dual u_{bc}+\xib_a \nab_4 u_{bc} -\xib_b\chi_{ad}u_{dc} -\xib_c\chi_{ad}u_{bd}+\chi_{ab}\xib_d u_{dc} +\chi_{ac}\xib_d u_{bd}\\
&-\chibh_{ad} \nab_d u_{bc} -\eta_b\chibh_{ad}u_{dc} - \eta_c\chibh_{ad}u_{bd}+\chibh_{ab}\eta_du_{dc} +\chibh_{ac}\eta_du_{bd},
   \end{split}
   \eea
   \bea
   \bsplit
\,  [\nab_4,\nab_a] u_{bc}    &=-\frac 1 2 \trch\, (\nab_a u_{bc}+\etab_bu_{ac}+\etab_c u_{ab}-\de_{a b}(\etab \c u)_c-\de_{a c}(\etab \c u)_b )\\
&-\frac 1 2 \atrch\, (\dual \nab_a u_{bc} +\etab_b\dual u_{ac}+\etab_c\dual u_{ab}- \in_{a b}(\etab \c u)_c- \in_{a c}(\etab \c u)_b )\\
&+(\etab_a+\ze_a)\nab_4 u_{bc}+\err_{4abc}[u],\\
\err_{4abc}[u]&= 2\dual \b_a \dual u_{bc}+\xi_a \nab_3 u_{bc} -\xi_b\chib_{ad}u_{dc} -\xi_c\chib_{ad}u_{bd}+\chib_{ab}\xi_d u_{dc} +\chib_{ac}\xi_d u_{bd}\\
& -\chih_{ad} \nab_d u_{bc} -\etab_b\chih_{ad}u_{dc} - \etab_c\chih_{ad}u_{bd}+\chih_{ab}\etab_du_{dc} +\chih_{ac}\etab_du_{bd}, 
     \end{split}
   \eea
   \bea\label{commutator-3-a-u-bc}
   \bsplit
   \, [\nab_4, \nab_3] u_{ab}&=2 \om \nab_3 u_{ab} -2\omb \nab_4 u_{ab} + 2(\etab_c-\eta_c ) \nab_c u_{ab} + 4 \eta \hot (\etab \c u)-4 \etab \hot (\eta \c u) -4 \dual \rho \dual u_{ab} \\
   &+\err_{43ab}[u],\\
\err_{43ab}[u]&= 2 \big( \xib_{a}  \xi_c- \xi_{a}  \xib_c )u^c\,_{b}+2 \big( \xib_{b}  \xi_c- \xi_{b}  \xib_c )u_{a} \,^c.
   \end{split}
\eea
       \end{enumerate}
 \end{lemma}

   We deduce the following commutation formulas.
      \begin{corollary}\label{corr:comm} The following commutation formulas hold true.
      \begin{enumerate}
      \item    Given  $u\in \SS_1$ we have,
    \bea
    \bsplit
  \,    [\nab_3,\div] u&=-\frac 1 2 \trchb \,\big( \div u-\eta\c u\big) +
  \frac 1 2 \atrchb\,\big(  \div\dual u  -\eta \c \dual u\big)      +(\eta-\ze)\c\nab_3 u  
+\err_{3\div}[u],\\
\err_{3\div}[u]&=-\dual \bb \c \dual u +\xib \c \nab_4 u -\xib \c \chih \c  u-\chibh \c \nab u-\eta \c \chibh \c u,
    \\
\,   [\nab_4,\div] u&=-\frac 1 2 \trch \, \big(\div u-\etab\c u\big) +\frac 1 2 \atrch\, \big(\div\dual u -   \etab\c \dual u\big) 
+(\etab+\ze)\c\nab_4 u    +\err_{4\div}[u], \\
\err_{4\div}[u]&=\dual \b \c \dual u +\xi \c \nab_3 u -\xi\c \chibh \c  u-\chih \c \nab u-\etab \c \chih \c u.
    \end{split}
    \eea
 Also, 
\bea\label{last-statement-item1}
\bsplit
\, [\nab_3, \nab \hot] u&=- \frac 1 2 \trchb \left( \nab\hot u +\eta \hot u\right)- \frac 1 2 \atrchb\, \dual  \left( \nab \hot  u+\etab \hot  u\right)+ (\eta-\ze) \hot \nab_3 u+\err_{3\hot}[u],\\
\err_{3\hot}[u] &=-\dual \bb \hot \dual u+\xib \hot \nab_4 u-\xib \hot ( \chi \c u)+\chih \,(\xib\c u)-\chibh \c \nab u-\eta \hot (\chibh \c u)+\chibh (\eta\c u),\\
\\
\, [\nab_4, \nab \hot] u&=- \frac 1 2 \trch \left( \nab\hot u +\etab \hot u\right)- \frac 1 2 \atrch\,  \dual  \left(\nab \hot  u+\etab \hot  u\right)+ (\etab+\ze) \hot \nab_4 u+\err_{4\hot}[u], \\
\err_{4\hot}[u] &=\dual \b \hot \dual u+\xi \hot \nab_3 u-\xi \hot ( \chib \c u)+\chibh \,(\xi\c u)-\chih \c \nab u-\etab \hot (\chih \c u)+\chih (\etab\c u).
\end{split}
\eea
\item Given $u\in \SS_2$ we have
\bea
\bsplit
\,[\nab_3,  \div] u &= -\frac 1 2 \trchb \big(  \div u -  2 \eta \c u\big) +\frac 1 2 \atrchb\big(  \div\dual u -2 \eta\c \dual u\big) +(\eta-\ze)\c\nab_3 u +\err_{3\div}[u],\\
 \err_{3\div}[u] &=-2\dual \bb \c \dual u +\xib\c \nab_4 u -\xib\c  \chi\c  u -(\chi\c u)\xib +\xib\c u\c\chi -\chibh\c \nab u\\
 & -\eta\c \chibh \c u-(\chibh\c u)\eta+\eta\c u\c\chibh, \\
\,[\nab_4,  \div] u &=-\frac 1 2 \trch \big(  \div u - 2\etab \c u\big) +\frac 1 2 \atrch\big(  \div\dual u -2 \etab\c \dual u\big) +(\etab+\ze)\c\nab_4 u +\err_{4\div}[u],\\
\err_{4\div}[u] &=2\dual \b \c \dual u +\xi\c \nab_3 u -\xi\c  \chib\c  u -(\chib\c u)\xi +\xi\c u\c\chib -\chih\c \nab u\\
& -\etab\c \chih \c u-(\chih\c u)\etab+\etab\c u\c\chih.
\end{split}
\eea
\end{enumerate} 
\end{corollary}
     
     \begin{proof}
     We check \eqref{last-statement-item1}. From \eqref{commutator-4-a-u-b} we have 
 \beaa
2 [\nab_4, \nab \hot] u_{ab}&=&  [\nab_4,\nab_a] u_b   + \,  [\nab_4,\nab_b] u_a - \de_{ab} [\nab_4, \div] u   \\
 &=&-   \trch \left( \nab\hot u +\etab \hot u\right) +(\etab +\ze)_a \nab_4 u_b-\frac 1 2  \atrch  H_{ab}\\
 &&+\err_{4ab}[u] +\err_{4ba}[u]-\de_{ab}\err_{4\div}[u]
 \eeaa
  where
 \beaa
 H_{ab}:&=&\big( \dual \nab_a u_b+\eta_b \dual u_a-\in_{ab} \eta\c u\big)+\big( \dual \nab_b u_a+\eta_b \dual u_a-\in_{ba} \eta\c u\big)- \de_{ab}\big( \dual\nab \c u +\eta\c \dual u \big)\\
 &=& 2 (\dual \nab \hot u)_{ab}+ 2 ( \eta\hot\dual  u)_{ab}.
 \eeaa
 Recalling that $\dual\xi \hot  \eta = \xi\hot\dual\eta=\dual\big(\xi \hot  \eta\big)$ we infer that
 $  H= 2 \dual \big(\nab \hot u+ \eta\hot  u\big)$.
  This proves the desired result.

 We check the last statement in item 2. From \eqref{commutator-3-a-u-bc}
 \beaa
 \,  [\nab_4,\nab_a] u_{bc}    &=&-\frac 1 2 \trch\, (\nab_a u_{bc}+\etab_bu_{ac}+\etab_c u_{ab}-\de_{a b}(\etab \c u)_c-\de_{a c}(\etab \c u)_b )\\
&&-\frac 1 2 \atrch\, (\dual \nab_a u_{bc} +\etab_b\dual u_{ac}+\etab_c\dual u_{ab}- \in_{a b}(\etab \c u)_c- \in_{a c}(\etab \c u)_b )\\
&&+(\etab_a+\ze_a)\nab_4 u_{bc}
 \eeaa
 we deduce, recalling that $\delta_{ab} u_{ab}=0$,  
 \beaa
 [\nab_4,  \div] u_c &=& \delta_{ab} [\nab_4, \nab_a] u_{bc}\\
 &=&-\frac 1 2 \trch\, (\div u_{c}+(\etab \c u)_{c}-2(\etab \c u)_c-(\etab \c u)_c )\\
&&-\frac 1 2 \atrch\, (-\div \dual u_{c} +(\etab \c\dual u)_{c}+ \dual(\etab \c u)_c )+((\etab+\ze) \c\nab_4 u)_{c}
 \eeaa
 which proves the desired result.
   \end{proof}


\subsection{Null structure and Bianchi   using conformally invariant derivatives}


Consider  frame transformations of the form
\beaa
e_3'=\la^{-1} e_3, \qquad  e'_4 = \la e_4 , \qquad e_a'= e_a.
\eeaa
Note that  under   the   above  mentioned  frame transformation we have
\beaa
 \trchb'&=&\la^{-1} \trchb, \quad \atrchb'=\la^{-1} \atrchb,  \quad   \trch'=\la\trch, \quad \atrch'=\la \atrch, \\
  \xi'&=& \la \xi, \quad   \eta'=\eta, \quad \etab'=\etab,  \quad \xib'=\la^{-1}\xib,\\
   \a'&=&\la^2 \a,\quad \b'=\la \b,    \quad \rho'=\rho,    \quad \rhod'=\rhod,  \quad \bb'=\la^{-1} \bb,\quad  \aa'=\la^{-2} \aa,
   \eeaa
   and
   \beaa
 \omb'&=& \la^{-1}\left(\omb +\frac{1}{2} e_3(\log \la)\right), \quad \om'= \la\left(\om -\frac{1}{2} e_4(\log \la)\right), \quad
 \ze'= \ze - \nab (\log \la).
\eeaa

\begin{remark}
If $f$ verifies $f'= \lambda^s f$, then  $\nab_3 f, \nab_4 f, \nab_a f$ are not conformal invariant.
\end{remark} 

 We correct the lacking of being conformal invariant by making the following  definition.  

\begin{lemma}
If $f$ verifies $f'= \lambda^s f$, then 
 \begin{enumerate}
 \item $\nabc_3 f:= \nab_3f-2 s \omb f$ is $(s-1)$-conformally invariant.
 \item $\nabc_4 f:= \nab_4f+2 s \om f$ is $(s+1)$-conformally invariant.
 \item $\nabc_A f:= \nab_Af+ s \ze_A f$ is $s$-conformally invariant. 
  \end{enumerate}
\end{lemma}

\begin{proof}
Immediate verification.
\end{proof}

\begin{remark} 
Note that $s$ is precisely what   in \cite{Ch-Kl} is called       the  signature of the tensor.
\end{remark}

Using these definitions we rewrite the main equations as follows
\begin{proposition}
\label{prop-nullstr-conformal}
\beaa
\nabc_3\trchb&=&-|\chibh|^2-\frac 1 2 \big( \trchb^2-\atrchb^2\big)+2\divc\xib +  2 \xib\c(\eta+\etab),\\
\nabc_3\atrchb&=&-\trchb\atrchb +2\curlc \xib + 2 \xib\wedge(-\eta+\etab),\\
\nabc_3\chibh&=&-\trchb\,  \chibh+2  \nab\hot \xib+   2  \xib\hot(\eta+\etab)-\aa,
\eeaa
\beaa
\nabc_3\trch
&=& -\chibh\c\chih -\frac 1 2 \trchb\trch+\frac 1 2 \atrchb\atrch    +   2   \divc\eta+ 2 \big(\xi\c \xib +|\eta|^2\big)+ 2\rho,\\
\nabc_3\atrch
&=&-\chibh\wedge\chih-\frac 1 2(\atrchb \trch+\trchb\atrch)+ 2 \curlc \eta  + 2 \xib\wedge\xi  -  2 \dual \rho,\\
\nabc_3\chih
&=&-\frac 1 2 \big( \trch \chibh+\trchb \chih\big)-\frac 1 2 \big(-\dual \chibh \, \atrch+\dual \chih\,\atrchb\big)
+2  \nabc\hot \eta +2 \xib\hot\xi +2\eta\hot\eta,
\eeaa
\beaa
\nabc_4\trchb
&=& -\chih\c\chibh -\frac 1 2 \trch\trchb+\frac 1 2 \atrch\atrchb    +  2   \divc \etab+ 2\big( \xi\c \xib +|\etab|^2\big)+2\rho,\\
\nabc_4\atrchb
&=&-\chih\wedge\chibh-\frac 1 2(\atrch \trchb+\trch\atrchb)+ 2 \curlc\etab  + 2 \xi\wedge\xib+2 \dual \rho,\\
\nabc_4\chibh
&=&-\frac 1 2 \big( \trchb \chih+\trch \chibh\big)-\frac 1 2 \big(-\dual \chih \, \atrchb+\dual \chibh\,\atrch\big)
+2 \nabc\hot \etab +2 \xi\hot\xib +2 \etab\hot\etab,
\eeaa
\beaa
\nabc_4\trch&=&-|\chih|^2-\frac 1 2 \big( \trch^2-\atrch^2\big)+ 2 \divc\xi  + 2   \xi\c(\etab+\eta),\\
\nabc_4\atrch&=&-\trch\atrch + 2 \curlc\xi + 2 \xi\wedge(-\etab+\eta),\\
\nabc_4\chih&=&-\trch\,  \chih+ 2 \nabc\hot \xi+  2   \xi\hot(\etab+\eta)-\a,
\eeaa
\beaa
\nabc_3 \etab -\nabc_4\xib &=& -\chibh\c(\etab-\eta) -\frac{1}{2}\trchb(\etab-\eta)+\frac{1}{2}\atrchb(\dual\etab-\dual\eta) -4 \om \xib  +\bb, \\
\nabc_4 \eta    -    \nabc_3\xi &=& -\chih\c(\eta-\etab) -\frac{1}{2}\trch(\eta-\etab)+\frac{1}{2}\atrch(\dual\eta-\dual\etab)-4\omb \xi -\b.\\
\eeaa
Also,
\beaa
\divc\chih &=& \frac{1}{2}\nabc(\trch) -\frac{1}{2}\dual\nabc (\atrch) -\atrch\dual\eta-\atrchb\dual\xi -\b,\\
\divc\chibh  &=& \frac{1}{2}\nabc(\trchb) -\frac{1}{2}\dual\nabc(\atrchb) -\atrchb\dual\etab-\atrch\dual\xib +\bb.
\eeaa
\end{proposition}

    \begin{proposition}\label{prop:bianchi-conformal} 
    We have,
    \beaa
    \nabc_3\a- 2   \nabc\hot \b&=&-\frac 1 2 \big(\trchb\a+\atrchb\dual \a)+
 2 \c  4\eta\hot \b - 3 (\rho\chih +\rhod\dual\chih),\\
\nabc_4\beta - \divc\a &=&-2(\trch\beta-\atrch \dual \b)  +\a\c  \etab + 3  (\xi\rho+\dual \xi\rhod),\\
     \nabc_3\b+\divc\varrho&=&-(\trchb \b+\atrchb \dual \b)+2\bb\c \chih+3 (\rho\eta+\rhod\dual \eta)+    \a\c\xib,\\
 \nabc_4 \rho-\divc \b&=&-\frac 3 2 (\trch \rho+\atrch \rhod)+2\etab \c\b-2\xi\c\bb-\frac 1 2 \chibh \c\a,\\
   \nabc_4 \rhod+\curlc \b&=&-\frac 3 2 (\trch \rhod-\atrch \rho)-2\etab\c\dual \b-2\xi\c\dual \bb+\frac 1 2 \chibh \c\dual \a, \\
     \nabc_3 \rho+\divc\bb&=&-\frac 3 2 (\trchb \rho -\atrchb \rhod) -2\eta \c\bb+2\xib\c\b-\frac 12\chih\c\aa,
 \\
   \nabc_3 \rhod+\curlc \bb&=&-\frac 3 2 (\trchb \rhod+\atrchb \rho)- 2\eta\c\dual \bb-2\xib\c\dual\b-\frac 1 2 \chih\c\dual \aa,\\
     \nabc_4\bb-\divc\varoc&=&-(\trch \bb+ \atrch \dual \bb)+2\b\c \chibh
    -3 (\rho\etab-\rhod\dual \etab)-    \aa\c\xi,\\
     \nabc_3\bb +\divc\aa &=&-2(\trchb\,\bb-\atrchb \dual \bb)   -\aa\c \eta - 3  (\xib\rho-\dual \xib \rhod),\\
     \nabc_4\aa+ 2  \nabc\hot \bb&=&-\frac 1 2 \big(\trch\aa-\atrch\dual \aa)
 -  2  \c 4\etab\hot \bb - 3  (\rho\chibh -\rhod\dual\chibh).
\eeaa
    \end{proposition}

    
\section{Main equations in complex notations} 


In this section we introduce complex notations for the Ricci coefficients and the curvature components with the objective of simplifying the main equations. From the real scalars, 1-tensors and symmetric traceless 2-tensors already introduced, we define their complexified version which results in anti-self dual tensors. 


\subsection{Complex notations}


Recall Definition \ref{definition-SS-real} of the set of real horizontal $k$-tensors $\SS_k=\SS_k(\MM, \mathbb{R})$ on $\MM$. For instance, 
\begin{itemize}
\item $a \in \SS_0$ is a real scalar function on $\MM$, 
\item $f \in \SS_1$ is a real horizontal 1-tensor on $\MM$,
\item $u \in \SS_2$ is a real horizontal symmetric traceless 2-tensor on $\MM$.
\end{itemize}

By Definition \ref{definition-hodge-duals}, the duals of real horizontal tensors are real horizontal tensors of the same type, i.e. $\dual f \in \SS_1$ and $\dual u \in \SS_2$.

We define the complexified version of horizontal tensors on $\MM$.

\begin{definition} We denote by $\SS_k(\mathbb{C})=\SS_k(\MM, \mathbb{C})$ the set of complex anti-self dual $k$-tensors on $\MM$. More precisely, 
\begin{itemize}
\item $a+ i b \in \SS_0(\mathbb{C})$ is a complex scalar function on $\MM$ if $(a, b) \in \SS_0$,
\item $F= f+ i \dual f  \in \SS_1(\mathbb{C})$ is a complex anti-self dual 1-tensor on $\MM$ if $f \in \SS_1$,
\item $U=u + i \dual u \in \SS_2(\mathbb{C})$ is a complex anti-self dual symmetric traceless 2-tensor on $\MM$ if $u \in \SS_2$.
\end{itemize}
\end{definition}

Observe that $F\in \SS_1(\mathbb{C})$ and $ U\in \SS_2(\mathbb{C})$ are indeed anti-self dual tensors, i.e.:
\beaa
\dual F=-i F, \qquad \dual U=-i U. 
\eeaa
More precisely 
\beaa
U_{12}=U_{21}= i \dual U_{12}=i \in_{12} U_{22}=-i U_{11}, \qquad U_{11} =i U_{12}.
\eeaa

Recall that the derivatives $\nab_3$, $\nab_4$ and $\nab_a$ are real derivatives. We can use the dual operators to define the complexified version of the $\nab_a$ derivative, which allows to simplify the notations in the main equations.

\begin{definition}
We define the complexified version of the horizontal derivative as 
\beaa
\DD= \nab+ i \dual \nab, \qquad \DDov=\nab- i \dual \nab
\eeaa
More precisely, we have
\begin{itemize}
\item for $a+i b \in \SS_0(\mathbb{C})$, 
\beaa
\DD(a+ib) &:=& (\nabla+i\dual\nabla)(a+ib),\qquad \DDov(a+ib) := (\nabla-i\dual\nabla)(a+ib).
\eeaa
\item For $f+ i \dual f \in\SS_1(\mathbb{C}) $, 
\beaa
\DD\c(f+i\dual f) &:=& (\nabla+i\dual\nabla)\c(f+i\dual f)\, =0, \\
\DDov\c(f+i\dual f) &:=& (\nabla-i\dual\nabla)\c(f+i\dual f),\\
\DD\hot(f+i\dual f) &:=& (\nabla+i\dual\nabla)\hot(f+i\dual f).
\eeaa
\item For  $u+ i \dual u \in \SS_2(\mathbb{C}) $, 
\beaa
\DD (u+i\dual u) &:=& (\nabla+i\dual\nabla) (u+i\dual u)=0\\
\DDov (u+i\dual u) &:=& (\nabla- i\dual\nabla) (u+i\dual u).
\eeaa
\end{itemize}
\end{definition}

Note that 
\beaa
\dual \DD=-i\DD. 
\eeaa
 
\begin{remark}
For $F= f+i \dual f\in\SS_1(\mathbb{C})$ the operator $-\frac 1 2 \DD\hot $ is formally adjoint to  the operator $\DDov \c U$ applied to $U\in \SS_2(\mathbb{C})$.
For $ h=a+ib \in \SS_0(\mathbb{C})$ the operator $-\DD h $ is formally adjoint to the operator $\DDov \c F$ applied to $F\in\SS_1(\mathbb{C})$.  These notions makes sense  literally only  if  the horizontal structure is integrable.
\end{remark}

\begin{lemma}
The following holds .
\begin{itemize}
\item If $\xi, \eta\in \SS_1$
\beaa
\xi\c \eta+i\dual\xi\c \eta &=& \frac{1}{2}\Big((\xi+i\dual\xi)\c (\ov{\eta+i\dual \eta})\Big),\\
\xi\hot \eta+i\dual(\xi\hot \eta) &=& \frac{1}{2}\Big((\xi+i\dual \xi)\hot (\eta+i\dual \eta)\Big).
\eeaa

\item If $\eta\in \SS_1 $,    $u\in\SS_2$ 
\beaa
u\c \eta+i\dual u\c \eta &=& \frac{1}{2}(u+i\dual u)\c (\ov{\eta+i\dual \eta}),\\
u\c \eta+i\dual (u\c \eta) &=& \frac{1}{2}(u+i\dual u)\c (\ov{\eta+i\dual \eta}).
\eeaa

\item If $u, v\in \SS_2$ 
\beaa
u\c v+i\dual u\c v &=& \frac{1}{2}(u+i\dual u)\c (\ov{v+i\dual v}).
\eeaa

\item If $a, b \in \SS_0 $
\beaa
\nabla a-\dual\nabla b+i(\dual\nabla a+\nabla b) &=& \DD(a+ib).
\eeaa

\item If $\xi \in \SS_1 $
\beaa
\div \xi+i\curl \xi &=& \frac{1}{2}\ov{\DD}\c(\xi+i\dual \xi)\\
\nabla\hot \xi+i\dual(\nabla\hot \xi) &=& \frac{1}{2}\mathcal{D}\hot(\xi+i\dual \xi).
\eeaa

\item If $u\in \SS_2 $
\beaa
\div u+i\dual(\div u) &=& \frac{1}{2}\ov{\DD}\c(u+i\dual u).
\eeaa
\end{itemize}
\end{lemma}

\begin{proof}
The first identities rely on Lemma \ref{lemma:usefulidentitiesforcomplexification}. The other rely on the following identities, for  $\xi\in \SS_1 $,  $u\in\SS_2 $.
\beaa
&&\nab\c\dual \xi=\curl \xi, \quad \dual\nab\c \xi=-\curl \xi, \quad \dual\nab\c\dual \xi=\nab \xi,\\
&& \nabla\hot \dual \xi=\dual\nabla\hot \xi=\dual(\nabla\hot \xi), \quad \dual\nabla\hot \dual \xi = -\nabla\hot \xi,\\
&& \dual(\div u)=\nab\c \dual u, \quad \dual\nab\c u=-\dual(\div u), \quad \dual\nab\c\dual u=\nab\c u.
\eeaa
 For example, we check that $\dual (\nab \hot \xi)= \nab \hot \dual \xi$. Let               $ C=\nab \hot \xi$ and $D=\nab \hot \dual \xi $. We have,
      \beaa
      C_{11}&=&\nab_1 \xi_1-\nab_2 \xi_2=-C_{22},\\
      C_{12}&=&\nab_1 \xi_2+\nab_2 \xi_1 =C_{21}.
      \eeaa
      Hence,
     \beaa
      (\dual C)_{11}&=& C_{21}=\nab_2 \xi_1+\nab_1 \xi_2,\\
      (\dual C)_{12}&=&C_{22}=-C_{11}=-(\nab_1 \xi_1-\nab_2 \xi_2 ),\\
       (\dual C)_{21}&=&-C_{11}=-(\nab_1 \xi_1-\nab_2 \xi_2 ),\\
       (\dual C)_{22}&=&- C_{12}=-(\nab_1 \xi_2+\nab_2 \xi_1).
      \eeaa
      On the other hand,
      \beaa
      D_{11}&=& \nab_1 \dual \xi_1-\nab_2 \dual \xi_2=\nab_1 \xi _2+\nab_2 \xi_1,\\
      D_{12}&=& \nab_1 \dual \xi_2+\nab_2 \dual \xi_1=-\nab_1  \xi_1+\nab_2 \xi_2. 
      \eeaa
      Hence,
      \beaa
     D_{11}  &=&(\dual C)_{11},\\
     D_{12}&=&(\dual C)_{12},
      \eeaa
      i.e. $\dual C= D$ as desired.
\end{proof}

\begin{lemma}\label{dot-hot-complex} 
Let  $E= \xi +i \dual \xi \in \SS_1(\mathbb{C})  $ and $F=\eta+i\dual \eta \in \SS_1(\mathbb{C}) $ and $U=u+i \dual u\in \SS_2(\mathbb{C}) $.  Then
\bea\label{simil-Leibniz}
E \hot ( \ov{F} \c U)+F \hot ( \ov{E} \c U)&=& ( E \c \ov{F}+ \ov{E} \c F) \ U.
\eea
\end{lemma}

\begin{proof} 
Recall, see Lemma \ref{dot-hot}, 
\beaa
\xi \hot ( \eta  \c u) + \eta  \hot ( \xi \c u)= (\xi\c \eta  ) u. 
\eeaa
Now
\beaa
E \hot ( \ov{F} \c U)&=& (\xi+ i \dual \xi ) \hot \big( \ov{(\eta + i \dual \eta)} \c (u + i \dual u)\big) \\
&=& 2(\xi+ i \dual \xi ) \hot \big( (u \c \eta) + i \dual ( u \c \eta)\big) \\
&=& 4\Big (\xi \hot  (u \c \eta) + i \dual ( \xi \hot  (u \c \eta) ) \Big), 
\\
F \hot ( \ov{E} \c U)&=& 4 \Big(\eta \hot  (u \c \xi) + i \dual ( \eta \hot  (u \c \xi) ) \Big) .
\eeaa
Therefore
\beaa
E \hot ( \ov{F} \c U)+F \hot ( \ov{E} \c U)&=&4 (\xi \hot  (u \c \eta) + i \dual ( \xi \hot  (u \c \eta) ) ) +4 (\eta \hot  (u \c \xi) + i \dual ( \eta \hot  (u \c \xi) ) )\\
 &=&4 (\xi \hot  (u \c \eta) +\eta \hot  (u \c \xi))+ 4i \dual ( \xi \hot  (u \c \eta)  +  \eta \hot  (u \c \xi) )\\
  &=&4 ( (\xi \c \eta) \ u)+ 4i \dual (  (\xi \c \eta) \ u)\\
    &=&4 (\xi \c \eta) \ ( u+ i \dual  u)
\eeaa
while
\beaa
E \c \ov{F}+ \ov{E} \c F&=& 2(\xi \c \eta + i \dual \xi \c \eta)  + 2(\eta \c \xi + i \dual \eta \c \xi)=4(\xi \c \eta ). 
\eeaa
Hence,
\beaa
E \hot ( \ov{F} \c U)+F \hot ( \ov{E} \c U)&=&\big(E \c \ov{F}+ \ov{E} \c F\big)U
\eeaa
as stated.
\end{proof}


\subsection{Leibniz formulas}


We collect here Leibniz formulas involving the derivative operators defined above.

\begin{lemma} 
Let $h \in \SS_0(\mathbb{C})$, $F= f+i \dual f \in \SS_1(\mathbb{C}) $,  $U=u+i \dual u\in \SS_2(\mathbb{C})$.  Then
\bea
\bsplit
 \ov{\DD} \c (h F)&= h \ov{\DD} \c F+  \ov{\DD}(h) \c F \lab{ov-HH-hF},\\
 \DD\hot (h F)&=h  \DD\hot  F+ \DD( h) \hot F  \lab{DD-hot-hF},\\
\ov{\DD}\c( h U)&=  \ov{\DD}(h) \c U  + h(\ov{\DD} \c U) \lab{ov-DD-hu},
\end{split}
\eea
\bea\label{Leibniz-hot}
\bsplit
 \DD\hot(\ov{F}\c U)&=   (\DD\c \ov{F} ) U +  (\ov{F}\c \DD) U.
\end{split}
\eea
\end{lemma}

\begin{proof} 
Note that
\beaa
\curl (hf)&=& \ep^{ab} \nab_a(h f)_b=\ep^{ab} \nab_a(h) f_b+\ep^{ab} h \nab_a(f)_b-\ep^{ba} \nab_a(h) f_b+ h \curl f\\
&=&-\dual \nab h \c f+ h \curl f,\\
\div(hu)&=& \nab^{a} (h u)_{ab}=\nab^{a} (h) u_{ab}+h \nab^a u_{ab}=\nab h \c u + h \div u,
\eeaa
and
 \beaa
 (\nab \hot (hf))_{AB}&=& \frac 1 2 \Big(\nab_A (hf)_B+\nab_B (hf)_A-\delta_{AB} (\div (hf))\Big)\\
 &=& \frac 1 2 (\nab_A (h)f_B+ h \nab_Af_B+\nab_B hf_A+h \nab_B f_A-\delta_{AB} (\nab h \c f +h\div (f))) \\
 &=&h (\nab \hot f)_{AB}+  (\nab h \hot f)_{AB}. 
 \eeaa

Now,
\beaa
\ov{\DD}\c( h F)&=& \ov{\DD}\c( h f + i \dual (h f ))= 2 \div (hf) + 2i \curl (hf)\\
&=& 2 (\nab h \c f + h \div f) + 2i (-\dual \nab h \c f+ h \curl f)\\
&=& h \ov{\DD} \c F+ 2 \nab h \c f  + 2i ( \nab h \c \dual f)\\
&=& h \ov{\DD} \c F+ F \c (\nab h - i \dual \nab h )\\
&=& h \ov{\DD} \c F+  \ov{\DD}(h) \c F. 
\eeaa

We have 
\beaa
\DD\hot (h F)&=& 2 \nab \hot (hf)+2i \dual (\nab \hot (hf))\\
 &=& 2 (h (\nab \hot f)+ (\nab h \hot f))+2i \dual (h (\nab \hot f)+ (\nab h \hot f)\\
 &=& 2 h (\nab \hot f)+2i h\dual  (\nab \hot f)+  2(\nab h \hot f)+ 2i \dual (\nab h \hot f)\\
  &=&  h \DD \hot F+  2(\nab h \hot f)+ 2i \dual (\nab h \hot f)\\
 &=&h  \DD\hot  F+ (\nab h+i \dual \nab h) \hot (f+i \dual f) \\
 &=&h  \DD\hot  F+ \DD( h) \hot F. 
 \eeaa

We have
\beaa
\ov{\DD}\c( h U)&=& \ov{\DD}\c( h u + i \dual (h u ))= 2 \div (hu) + 2i \dual (\div (hu))\\
&=&  2 (\nab h \c u + h \div u) + 2i \dual (\nab h \c u + h \div u)\\
&=&  2 \nab h \c u + 2i \dual (\nab h \c u) + h(\ov{\DD} \c U)\\
&=&   \nab h \c u +\dual \nab h \c \dual u + i \nab h \c \dual u- i \dual \nab h \c u + h(\ov{\DD} \c U)\\
&=&   \nab h \c u- i \dual \nab h \c u+ i (\nab h \c \dual u -i\dual \nab h \c \dual u ) + h(\ov{\DD} \c U)\\
&=&   (\nab h- i \dual \nab h )\c (u+ i\dual u)  + h(\ov{\DD} \c U)\\
&=&   \ov{\DD}(h) \c U  + h(\ov{\DD} \c U)
\eeaa 
as desired.

We write
\beaa
2 \DD\hot(\ov{F}\c U)_{ab}&=& \DD_a(\ov{F}\c U)_b+\DD_b(\ov{F}\c U)_a- \de_{ab}\DD^c (\ov{F}\c U)_c\\
&=& \DD_a ( \ov{F} ^c  U_{cb}  ) +\DD_b( \ov{F} ^c  U_{ca}  ) -\de_{ab} \DD^d ( \ov{F}^c U_{cd}) \\
&=& \DD_a  \ov{F} ^c  U_{cb}   +\DD_b  \ov{F} ^c  U_{ca}   -\de_{ab} \DD^d  \ov{F}^c U_{cd}\\
&+& \ov{F}^c\big(   \DD_a  U_{cb} +  \DD_b  U_{ca}- \de_{ab} \DD^d U_{cd}\big).
\eeaa
Now, in view of Lemma \ref{le:nonsym-product},
\beaa
 \DD_a  \ov{F} ^c  U_{cb}   +\DD_b  \ov{F} ^c  U_{ca} &=&\de_{ab}   (\DD^d  \ov{F} ^c) U_{cd}+(\DD\c \ov{F} ) U_{ab} 
 +\frac 1 2\Big(\big( \DD_a \ov{F} _c-\DD_c \ov{F}_a \big) U_{cb}+ \big( \DD_b \ov{F}_c -\DD_c \ov{F}_b \big) U_{ca}\Big).
\eeaa
Hence
\beaa
 \DD_a  \ov{F} ^c  U_{cb}   +\DD_b  \ov{F} ^c  U_{ca}   -\de_{ab} \DD^d  \ov{F}^c U_{cd}&=&(\DD\c \ov{F} ) U_{ab} +\frac 1 2\Big(\big( \DD_a \ov{F} _c-\DD_c \ov{F}_a \big) U_{cb}+ \big( \DD_b \ov{F}_c -\DD_c \ov{F}_b \big) U_{ca}\Big).
 \eeaa
 Recall that $\dual F=-i F, \,  \dual U=-i U, \,  \dual \DD=-i\DD $. We deduce,
 \bea
 \DD_a \ov{F}_b -\DD_b \ov{F} _a&=& i\in_{ab}( \DD\c \ov{F} ).
 \eea
To check note that 
 \beaa
   \DD_1 \ov{F}_2 -\DD_2 \ov{F} _1&=& 2\Big[(\nab_1 f_2 -\nab_2 f_1) + i (\nab_1 f_1+\nab_2 f_2)\Big],\\
   ( \DD\c \ov{F} )&=& 2 \Big[  (\nab_1 f_1+\nab_2 f_2)- i (\nab_1 f_2 -\nab_2 f_1)\Big].
 \eeaa
 We deduce,
 \beaa
 \DD_a  \ov{F} ^c  U_{cb}   +\DD_b  \ov{F} ^c  U_{ca}   -\de_{ab} \DD^d  \ov{F}^c U_{cd}&=&(\DD\c \ov{F} ) U_{ab}+\frac 1 2 i  (\DD\c \ov{F} )\big( \in_{ac} U_{cb}+\in_{bc} U_{ca}\big)\\
 &=& (\DD\c \ov{F} ) U_{ab} +\frac 1 2 i  (\DD\c \ov{F} )\big( - 2 i U_{ab}\big)\\
&=& 2  (\DD\c \ov{F} ) U_{ab}.
 \eeaa
 Therefore,
 \beaa
2  \DD\hot(\ov{F}\c U)_{ab}&=&  2  (\DD\c \ov{F} ) U_{ab}+\ov{F}^c\big(   \DD_a  U_{cb} +  \DD_b  U_{ca}- \de_{ab} \DD^d U_{cd}\big).
 \eeaa
 It remains to re-express the tensor 
 \beaa
   \DD_a  U_{cb} +  \DD_b  U_{ca}- \de_{ab} \DD^d U_{cd}.
   \eeaa
 Note also that  $ \DD^d U_{cd}=0$. We claim
 \beaa
  \DD_a  U_{cb} +  \DD_b  U_{ca}&=& 2 \DD_c U_{ab}.
 \eeaa
Indeed, for $a=b=1$, $c=2$,
 \beaa
 2\DD_1  U_{21} &=&2(\nab_1+i \dual \nab_1)  U_{21}=2(\nab_1+i  \nab_2)  U_{12}=-2i(\nab_1+i  \nab_2)  U_{11}=2( \nab_2-i\nab_1)  U_{11},\\
 2\DD_2 U_{11}&=& 2(\nab_2+i \dual \nab_2)  U_{11}=2(\nab_2-i \nab_1)  U_{11}.
 \eeaa
 For $a=c=1$, $b=2$,
 \beaa
 \DD_1  U_{12} +  \DD_2  U_{11}&=&(\nab_1+i\nab_2) U_{12} + (\nab_2-i\nab_1 )  U_{11}=(\nab_1+i\nab_2) U_{12} + i(\nab_2-i\nab_1 )  U_{12}\\
&=&2(\nab_1+i\nab_2) U_{12} =2\DD_1 U_{12}.
 \eeaa
 
 We deduce,
 \beaa
2 \DD\hot(\ov{F}\c U)_{ab}&=&  2  (\DD\c \ov{F} ) U_{ab}+ 2 \ov{F}^c \DD_c U_{ab}.
 \eeaa
This implies the lemma.
\end{proof}

\begin{lemma}\label{simplification-angular} 
Let $F= f+i \dual f \in \SS_1(\mathbb{C}) $,  $U=u+i \dual u\in \SS_2(\mathbb{C})$.  Then
\beaa
F \hot (\DDb \c U)&=& (F\c \DDb) U, \\
(F \c \DDb) U + (\ov{F} \c \DD) U&=& 4f \c \nab U, \\
 (F \c  \DDb) U &=&2 F  \c\nab U. 
 \eeaa
 As a consequence,
 \bea\label{relation0angular=der}
 2f \c \nab U&=&  \left(F +\ov{F} \right) \c\nab U.  
 \eea
\end{lemma}

\begin{proof}
We have
\beaa
( \DDb \c U )_1&=& \DDb^a U_{1a}=(\nab - i \dual \nab)^a U_{1a}\\
&=&(\nab_1 - i \dual \nab_1) U_{11}+(\nab_2 - i \dual \nab_2) U_{12}\\
&=&(\nab_1 - i  \nab_2) U_{11}+(\nab_2 + i  \nab_1) (-i U_{11})\\
&=&2(\nab_1 - i  \nab_2) U_{11},
\eeaa
and
\beaa
( \DDb \c U )_2&=& \DDb^a U_{2a}=(\nab - i \dual \nab)^a U_{2a}\\
&=&(\nab_1 - i \dual \nab_1) U_{21}+(\nab_2 - i \dual \nab_2) U_{22}\\
&=&(\nab_1 - i  \nab_2)(-i U_{11})+(\nab_2 + i  \nab_1) (-U_{11})\\
&=&-2i(\nab_1 -i  \nab_2) U_{11}.
\eeaa
Therefore 
\beaa
2(F \hot ( \DDb \c U ))_{11}&=& 2F_1( \DDb \c U )_1-\de_{11} F \c ( \DDb \c U )\\
&=& F_1( \DDb \c U )_1- F_2 ( \DDb \c U )_2\\
&=& (f_1+i \dual f_1)2(\nab_1 - i  \nab_2) U_{11}- (f_2+i\dual f_2)(-2i(\nab_1 -i  \nab_2) U_{11})\\
&=& 4f_1\nab_1U_{11} - 4i  f_1\nab_2U_{11} +4i f_2\nab_1U_{11} +4 f_2   \nab_2U_{11}. 
\eeaa
On the other hand,
\beaa
2(F \hot ( \DDb \c U ))_{12}&=& F_1( \DDb \c U )_2+F_2( \DDb\c U )_1\\
&=& (f_1+i \dual f_1)(2(\nab_1 -i  \nab_2) U_{12})+ (f_2+i\dual f_2)2(\nab_1 - i  \nab_2) iU_{12}\\
&=& (f_1+i  f_2)(2(\nab_1 -i  \nab_2) U_{12})+ (f_2-i f_1)2(\nab_1 - i  \nab_2) iU_{12}\\
&=& 4f_1\nab_1U_{12} - 4i  f_1\nab_2U_{12} +4i f_2\nab_1U_{12} +4 f_2   \nab_2U_{12} 
\eeaa
which therefore gives
\beaa
(F \hot ( \DDb \c U ))_{ab}&=& 2f_1\nab_1U_{ab} - 2i  f_1\nab_2U_{ab} +2i f_2\nab_1U_{ab} +2 f_2   \nab_2U_{ab}. 
\eeaa
On the other hand,
\beaa
(F\c \DDb U)_{ab}&=& F^c\DDb_c U_{ab}= F_1\DDb_1 U_{ab}+F_2\DDb_2 U_{ab}\\
&=&(f_1+ i \dual f_1)(\nab_1- i \dual \nab_1) U_{ab}+(f_2+ i \dual f_2)(\nab_2- i \dual \nab_2)  U_{ab}\\
&=&(f_1+ i  f_2)(\nab_1- i  \nab_2) U_{ab}+(f_2-i  f_1)(\nab_2+ i  \nab_1)  U_{ab}\\
&=&2f_1\nab_1U_{ab} - 2i  f_1\nab_2U_{ab} +2i f_2\nab_1U_{ab} +2 f_2   \nab_2U_{ab}.
\eeaa
From the above we also have
\beaa
(\ov{F}\c \DD U)_{ab}&=&2f_1\nab_1U_{ab} + 2i  f_1\nab_2U_{ab} -2i f_2\nab_1U_{ab} +2 f_2   \nab_2U_{ab}
\eeaa
which implies
\beaa
(F\c \DDb U)_{ab}+(\ov{F}\c \DD U)_{ab}&=&4f_1\nab_1U_{ab} +4 f_2   \nab_2U_{ab}=4f \c \nab U
\eeaa
as stated. Finally,
\beaa
 F^c\ov{ \DD_c}U &=&  f^c\ov{ \DD_c}U +i( \dual f^c) \ov{ \DD_c}U = f^c\ov{ \DD_c}U - i f^c( \dual \ov{ \DD_c}U) = 2f^c\ov{ \DD_c}U =2 F^c \nab_c U.
 \eeaa
\end{proof}


   \subsection{Complex notations for the Ricci coefficients and curvature components}
   
   We now extend the definitions for the Ricci coefficients and curvature components given in Section \ref{section-general-formalism} to the complex case by using the anti-self dual tensors defined above. 
 
\begin{definition} Let $(\MM, \g)$ be a manifold satisfying the Einstein vacuum equation. Then we define the following complex anti-self dual tensors:
\beaa
A:=\a+i\dual\a, \quad B:=\b+i\dual\b, \quad P:=\rho+i\dual\rho,\quad \Bb:=\bb+i\dual\bb, \quad \Ab:=\aa+i\dual\aa,
\eeaa      
 and   
\beaa
&& X=\chi+i\dual\chi, \quad \Xb=\chib+i\dual\chib, \quad H=\eta+i\dual \eta, \quad \Hb=\etab+i\dual \etab, \quad Z=\ze+i\dual\ze, \\ 
&& \Xi=\xi+i\dual\xi, \quad \Xib=\xib+i\dual\xib.
\eeaa    
In particular, note that 
\beaa
\tr X = \trch-i\atrch, \quad \Xh=\chih+i\dual\chih, \quad \tr\Xb = \trchb -i\atrchb, \quad \Xbh=\chibh+i\dual\chibh.
\eeaa
\end{definition}


\subsection{Main equations  in complex form}


The complex notations allow us to rewrite the Ricci equations in  a more compact  form. 
\begin{proposition}
\label{prop-nullstr:complex}
\beaa
\nab_3\tr\Xb +\frac{1}{2}(\tr\Xb)^2+2\omb\,\tr\Xb &=& \DD\c\ov{\Xib}+\Xib\c\ov{\Hb}+\ov{\Xib}\c(H-2Z)-\frac{1}{2}\Xbh\c\ov{\Xbh},\\
\nab_3\Xbh+\Re(\tr\Xb) \Xbh+ 2\omb\,\Xbh&=& \DD\hot \Xib+   \Xib\hot(H+\Hb-2Z)-\Ab,
\eeaa
\beaa
\nab_3\tr X +\frac{1}{2}\tr\Xb\tr X-2\omb\tr X &=& \DD\c\ov{H}+H\c\ov{H}+2P+\Xib\c\ov{\Xi}-\frac{1}{2}\Xbh\c\ov{\Xh},\\
\nab_3\widehat{X} +\frac{1}{2}\tr\Xb\, \widehat{X} -2\omb\widehat{X} &=& \DD\hot H  +H\hot H -\frac{1}{2}\ov{\tr X} \widehat{\Xb}+\frac{1}{2}\Xib\hot\Xi,
\eeaa
\beaa
\nab_4\tr\Xb +\frac{1}{2}\tr X\tr\Xb -2\om\tr\Xb &=& \DD\c\ov{\Hb}+\Hb\c\ov{\Hb}+2\ov{P}+\Xi\c\ov{\Xib}-\frac{1}{2}\Xh\c\ov{\Xbh},\\
\nab_4\widehat{\Xb} +\frac{1}{2}\tr X\, \widehat{\Xb} -2\om\widehat{\Xb} &=& \DD\hot\Hb  +\Hb\hot\Hb -\frac{1}{2}\ov{\tr\Xb} \widehat{X}+\frac{1}{2}\Xi\hot\Xib,
\eeaa
\beaa
\nab_4\tr X +\frac{1}{2}(\tr X)^2+2\om\tr X &=& \DD\c\ov{\Xi}+\Xi\c\ov{H}+\ov{\Xi}\c(H+2Z)-\frac{1}{2}\Xh\c\ov{\Xh},\\
\nab_4\Xh+\Re(\tr X)\Xh+ 2\om\Xh&=& \DD\hot \Xi+  \Xi\hot(\Hb+H+2Z)-A.
\eeaa
Also,
\beaa
\nab_3Z +\frac{1}{2}\tr\Xb(Z+H)-2\omb(Z-H) &=& -2\DD\omb -\frac{1}{2}\widehat{\Xb}\c(\ov{Z}+\ov{H})\\
&&+\frac{1}{2}\tr X\Xib+2\om\Xib -\Bb+\frac{1}{2}\ov{\Xib}\c\Xh,\\
\nab_4Z +\frac{1}{2}\tr X(Z-\Hb)-2\om(Z+\Hb) &=& 2\DD\om +\frac{1}{2}\widehat{X}\c(-\ov{Z}+\ov{\Hb})\\
&&-\frac{1}{2}\tr\Xb\Xi-2\omb\Xi -B-\frac{1}{2}\ov{\Xi}\c\Xbh,\\
\nab_3\Hb -\nab_4\Xib &=&  -\frac{1}{2}\ov{\tr\Xb}(\Hb-H) -\frac{1}{2}\Xbh\c(\ov{\Hb}-\ov{H}) -4\om\Xib+\Bb,\\
\nab_4H -\nab_3\Xi &=&  -\frac{1}{2}\ov{\tr X}(H-\Hb) -\frac{1}{2}\Xh\c(\ov{H}-\ov{\Hb}) -4\omb\Xi-B,
\eeaa
and
\beaa
\nab_3\om+\nab_4\omb -4\om\omb -\xi\c \xib -(\eta-\etab)\c\ze +\eta\c\etab&=&   \rho.
\eeaa
Also,
\beaa
\frac{1}{2}\ov{\DD}\c\Xh +\frac{1}{2}\Xh\c\ov{Z} &=& \frac{1}{2}\DD\ov{\tr X}+\frac{1}{2}\ov{\tr X}Z-i\Im(\tr X)H-i\Im(\tr \Xb)\Xi-B,\\
\frac{1}{2}\ov{\DD}\c\Xbh -\frac{1}{2}\Xbh\c\ov{Z} &=& \frac{1}{2}\DD\ov{\tr\Xb}-\frac{1}{2}\ov{\tr\Xb}Z-i\Im(\tr\Xb)\Hb-i\Im(\tr X)\Xib+\Bb,
\eeaa
and,
\beaa
\curl\ze&=&-\frac 1 2 \chih\wedge\chibh   +\frac 1 4 \big(  \trch\atrchb-\trchb\atrch   \big)+\om \atrchb -\omb\atrch+\dual \rho.
\eeaa
\end{proposition}

The complex notations allow us to rewrite the Bianchi identities as follows.  
  \begin{proposition}\label{prop:bianchi:complex} 
    We have,
 \beaa
 \nab_3A -\DD\hot B &=& -\frac{1}{2}\tr\Xb A+4\omb A +(Z+4H)\hot B -3\ov{P}\Xh,\\
\nab_4B -\frac{1}{2}\ov{\DD}\c A &=& -2\ov{\tr X} B -2\om B +\frac{1}{2}A\c  (\ov{2Z +\Hb})+3\ov{P} \,\Xi,\\
\nab_3B-\DD\ov{P} &=& -\tr\Xb B+2\omb B+\ov{\Bb}\c \Xh+3\ov{P}H +\frac{1}{2}A\c\ov{\Xib},\\
\nab_4P -\frac{1}{2}\DD\c \ov{B} &=& -\frac{3}{2}\tr X P +\frac{1}{2}(2\Hb+Z)\c\ov{B} -\ov{\Xi}\c\Bb -\frac{1}{4}\Xbh\c \ov{A}, \\
\nab_3P +\frac{1}{2}\ov{\DD}\c\Bb &=& -\frac{3}{2}\ov{\tr\Xb} P -\frac{1}{2}(\ov{2H-Z})\c\Bb +\Xib\c \ov{B} -\frac{1}{4}\ov{\Xh}\c\Ab, \\
\nab_4\Bb+\DD P &=& -\tr X\Bb+2\om\Bb+\ov{B}\c \Xbh-3P\Hb -\frac{1}{2}\Ab\c\ov{\Xi},\\
\nab_3\Bb +\frac{1}{2}\ov{\DD}\c\AA &=& -2\ov{\tr\Xb}\,\Bb -2\omb\,\Bb -\frac{1}{2}\Ab\c (\ov{-2Z +H})-3P \,\Xib,\\
\nab_4\Ab +\frac{1}{2}\DD\hot\Bb &=& -\frac{1}{2}\ov{\tr X} \Ab+4\om\Ab +\frac{1}{2}(Z-4\Hb)\hot \Bb -3P\Xbh.
\eeaa
    \end{proposition} 
    
    \begin{proof} 
    We derive the equation for $A$ and $B$. Observe that from the Bianchi identity
    \beaa
      \nab_3\a&=&2\nab\hot \b-\frac 1 2 \big(\trchb\a+\atrchb\dual \a)+4\omb \a+
 2(\ze+4\eta)\hot \b - 3 (\rho\chih+\rhod\dual\chih)
    \eeaa 
    we obtain
    \beaa
     \dual\nab_3\a&=&2\dual(\nab\hot \b)-\frac 1 2 \big(\trchb \dual\a-\atrchb 
     \a)+4\omb \dual \a+
2\dual( (\ze+4\eta)\hot \b )- 3 (\rho \dual\chih -\rhod\chih).
    \eeaa
    This implies
    \beaa
    \nab_3 A&=& \nab_3( \a + i \dual \a)\\
    &=& 2\nab\hot \b+2i \dual(\nab\hot \b)-\frac 1 2 \big(\trchb\a+\atrchb\dual \a)-\frac 1 2 i\big(\trchb \dual\a-\atrchb 
     \a)+4\omb (\a+i\dual a)\\
     &&+2(\ze+4\eta)\hot \b+ 2i\dual( (\ze+4\eta)\hot \b ) - 3 (\rho\chih +\rhod\dual\chih)- 3i (\rho \dual\chih -\rhod\chih)\\
      &=&  \DD \hot (\b + i \dual \b)-\frac 1 2 (\trchb- i \atrchb )
     \a-\frac 1 2 (\trchb- i \atrchb) i\dual\a+4\omb (\a+i\dual a)\\
     &&+ ((\ze+4\eta+ i \dual (\ze +4 \eta))\hot (\b+i \dual \b))  - 3 (\rho- i \rhod)\chih - 3 (\rho- i \rhod) i\dual\chih     \eeaa
     which finally gives
     \beaa
      \nab_3 A      &=&  \DD \hot B-\frac 1 2 \tr\Xb A+4\omb A+ (Z+4H)\hot B  - 3 \ov{P}\hat{X}. 
     \eeaa
     From the equation 
     \beaa
     \nab_4\beta  &=&\div\a-2(\trch\beta-\atrch \dual \b) - 2  \om\b +\a\c  (2 \ze +\etab) + 3  (\xi\rho+\dual \xi\rhod)
     \eeaa
     we obtain
         \beaa
     \dual\nab_4\beta  &=&\dual\div\a-2(\trch \dual\beta+\atrch \b) - 2  \om\dual\b +\dual(\a\c  (2 \ze +\etab) )+ 3  (\dual\xi\rho-\xi\rhod).
     \eeaa
     This implies
     \beaa
     \nab_4B &=& \nab_4(\b +i \dual \b)\\
     &=&\div\a+i\dual\div\a-2(\trch\beta-\atrch \dual \b)-2i(\trch \dual\beta+\atrch \b)  - 2  \om(\b+i \dual\b) \\
     &&+\a\c  (2 \ze +\etab)+i\dual(\a\c  (2 \ze +\etab) ) + 3  (\xi\rho+\dual \xi\rhod)+ 3i  (\dual\xi\rho-\xi\rhod)\\
      &=&\frac 1 2 \ov{\DD} \c (\a + i \dual \a)-2(\trch+i\atrch )\b-2(\trch  +i\atrch )i\dual \b - 2  \om(\b+i \dual\b) \\
     &&+\frac 1 2 (\a+ i \dual \a)\c  (2 \ze +\etab-i\dual(2 \ze +\etab) ) + 3  (\rho-i\rhod) \xi+ 3 (\rho- i \rhod) i\dual\xi 
     \eeaa
     which finally gives
     \beaa
      \nab_4B       &=&\frac 1 2 \ov{\DD} \c A-2\ov{\tr X}B - 2  \om B+\frac 1 2 A\c  (2 \ov{Z} + \ov{\Hb}  ) + 3  \ov{P} \Xi
     \eeaa
     as stated. 
    \end{proof}


\subsection{Main equations in complex form using conformal operators}


\begin{definition}
 We define the following conformal angular derivatives in the complex notation:
 \begin{itemize}
\item For $a+i b \in \SS_0(\mathbb{C}) $  we define
\beaa
\DDc(a+ib) &:=& \big(\nabc +i\dual \nabc\big)(a+ib).
\eeaa

\item For  $f+i \dual f \in\SS_1(\mathbb{C}) $ we define
\beaa
\DDc(f+i\dual f) &:=& \big(\nabc+i\dual \nabc\big) \c (f+i\dual f),
\\
\DDc \hot(f+i\dual f) &:=& (\nabc+i\dual\nabc )\hot(f+i\dual f).
\eeaa
\item For $u+ i \dual u \in \SS_2(\mathbb{C})$ we define
\beaa
\DDc \c(u+i\dual u) &:=& \big(\nabc +i\dual\nabc \big)\c(u+i\dual u).
\eeaa
\item In all the above cases we set
\beaa
\DDbc&:=&\nabc-i\nabc.
\eeaa
\end{itemize}
\end{definition}

These complex notations allow us to rewrite the null structure equations as follows.  
\begin{proposition}
\label{prop-nullstr:complex-conf}
\beaa
\nabc_3\tr\Xb +\frac{1}{2}(\tr\Xb)^2 &=& \DDc\c\ov{\Xib}+\Xib\c\ov{\Hb}+\ov{\Xib}\c H-\frac{1}{2}\Xbh\c\ov{\Xbh},\\
\nabc_3\Xbh+\Re(\tr\Xb) \Xbh&=& \DDc\hot \Xib+   \Xib\hot(H+\Hb)-\Ab,
\eeaa
\beaa
\nabc_3\tr X +\frac{1}{2}\tr\Xb\tr X &=& \DDc\c\ov{H}+H\c\ov{H}+2P+\Xib\c\ov{\Xi}-\frac{1}{2}\Xbh\c\ov{\Xh},\\
\nabc_3\widehat{X} +\frac{1}{2}\tr\Xb\, \widehat{X} &=&\DDc\hot H  +H\hot H -\frac{1}{2}\ov{\tr X} \widehat{\Xb}+\Xib\hot\Xi,
\eeaa
\beaa
\nabc_4\tr\Xb +\frac{1}{2}\tr X\tr\Xb &=& \DDc\c\ov{\Hb}+\Hb\c\ov{\Hb}+2\ov{P}+\Xi\c\ov{\Xib}-\frac{1}{2}\Xh\c\ov{\Xbh},\\
\nabc_4\widehat{\Xb} +\frac{1}{2}\tr X\, \widehat{\Xb} &=& \DDc\hot\Hb  +\Hb\hot\Hb -\frac{1}{2}\ov{\tr\Xb} \widehat{X}+\Xi\hot\Xib,
\eeaa
\beaa
\nabc_4\tr X +\frac{1}{2}(\tr X)^2 &=& \DDc\c\ov{\Xi}+\Xi\c\ov{H}+\ov{\Xi}\c H-\frac{1}{2}\Xh\c\ov{\Xh},\\
\nabc_4\Xh+\Re(\tr X)\Xh&=& \DDc\hot \Xi+  \Xi\hot(\Hb+H)-A,
\eeaa
\beaa
\nabc_3\Hb -\nabc_4\Xib &=&  -\frac{1}{2}\ov{\tr\Xb}(\Hb-H) -\frac{1}{2}\Xbh\c(\ov{\Hb}-\ov{H}) +\Bb,\\
\nabc_4H -\nabc_3\Xi &=&  -\frac{1}{2}\ov{\tr X}(H-\Hb) -\frac{1}{2}\Xh\c(\ov{H}-\ov{\Hb}) -B.
\eeaa
Also,
\beaa
\frac{1}{2}\ov{\DDc}\c\Xh &=& \frac{1}{2}\DDc\ov{\tr X}-i\Im(\tr X)H-i\Im(\tr \Xb)\Xi-B,\\
\frac{1}{2}\ov{\DDc}\c\Xbh &=& \frac{1}{2}\DDc\ov{\tr\Xb}-i\Im(\tr\Xb)\Hb-i\Im(\tr X)\Xib+\Bb.
\eeaa
\end{proposition}
    
The complex notations allow us to rewrite the Bianchi identities as follows.  
  \begin{proposition}\label{prop:bianchi:complex} 
    We have,
 \beaa
 \nabc_3A -\DDc\hot B &=& -\frac{1}{2}\tr\Xb A + 4 H   \hot B -3\ov{P}\Xh,\\
\nabc_4B -\frac{1}{2} \DDbc \c A &=& -2\ov{\tr X} B +\frac{1}{2}A\c \ov{\Hb}+3\ov{P} \,\Xi,\\
\nabc_3B-\DDc\ov{P} &=& -\tr\Xb B+\ov{\Bb}\c \Xh+3\ov{P}H +\frac{1}{2}A\c\ov{\Xib},\\
\nabc_4P -\frac{1}{2}\DDc\c \ov{B} &=& -\frac{3}{2}\tr X P + \Hb \c\ov{B} -\ov{\Xi}\c\Bb -\frac{1}{4}\Xbh\c \ov{A}, \\
\nabc_3P +\frac{1}{2}\DDbc \c\Bb &=& -\frac{3}{2}\ov{\tr\Xb} P - \ov{H} \c\Bb +\Xib\c \ov{B} -\frac{1}{4}\ov{\Xh}\c\Ab, \\
\nabc_4\Bb+\DDc P &=& -\tr X\Bb+\ov{B}\c \Xbh-3P\Hb -\frac{1}{2}\Ab\c\ov{\Xi},\\
\nabc_3\Bb +\frac{1}{2}\DDbc \c\AA &=& -2\ov{\tr\Xb}\,\Bb  -\frac 1 2  \Ab\c \ov{H}-3P \,\Xib,\\
\nabc_4\Ab +\frac{1}{2}\DDc\hot\Bb &=& -\frac{1}{2}\ov{\tr X} \Ab - 2 \Hb\hot \Bb -3P\Xbh.
\eeaa
    \end{proposition}


\section{Invariant wave operators}


Recall that given a  horizontal   structure $\O(\MM)$ and  $X, Y \in \O(\MM) $ we have defined $\nab_X Y$ to be the  horizontal projection of $D_X Y$.  We extend this definition to $X\in \T(\MM)$  and $Y\in\O(\MM)$ as follows

\begin{definition}
Given  $X\in \T(\MM)$ and $Y \in \O(\MM)$ we define,
\beaa
\Ddot_X Y&:=& ( \D_X Y)^\perp.
\eeaa
Given an orthonormal  frame  $e_1, e_2\in \O(\MM)$ we write 
\beaa
\Ddot_\mu  e_a&=&\sum_{b=1,2}(\La_\mu)_{ab}\, e_b,  \qquad  (\La_\mu)_{\a\b}:=\g(\D_\mu e_\b, e_\a).
\eeaa
\end{definition} 

 \begin{definition}
 \label{definition:S-covariantderivative}
 Given a general, covariant,  $S$- horizontal tensor-field  $U$
 we define its horizontal covariant derivative according to
 the formula,
 \bea
 \Ddot_X U(Y_1,\ldots Y_k)=X (U(Y_1,\ldots Y_k))&-&U(\Ddot_X Y_1,\ldots Y_k)-\ldots- U(Y_1,\ldots \Ddot_XY_k)
 \eea
 where $X\in \T\MM$ and $Y_1,\ldots Y_k\in \T_S\MM$.
 \end{definition}
 
   \begin{proposition}
 For  all  $X\in\T\MM$   and $Y_1, Y_2 \in \T_S\MM$,
 \beaa
 X h (Y_1,Y_2)= h (\Ddot_X Y_1, Y_2)+ h(Y_1, \Ddot_X Y_2).
 \eeaa
  \end{proposition}
  
  \begin{proof}
  Indeed,
  \beaa
 X h (Y_1,Y_2)&=&X \g (Y_1,Y_2)=\g (\D_X Y_1,Y_2)+\g ( Y_1,\D_XY_2)= \g (\Ddot_X Y_1,Y_2)+\g ( Y_1,\Ddot_XY_2)\\
&=&  h (\Ddot_X Y_1,Y_2)+h ( Y_1,\Ddot_XY_2).
  \eeaa
   \end{proof}


\subsection{Mixed tensors}


We consider tensors  $\T^k \MM\otimes   \O \MM  $, i.e. tensors  of the form  $U_{\mu_1\ldots \mu_k,  A_1\ldots A_L}$
for which we define,
\beaa
\Ddot_\mu U_{\nu_1\ldots \nu_k,  A_1\ldots A_L}&=& e_\mu U_{\nu_1\ldots \nu_k,  A_1\ldots A_l}-U_{\D_\mu\nu_1\ldots \nu_k,  A_1\ldots A_l}-\ldots- U_{\nu_1\ldots  \D_\mu\nu_k,  A_1\ldots A_l}\\
&-& U_{\nu_1\ldots \nu_k,   \Ddot_\mu A_1\ldots A_l}-  U_{\nu_1\ldots \nu_k,   A_1 \ldots \Ddot_\mu A_l}.
\eeaa
We are now ready   to prove the following,
\begin{proposition}
We   have  the curvature formula
 \beaa
( \Ddot _\mu\Ddot_\nu -\Ddot_\nu\Ddot _\mu)\Psi_A=\R_{A}\, ^   B\,_{   \mu\nu}\Psi_B.
 \eeaa
 More  generally for a mixed tensor $\Psi_{\la A}$
  \beaa
( \Ddot _\mu\Ddot_\nu -\Ddot_\nu\Ddot _\mu)\Psi_{\la A}=    \R_{\la }\, ^   \si \,_{   \mu\nu}\Psi_{\si A}+    
            \R_{A}\, ^   B\,_{   \mu\nu}\Psi_{\la B}.
 \eeaa
 \end{proposition}

\begin{proof}
Straightforward verification. 
\end{proof}

Observe that all the definitions above hold true for both real and complex tensors. 


\subsection{Invariant lagrangian}


Recall from Definition \ref{def:proxyfirstandsecondfundameentalform} that $\ga$ is a horizontal 2-tensor with  $\ga_{AB}=\g(e_A, e_B)$. We introduce for a complex tensor $\Psi \in \SS_k(\mathbb{C})$
 \beaa
 \LL&=& \g^{\mu\nu} \ga_{AB}\Ddot_\mu  \Psi^A \Ddot_\mu  \ov{\Psi}^B  +V \ga_{AB} \Psi^A \ov{\Psi}^B.
 \eeaa
 for $V$ real. 
\begin{proposition}
The Euler Lagrange equations are given by:
\beaa
\squared\Psi^A= V \Psi^A
\eeaa
where   $\squared\Psi^A:= \g^{\mu\nu} \Ddot_\mu\Ddot_ \nu \Psi^A.$
 \end{proposition}
 
\begin{proof}
The variation of the action is given by,
\beaa
0&=&\int_\M  \ga_{AB}\left(  \g^{\mu\nu} \Ddot_\mu   \Psi^A   \Ddot_ \nu( \de\ov{\Psi})^B+ \g^{\mu\nu} \Ddot_\mu   \ov{\Psi}^A   \Ddot_ \nu( \de\Psi)^B +V\Psi^A \de\ov{\Psi}^B+V\ov{\Psi}^A \de\Psi^B\right) dv_\g\\
&=&\int_\M   \Ddot_ \nu\left ( \g^{\mu\nu}  \ga_{AB}\Ddot_\mu   \ov{\Psi}^A   ( \de\Psi)^B \right) dv_\g+\int_\M   \Ddot_ \nu\left ( \g^{\mu\nu}  \ga_{AB}\Ddot_\mu   \Psi^A   ( \de\ov{\Psi})^B \right) dv_\g\\
&&- \int_\M  \ga_{AB} \left ( \g^{\mu\nu}  \Ddot_ \nu \Ddot_\mu  \ov{ \Psi}^A   ( \de\Psi)^B  -V\ov{\Psi}^A \de\Psi^B+ \g^{\mu\nu}  \Ddot_ \nu \Ddot_\mu   \Psi^A   ( \de\ov{\Psi})^B  -V\Psi^A \de\ov{\Psi}^B \right)dv_\g\\
&=&- \int_\M  \ga_{AB} \left ( \g^{\mu\nu}  \Ddot_ \nu \Ddot_\mu  \ov{ \Psi}^A   ( \de\Psi)^B  -V\ov{\Psi}^A \de\Psi^B \right)dv_\g\\
&&- \int_\M  \ga_{AB} \left ( \g^{\mu\nu}  \Ddot_ \nu \Ddot_\mu   \Psi^A   ( \de\ov{\Psi})^B  -V\Psi^A \de\ov{\Psi}^B \right)dv_\g
\eeaa
from which the proposition follows.
\end{proof}

\begin{definition}
We introduce  the   energy-momentum tensor
 \beaa
 \QQ_{\mu\nu}:&=&\frac 1 2 \Ddot_\mu  \Psi \c \Ddot _\nu \ov{\Psi} +\frac 1 2 \Ddot_\mu  \ov{\Psi} \c \Ddot _\nu \Psi 
          -\frac 12 \g_{\mu\nu} \left(\Ddot_\la \Psi\c\Ddot^\la \ov{\Psi} +  V \Psi \c \ov{\Psi}\right)
 \eeaa
 where  the dot product   here denotes full contraction with respect to the  horizontal indices, i.e.
 $\Ddot_\mu  \Psi \c \Ddot _\nu \ov{\Psi}= h^{ac} h^{bd} \Ddot_\mu  \Psi _{ab} \c \Ddot _\nu \ov{\Psi}_{cd}$.
 \end{definition}

\begin{lemma}
We have,
  \beaa
 \D^\nu\QQ_{\mu\nu}  &=& \frac 1 2 \Ddot_\mu  \Psi \c \left(\squared \ov{\Psi}-V\ov{\Psi}\right)+\frac 1 2 \Ddot_\mu  \ov{\Psi} \c \left(\squared \Psi-V\Psi\right)\\
 && +\frac 1 2 \Ddot ^\nu \ov{\Psi}^A \R_{AB\nu\mu} \Psi^B  +\frac 1 2  \Ddot ^\nu \Psi^A \R_{AB\nu\mu} \ov{\Psi}^B   -\frac 12  \D_{\mu} V  \Psi \c \ov{\Psi}
 \eeaa
\end{lemma}

\begin{proof}
We have,
\beaa
 \D^\nu\QQ_{\mu\nu}&=&\frac 1 2 \Ddot^\nu \Ddot_\nu  \Psi \c \Ddot _\mu \ov{\Psi}+\frac 1 2 \Ddot^\nu \Ddot_\nu  \ov{\Psi} \c \Ddot _\mu \Psi\\
 &&+ \frac 1 2  \Ddot^\nu  \Psi\c  \left( \Ddot_\nu \Ddot _\mu -\Ddot_\mu \Ddot _\nu \right)\ov{\Psi}+ \frac 1 2  \Ddot^\nu  \ov{\Psi}\c  \left( \Ddot_\nu \Ddot _\mu -\Ddot_\mu \Ddot _\nu \right)\Psi\\
 &&-\frac 1 2 V\D_\mu \Psi \c \ov{\Psi}   -\frac 1 2 V\D_\mu \ov{\Psi} \c \Psi    -\frac 1 2 \D_\mu V \Psi\c\ov{\Psi}
\\
&=&\frac 1 2  \Ddot_\mu \ov{\Psi}\left(\squared \Psi- V\Psi\right)+\frac 1 2 \Ddot_\mu \Psi\left(\squared \ov{\Psi}- V\ov{\Psi}\right)+ \frac 1 2  \Ddot^\nu  \Psi^A\R_{ A   B   \nu\mu}\ov{\Psi}^B+ \frac 1 2  \Ddot^\nu  \ov{\Psi}^A\R_{ A   B   \nu\mu}\Psi^B\\
 &&  -\frac 1 2 \D_\mu V \Psi\c\ov{\Psi}
\eeaa
 as desired.
\end{proof}


\subsection{Standard calculation}\lab{sec:standardcalculationvectorfieldmethodwaveequationpsi}


 \begin{proposition}\label{prop-app:stadard-comp-Psi}
 Consider $\Psi\in \SS_2(\mathbb{C})$ and   $X$ a vectorfield on $\MM$.       
 \begin{enumerate}
\item 
 The $1$-form   $\PP_\mu=\QQ_{\mu\nu} X^\nu$   verifies,
\beaa
\D^\mu \PP_\mu&=& \frac 1 2 X^\mu \Ddot_\mu  \ov{\Psi} \c\left(\squared  \Psi-V\Psi\right) +\frac 1 2  X^\mu \Ddot_\mu  \Psi \c\left(\squared \ov{ \Psi}-V\ov{\Psi}\right) 
- X( V )  \Psi\c \ov{\Psi}.
\eeaa
\item
   Let $X$ as above,   $w$ a scalar   and $M$  a one form. Define,
 \beaa
 \PP_\mu &=&\PP_\mu[X, w, M]=\QQ_{\mu\nu} X^\nu +\frac 1 4  w \ov{\Psi} \c \Ddot_\mu \Psi +\frac 1 4  w \Psi \c \Ddot_\mu \ov{\Psi} -\frac 1 4|\Psi|^2   \pr_\mu w +\frac 1 4 |\Psi|^2 M_\mu.
  \eeaa
  with $|\Psi|^2:=\Psi\c\ov{\Psi}$. Then,
 \beaa
  \D^\mu  \PP_\mu[X, w, M] &=& \frac 1 2 \QQ  \c\piX - \frac 1 2 X( V )  \Psi\c \ov{\Psi}+\frac 12  w \LL[\Psi] -\frac 1 4|\Psi|^2   \square_\g  w\\
   &+&\frac 1 4 |\Psi|^2 \Div M+\frac 1 4 \ov{\Psi}\c \Ddot_\mu\Psi \,  M^\mu+\frac 1 4 \Psi\c \Ddot_\mu\ov{\Psi} \,  M^\mu\\
   &&+\frac 1 2  \left(X( \ov{\Psi} )+\frac 1 2   w \ov{\Psi}\right)\c \left(\squared  \Psi-V\Psi\right) + \frac 1 2 \left(X( \Psi )+\frac 1 2   w \Psi\right)\c \left(\squared  \ov{\Psi}-V\ov{\Psi}\right) \\
   &+&\frac 1 2 X^\mu \Ddot^\nu  \ov{\Psi} ^A\R_{ A   B   \nu\mu}\Psi^B+\frac 1 2 X^\mu \Ddot^\nu  \Psi ^A\R_{ A   B   \nu\mu}\ov{\Psi}^B.
 \eeaa
 \end{enumerate}
\end{proposition}
\begin{proof} See Chapter 10 in \cite{KS}.
\end{proof}


\subsection{Canonical form of the wave equation}


 Consider the wave equation for  $\SS_2$-tensors
 \beaa
 \squared\Psi_{ab}:= \g^{\mu\nu} \Ddot_\mu\Ddot_ \nu \Psi_{ab}.
\eeaa

\begin{lemma}
We have, for $\Psi\in \SS_2(\mathbb{C})$
\begin{enumerate}
\item We have
\bea
\bsplit
\square_2 \Psi&=-\frac 1 2 \big(\nab_3\nab_4\Psi+\nab_4 \nab_3 \Psi\big)+\lap_2 \Psi +\left(\omb -\frac 1 2 \trchb\right) \nab_4\Psi+\left(\om -\frac 1 2 \trch\right) \nab_3\Psi \\
&+ (\eta+\etab) \c\nab \Psi.
\end{split}
\eea
\item We have
\bea
\,[\nab_3,  \nab_4 ] \Psi &=& 4\rhod \dual \Psi- 2\om \nab_3 \Psi+ 2\omb \nab_4 \Psi- 2(\etab-\eta)\c\nab \Psi.
\eea

\item 
We also have
\bea\label{first-equation-square}
\bsplit
\square_2 \Psi&=&-\nab_4 \nab_3 \Psi +\lap_2\Psi+\left(2\om -\frac 1 2 \trch\right) \nab_3\Psi- \frac 1 2 \trchb \nab_4\Psi+2 \etab \c \nab  \Psi-2\rhod \dual \Psi.
\end{split}
\eea
\end{enumerate}
\end{lemma}

\begin{proof}
By definition
\beaa
\squared_2 \Psi_{ab}&=& \g^{34} \Ddot_3\Ddot_4 \Psi_{ab} + \g^{43} \Ddot_4\Ddot_3 \Psi_{ab}+ \g^{cd}\Ddot_c\Ddot_d \Psi_{ab}.
\eeaa
We write
\beaa
\Ddot_4 \Psi_{ab}&=& \nab_4 \Psi_{ab},\\
\Ddot_3 \Ddot_4 \Psi_{ab}&=& \nab_3 \nab_4 \Psi_{ab} - 2\omb \nab_4 \Psi_{ab}- 2\eta\c  \nab \Psi_{ab}, \\
\Ddot_4 \Ddot_3 \Psi_{ab}&=& \nab_4 \nab_3 \Psi_{ab} - 2\om \nab_3 \Psi_{ab}- 2\etab\c  \nab \Psi_{ab}, \\
\Ddot_d\Psi_{ab}&=&\nab_d\Psi_{ab},\\
\Ddot_c \Ddot_d \Psi_{ab}&=&\nab_c\nab_d \Psi_{ab} - \frac 1 2 \chi_{cd}\nab_3 \Psi_{ab} -\frac 1 2 \chib_{cd}\nab_4 \Psi_{ab}. 
\eeaa
Hence
\beaa
\squared_2 \Psi_{ab}&=&-\frac 1 2  \Ddot_3\Ddot_4 \Psi_{ab}    -\frac 1 2  \Ddot_4\Ddot_3 \Psi_{ab}     + \g^{cd}\Ddot_c\Ddot_d \Psi_{ab}\\
&=&-\frac 1 2 \big(\nab_3\nab_4\Psi_{ab}+\nab_4 \nab_3 \Psi_{ab}\big)+\g^{cd}\left(\nab_c\nab_d \Psi_{ab} - \frac 1 2 \chi_{cd}\nab_3 \Psi_{ab} -\frac 1 2 \chib_{cd}\nab_4 \Psi_{ab} \right)\\
& + &\omb \nab_4 \Psi_{ab}+\eta\c  \nab \Psi_{ab} +\omb\nab_3 \Psi_{ab}+\etab\c  \nab \Psi_{ab}\\
&=&-\frac 1 2 \big(\nab_3\nab_4\Psi_{ab}+\nab_4 \nab_3 \Psi_{ab}\big)+\lap_2 \Psi_{ab}- \frac 1 2 \trch\nab_3\Psi_{ab} -\frac1  2 \trchb \nab_4\Psi_{ab}\\
& + &\omb \nab_4 \Psi_{ab}+\eta\c  \nab \Psi_{ab} +\omb\nab_3 \Psi_{ab}+\etab\c  \nab \Psi_{ab}.
\eeaa
Hence
\beaa
\square_2 \Psi&=& -\frac 1 2 \big(\nab_3\nab_4\Psi_{ab}+\nab_4 \nab_3 \Psi_{ab}\big)+\lap_2 \Psi_{ab} +\left(\omb -\frac 1 2 \trchb\right) \nab_4\Psi+\left(\om -\frac 1 2 \trch\right) \nab_3\Psi \\
&&+(\eta+\etab) \c\nab \Psi.
\eeaa
To prove the second statement we recall 
\beaa
( \Ddot _\mu\Ddot_\nu -\Ddot_\nu\Ddot _\mu)\Psi_{ab}&=&\R_{a}\, ^   c\,_{   \mu\nu}\Psi_{cb}+\R_{b}\, ^   c\,_{   \mu\nu}\Psi_{ac}.
\eeaa
Hence,
\beaa
( \Ddot _4 \Ddot_3  -\Ddot_ 3 \Ddot _4)\Psi_{ab}&=&\R_{a}\, ^   c\,_{   43}\Psi_{cb}+\R_{b}\, ^   c\,_{   43}\Psi_{ac}= - 2 \in_{ac}\rhod \Psi_{cb}- 2 \in_{bc}\rhod \Psi_{ac} \\
&=&-4\rhod \dual \Psi_{ab}.
\eeaa
We deduce,
\beaa
  -4\rhod \dual \Psi_{ab}&=&     ( \Ddot _4 \Ddot_3  -\Ddot_ 3 \Ddot _4)\Psi_{ab}  \\
  &     =  &     (\nab_4 \nab_3 -\nab_3\nab_4\psi) \Psi_{ab}  +2\omb \nab_4 \Psi_{ab}+2\eta\c  \nab \Psi_{ab}- 2\om \nab_3 \Psi_{ab}-2\etab\c  \nab \Psi_{ab}
\eeaa
and thus
\beaa
\,[\nab_4,  \nab_3 ] \Psi &=& -4\rhod \dual \Psi+ 2\om \nab_3 \Psi- 2\omb \nab_4 \Psi+ 2(\etab-\eta)\c\nab \Psi
\eeaa
as stated. 
\end{proof}

    
\section{Perturbations of Kerr}
\lab{section:Kerr}


Before discussing perturbations of Kerr, we first provide basic facts concerning Kerr in sections \ref{sec:recallbasicthingsinKerr} and \ref{sec:recallbasicthingsinKerrbis}.


\subsection{The Kerr metric}
\lab{sec:recallbasicthingsinKerr}


We consider the Kerr metric in standard Boyer-Lindquist coordinates 
$$\g=-\frac{|q|^2\Delta}{\Sigma^2}(dt)^2+\frac{\Sigma^2(\sin\theta)^2}{|q|^2}\left(d\varphi-\frac{2amr}{\Sigma^2}dt\right)^2+\frac{|q|^2}{\Delta}(dr)^2+|q|^2(d\theta)^2,$$
where 
\bea\label{definition-q}
q=r+ i a \cos\th
\eea
and
$$
\left\{\ba{lll}
\Delta &=& r^2-2mr+a^2,\\
|q|^2 &=& r^2+a^2(\cos\theta)^2,\\
\Sigma^2 &=& (r^2+a^2)|q|^2+2mra^2(\sin\theta)^2=(r^2+a^2)^2-a^2(\sin\theta)^2\Delta.
\ea\right.
$$

The null frame is given by
\bea\label{null-frames}
\begin{split}
& e_4=\frac{r^2+a^2}{\Delta}\pr_t+\pr_r+\frac{a}{\Delta}\pr_\varphi,\quad e_3=\frac{r^2+a^2}{|q|^2}\pr_t-\frac{\Delta}{|q|^2}\pr_r+\frac{a}{|q|^2}\pr_\varphi,\\ & e_1=\frac{1}{|q|}\pr_\th,\quad e_2=\frac{a\sin\th}{|q|}\pr_t+\frac{1}{|q|\sin\th}\pr_\varphi.
\end{split}
\eea

\begin{remark}
There is an indeterminacy in the principal null frame as one may replace the null pair $(e_3, e_4)$ with $(\la^{-1}e_3, \la e_4)$ for any $\la>0$. In this section, the formulas correspond to the arbitrary choice of $\la>0$ ensuring $\nab_4e_4=0$ and thus $\om=0$. Note that our main result, on generalized Regge-Wheeler   equation,  is independent of this choice.
\end{remark}

The complex Ricci coefficients are given by  
\beaa
&& \Xh=\Xbh=\Xi=\Xib=\om=0, \qquad  \Hb=-Z,
\eeaa
\beaa
&&\tr X=\frac{2}{q}, \qquad\qquad \tr\Xb=-\frac{2\Delta q}{|q|^4},\\
&&H_1=\frac{ai\sin\th\, q}{|q|^3}, \qquad H_2=\frac{a\sin\th \, q}{|q|^3}, \\
&& Z_1=\frac{ai\sin\th\, \ov{q}}{|q|^3},\qquad\,\, Z_2=\frac{a\sin\th\,\ov{q}}{|q|^3},\\
&& \omb=\frac{a^2\cos^2\th (r-m)+mr^2-a^2r}{|q|^4}.
\eeaa

\begin{remark}
Note the identities 
\beaa
H_1=-\ov{Z_1}, \quad H_2=\ov{Z_2}
\eeaa
and also
\beaa
H_1=\ov{\Hb_1}, \quad H_2=-\ov{\Hb_2}.
\eeaa
\end{remark}

The complex curvature components are given by
\beaa
A=B=\Bb=\Ab=0,\quad P=-\frac{2m}{q^3}.
\eeaa


\subsection{Equations for $q$ in Kerr}
\lab{sec:recallbasicthingsinKerrbis}


Recall the definition \eqref{definition-q} of $q$, $q=r+ i a \cos\th$. We have the following equations. 
\begin{lemma} 
The scalar $q$ satisfies
\bea\lab{equations:forq}
\nab_4 q&= \frac 1 2 \tr X q,  \qquad  \nab_3 q=\frac 1 2 \ov{\tr \Xb }  q, \qquad  \DD q= q  \Hb , \qquad   \DDb q=q \ov{H}.
\eea
Also
\bea
\lab{eq:qHb=ovqH}
q \Hb=-\ov{q} H.
\eea
In particular $ |H|^2=|\Hb|^2 
$.
\end{lemma}

\begin{proof} 
From the value of $\tr X=\frac{2}{q}$, and the reduced equation $\nab_4\tr X +\frac{1}{2}(\tr X)^2 =0$
we deduce the equation for $\nab_4 q$. From the value of $P=-\frac{2m}{q^3}$ and the reduced Bianchi identity $\nab_3P  = -\frac{3}{2}\ov{\tr\Xb} P$ we deduce the equation for $\nab_3 q$. 
Similarly, the equation $ \DD\ov{P} =- 3\ov{P}H $ becomes $\DD\ov{q} =H\ov{q}$
  or $ \DDb q=q \ov{H}$. The last equation in \eqref{equations:forq}  follows in the same manner from  $\DD P = -3P\Hb$.
    
 To derive \eqref{eq:qHb=ovqH} we write, recalling that $q=r+ia \cos\th$ and using $e_a(r)=0$, i.e. $\DD r=0$, 
  \beaa
   q  \Hb&=& \DD q=\DD( r+i a \cos\th) =\nab(r+i a \cos \th)+ i \dual \nab(r+i a \cos\th)\\
   &=&ia \nab\cos\th -a \dual\nab \cos\th, \\
   q \ov{H}&=&\DDb q =\DDb( r+i a \cos\th)  =\nab(r+i a \cos \th)- i \dual \nab(r+i a \cos\th)\\
   &=& ia \nab\cos\th +a \dual\nab \cos\th.
  \eeaa
   We deduce $ \ov{ q  \Hb}= - q \ov{H}$ i.e. $  q\Hb=-\ov{q} H$  as stated.
 \end{proof}

\begin{lemma}\label{derivatives-tr} 
In Kerr we have
\beaa
\nab (\trchb)&=&  -\frac 3 2\trchb    (\etab+ \eta) -\frac 1 2\atrchb  (\dual \eta-  \dual \etab ),  \\
\nab(\atrchb)&=& -\frac 3 2  \atrchb  (\etab+ \eta) +\frac 1 2\trchb  (\dual \eta-  \dual \etab ).
\eeaa
\end{lemma}

\begin{proof} 
Using that $\nab q=\frac 1 2 q (\Hb+ \ov{H})$, and since $\tr\Xb= -\frac{2\Delta}{q \ov{q}^2}$ we have
\beaa
\nab \tr\Xb &=&  \frac{2\Delta}{q^2 \ov{q}^2}\nab q  +\frac{4\Delta}{q \ov{q}^3}\nab\ov{q}\\
&=&  \frac{2\Delta}{q \ov{q}^2}\frac 1 2  (\Hb+ \ov{H})  +\frac{2\Delta}{q \ov{q}^2} (\ov{\Hb}+ H)\\
&=&  \frac{2\Delta}{q \ov{q}^2}  (\frac 1 2\Hb+ \frac 1 2\ov{H}+\ov{\Hb}+ H)\\
&=& -\tr\Xb  \left(\frac 1 2\Hb+ \frac 1 2\ov{H}+\ov{\Hb}+ H\right).
\eeaa
We compute
\beaa
&& -\tr\Xb  \left(\frac 1 2\Hb+ \frac 1 2\ov{H}+\ov{\Hb}+ H\right)\\
&=&  -(\trchb - i \atrchb)  \left(\frac 1 2(\etab+ i \dual \etab)+ \frac 1 2(\eta-i \dual \eta)+\etab - i \dual \etab+ \eta + i \dual \eta\right)\\
&=&  -(\trchb - i \atrchb)  \left(\frac 3 2 \etab +\frac 3 2 \eta+i \frac 1 2 \dual \eta-i \frac 1 2  \dual \etab \right).
\eeaa
Since $\nab \tr\Xb= \nab \trchb - i \nab\atrchb$, and each term is real, we obtain the lemma. 
\end{proof}

\begin{lemma} 
In Kerr we have
\bea\label{relation=tr-eta}
 \frac 1 2\trchb  (  \eta +\etab)-\frac 1 2   \atrchb  ( \dual \eta-\dual\etab )=0.
\eea
\end{lemma}

\begin{proof}
Recall that 
\beaa
 \trchb=-\frac{2r\Delta}{|q|^4}, \quad \atrchb=\frac{2a\Delta\cos\th}{|q|^4},
 \eeaa
 and 
\beaa
&&H_1=-\frac{a^2\sin\th \cos\th}{|q|^3}+i\frac{a\sin\th r}{|q|^3}, \qquad H_2=\frac{a\sin\th r}{|q|^3}+i\frac{a^2\sin\th \cos\th}{|q|^3}, \\
&& \Hb_1=-\frac{a^2\sin\th \cos\th }{|q|^3}-i\frac{a\sin\th(r)}{|q|^3},\qquad \Hb_2=-\frac{a\sin\th r}{|q|^3}+i\frac{a^2\sin\th \cos\th}{|q|^3}.
\eeaa
Therefore
\beaa
\eta_1&=& -\frac{a^2\sin\th \cos\th}{|q|^3}, \qquad \eta_2=\frac{a\sin\th r}{|q|^3}, \\
\dual \eta_1&=& \frac{a\sin\th r}{|q|^3}, \qquad \dual \eta_2= \frac{a^2\sin\th \cos\th}{|q|^3},\\
\etab_1&=& -\frac{a^2\sin\th \cos\th }{|q|^3}, \qquad \etab_2 =-\frac{a\sin\th r}{|q|^3},\\
\dual \etab_1&=& -\frac{a\sin\th(r)}{|q|^3}, \qquad \dual \etab_2 =\frac{a^2\sin\th \cos\th}{|q|^3}.
\eeaa
We evaluate the above expression in $1$:
\beaa
&&  -\frac 1 2\trchb    \eta_1 -\frac{1}{2}\trchb\, \etab_1+\frac 1 2   \atrchb  ( \dual \eta_1) -\frac{1}{2}\atrchb\dual\etab_1 \\
&=& - \frac 1 2\frac{2r\Delta}{|q|^4}  \frac{a^2\sin\th \cos\th}{|q|^3} -\frac{1}{2}\frac{2r\Delta}{|q|^4} \frac{a^2\sin\th \cos\th }{|q|^3}\\
&&+\frac 1 2  \frac{2a\Delta\cos\th}{|q|^4}  \left( \frac{a\sin\th r}{|q|^3}\right) +\frac{1}{2}\frac{2a\Delta\cos\th}{|q|^4} \frac{ar\sin\th}{|q|^3}=0.
\eeaa
We evaluate the expression in $2$:
\beaa
&&  -\frac 1 2\trchb    \eta_2 -\frac{1}{2}\trchb\, \etab_2+\frac 1 2   \atrchb  ( \dual \eta_2) -\frac{1}{2}\atrchb\dual\etab_2 \\
&=&  \frac 1 2\frac{2r\Delta}{|q|^4} \frac{a\sin\th r}{|q|^3} -\frac{1}{2}\frac{2r\Delta}{|q|^4} \frac{a\sin\th r}{|q|^3}\\
&&+\frac 1 2  \frac{2a\Delta\cos\th}{|q|^4}  \left( \frac{a^2\sin\th \cos\th}{|q|^3}\right) -\frac{1}{2}\frac{2a\Delta\cos\th}{|q|^4} \frac{a^2\sin\th \cos\th}{|q|^3}=0.
\eeaa
This proves the lemma. 
\end{proof}


\subsection{Perturbations of Kerr}\label{perturbations-section}


Recall that in Kerr the Ricci coefficients  $\Xh$, $\Xbh$, $\Xi$, $\Xib$ and the curvature components $A$, $B$, $\Bb$, $\Ab$ vanish identically. We therefore expect that in perturbations of Kerr these quantities stay small, i.e. of order $O(\ep)$ for a sufficiently small $\ep$.

\begin{definition}
We say that a spacetime $\MM$ is  an $O(\ep)$ perturbation of   $Kerr(a, m)$   if  $\MM$ comes equipped with two functions $r:\MM\longrightarrow (0, \infty), \,\,\th: \MM\longrightarrow [0, \pi] $
 and a null frame $(e_3, e_4, e_1, e_2)$ such that the following conditions are verified.
 \begin{enumerate}
 \item  The  connection  and curvature coefficients  verify
 \beaa
 \Xi, \quad \Xib, \quad \Xh, \quad \Xbh, \quad A, \quad B,\quad \Bb, \quad \Ab  &=& O(\ep), 
 \\
 \widecheck{H}, \quad \widecheck{\Hb},\quad  \widecheck{Z}, \quad   \widecheck{\tr X}, \quad  \widecheck{\tr \Xb}, \quad \widecheck{\om}, \quad \widecheck{\omb}, \quad \widecheck{P}&=& O(\ep),
\eeaa
where, given a quantity $Q$,   we denoted by $\widecheck{Q}  $   the renormalized  quantity $\widecheck{Q}= Q-Q_{Kerr} $ where $Q_{Kerr} $ denotes the value of $Q $  expressed in terms of the variables  $r$ and $\th$.

\item  The complex scalar $q= r+i a \cos\th$ verifies
\beaa
\nab_4 q- \frac 1 2 \tr X q=O(\ep) , \quad \nab_3 q-\frac 1 2 \ov{\tr \Xb }  q= O(\ep), \quad 
\DD q- q  \Hb= O(\ep),   \quad  \DDb q-q \ov{H}= O(\ep).
\eeaa
 \end{enumerate}
\end{definition}

We introduce a schematic notation, similar  to the one used in \cite{KS}, see  Definition 2.3.8,  to keep track of the error terms. We divide the connection coefficient terms into
\bea\label{definition-Gammas}
\bsplit
\Ga_g^{(0)}&=\left\{ r \Xi, \quad \Xh, \quad \widecheck{Z}, \quad \widecheck{\Hb}, \quad \frac 1 r\nab_4 q- \frac {1}{ 2r} \tr X q, \quad  \frac 1 r  (\DD q- q  \Hb),  \quad  \frac 1 r ( \DDb q-q \ov{H})\right\}, \\
\Ga_b^{(0)}&=\left\{ \widecheck{H},\quad  \Xbh,\quad  \Xib, \quad  \frac 1 r \nab_3 q-\frac {1}{ 2r} \ov{\tr \Xb }  q\right\}.
\end{split}
\eea
For higher derivatives we denote 
\beaa
\Ga_g^{(s)}&=& \dk^{\leq s} \Ga_g, \qquad \Ga_b^{(s)}= \dk^{\leq s} \Ga_b,
\eeaa
where
\beaa
\dk&=& \{\nab_3, r\nab_4, r\DD \}.
\eeaa

For a $O(\ep)$-perturbation of Kerr we can also define the following two vectorfields:
\bea
\T&=& \frac 1 2 e_3+\frac{\Delta}{2|q|^2} e_4-\frac{a \sin\th}{|q|} e_2, \label{definition-T-vectorfield}\\
 \Z&=&\left(|q|\sin\th +\frac{a^2 \sin^3\th}{|q|}\right) e_2-\frac 1 2a\sin^2\th e_3-\frac{a\sin^2\th\Delta}{2|q|^2} e_4.\label{definition-Z-vectorfield}
\eea
In Kerr spacetime we have $\T=\partial_t$ and $\Z=\partial_\vphi$.


\subsection{Commutation formulas for perturbations of Kerr}


In this section, we collect commutation formulas where   the error terms are written schematically by making use of  the notation introduced in \eqref{definition-Gammas}.


  \subsubsection{Commutation formulas in complex form}
   

   \begin{lemma}
   The following commutation formulas hold true.
   \begin{enumerate}
   \item
    Let $F=f+i\dual f \in \SS_1 (\mathbb{C})$.  Then
\bea
 \lab{commutator-nab-43-D-hot} 
\bsplit
\, [\nab_4, \mathcal{D}\hot ]F&=- \frac 1 2 \tr X( \mathcal{D}\hot F + \underline{H} \hot F)+ (\underline{H}+Z) \hot \nab_4 F+\err_{4\DD \hot}[F], \\
\, [\nab_3, \mathcal{D}\hot] F&=- \frac 1 2 \tr \Xb( \mathcal{D}\hot F + H \hot F)+ (H-Z) \hot \nab_3 F+\err_{3\DD \hot}[F],
\end{split}
\eea
where 
\beaa
\err_{4 \DD\hot}[F]&=&  r^{-1} \Ga_g \c  \dk^{\leq 1} F+\lot, \qquad \err_{3 \DD\hot}[F]= \Ga_g \c \dk^{\leq 1} F+ \lot
\eeaa

\item Let $U=u+i \dual u\in \SS_2(\mathbb{C})$. Then
\bea
 \lab{commutator-nab-4-D-c} 
\bsplit
\, [\nab_3, \nab_4] U&=-2 \om \nab_3  U +2\omb \nab_4 U + 2(\eta-\etab ) \c  \nab  U - 4  \eta \hot (\etab \c U))+4 \etab \hot (\eta \c U) +4 i  \rhod \, U \\
&+ \err_{34}[U]
 \end{split}
 \eea
 where
 \beaa
 \err_{34}[U]&=& \left(\Ga_g  \c \Ga_g \right) U = \lot
 \eeaa

 \item  For $U\in \SS_2(\mathbb{C})$
 \bea
 \bsplit
\, [\nab_4, \ov{\DD}\c] U&=- \frac 1 2\ov{\tr X}\, ( \ov{\DD} \c U - 2 \ov{\Hb} \c U)+\ov{(\Hb+Z)} \c \nab_4 U +\err_{4\ov{\DD}}[U] \\
\, [\nab_3, \ov{\DD}\c] U&=- \frac 1 2\ov{\tr\Xb}\, ( \ov{\DD} \c U -  2 \ov{H} \c U)+\ov{(H-Z)} \c \nab_3 U +\err_{3\ov{\DD}}[U] 
\end{split}
\eea
where
\beaa
\err_{4\ov{\DD}}[U] &=& r^{-1} \Ga_g  \c \dk^{\leq 1} U+\lot , \qquad \err_{3\ov{\DD}}[U]=  \Ga_g \c  \dk^{\leq 1} U+\lot
\eeaa
Also,
\bea\label{commutator-nab-3-nab-a-U}
\bsplit
\, [\nab_3, \nab_a] U_{bc}&=-\frac  1 2   \trchb\, \Big(\nab_a U_{bc}+\eta_bU_{ac}+\eta_c U_{ab}-\de_{a b}(\eta \c U)_c-\de_{a c}(\eta \c U)_b \Big)\\
&-\frac 1 2 \atrchb\, \Big(\dual \nab_a  U_{bc} +\eta_b \dual U_{ac}+\eta_c \dual U_{ab}- \in_{a b}(\eta \c  U)_c- \in_{a c}(\eta \c  U)_b \Big)\\
 &+(\eta_a-\ze_a)\nab_3 U_{bc}+ \err_{3abc}[U]
 \end{split}
\eea
where
\beaa
 \err_{3abc}[U]&=&\Ga_g \c \dk^{\leq 1} U+\lot
\eeaa
\end{enumerate}
where  $\lot$ denotes error terms which are 
quadratic   in  the perturbation and  enjoy  better  decay  properties,  or are higher order  and decay at least as good.
\end{lemma}

\begin{proof}
We  derive the commutation formula for  $ [\nab_4, \DD \hot] f$. According to Corollary \ref{corr:comm}, from \eqref{last-statement-item1} we have
\beaa
\, [\nab_4, \nab \hot] f&=- \frac 1 2 \trch \left( \nab\hot f +\etab \hot f\right)- \frac 1 2 \atrch \dual  \left(\nab \hot  f+\etab \hot f\right)+ (\etab+\ze) \hot \nab_4 f+\err_{4\hot}[f].
\eeaa
Hence for $F= f+i\dual f$
\beaa
\, [\nab_4, \nab \hot] F&=&- \frac 1 2 \trch \left( \nab\hot F +\etab \hot F\right)- \frac 1 2 \atrch \dual  \left(\nab \hot  F+\etab \hot F\right)+ (\etab+\ze) \hot \nab_4 F+\err_{4\hot}[F].
\eeaa
 Recalling that  $ \dual F=-iF$  and $\dual (\nab \hot f)=\dual \nab\hot f=\nab\hot \dual f$ we deduce,
 \beaa
 \, [\nab_4, \nab \hot] F&=&- \frac 1 2 \trch \left( \nab\hot F +\etab \hot F\right)+ i \frac 1 2 \atrch  \left(\nab \hot  F+\etab \hot F\right)
 + (\etab+\ze) \hot \nab_4 F+\err_{4\hot}[F]\\
 &=&-\frac 1 2\big( \trch-i\atrch)   \left(\nab \hot  F+\etab \hot F\right)  + (\etab+\ze) \hot \nab_4 F+\err_{4\hot}[F]\\
 &=&-\frac 1 2 \tr X\left(\nab \hot  F+\etab \hot F\right)  + (\etab+\ze) \hot \nab_4 F+\err_{4\hot}[F].
 \eeaa
Taking the dual
\beaa
\, [\nab_4,  \dual \nab \hot] F &=&-\frac 1 2 \tr X\left( \dual \nab \hot  F+\dual \etab \hot F\right)  +\dual (\etab+\ze) \hot \nab_4 F+\dual \err_{4\hot}[ F].
\eeaa
Finally adding we derive
\beaa
 [\nab_4,  \DD \hot] F &=&-\frac 1 2 \tr X\left( \DD \hot  F+\Hb\hot F\right)+(\Hb+Z)\hot\nab_4 F+\err_{4 \DD\hot}[F]
\eeaa
as stated.

   Recall from \eqref{commutator-3-a-u-bc}, 
\beaa
\, [\nab_3, \nab_4] u&=-2 \om \nab_3 u +2\omb \nab_4 u + 2(\eta_c-\etab_c ) \nab_c u - 4 \eta \hot (\etab \c u)+4 \etab \hot (\eta \c u) +4 \dual \rho \dual u+\err_{43}[u].
\eeaa
Adding the corresponding formula for the dual we  derive for $U=u+i\dual u$,
\beaa
\, [\nab_3, \nab_4] U&=-2 \om \nab_3  U +2\omb \nab_4 U + 2(\eta_c-\etab_c )  \nab_c  U - 4  \eta \hot (\etab \c U)+4 \etab \hot (\eta \c U) +4 i  \rhod  U+\err_{43}[U].
\eeaa

To prove the last part of the Lemma we  recall from Corollary \ref{corr:comm}
\beaa
\,[\nab_4,  \div] u &=&-\frac 1 2 \trch \big(  \div u - 2\etab \c u\big) +\frac 1 2 \atrch\big(  \div\dual u -2 \etab\c \dual u\big) +(\etab+\ze)\c\nab_4 u+\err_{4\div}[u]. 
\eeaa
Similarly
\beaa
\,[\nab_4,  \div] \dual u&=& -\frac 1 2 \trch \big(  \div \dual u - 2\etab \c \dual u\big)- \frac 1 2 \atrch \big(  \div u -2 \etab\c  u\big) +(\etab+\ze)\c\nab_4 \dual u+\err_{4\div}[\dual u]. 
\eeaa
Thus, for $U=u+i\dual u$,
\beaa
\,[\nab_4,  \div] U &=&-\frac 1 2 \trch \big(  \div U - 2\etab \c U\big)-i \frac 1 2\atrch\big(  \div U - 2\etab \c U\big) +(\etab+\ze)\c\nab_4 U+\err_{4\ov{\DD}}[U]\\
&=&-\frac 1 2 (\trch+i\atrch) \big(  \div U - 2\etab \c U\big) +(\etab+\ze)\c\nab_4 U+\err_{4\ov{\DD}}[U]\\
&=&-\frac 1 2 \ov{\tr X}  \big(  \div U - 2\etab \c U\big) +(\etab+\ze)\c\nab_4 U+\err_{4\ov{\DD}}[U].
\eeaa
Hence, since  $\dual U=-i U$,
\beaa
\, [\nab_4, \ov{\DD}]\c  U&=& \,[\nab_4, \nab   ] \c U- i \,[\nab_4, \dual \nab ] \c U= \,[\nab_4, \nab   ] \c U+i\,[\nab_4, \nab ] \c\dual  U= 2 \,[\nab_4, \nab   ] \c U\\
&=&2 \,[\nab_4, \div  ] U =- \ov{\tr X}  \big(  \div U - 2\etab \c U\big)  +2(\etab+\ze)\c\nab_4 U+\err_{4\ov{\DD}}[U].
\eeaa
Also,
\beaa
  \ov{\DD}  U &=&\nab\c U-i \dual \nab \c U=\nab\c U+i  \nab\c \dual U= 2\nab\c U.
 \eeaa
 Hence,
 \beaa
 \, [\nab_4, \ov{\DD}]\c  U&=&2 \,[\nab_4, \div  ] U =- \frac 1 2  \ov{\tr X}  \big( \ov{\DD} U - 4\etab \c U\big)  +2(\etab+\ze)\c\nab_4 U+\err_{4\ov{\DD}}[U].
\eeaa
Note also that since $U=i \dual U$ and $\etab\c \dual U=-\dual \etab \c U$
\beaa
2 \etab\c U=\etab\c U+ \etab \c (i\dual U)= \etab\c U-i \dual  \etab \c  U=(\eta-i\dual \eta) \c U=\ov{\Hb} \c U.
\eeaa
Similarly
\beaa
2(\etab+\ze)\c\nab_4 U= (\ov{\Hb}+\ov{Z} )\c \nab_4 U.
\eeaa
Hence,
\beaa
\, [\nab_4, \ov{\DD}]\c  U&=&- \frac 1 2  \ov{\tr X}  \big( \ov{\DD} U-2 \ov{\Hb} \c U\big)+ (\ov{\Hb}+\ov{Z} )\c \nab_4 U+\err_{4\ov{\DD}}[U]
\eeaa
as stated.

Starting with \eqref{commutator-3-a-u-bc}
\beaa
\,  [\nab_3,\nab_a] u_{bc}    &=&-\frac 1 2 \trchb\, (\nab_a u_{bc}+\eta_bu_{ac}+\eta_c u_{ab}-\de_{a b}(\eta \c u)_c-\de_{a c}(\eta \c u)_b )\\
&&-\frac 1 2 \atrchb\, (\dual \nab_a u_{bc} +\eta_b\dual u_{ac}+\eta_c\dual u_{ab}- \in_{a b}(\eta \c u)_c- \in_{a c}(\eta \c u)_b )\\
&&+(\eta_a-\ze_a)\nab_3 u_{bc}+\err_{3abc}[u]
\eeaa
and adding the  same expression  for $u$ replaced with $i \dual u$ we derive
\beaa
\,  [\nab_3,\nab_a] U_{bc}    &=&-\frac  1 2   \trchb\, \Big(\nab_a U_{bc}+\eta_bU_{ac}+\eta_c U_{ab}-\de_{a b}(\eta \c U)_c-\de_{a c}(\eta \c U)_b \Big)\\
&-&\frac 1 2 \atrchb\, \Big(\dual \nab_a  U_{bc} +\eta_b \dual U_{ac}+\eta_c \dual U_{ab}- \in_{a b}(\eta \c  U)_c- \in_{a c}(\eta \c  U)_b \Big)\\
 &+&(\eta_a-\ze_a)\nab_3 U_{bc}+ \err_{3abc}[U]
\eeaa
as stated.
\end{proof}

    
\subsubsection{Commutation formulas for conformally invariant  derivatives}
    

     \begin{lemma}\label{commutator-nab-c-3-DD-c-hot}\label{commutator-nab-3-nab-4}\label{commtator-3-a}
   The following commutation formulas hold true.
   \begin{enumerate}
   \item
    Let $F=f+i\dual f \in \SS_1 (\mathbb{C})$.  Then
 \bea
 \bsplit
  \, [\nabc_4 , \DDc \hot ]F &=- \frac 1 2 \tr X\left( \DDc\hot F + (1-s)\Hb\hot F\right)+ \underline{H} \hot \nabc_4 F+ \err_{4\DD\hot}[F], \label{commutator-nabc-4-F-formula} \\
 \, [\nabc_3, \DDc \hot ]F &=- \frac 1 2 \tr \Xb \left( \DDc \hot F + (1+s)H \hot F \right)  + H \hot \nabc_3 F+ \err_{3\DD\hot}[F],
 \end{split}
 \eea
 where
\beaa
\err_{4 \DD\hot}[F]&=&  r^{-1} \Ga_g  \c \dk^{\leq 1} F+\lot, \qquad \err_{3\DD\hot}[F]= \Ga_g \c \dk^{\leq 1} F+ \lot
\eeaa

\item Let $U=u+i \dual u\in \SS_2(\mathbb{C})$. Then
\bea
\bsplit
 [\nabc_3, \nabc_4] U  &= 2((\eta-\etab ) \c \nabc) U  +((s-2)P+(s+2)\ov{P}  -2s\eta\c\etab)U\\
 & - 4  \eta \hot (\etab \c U)+4 \etab \hot (\eta \c U)+ \err_{34}[U]
 \end{split}
 \eea
 where
 \beaa
 \err_{34}[U]&=& (\Ga_g \cdot \Ga_g) U = \lot
 \eeaa

 \item  For $U\in \SS_2(\mathbb{C})$
 \bea
\, [\nabc_3, \ov{\DDc}\c] U&=- \frac 1 2\ov{\tr\Xb}\, ( \ov{\DDc} \c U +(s-  2) \ov{H} \c U)+\ov{H} \c \nab_3 U+\err_{3\ov{\DD}}[U] \label{commutator-nabc-3-ov-DDc-U}
\eea
where
\beaa
\err_{3\ov{\DD}}[U]&=&  \Ga_g \c  \dk^{\leq 1} U+\lot
\eeaa
Also,
\beaa
 [\nabc_3, \nabc_a] U_{bc}&=&\eta_a\nabc_3 U_{bc}\\
&&-\frac  1 2   \trchb\, \Big(\nabc_a U_{bc}+s(\eta_a) U_{bc}+\eta_bU_{ac}+\eta_c U_{ab}-\de_{a b}(\eta \c U)_c-\de_{a c}(\eta \c U)_b \Big)\\
&&-\frac 1 2 \atrchb\, \Big(\dual \nabc_a  U_{bc}+s (\dual\eta_a) U_{bc} +\eta_b \dual U_{ac}+\eta_c \dual U_{ab}\\
&&- \in_{a b}(\eta \c  U)_c- \in_{a c}(\eta \c  U)_b \Big)+ \err_{3abc}[U]
 \eeaa
where
\beaa
 \err_{3abc}[U]&=&\Ga_g \c  \dk^{\leq 1} U+\lot
\eeaa
\end{enumerate}
Here,  $\lot$ denotes error terms which are 
quadratic   in  the perturbation and  enjoy  better  decay  properties,  or are higher order  and decay at least as good.
\end{lemma}

\begin{proof}
 We have 
 \beaa
 \nabc_3 (\DDc \hot F )&=&\nabc_3 (\DD\hot F+s Z \hot F )\\
 &=&\nab_3 (\DD\hot F+s Z \hot F )-2s \omb (\DD\hot F+s Z \hot F ) \\
  &=&\nab_3 (\DD\hot F)+s \nab_3(Z) \hot F +s Z \hot \nab_3F -2s \omb (\DD\hot F+s Z \hot F ).
 \eeaa
 Using \eqref{commutator-nab-43-D-hot} and the null structure equation 
 \beaa
 \nab_3Z &=&-\frac{1}{2}\tr\Xb(Z+H)+2\omb(Z-H)  -2\DD\omb\\
 && -\frac{1}{2}\widehat{\Xb}\c(\ov{Z}+\ov{H})+\frac{1}{2}\tr X\Xib+2\om\Xib -\Bb+\frac{1}{2}\ov{\Xib}\c\Xh
 \eeaa
  we obtain
 \beaa
 \nabc_3 (\DDc \hot F )  &=&\mathcal{D}\hot(\nab_3F)- \frac 1 2 \tr \Xb( \mathcal{D}\hot F + H \hot F)+ (H-Z) \hot \nab_3 F\\
 &&+s \Big(-\frac{1}{2}\tr\Xb(Z+H)+2\omb(Z-H)  -2\DD\omb\Big) \hot F \\
 &&+s Z \hot \nab_3F -2s \omb (\DD\hot F+s Z \hot F )+\err_{3 \DDc\hot}[F]
 \eeaa
 where 
 \beaa
 \err_{3\DDc\hot}[F]&=& \err_{3 \DD\hot}[F]+ s (-\frac{1}{2}\widehat{\Xb}\c(\ov{Z}+\ov{H})+\frac{1}{2}\tr X\Xib+2\om\Xib -\Bb+\frac{1}{2}\ov{\Xib}\c\Xh) \hot F.
 \eeaa
 We can rewrite the above as
  \beaa
  \nabc_3 (\DDc \hot F )  &=&\mathcal{D}\hot(\nab_3F)+s Z \hot \nab_3F+ (H-Z) \hot \nab_3 F+2s\omb(Z-H)\hot F \\
 &&- \frac 1 2 \tr \Xb( \mathcal{D}\hot F + H \hot F+ s(Z+H) \hot F ) \\
 && -2s \omb \DD\hot F-2s^2\omb  Z \hot F -2s \DD \omb \hot F + \err_{3\DDc\hot}[F]
 \eeaa
 which gives
  \beaa
 \nabc_3 (\DDc \hot F )  &=&\mathcal{D}\hot(\nabc_3F)+s Z \hot \nabc_3F+ (H-Z) \hot \nabc_3 F\\
 &&- \frac 1 2 \tr \Xb( \mathcal{D}\hot F + H \hot F+ s(Z+H) \hot F )  + \err_{3\DDc\hot}[F].
 \eeaa
 Recall that $\nabc_3F$ is conformal of type $s-1$, therefore $$\DDc \hot(\nabc_3F)= \mathcal{D}\hot(\nabc_3F)+(s-1) Z \hot \nabc_3F.$$ 
 We finally obtain 
   \beaa
 \nabc_3 (\DDc \hot F )  &=&\DDc \hot(\nabc_3F)- \frac 1 2 \tr \Xb \left( \DDc \hot F + (s+1)H \hot F \right)  + H \hot \nabc_3 F+ \err_{3\DDc\hot}[F]
 \eeaa
 as stated.  
 
 Similarly,  
 \beaa
 \nabc_4 (\DDc \hot F )&=&\nabc_4 (\DD\hot F+s Z \hot F )\\
 &=&\nab_4 (\DD\hot F+s Z \hot F )+2s \om (\DD\hot F+s Z \hot F ) \\
  &=&\nab_4 (\DD\hot F)+s \nab_4(Z) \hot F +s Z \hot \nab_4F +2s \om (\DD\hot F+s Z \hot F ).
 \eeaa
 Using \eqref{commutator-nab-43-D-hot} and the null structure equation for $\nab_4 Z$,
  we obtain
 \beaa
 \nabc_4 (\DDc \hot F )  &=&\mathcal{D}\hot(\nab_4F)- \frac 1 2 \tr X( \mathcal{D}\hot F + \underline{H} \hot F)+ (\underline{H}+Z) \hot \nab_4 F\\
 &&+s (-\frac{1}{2}\tr X(Z-\Hb)+2\om(Z+\Hb)+ 2\DD\om ) \hot F \\
 &&+s Z \hot \nab_4F +2s \om (\DD\hot F+s Z \hot F )++\err_{4 \DDc\hot}[F].
 \eeaa
 We can rewrite the above as
 \beaa
 \nabc_4 (\DDc \hot F )  &=&\mathcal{D}\hot(\nab_4F)+s Z \hot \nab_4F+ (\underline{H}+Z) \hot \nab_4 F+2s\om(Z+\Hb)\hot F\\
 &&- \frac 1 2 \tr X( \mathcal{D}\hot F + \underline{H} \hot F+s (Z-\Hb)\hot F)\\
 && +2s \om \DD\hot F+2s^2\om Z \hot F +2s \DD\om  \hot F+\err_{4 \DDc\hot}[F]
 \eeaa
 which gives
  \beaa
 \nabc_4 (\DDc \hot F )  &=&\mathcal{D}\hot(\nabc_4F)+s Z \hot \nabc_4F+ (\underline{H}+Z) \hot \nabc_4 F\\
 &&- \frac 1 2 \tr X( \mathcal{D}\hot F + \underline{H} \hot F+s (Z-\Hb)\hot F) +\err_{4 \DDc\hot}[F].
 \eeaa
 Recall that $\nabc_4F$ is conformal of type $s+1$, therefore $$\DDc \hot(\nabc_4F)= \mathcal{D}\hot(\nabc_4F)+(s+1) Z \hot \nabc_4F.$$ 
 We finally obtain 
 \beaa
 \nabc_4 (\DDc \hot F )  &=&\DDc \hot(\nabc_4F)+ \underline{H} \hot \nabc_4 F- \frac 1 2 \tr X( \DDc\hot F + (1-s)\Hb\hot F) +\err_{4 \DDc\hot}[F]
 \eeaa
 as stated. 

 We have 
 \beaa
 [\nabc_3, \nabc_4] U&=&\nabc_3 \nabc_4 U -\nabc_4 \nabc_3 U\\
 &=&(\nab_3-2(s+1) \omb ) (\nab_4 U+2s \om U) -(\nab_4+2(s-1)\om) (\nab_3 U -2s \omb U)\\
 &=&\nab_3\nab_4 U+2s\nab_3(\om U)-2(s+1) \omb (\nab_4 U)-4s(s+1) \om  \omb U \\
 &&-\nab_4 \nab_3 U +2s \nab_4(\omb U)-2(s-1)\om\nab_3 U+4s(s+1) \om  \omb U\\
 &=&[\nab_3, \nab_4] U-2 \omb \nab_4 U+2\om \nab_3 U  +2s(\nab_3 \om+ \nab_4\omb -4\om\omb)U.
 \eeaa
 Using \eqref{commutator-nab-4-D-c}   and
 \beaa
 \nab_3\om+\nab_4\omb -4\om\omb &=&   \rho  +(\eta-\etab)\c\ze -\eta\c\etab+\xi\c \xib 
 \eeaa
 we obtain
  \beaa
 [\nabc_3, \nabc_4] U &=& -2 \om \nab_3  U +2\omb \nab_4 U + 2(\eta_c-\etab_c )  \nab_c  U - 4  \eta \hot (\etab \c U)+4 \etab \hot (\eta \c U)+ \err_{34}[U] \\
 &&- 2(P - \ov{P})  U-2 \omb \nab_4 U+2\om \nab_3 U  +2s(\rho  +(\eta-\etab)\c\ze -\eta\c\etab)U \\
 &=&  2(\eta_c-\etab_c )  \nabc_c  U  +((s-2)P+(s+2)\ov{P}  -2s\eta\c\etab)U \\
 &&- 4  \eta \hot (\etab \c U)+4 \etab \hot (\eta \c U) + \err_{34}[U]
 \eeaa
 as stated.

 We have 
\beaa
\nabc_3 \ov{\DDc}\c U&=& \nabc_3 ((\ov{\DD}+ s \ov{Z})\c U )\\
&=& \nab_3 ((\ov{\DD}+ s \ov{Z})\c U )-2s\omb((\ov{\DD}+ s \ov{Z})\c U )\\
&=& \nab_3(\ov{\DD}\c U)+ s \nab_3\ov{Z}\c U+ s \ov{Z}\c \nab_3U -2s\omb((\ov{\DD}+ s \ov{Z})\c U ).
\eeaa
Using the equation for $ \nab_3\ov{Z}$ and \eqref{commutator-nabc-3-ov-DDc-U}, 
we have
\beaa
\nabc_3 \ov{\DDc}\c U&=& \ov{\DD}\c \nab_3U- \frac 1 2\ov{\tr\Xb}\, ( \ov{\DD} \c U -  2 \ov{H} \c U)+\ov{(H-Z)} \c \nab_3 U +\err_{3\ov{\DD}}[U]\\
&& + s \left(-\frac{1}{2}\ov{\tr\Xb}(\ov{Z}+\ov{H})+2\omb(\ov{Z}-\ov{H})  -2\ov{\DD}\omb \right)\c U+ s \ov{Z}\c \nab_3U -2s\omb((\ov{\DD}+ s \ov{Z})\c U ).
\eeaa
We can rewrite the above as
\beaa
\nabc_3 \ov{\DDc}\c U&=& \ov{\DDc}\c (\nabc_3U)- \frac 1 2\ov{\tr\Xb}\, ( \ov{\DDc} \c U+(s -  2) \ov{H} \c U)+\ov{H} \c \nabc_3U +\err_{3\ov{\DD}}[U]
\eeaa
which gives the desired formula.

We have
\beaa
\nabc_3 \nabc_a U_{bc}&=&\nabc_3 (\nab_a U_{bc}+s \ze_a U_{bc})\\
&=&\nab_3 (\nab_a U_{bc}+s \ze_a U_{bc})-2s\omb (\nab_a U_{bc}+s \ze_a U_{bc})\\
&=&\nab_3 \nab_a U_{bc}+s \nab_3\ze_a U_{bc}+s \ze_a \nab_3U_{bc}-2s\omb (\nab_a U_{bc}+s \ze_a U_{bc}).
\eeaa
Using \eqref{commutator-nab-3-nab-a-U} and the null structure equation 
\beaa
\nab_3 \ze+2\nab\omb&=& -\frac{1}{2}\trchb(\ze+\eta)-\frac{1}{2}\atrchb(\dual\ze+\dual\eta)+ 2 \omb(\ze-\eta)\\
&&+\hch\c\xib+\frac{1}{2}\trch\,\xib+\frac{1}{2}\atrch\dual\xib +2 \om \xib-\chibh\c(\ze+\eta) -\bb
\eeaa
we obtain
\beaa
\nabc_3 \nabc_a U_{bc}&=&\nab_a \nab_3 U_{bc}-\frac  1 2   \trchb\, \Big(\nab_a U_{bc}+\eta_bU_{ac}+\eta_c U_{ab}-\de_{a b}(\eta \c U)_c-\de_{a c}(\eta \c U)_b \Big)\\
&&-\frac 1 2 \atrchb\, \Big(\dual \nab_a  U_{bc} +\eta_b \dual U_{ac}+\eta_c \dual U_{ab}- \in_{a b}(\eta \c  U)_c- \in_{a c}(\eta \c  U)_b \Big)\\
 &&+(\eta_a-\ze_a)\nab_3 U_{bc}\\
 &&+s (-2\nab_a\omb -\frac{1}{2}\trchb(\ze_a+\eta_a)-\frac{1}{2}\atrchb(\dual\ze_a+\dual\eta_a)+ 2 \omb(\ze_a-\eta_a)) U_{bc}\\
 &&+s \ze_a \nab_3U_{bc}-2s\omb (\nab_a U_{bc}+s \ze_a U_{bc})+\err_{3abc}[U].
\eeaa
We can rewrite the above as
\beaa
\nabc_3 \nabc_a U_{bc}&=&\nab_a \nab_3 U_{bc}+s \ze_a \nab_3U_{bc}+(\eta_a-\ze_a)\nab_3 U_{bc}+ 2s \omb(\ze_a-\eta_a) U_{bc}\\
&&-\frac  1 2   \trchb\, \Big(\nab_a U_{bc}+s(\ze_a+\eta_a) U_{bc}+\eta_bU_{ac}+\eta_c U_{ab}-\de_{a b}(\eta \c U)_c-\de_{a c}(\eta \c U)_b \Big)\\
&&-\frac 1 2 \atrchb\, \Big(\dual \nab_a  U_{bc}+s (\dual\ze_a+\dual\eta_a) U_{bc} +\eta_b \dual U_{ac}+\eta_c \dual U_{ab}\\
&&- \in_{a b}(\eta \c  U)_c- \in_{a c}(\eta \c  U)_b \Big)\\
 &&-2s\omb \nab_a U_{bc}-2s^2\omb  \ze_a U_{bc}-2s\nab_a\omb  U_{bc}+\err_{3abc}[U]
\eeaa
which gives
\beaa
\nabc_3 \nabc_a U_{bc}&=&\nab_a \nabc_3 U_{bc}+s \ze_a \nabc_3U_{bc}+(\eta_a-\ze_a)\nabc_3 U_{bc}\\
&&-\frac  1 2   \trchb\, \Big(\nabc_a U_{bc}+s(\eta_a) U_{bc}+\eta_bU_{ac}+\eta_c U_{ab}-\de_{a b}(\eta \c U)_c-\de_{a c}(\eta \c U)_b \Big)\\
&&-\frac 1 2 \atrchb\, \Big(\dual \nabc_a  U_{bc}+s (\dual\eta_a) U_{bc} +\eta_b \dual U_{ac}+\eta_c \dual U_{ab}\\
&&- \in_{a b}(\eta \c  U)_c- \in_{a c}(\eta \c  U)_b \Big)+\err_{3abc}[U].
\eeaa
Recall that $\nabc_3 U$ is of conformal type $s-1$, therefore
\beaa
\nabc_a \nabc_3 U_{bc}&=& \nab_a \nabc_3 U_{bc}+(s-1) \ze_a \nabc_3 U_{bc}.
\eeaa
We finally obtain
\beaa
\nabc_3 \nabc_a U_{bc}&=&\nabc_a \nabc_3 U_{bc}+\eta_a\nabc_3 U_{bc}\\
&&-\frac  1 2   \trchb\, \Big(\nabc_a U_{bc}+s(\eta_a) U_{bc}+\eta_bU_{ac}+\eta_c U_{ab}-\de_{a b}(\eta \c U)_c-\de_{a c}(\eta \c U)_b \Big)\\
&&-\frac 1 2 \atrchb\, \Big(\dual \nabc_a  U_{bc}+s (\dual\eta_a) U_{bc} +\eta_b \dual U_{ac}+\eta_c \dual U_{ab}\\
&&- \in_{a b}(\eta \c  U)_c- \in_{a c}(\eta \c  U)_b \Big)+\err_{3abc}[U]
\eeaa
as stated.
\end{proof}


\section{The wave operator for perturbations of Kerr}


In this section we derive the expression for the wave operator for anti-self dual tensors in perturbations of Kerr.


\subsection{The Gauss equation}


To derive the wave operator, 
we need to specialize the Gauss equation \eqref{Gauss-eq-horizontal} to a tensor $\Psi \in \SS_2(\mathbb{C})$ in a perturbation of Kerr. We will use this formula in Proposition \ref{Complexwave-decomp}. 
\begin{proposition}\label{Gauss-equation} 
We have for $\Psi \in \SS_2(\mathbb{C})$:
\bea\label{Gauss-eq-first-com}
\begin{split}
 \big( \nab_1 \nab_2- \nab_2 \nab_1\big)  \Psi &=\frac 1 2(\atrch\nab_3+\atrchb \nab_4) \Psi  + 2 i  \left( \frac 1 4\trch\trchb+ \frac 1 4 \atrch\atrchb+\rho\right) \Psi \\
 &+ \err_{12}[\Psi]
    \end{split}
    \eea
    or also 
    \bea\label{Gauss-eq-2-complex}
\begin{split}
 \big( \nab_1 \nab_2- \nab_2 \nab_1\big)  \Psi &=\frac 1 2(\atrch\nab_3+\atrchb \nab_4) \Psi  +  i  \left(  \frac 1 4 \tr X \ov{\tr \Xb}+ \frac 1 4 \tr \Xb  \ov{\tr X} + P + \ov{P}\right) \Psi\\
  &+ \err_{12}[\Psi]
    \end{split}
    \eea
    where
    \beaa
     \err_{12}[\Psi]&=& r^{-1}\Ga_g \c \Psi.
    \eeaa
      \end{proposition}
\begin{proof} 
From \eqref{Gauss-eq-horizontal}, we have
\beaa
\big( \nab_a \nab_b- \nab_b \nab_a\big)  \Psi_{st} &=&\frac 1 2 \in_{ab}(\atrch\nab_3+\atrchb \nab_4) \Psi_{st}
 -\frac  12 E_{sdab} \Psi_{dt}    -\frac 1 2 E_{tdab} \Psi_{sd} \\
    &+&\left(-\in_{sd}\in_{ab} \rho  \right)\Psi_{dt}+ \left(-\in_{td}\in_{ab} \rho \right) \Psi_{sd}
\eeaa
where 
\bea\label{definition-E}
E_{cdab}:&=&  \chi_{ac}\chib_{bd} + \chib_{ac}\chi_{bd} - \chi_{bc}\chib_{ad}- \chib_{bc}\chi_{ad}. 
\eea
We now compute the $E$ as given by \eqref{definition-E}. We have:
\beaa
\chi_{ac}\chib_{bd} &=&    \frac 1 4 \big( \trch \de_{ac}+ \atrch \in_{ac}\big) \big( \trchb \de_{bd}+ \atrchb \in_{bd}\big)+r^{-1}\Ga_g\\
&=& \frac 1 4\Big( \trch \trchb \de_{ac} \de_{bd}+  \atrch \atrchb \in_{ac} \in_{bd}    +\de_{ac}\in_{bd}  \trch \atrchb     +    \in_{ac}\de_{bd}  \atrch \trchb \Big)+r^{-1}\Ga_g,\\
\chib_{ac}\chi_{bd} &=& \frac 1 4\Big( \trchb \trch \de_{ac} \de_{bd}+  \atrchb \atrch \in_{ac} \in_{bd}    +\de_{ac}\in_{bd}  \trchb \atrch    +    \in_{ac}\de_{bd}  \atrchb \trch \Big)+r^{-1}\Ga_g,\\
 \chi_{bc}\chib_{ad}&=&  \frac 1 4\Big( \trch \trchb \de_{bc} \de_{ad}+  \atrch \atrchb \in_{bc} \in_{ad}    +\de_{bc}\in_{ad}  \trch \atrchb     +    \in_{bc}\de_{ad}  \atrch \trchb \Big)+r^{-1}\Ga_g,\\
 \chib_{bc}\chi_{ad}&=&  \frac 1 4\Big( \trchb \trch \de_{bc} \de_{ad}+  \atrchb \atrch \in_{bc} \in_{ad}    +\de_{bc}\in_{ad}  \trchb \atrch     +    \in_{bc}\de_{ad}  \atrchb \trch \Big)+r^{-1}\Ga_g.
\eeaa
Thus,
\beaa
E_{cdab}&=&  \chi_{ac}\chib_{bd} + \chib_{ac}\chi_{bd} - \chi_{bc}\chib_{ad}- \chib_{bc}\chi_{ad} \\
&=&\frac 12  \trch \trchb (\de_{ac} \de_{bd}-\de_{bc} \de_{ad})+\frac 1 2 \atrch \atrchb( \in_{ac} \in_{bd}- \in_{bc} \in_{ad}   ) \\
&&  +  \frac 1 4 (\de_{ac}\in_{bd}  +    \in_{ac}\de_{bd}- \de_{bc}\in_{ad} -  \in_{bc}\de_{ad} )  \trch \atrchb  \\
&&+  \frac 1 4 (  \in_{ac}\de_{bd}  +\de_{ac}\in_{bd}-\in_{bc}\de_{ad}- \de_{bc}\in_{ad}  )\atrch \trchb+r^{-1}\Ga_g.
\eeaa

 Observe that 
 \beaa
 E_{1112}&=&   \frac 1 4 (\de_{11}\in_{21}  -  \in_{21}\de_{11} )  \trch \atrchb  +  \frac 1 4 (   \de_{11}\in_{21}-\in_{21}\de_{11} )\atrch \trchb+r^{-1}\Ga_g=r^{-1}\Ga_g, \\
 E_{2212}&=&  \frac 1 4 (   \in_{12}\de_{22}- \de_{22}\in_{12} )  \trch \atrchb  +  \frac 1 4 (  \in_{12}\de_{22}  - \de_{22}\in_{12}  )\atrch \trchb+r^{-1}\Ga_g=r^{-1}\Ga_g,
 \eeaa
 and 
 \beaa
 E_{1212}&=& \frac 12  \trch \trchb (\de_{11} \de_{22})+\frac 1 2 \atrch \atrchb( - \in_{21} \in_{12}   ) +r^{-1}\Ga_g= \frac 12  \trch \trchb+\frac 1 2 \atrch \atrchb+r^{-1}\Ga_g, \\
 E_{2112}&=& \frac 12  \trch \trchb (-\de_{22} \de_{11})+\frac 1 2 \atrch \atrchb( \in_{12} \in_{21} )+r^{-1}\Ga_g =-\frac 12  \trch \trchb-\frac 1 2 \atrch \atrchb+r^{-1}\Ga_g.
 \eeaa

Define  $ Y_{st}=-\frac  12 E_{sd12} \Psi_{dt}    -\frac 1 2 E_{td12} \Psi_{sd}$. Since $\Psi \in \SS_2$ we have that $\Psi_{12}=-i\Psi_{11}$ and $\Psi_{11}=i \Psi_{12}$ and therefore evaluating $Y$ in coordinates:
\beaa
Y_{11}&=&-\frac  12 E_{1d12} \Psi_{d1}    -\frac 1 2 E_{1d12} \Psi_{1d}=-\frac  12 E_{1112} \Psi_{11} -\frac  12 E_{1212} \Psi_{21}    -\frac 1 2 E_{1112} \Psi_{11}   -\frac 1 2 E_{1212} \Psi_{12}\\
&=&\left(- E_{1112} +i  E_{1212} \right) \Psi_{11},      \\
Y_{12}&=&-\frac  12 E_{1112} \Psi_{12}  -\frac  12 E_{1212} \Psi_{22}    -\frac 1 2 E_{2112} \Psi_{11}  -\frac 1 2 E_{2212} \Psi_{12}=\left(- E_{1112}  + i E_{1212}\right) \Psi_{12}.  
\eeaa
This implies 
\beaa
-\frac  12 E_{sd12} \Psi_{dt}    -\frac 1 2 E_{td12} \Psi_{sd}&=&  \frac 1 2 i  \left( \trch\trchb+  \atrch\atrchb\right) \Psi_{st}+r^{-1}\Ga_g \Psi.    
\eeaa

Define $W_{st}= \R_{s d   12}\Psi_{dt}+\R_{td 12} \Psi_{sd}$. Evaluating $W$ in coordinates we have
\beaa
W_{11}&=& \R_{1 d   12}\Psi_{d1}+\R_{1d 12} \Psi_{1d}=2 \R_{1 2   12}\Psi_{12}=-2i \R_{1 2   12}\Psi_{11}=2i \rho \Psi_{11}.
\eeaa
We obtain
\beaa
\R_{s d   12}\Psi_{dt}+\R_{td 12} \Psi_{sd}=2i \rho \Psi_{st}.
\eeaa
Putting the above together we obtain the final formula. 
\end{proof}


\subsection{The wave equation using complex operators}


We now express the Laplacian appearing in the canonical form of the wave equation in terms of the complex derivative. We summarize the result in the following.
\begin{proposition}
\lab{Complexwave-decomp}
Given $\Psi \in \SS_2(\CCC) $ we have
\bea\label{expression-DD-laplacian-0}
\begin{split}
   \DD \hot (\DDb \c \Psi)&= 2\lap_2  \Psi - i (\atrch\nab_3+\atrchb \nab_4) \Psi  \\
   &+  \left(  \frac 1 2 \tr X \ov{\tr \Xb}+ \frac 1 2 \tr \Xb  \ov{\tr X} + 2P + 2\ov{P}\right) \Psi +\err_{ \DD \hot \DDb}[\Psi]
   \end{split}
\eea
where
\beaa
\err_{ \DD \hot \DDb}[\Psi]&=&r^{-1} \Ga_g \c \Psi.
\eeaa
\end{proposition}

\begin{proof} 
We define $Y_{ab}:=2\big(\DD\hot( \DDb \c \Psi)\big)_{ab}$, and express $Y$ in coordinates. 
We have
\beaa
Y_{ab}&=& \DD_a   \DDb^c \Psi_{cb} +\DD_b    \DDb^c \Psi_{c a } -\de_{ab} \DD^d    \DDb^c \Psi_{dc}.
\eeaa
By construction, $Y$ is symmetric and traceless. 
We calculate first $Y_{11}=-Y_{22}$. For $a=b=1$ we derive
\beaa
Y_{11}&=& \DD_1   \DDb^c \Psi_{c1} +\DD_1    \DDb^c \Psi_{c 1 } -\de_{11} \DD^d    \DDb^c \Psi_{dc}\\
&=& 2\DD_1   \DDb_c \Psi_{c1}  - \DD^d    \DDb^c \Psi_{dc}\\
&=& 2\DD_1   \DDb_1 \Psi_{11}+2\DD_1   \DDb_2 \Psi_{21}  - \DD^d    \DDb^1 \Psi_{d1}- \DD^d    \DDb^2 \Psi_{d2}\\
&=& 2\DD_1   \DDb_1 \Psi_{11}+2\DD_1   \DDb_2 \Psi_{21}  - \DD_1    \DDb_1 \Psi_{11}- \DD_2    \DDb_1 \Psi_{21}- \DD_1    \DDb_2 \Psi_{12}- \DD_2    \DDb_2 \Psi_{22}\\
&=& \DD_1   \DDb_1 \Psi_{11}+\DD_1   \DDb_2 \Psi_{21} - \DD_2    \DDb_1 \Psi_{21}- \DD_2    \DDb_2 \Psi_{22}.
\eeaa
Writing $\Psi_{22}=-\Psi_{11}$, we have
\beaa
Y_{11}&=& (\DD_1   \DDb_1 + \DD_2    \DDb_2) \Psi_{11}+(\DD_1   \DDb_2  - \DD_2    \DDb_1) \Psi_{12}.
\eeaa
We compute
\beaa
Y_{11}&=& (\big(\nab_1+i\dual\nab_1\big)   \big(\nab_1-i\dual\nab_1\big) + \big(\nab_2+i\dual\nab_2\big)   \big(\nab_2-i\dual\nab_2\big)) \Psi_{11}\\
&&+(\big(\nab_1+i\dual\nab_1\big)   \big(\nab_2-i\dual\nab_2\big)  - \big(\nab_2+i\dual\nab_2\big)   \big(\nab_1-i\dual\nab_1\big)) \Psi_{12}\\
&=& (\big(\nab_1+i\nab_2\big)   \big(\nab_1-i\nab_2\big) + \big(\nab_2-i\nab_1\big)   \big(\nab_2+i\nab_1\big)) \Psi_{11}\\
&&+(\big(\nab_1+i\nab_2\big)   \big(\nab_2+i\nab_1\big)  - \big(\nab_2-i\nab_1\big)   \big(\nab_1-i\nab_2\big)) \Psi_{12}\\
&=& 2\lap  \Psi_{11} -2 i (\nab_1\nab_2-\nab_2\nab_1) \Psi_{11}+2 \big(\nab_1\nab_2-\nab_2\nab_1\big) \Psi_{12} + 2 i \lap  \Psi_{12}.
\eeaa
Using that $\Psi_{12}=-i \Psi_{11}$, we obtain
\beaa
Y_{11}&=& 2\lap  \Psi_{11} -2 i (\nab_1\nab_2-\nab_2\nab_1) \Psi_{11}-2i \big(\nab_1\nab_2-\nab_2\nab_1\big)  \Psi_{11} + 2  \lap \Psi_{11}\\
&=& 4\lap  \Psi_{11} -4 i (\nab_1\nab_2-\nab_2\nab_1) \Psi_{11}.
\eeaa

We now compute $Y_{12}=Y_{21}$. For $a=1, b=2$ we derive
\beaa
Y_{12}&=& \DD_1   \DDb^c \Psi_{c2} +\DD_2    \DDb^c \Psi_{c 1 } -\de_{12} \DD^d    \DDb^c \Psi_{dc}\\
&=& \DD_1   \DDb_1 \Psi_{12}+\DD_1   \DDb_2 \Psi_{22} +\DD_2    \DDb_1 \Psi_{1 1 }+\DD_2    \DDb_2 \Psi_{2 1 }\\
&=& (\DD_1   \DDb_1 +\DD_2    \DDb_2) \Psi_{12 } +(\DD_2    \DDb_1 -\DD_1   \DDb_2) \Psi_{11}.
\eeaa
We therefore obtain
\beaa
Y_{12}&=& 2\lap  \Psi_{12} -2 i (\nab_1\nab_2-\nab_2\nab_1) \Psi_{12} -2 \big(\nab_1\nab_2-\nab_2\nab_1\big) \Psi_{11} - 2 i \lap  \Psi_{11}.
\eeaa
Using that $\Psi_{11}=i \Psi_{12}$, we obtain
\beaa
Y_{12}&=& 4\lap  \Psi_{12} -4 i (\nab_1\nab_2-\nab_2\nab_1) \Psi_{12}.
\eeaa

Putting the above together we have
\beaa
Y_{ab}&=& 4\lap  \Psi_{ab} -4 i (\nab_1\nab_2-\nab_2\nab_1) \Psi_{ab}.
\eeaa
Using the Gauss formula \eqref{Gauss-eq-2-complex}, we prove the desired formula. 
\end{proof}

By putting together the canonical expression for the wave equation and the above expression for the Laplacian, we obtain the following. 

\begin{corollary}\label{corollary-wave-complex} 
The wave equation for $\Psi \in \SS_2(\mathbb{C})$ can be written as 
\bea
\bsplit
\square_2 \Psi&=-\nab_4 \nab_3 \Psi + \frac 1 2  \DD \hot (\DDb \c \Psi)+\left(2\om -\frac 1 2 \tr X\right) \nab_3\Psi- \frac 1 2 \tr\Xb \nab_4\Psi \\
& +\frac 1 2 (\Hb \c \DDb) \Psi +\frac 12 (\ov{\Hb} \c \DD) \Psi+  \left( - \frac 1 4 \tr X \ov{\tr \Xb}- \frac 1 4 \tr \Xb  \ov{\tr X}  - 2P\right) \Psi + \err_{\square_2}[\Psi]
\end{split}
\eea
where
\beaa
 \err_{\square_2}[\Psi]&=& r^{-1}\Ga_g \c \Psi.
\eeaa
\end{corollary}

\begin{proof} 
Consider \eqref{first-equation-square}:
\beaa
\square_2 \Psi&=&-\nab_4 \nab_3 \Psi +\lap_2\Psi+\left(2\om -\frac 1 2 \trch\right) \nab_3\Psi- \frac 1 2 \trchb \nab_4\Psi+2 \etab \c \nab  \Psi+(P - \ov{P}) \Psi
\eeaa
where $\lap_2$ is the Laplacian on the horizontal structure for a $2$-tensor. Using \eqref{expression-DD-laplacian-0} to write
\beaa
\lap_2  \Psi  &=& \frac 1 2  \DD \hot (\DDb \c \Psi)+ \frac 1 2 i (\atrch\nab_3+\atrchb \nab_4) \Psi  -  \left(  \frac 1 4 \tr X \ov{\tr \Xb}+ \frac 1 4 \tr \Xb  \ov{\tr X} + P + \ov{P}\right) \Psi + \Ga_g \Psi
   \eeaa
we obtain
\beaa
\square_2 \Psi&=&-\nab_4 \nab_3 \Psi + \frac 1 2  \DD \hot (\DDb \c \Psi)+\left(2\om -\frac 1 2 \tr X\right) \nab_3\Psi- \frac 1 2 \tr\Xb \nab_4\Psi  +\frac 1 2 (\Hb \c \DDb) \Psi +\frac 12 (\ov{\Hb} \c \DD) \Psi\\
&&+  \left( - \frac 1 4 \tr X \ov{\tr \Xb}- \frac 1 4 \tr \Xb  \ov{\tr X}  - 2\ov{P}\right) \Psi +\Ga_g \Psi
\eeaa
where we used $(\Hb \c \DDb) \Psi + (\ov{\Hb} \c \DD) \Psi= 4\etab  \c \nab \Psi $. 
\end{proof}

For $\Psi \in \SS_2(\mathbb{C})$ of conformal type $0$ we can write
\bea\label{wave-equation-Psi}
\begin{split}
\square_2 \Psi&=-\nabc_4 \nabc_3 \Psi +\frac 1 2  \DDc\hot( \DDbc \c \Psi) +  \left( - \frac 1 4 \tr X \ov{\tr \Xb}- \frac 1 4 \tr \Xb  \ov{\tr X} - 2P\right) \Psi\\
&-\frac 1 2 \tr \Xb \nabc_4\Psi- \frac 1 2 \tr X \nabc_3\Psi+\frac 1 2 (\Hb \c \DDbc) \Psi +\frac 12 (\ov{\Hb} \c \DDc) \Psi+ \err_{\square_2}[\Psi]
\end{split}
\eea
which is the invariant conformal wave operator for $\Psi \in \SS_2$ of conformal type $0$.


\subsection{The projection of the wave operator}
\lab{section:projwave}


In what follows, we consider the wave equation satisfied by symmetric traceless complex $2$-tensors, $A, Q(A), \qf \in \SS_{2}(\mathbb{C})$.  On the other hand, the equations governing linearized gravity around Kerr which are known in the literature, for instance the Teukolsky equation, are equations for complex scalars of spin $\pm2$. Such equations for complex scalars can be obtained by projecting the corresponding 2-tensor to its first components.

 In the following Lemma, we make the explicit connection between the wave operator applied to a 2-tensor and the wave equation verified by its projection, which is a scalar function of spin $\pm2$. 

\begin{proposition} 
Let $\Psi \in \SS_2(\mathbb{C})$ and let $\psi$ be its projection to its  first components, i.e. $\psi = \Psi(e_1, e_1)=\Psi_{11}$. Then
\bea\label{formula-projection-wave}
(\square_2 \Psi)_{11}&=&\square_{\g} \psi +i  \frac{4}{|q|^2}\frac{\cos\th}{\sin^2\th}  \Z (\psi) - \left( \frac{4}{|q|^2}\cot^2\th + a \tilde{V}\right) \psi + r^{-1} \Ga_g \psi
\eea
where $\square_2$ and $\square_{\g}$ are the wave operators for 2-tensors and scalars respectively, for perturbations of Kerr. 
\end{proposition}

\begin{proof} 
Using Corollary \ref{corollary-wave-complex}  and Lemma \ref{lemma-projection}, the projection of the wave operator $\Box_2$ for a perturbation of Kerr applied to a $2$-tensor $\Psi$ can be computed:
\beaa
(\square_2 \Psi)_{11}&=&-(\nab_4 \nab_3 \Psi )_{11}+ \frac 1 2  \DD \hot (\DDb \c \Psi)_{11}+\left(2\om -\frac 1 2 \tr X\right) (\nab_3\Psi)_{11}- \frac 1 2 \tr\Xb (\nab_4\Psi )_{11} +(2\etab \c \nab \Psi)_{11}\\
&&+  \left( - \frac 1 4 \tr X \ov{\tr \Xb}- \frac 1 4 \tr \Xb  \ov{\tr X}  - 2\ov{P}\right) \Psi_{11} + r^{-1} \Ga_g \Psi_{11}\\
&=&-e_4 e_3 \psi-i  \atrchb  e_4 ( \psi)- i  \atrch  e_3 ( \psi)\\
&&+\lap \psi - \frac 1 2  i \atrch e_3( \psi) -\frac 1 2 i\atrchb e_4( \psi) +4 i \La  e_2(\psi)+\left( -4 \La^2 +\frac 1 2 \trch\trchb +2\rho \right)\psi \\
&& -\frac 1 2( \trch - i \atrch) e_3 ( \psi)- \frac 1 2 ( \trchb - i \atrchb)  e_4 ( \psi)\\
&& +2 \etab_1 e_1(\psi)+2 \etab_2 e_2(\psi )+  \left(-\frac 1 2 \trch\trchb  - 2\rho \right) \psi+ a \tilde{V}\psi+ r^{-1} \Ga_g \psi
\eeaa
where $a \tilde{V} \psi $ denotes zero-th order term in $\psi$.
This gives, using \eqref{wave-GKS}:
\beaa
(\square_2 \Psi)_{11}&=&-e_4 e_3 \psi-\frac 1 2 \trch e_3 ( \psi)- \frac 1 2  \trchb  e_4 ( \psi)+\lap \psi+2 \etab_1 e_1(\psi)+2 \etab_2 e_2(\psi ) \\
&&-i  \atrchb  e_4 ( \psi)- i  \atrch  e_3 ( \psi) +4 i \La  e_2(\psi) -4 \La^2 \psi + a \tilde{V} \psi \\
&=&\square_{\g} \psi -i  \left( \atrchb  e_4+  \atrch  e_3 -4  \La  e_2\right) (\psi) -4 \La^2 \psi + a \tilde{V}\psi + r^{-1} \Ga_g \psi.
\eeaa
Recall that $\La= \frac{r^2+a^2}{|q|^3}\cot\th$. Using \eqref{definition-T-vectorfield} and \eqref{definition-Z-vectorfield} we have
\beaa
&&\atrchb  e_4+  \atrch  e_3 -4  \La  e_2\\
&=&\frac{2a\Delta\cos\th}{|q|^4}  e_4+  \frac{2a\cos\th}{|q|^2}  e_3 -4  \frac{r^2+a^2}{|q|^3}\frac{\cos\th}{\sin\th}  e_2\\
&=&\frac{2\cos\th}{|q|^2} \left(\frac{a\Delta}{|q|^2}  e_4+  a  e_3 -2 \frac{r^2+a^2}{|q|}\frac{1}{\sin\th}  e_2\right) \\
&=&\frac{2\cos\th}{|q|^2} \left( a  \left(2\frac{r^2+a^2}{|q|^2}\T+2\frac{a}{|q|^2}\Z\right) -2 \frac{r^2+a^2}{|q|}\frac{1}{\sin\th}   \left(\frac{a\sin\th}{|q|}\T+\frac{1}{|q|\sin\th}\Z\right)\right) \\
&=&\frac{4\cos\th}{|q|^4} \left( a^2 - (r^2+a^2)\frac{1}{\sin^2\th} \right) \Z\\
&=&-\frac{4\cos\th}{\sin^2\th|q|^4} \left( r^2  +a^2\cos^2\th\right) \Z=-\frac{4}{|q|^2}\frac{\cos\th}{\sin^2\th}  \Z
\eeaa
which therefore gives \eqref{formula-projection-wave}.
\end{proof}

\begin{remark} 
More generally, for a $k$-tensor $\Psi \in \SS_k(\mathbb{C})$, the projection of the wave equation for $\Psi$ to its first component $\psi=\Psi_{1\dots1}$ satisfies
\beaa
(\square_k \Psi)_{11}&=&\square_{\g} \psi +i  \frac{2k}{|q|^2}\frac{\cos\th}{\sin^2\th}  \Z (\psi) - \left( \frac{k^2}{|q|^2}\cot^2\th + a \tilde{V}\right) \psi + r^{-1} \Ga_g \psi.
\eeaa
In the physics literature, this is referred to as the $k$-spin weighted wave equation, see for example  \cite{Kerr-lin1}.
\end{remark}


\section{The Teukolsky equation}


The curvature components $A$ and $\Ab$ are the only quantities which are invariant up to quadratic and higher order error terms to null frame transformations and vanish in Kerr.  In the language of the   previous  section they are   $O(\ep^2)$-invariant.

It is known that these curvature components satisfy  wave equations  which   decouple from  all other components at the linear level;  the celebrated Teukolsky equations. 
In this section we derive, using our formalism,  the corresponding   Teukolsky equation  for $A$ while keeping track of  the error terms generated  by the perturbation from Kerr.

 \begin{proposition}\label{Teukolsky-proposition} 
 The complex tensor $A$ satisfies the following equation:
 \bea\label{Teukolsky-equation-tens}
 \LL(A)&=& \err[\LL(A)]
 \eea
 where
\bea\label{Teukolsky-operator}
\begin{split}
\LL(A) &=-\nabc_4\nabc_3A+ \frac{1}{2}\DDc\hot (\DDbc \c A)+\left(- \frac 1 2 \tr X -2\ov{\tr X} \right)\nabc_3A-\frac{1}{2}\tr\Xb \nabc_4A\\
&+\left( 4H+\Hb +\ov{\Hb} \right)\c \nabc A+ \left(-\ov{\tr X} \tr \Xb +2\ov{P}\right) A+  2H   \hot (\ov{\Hb} \c A)
\end{split}
\eea
with error term  expressed schematically\footnote{The error terms  denoted  $\lot$ are 
quadratic   in  the perturbation and  enjoy  better  decay  properties,  or are higher order  and decay at least as good.}
\bea
\err[\LL(A)]&=& r^{-1}\frak{d}^{\leq 1} \left( \Ga_g B\right)+\Xi \nab_3 B+\lot
\eea
\end{proposition}

\begin{proof}
According to Proposition \ref{prop:bianchi:complex}, we have the following Bianchi identity for $A$:
\bea\label{Bianchi-identity-A}
 \nabc_3A -\DDc\hot B &=& -\frac{1}{2}\tr\Xb A +  4H   \hot B -3\ov{P}\Xh.
\eea
Apply $\nabc_4$ to \eqref{Bianchi-identity-A}:
\beaa
\nabc_4\nabc_3A &=&\nabc_4(\DDc\hot B )-\frac{1}{2}\nabc_4\tr\Xb A-\frac{1}{2}\tr\Xb \nabc_4A \\
&&+ 4\nabc_4( H )  \hot B + 4 H   \hot\nabc_4( B )-3\nabc_4\ov{P}\Xh-3\ov{P}\nabc_4\Xh.
\eeaa
We compute the right hand side of the above equation.

According to Lemma \ref{commutator-nab-c-3-DD-c-hot}, we apply the commutation formula \eqref{commutator-nabc-4-F-formula} to $B$, which is a one form of conformal type $1$. Using the definition of $\Ga_b$ and $\Ga_g$ as in \eqref{definition-Gammas}, we have
 \beaa
 \, [\nabc_4 , \DDc \hot ]B &=&- \frac 1 2 \tr X( \DDc\hot B )+ \underline{H} \hot \nabc_4 B+ \err_{4\DDc\hot}[B]
 \eeaa
 where
 \beaa
  \err_{4\DDc\hot}[B]
 &=&\frac{1}{r} \Ga_g \left(\nab_3 +r \DD\right)B+ \left(B  +\frac{a}{r^2} \Ga_g\right) B.  
 \eeaa

We therefore obtain
 \beaa
&&  \nabc_4(\DDc\hot B )+ 4 H   \hot\nabc_4( B )\\
&=&\DDc\hot(\nabc_4 B )+ 4 H   \hot\nabc_4( B ) - \frac 1 2 \tr X \DDc\hot B + \underline{H} \hot \nabc_4 B+ \err_{4\DDc\hot}[B]\\
 &=&\DDc\hot(\nabc_4 B )  + \left(4 H  + \underline{H} \right)\hot \nabc_4 B- \frac 1 2 \tr X \DDc\hot B  +\err_{4\DDc\hot}[B].
 \eeaa
According to Proposition \ref{prop:bianchi:complex}, we have the following Bianchi identity for $B$: 
\beaa
\nabc_4B -\frac{1}{2} \DDbc \c A &=& -2\ov{\tr X} B +\frac{1}{2}A\c \ov{\Hb}+3 \ov{P} \  \Xi.
\eeaa
We therefore obtain
 \bea\label{term1}
 \begin{split}
& \nabc_4(\DDc\hot B )+ 4 H   \hot\nabc_4( B )\\
 &=\DDc\hot\left(\frac{1}{2} \DDbc \c A -2\ov{\tr X} B +\frac{1}{2}A\c \ov{\Hb}+3 \ov{P} \  \Xi\right)\\
 &  + \left(4 H  + \underline{H} \right)\hot \left(\frac{1}{2} \DDbc \c A -2\ov{\tr X} B +\frac{1}{2}A\c \ov{\Hb}+3 \ov{P} \  \Xi\right)- \frac 1 2 \tr X \DDc\hot B +\err_{4\DDc\hot}[B]\\
  &=\frac{1}{2}\DDc\hot (\DDbc \c A + \ov{\Hb} \c A) +\left(- \frac 1 2 \tr X -2\ov{\tr X} \right)\DDc \hot B \\
 &  + \left( 2H  +\frac 1 2 \underline{H} \right)\hot ( \DDbc \c A + \ov{\Hb} \c A)+2(-\DDc\ov{\tr X}-\ov{\tr X} \left(4 H  + \underline{H} \right))\hot  B \\
 &+3 \ov{P}\left( \DDc \hot \Xi+(4H+\underline{H})\hot\Xi\right)+\err_{4\DDc\hot}[B]+ 3 \DDc \ov{P} \hot \Xi.
 \end{split}
 \eea
Using the null structure equation 
\beaa
\nabc_4\tr\Xb +\frac{1}{2}\tr X\tr\Xb &=& \DDc\c\ov{\Hb}+\Hb\c\ov{\Hb}+2\ov{P}+\Ga_b \Ga_g
\eeaa
we obtain
\bea\label{term2}
\begin{split}
&-\frac{1}{2}\nabc_4\tr\Xb A-\frac{1}{2}\tr\Xb \nabc_4A\\
&=\left(\frac{1}{4}\tr X\tr\Xb -\frac 1 2 ( \DDc\c\ov{\Hb}+\Hb\c\ov{\Hb})-\ov{P}\right) A-\frac{1}{2}\tr\Xb \nabc_4A+\text{cubic terms}.
\end{split}
\eea
Using the null structure equation 
\beaa
\nabc_4H-\nabc_3 \Xi &=&  -\frac{1}{2}\ov{\tr X}(H-\Hb) -\frac{1}{2}\Xh\c(\ov{H}-\ov{\Hb}) -B
\eeaa
we obtain
\bea\label{term3}
4\nabc_4( H )  \hot B&=&  -2\ov{\tr X}(H-\Hb)  \hot B+\left(\nabc_3\Xi +\frac{a}{r^2} \Ga_g+B\right)B.
\eea
Finally using 
\beaa
\nabc_4\Xh+\Re(\tr X)\Xh&=& \DDc \hot \Xi + \Xi \hot (\Hb +H) -A, \\
\nabc_4P -\frac{1}{2}\DDc\c \ov{B} &=& -\frac{3}{2}\tr X P + \Hb \c\ov{B} + \left(\Xi \right) \Bb +\Ga_b  A,
\eeaa
we obtain
\bea\label{term4}
\begin{split}
&-3\nabc_4\ov{P}\Xh-3\ov{P}\nabc_4\Xh\\
&=-3(-\frac{3}{2}\ov{\tr X} \ \ov{ P})\Xh-3\ov{P}(-\frac 1 2 (\tr X+\ov{\tr X})\Xh -A+ \DDc \hot \Xi + \Xi \hot (\Hb +H))+\Ga_g \DDc B+\frac{a}{r^2} \Ga_g B \\
&= \left(\frac 3 2\tr X+6\ov{\tr X}\right)\ov{P}\Xh +3\ov{P}A-3\ov{P}( \DDc \hot \Xi + \Xi \hot (\Hb +H))+\Ga_g \DDc B+\frac{a}{r^2} \Ga_g B.
\end{split}
\eea
Summing \eqref{term1}, \eqref{term2}, \eqref{term3} and \eqref{term4}, we obtain
\beaa
\nabc_4\nabc_3A &=& \left(- \frac 1 2 \tr X -2\ov{\tr X} \right)\DDc \hot B +\left(\frac 3 2\tr X+6\ov{\tr X}\right)\ov{P}\Xh\\
 &&  +2(-\DDc\ov{\tr X}-5\ov{\tr X}  H  )\hot  B \\
&&+ \left(\frac{1}{4}\tr X\tr\Xb -\frac 1 2 ( \DDc\c\ov{\Hb}+\Hb\c\ov{\Hb})+2\ov{P}\right) A-\frac{1}{2}\tr\Xb \nabc_4A\\
&&+\frac{1}{2}\DDc\hot (\DDbc \c A + \ov{\Hb} \c A)+ \left( 2H  +\frac 1 2 \underline{H} \right)\hot ( \DDbc \c A + \ov{\Hb} \c A)\\
&&+\frac{1}{r} \Ga_g \left(\nab_3 +r \DD\right)B+\left(\nabc_3\Xi +\frac{a}{r^2} \Ga_g+B\right)B.
\eeaa
Using again \eqref{Bianchi-identity-A} we have 
\beaa
&&\left(- \frac 1 2 \tr X -2\ov{\tr X} \right)\DDc \hot B +\left(\frac 3 2\tr X+6\ov{\tr X}\right)\ov{P}\Xh\\
&=&\left(- \frac 1 2 \tr X -2\ov{\tr X} \right)( \DDc \hot B -3\ov{P}\Xh)\\
&=&\left(- \frac 1 2 \tr X -2\ov{\tr X} \right)\left(\nabc_3A+\frac{1}{2}\tr\Xb A-4 H   \hot B \right).
\eeaa
This gives
\beaa
\nabc_4\nabc_3A &=& \left(- \frac 1 2 \tr X -2\ov{\tr X} \right)\nabc_3A-\frac{1}{2}\tr\Xb \nabc_4A\\
 &&  +2\left(-\DDc\ov{\tr X}+(\tr X-\ov{\tr X} ) H  \right)\hot  B \\
&&+ \left(-\ov{\tr X} \tr \Xb -\frac 1 2 ( \DDc\c\ov{\Hb}+\Hb\c\ov{\Hb})+2\ov{P}\right) A\\
&&+\frac{1}{2}\DDc\hot (\DDbc \c A + \ov{\Hb} \c A)+ \left( 2H  +\frac 1 2 \underline{H} \right)\hot ( \DDbc \c A + \ov{\Hb} \c A)\\
&&+\frac{1}{r} \Ga_g \left(\nab_3 +r \DD\right)B+\left(\nabc_3\Xi +\frac{a}{r^2} \Ga_g+B\right)B.
\eeaa
By Codazzi equation
\beaa
\frac{1}{2}\DDbc\c\Xh  &=& \frac{1}{2}\DDc\ov{\tr X} -i\Im(\tr X)(H+\Xi)-B
\eeaa
therefore the second line is absorbed by the quadratic terms in the last line. We therefore have
\beaa
\nabc_4\nabc_3A &=& \left(- \frac 1 2 \tr X -2\ov{\tr X} \right)\nabc_3A-\frac{1}{2}\tr\Xb \nabc_4A\\
&&+ \left(-\ov{\tr X} \tr \Xb -\frac 1 2 ( \DDc\c\ov{\Hb}+\Hb\c\ov{\Hb})+2\ov{P}\right) A\\
&&+\frac{1}{2}\DDc\hot (\DDbc \c A + \ov{\Hb} \c A)+ \left( 2H  +\frac 1 2 \underline{H} \right)\hot ( \DDbc \c A + \ov{\Hb} \c A)+\err[\LL(A)]
\eeaa
where
\beaa
\err[\LL(A)]&=&\frac{1}{r} \Ga_g \left(\nab_3 +r \DD\right)B+\left(\nabc_3\Xi +\frac{a}{r^2} \Ga_g+B\right)B.
\eeaa

The last line of the above equation involving the angular derivatives of $A$ can be written as
\beaa
&&\frac{1}{2}\DDc\hot (\DDbc \c A + \ov{\Hb} \c A)+ \left( 2H  +\frac 1 2 \underline{H} \right)\hot ( \DDbc \c A + \ov{\Hb} \c A)\\
&=& \frac{1}{2}\DDc\hot (\DDbc \c A)+\frac{1}{2}\DDc\hot ( \ov{\Hb} \c A)+ \left( 2H  +\frac 1 2 \underline{H} \right)\hot ( \DDbc \c A )+ \left( 2H  +\frac 1 2 \underline{H} \right)\hot (\ov{\Hb} \c A).
\eeaa
Applying \eqref{simil-Leibniz} and \eqref{Leibniz-hot}, we write
\beaa
\DDc\hot ( \ov{\Hb} \c A)&=&  (\DDc\c \ov{\Hb} ) A + (\ov{\Hb}\c \DDc) A,\\
\Hb \hot ( \ov{\Hb} \c A)&=& ( \Hb \c \ov{\Hb}) \ A,
\eeaa
which implies
\beaa
\nabc_4\nabc_3A &=& \frac{1}{2}\DDc\hot (\DDbc \c A)+\left(- \frac 1 2 \tr X -2\ov{\tr X} \right)\nabc_3A-\frac{1}{2}\tr\Xb \nabc_4A\\
&&+ \left(-\ov{\tr X} \tr \Xb -\frac 1 2 ( \DDc\c\ov{\Hb}+\Hb\c\ov{\Hb})+2\ov{P}\right) A\\
&&+\frac{1}{2}\left( (\DDc\c \ov{\Hb} ) A + (\ov{\Hb}\c \DDc) A\right)\\
&&+ \left( 2H  +\frac 1 2 \underline{H} \right)\hot ( \DDbc \c A )+ \left( 2H   \right)\hot (\ov{\Hb} \c A)+  \frac 1 2( \Hb \c \ov{\Hb}) \ A+\err[\LL(A)]\\
&=& \frac{1}{2}\DDc\hot (\DDbc \c A)+\left(- \frac 1 2 \tr X -2\ov{\tr X} \right)\nabc_3A-\frac{1}{2}\tr\Xb \nabc_4A\\
&&+\frac{1}{2}  (\ov{\Hb}\c \DDc) A+ \left( 2H  +\frac 1 2 \underline{H} \right)\hot ( \DDbc \c A )+ \left(-\ov{\tr X} \tr \Xb +2\ov{P}\right) A\\
&&+ \left( 2H   \right)\hot (\ov{\Hb} \c A)+\err[\LL(A)].
\eeaa
Using Lemma \ref{simplification-angular}, we further simplify the angular part writing 
\beaa
\frac 12 (\ov{\Hb}\c \DDc )A+ \left( 2H  +\frac 1 2 \underline{H} \right)\hot ( \DDbc \c A )&=& \frac 12 (\ov{\Hb}\c \DDc )A+\left( \left( 2H  +\frac 1 2 \underline{H} \right)\c \DDbc \right) A \\
&=& \frac 12 (\ov{\Hb}\c \DDc+\Hb \c \DDbc )A+(2H \c \DDbc ) A \\
&=& \left( 2\etab +4H \right)\c \nabc A \\
&=& \left( 4H+\Hb +\ov{\Hb} \right)\c \nabc A.
\eeaa
This proves the proposition. 
\end{proof}

\begin{remark} 
The Teukolsky equation \eqref{Teukolsky-equation-tens} is a tensorial equation for  $A\in \SS_2(\mathbb{C})$, as defined in our formalism.   The standard derivation  of the equation, in linear theory,  is done instead with respect to  the Newmann-Penrose  formalism, see section \ref{section:NPformalism}.  To  relate the Teukolsky equation in our formalism to the classical one  in NP formalism we have to project it with respect to the horizontal frame $e_1, e_2$.  One can check in fact  that  the standard  Teukolsky variable, which we denote by 
$\a^{[+2]}$, is related  to $A$ via the formula 
\beaa
\a^{[+2]}&:=& -\frac{\ov{q}}{ q} A_{11}
\eeaa
where $A_{11}=A(e_1, e_1)$.   One can   check,  see  Appendix \ref{Teuk-literature}, that
in the particular case of the Kerr metric  we have   
\beaa
|q|^2\square_{{m,a}} \a^{[+2]}&=&-4(r-m) \pr_r \a^{[+2]}-4 \left(\frac{m(r^2-a^2)}{\Delta} - r -i a \cos\th \right)\pr_t \a^{[+2]} \\
&&- 4\left(\frac{a(r-m)}{\Delta} + i \frac{\cos\th}{\sin^2\th} \right) \pr_\vphi \a^{[+2]}+ (4 \cot^2\th-2)\a^{[+2]}
\eeaa
 where $\square_{{m,a}}  $ is the D'Alembertian  relative to the  Kerr metric.  This  is  the standard form of the Teukolsky equation in Boyer-Lindquist coordinates, see \cite{Teuk}. 
\end{remark}


\section{The Regge-Wheeler-type equation}


Just like in Schwarzschild, there are other $O(\ep^2)$ invariant quantities to be considered in Kerr. 
\begin{lemma}\label{lemma:basicinvarianceandconformalinvarianceofQofA} 
The expression 
\bea\label{definition-Q(A)}
 Q(A)&=& \nabc_3\nabc_3 A + \C \  \nabc_3A +\D \  A \in \SS_2
 \eea
 for any scalar functions $\C$ and $\D$ is $O(\ep^2)$ invariant. Moreover, it is conformally invariant of type $0$ provided $\C$ is conformally invariant of type $-1$ and $\D$ is conformally invariant of type $-2$. 
\end{lemma}

\begin{proof} 
Clearly the quantity vanishes in Kerr and is an $O(\ep^2)$ invariant. By construction, it is also conformally invariant under the above conditions on $\C$ and $\D$. 
\end{proof}

We will also need the following rescaled version of $Q$.
\begin{definition} 
Given a general null frame $(e_3, e_4, e_1, e_2)$ and given scalar functions $r$ and $\th$ satisfying the assumptions in Section \ref{perturbations-section}, we define our main quantity $\qf \in \SS_2$ as 
\bea
\qf&=& q \ov{q}^{3} Q(A)=q \ov{q}^{3} \left( \nabc_3\nabc_3 A + \C \  \nabc_3A +\D \  A\right)
\eea
where $q=r+ i a \cos\th$. 
\end{definition}

The quantity $\qf$ can be seen as a second order differential operator in $A$. In particular, this is a physical space version of the Chandrasekhar transformation, which transforms the Teukolsky equation satisfied by $A$, as in Proposition \ref{Teukolsky-proposition}, into a Regge-Wheeler-type equation. We now state the main result of the paper  concerning  the wave equation satisfied by $\qf$. 
 
 \begin{theorem}\label{main-theorem} 
 There exists choices of complex scalar functions $\C$, $\D$,  see \eqref{finalchoicefordefinition-C} \eqref{finalchoicefordefinition-D},  such that  the invariant symmetric traceless $2$-tensor $\qf \in \SS_2$ verifies the equation
 \bea\label{wave-equation-qf}
 \square_2 \qf   - \frac{4i a\cos\th}{|q|^2} \T( \qf )   - V  \qf &=& a\,  L_{\qf}[A] + \err[\square_2 \qf]
 \eea
 where
 \begin{itemize}
 \item the potential $V$ is a complex scalar function, see also Remark \ref{rmk:propertiesofthepotentialV} below,
 
 \item $ L_{\qf}[A ]$ is  a  linear second order   operator  in $A$,   given by
  \beaa
L_{\qf}[A] &=& c_1\nabc_2 ( \nabc_3 A) + c_2\nabc_3 A + c_3 \nabc (A) + c_4 A
\eeaa
with $c_1,\ldots, c_4$  smooth  functions  of $(r, \th)$ which have the following fall-off in $r$ 
\beaa
c_1=O\left(r\right), \qquad c_2=O\left(1\right), \qquad c_3=O\left(1\right), \qquad c_4=O\left(\frac{1}{r}\right),
\eeaa

\item  The error term is given schematically by\footnote{The error terms  denoted  $\lot$ are 
quadratic   in  the perturbation and  enjoy  better  decay  properties,  or are higher order  and decay at least as good.} 
 \beaa
 \err[\square_2 \qf]&=& r^2 \frak{d}^{\leq 2} (\Ga_g \c (A, B)) + \nab_3 (r^3 \frak{d}^{\leq 2}( \Ga_g \c (A, B)))\\
 &&+\frak{d}^{\leq 1} (\Ga_g \qf)  + r^2 \widecheck{H} \nab^{\leq 1} \nab_3^{\leq 1} A+\lot
 \eeaa
 \end{itemize}
 \end{theorem}

\begin{remark} 
The Regge-Wheeler type equation in Kerr was obtained by Ma \cite{Ma}  and  Dafermos,  Holzegel  and Rodnianski \cite{D-H-R-Kerr}  starting with  the NP  complex scalar curvature component $\Psi_0$, corresponding to $A_{11}=\a_{11}+i \dual \a_{11}$. Our  equation \eqref{wave-equation-qf} is instead a tensorial equation in curved background  with 
precise reference to the error terms, see also the similar equation derived in \cite{KS}. We note that, in the particular case of Kerr,   equations derived in \cite{Ma},  \cite{D-H-R-Kerr}    can be obtained      from \eqref{wave-equation-qf} by projection to the $1$-$1$ component. Such a projection modifies the equations  by    the appearance of Christoffel symbols of the horizontal distributions, see Section \ref{section-projection-equation} for the projection of the Regge-Wheeler type equation to its first component.  
\end{remark}

\begin{remark}\label{rmk:propertiesofthepotentialV}
The potential term $V$  coincides with the potential $-\trch\trchb$, appearing  in \cite{KS} in the context of perturbations of Schwarzschild, plus terms multiplied by the angular momentum $a$. More precisely,
\beaa
\Re(V)-V_0 = O\left(\frac{|a|}{r^4}\right), \qquad \Im(V)=O\left(\frac{|a|}{r^3}\right), \qquad V_0:=-\trch\trchb.
\eeaa
 In particular,  for small angular momentum, $\Re(V)-V_0$ and $\Im(V)$ can be treated as  lower order terms, and absorbed by the left hand side. 
\end{remark}


\subsection{Proof of Theorem \ref{main-theorem}}


 Recall the Teukolsky equation as in Proposition \ref{Teukolsky-proposition}, i.e.
\bea\label{Teuk-repeat}
 \LL(A)&=& \err[\LL(A)].
\eea
We want to apply the Chandrasekhar transformation, i.e. the operator $Q$, to the above. We compute the commutator $[Q,\LL ]A$ between $\LL$ and the second order differential operator $Q$, as defined in \eqref{definition-Q(A)}, for any scalar functions $\C$ and $\D$. In order to cancel out the highest order terms in the equation for $Q(A)$ we need to impose differential equations on the functions $\C$ and $\D$.  We obtain the following.
\begin{proposition}\label{first-intermediate-step-main-theorem} 
Let $Q(A)= \nabc_3 \nabc_3 A+ \C \nabc_3 A + \D \ A$ such that $\C$ and $\D$ satisfy
 \bea
 \nabc_3 \C+\frac {\C }{2} (\tr \Xb+\ov{\tr\Xb}) - \ov{\tr\Xb} \tr\Xb&=& \Ga_g, \label{equation-nab-3-c-1}\\
 \nabc_3 \D +\D \left(\tr \Xb+\ov{\tr \Xb}\right)-\frac \C 4\tr\Xb(\ov{\tr\Xb})&=& r^{-1}\Ga_g. \label{equation-nab-3-d-1}
 \eea
 Then the commutator between $Q$ and $\LL$ is given by 
 \bea\label{final-commutator}
 \begin{split}
 [Q, \LL](A) &= 4\etab  \c \nabc Q(A)-  \left(\tr \Xb+\ov{\tr \Xb}\right)  \nabc_4Q(A)+C_0(  Q(A))\\
 &+ a\, L_Q(A)+\err[ [Q, \LL]A]
 \end{split}
 \eea
 where $C_0(  Q(A))$ are zero-th order terms in $Q(A)$,  $ L_Q(A)$ are linear lower order terms in $Q(A)$ of the schematic form 
 \beaa
L_Q(A) &=& d_1 \nabc ( \nabc_3 A) + d_2 \nabc_3 A + d_3 \nabc(A) + d_4 A,
\eeaa
with $d_1,\ldots, d_4$  smooth  functions  of $(r, \th)$ which have the following fall-off in $r$ 
\beaa
d_1=O\left(\frac{1}{r^3}\right), \qquad d_2=O\left(\frac{1}{r^4}\right), \qquad d_3=O\left(\frac{1}{r^4}\right), \qquad d_4=O\left(\frac{1}{r^5}\right).
\eeaa
The error terms $\err[ [Q, \LL]A]$ are given by
\beaa
\err[ [Q, \LL]A]&=&  \nabc_3\left( \frac 1 r \Ga_g \c \frak{d}^{\leq 2} A\right)+\lot
\eeaa
 \end{proposition}
 
 \begin{proof}
 See Appendix \ref{proof-appendix}. 
 \end{proof}

 Observe that the transport equation given by \eqref{equation-nab-3-c-1} only imposes conditions on the real part of the coefficient $\C$. Indeed, any function $\C$ of the form $\C= 2 \trchb + i \frak{c} \atrchb$
for any constant $\frak{c}$ satisfies \eqref{equation-nab-3-c-1}. Indeed
\beaa
&& \nabc_3 \C+\frac {\C }{2} (\tr \Xb+\ov{\tr\Xb}) - \ov{\tr\Xb} \tr\Xb\\
&=&  \nabc_3(2 \trchb + i \frak{c} \atrchb)+(2 \trchb + i \frak{c} \atrchb) \trchb  - (\trchb^2+\atrchb^2)\\
&=&  2\left(-\frac 1 2 \big( \trchb^2-\atrchb^2\big)+\Ga_g\right) + i \frak{c} (-\trchb \atrchb +\Ga_g)\\
&&+(2 \trchb + i \frak{c} \atrchb) \trchb  - (\trchb^2+ \atrchb^2)=\Ga_g.
\eeaa
Similarly, let $\D$ be a function of the form $\D= \frac 1 2 \trchb^2 + \frak{f} \atrchb^2+i \frak{e} \trchb \atrchb$. 
Then we have
\beaa
&& \nabc_3 \D +\D \left(\tr \Xb+\ov{\tr \Xb}\right)-\frac \C 4\tr\Xb(\ov{\tr\Xb})\\
&=& \nabc_3 \left( \frac 1 2 \trchb^2 + \frak{f} \atrchb^2+i \frak{e} \trchb \atrchb\right) +\left( \frac 1 2 \trchb^2 + \frak{f} \atrchb^2+i \frak{e} \trchb \atrchb \right) 2\trchb\\
&&-\frac 1 4( 2 \trchb + i \frak{c} \atrchb) (\trchb^2+ \atrchb^2)\\
&=&  \trchb  \left(-\frac 1 2 \big( \trchb^2-\atrchb^2\big)+\Ga_g\right) +2 \frak{f} \atrchb (-\trchb\atrchb+\Ga_g)\\
&& +i \frak{e} \left(-\frac 1 2 \big( \trchb^2-\atrchb^2\big)+\Ga_g\right) \atrchb+i \frak{e}  \trchb (-\trchb\atrchb+\Ga_g) \\
&&+\left( \frac 1 2 \trchb^2 + \frak{f} \atrchb^2+i \frak{e} \trchb \atrchb \right) 2\trchb-\frac 1 4( 2 \trchb + i \frak{c} \atrchb) (\trchb^2+ \atrchb^2)\\
&=&i\atrchb  \left( \trchb^2+ \atrchb^2\right)\left( \frac 1 2 \frak{e}-\frac 1 4 \frak{c} \right)+r^{-1} \Ga_g.
\eeaa
In particular, for any constants $\frak{f}$ and $\frak{c}$
\bea
\C&=& 2 \trchb + i \frak{c} \atrchb, \label{definition-c-general}\\
\D&=& \frac 1 2 \trchb^2 + \frak{f} \atrchb^2+i \frac 1 2\frak{c} \trchb \atrchb, \label{definition-d-general}
\eea
satisfy  \eqref{equation-nab-3-c-1} and \eqref{equation-nab-3-d-1}.

We now select a choice for the imaginary part of the function $\C$ so that the linear lower order terms can be simplified. 
We have the following lemma.

\begin{lemma}\label{lemma-lot} 
Let $\C$ be given by 
\bea\label{finalchoicefordefinition-C}
\C&=& 2 \trchb -4 i  \atrchb.
\eea
Then $\C$ satisfies the assumptions of Proposition \ref{first-intermediate-step-main-theorem} and the linear lower order terms have the following structure:
 \beaa
L_Q(A) &=& d_1 \nabc_2 ( \nabc_3 A) + d_2 \nabc_3 A + d_3 \nabc(A) + d_4 A,
\eeaa
with the fall-off in $r$ given by Proposition \ref{first-intermediate-step-main-theorem}. 
\end{lemma}

\begin{proof} 
See Appendix \ref{appendiz-s}. 
\end{proof}

\begin{remark}
The choice \eqref{finalchoicefordefinition-C} corresponds to fixing $\frak{c}=-4$ which yields for $\D$
\bea\label{finalchoicefordefinition-D}
\D&=& \frac 1 2 \trchb^2 + \frak{f} \atrchb^2 -2i \trchb \atrchb
\eea
for any constant $\frak{f}$. Note that with such choices, $\C$ is conformally invariant of type $-1$ and $\D$ is conformally invariant of type $-2$, so that $Q(A)$ is conformally invariant of type $0$ according to Lemma \ref{lemma:basicinvarianceandconformalinvarianceofQofA}. 
\end{remark}

We therefore consider $Q(A)$ defined using these functions $\C$ and $\D$, i.e.
\bea\label{definition-Q-A-C-D}
 Q(A)&=& \nabc_3\nabc_3 A + \left(2\trchb -4 i \atrchb \right) \  \nabc_3A +\left( \frac 1 2 \trchb^2  -2i \trchb \atrchb \right)A.
\eea
 
  Applying the operator $Q$ to \eqref{Teuk-repeat}, we obtain
 \bea\label{Teuk-repeat=2}
 \LL(Q(A))+[Q, \LL](A)&=& Q( \err[\LL(A)]).
 \eea
 Recall that 
 \beaa
 \LL(A) &=&-\nabc_4\nabc_3A+ \frac{1}{2}\DDc\hot (\DDbc \c A)+\left(- \frac 1 2 \tr X -2\ov{\tr X} \right)\nabc_3A-\frac{1}{2}\tr\Xb \nabc_4A\\
&&+\left( 4H+\Hb +\ov{\Hb} \right)\c \nabc A+ \left(-\ov{\tr X} \tr \Xb +2\ov{P}\right) A+  2H   \hot (\ov{\Hb} \c A)
 \eeaa
 which gives\footnote{Recall that $Q(A)$ is of conformal type $0$, therefore all conformal derivatives coincide with the non-conformal ones.} for $Q=Q(A)$
 \beaa
  \LL(Q) &=&-\nab_4\nab_3Q+ \frac{1}{2}\DD\hot (\DDb \c Q)+\left(- \frac 1 2 \tr X -2\ov{\tr X}-2\om \right)\nab_3Q-\frac{1}{2}\tr\Xb \nab_4Q\\
&&+\left( 4H+\Hb +\ov{\Hb} \right)\c \nab Q+ \left(-\ov{\tr X} \tr \Xb +2\ov{P}\right) Q+  2H   \hot (\ov{\Hb} \c Q).
 \eeaa
Using Proposition \ref{first-intermediate-step-main-theorem}, equation \eqref{Teuk-repeat=2} gives
 \beaa
&& -\nab_4\nab_3Q+ \frac{1}{2}\DD\hot (\DDb \c Q)+\left(- \frac 1 2 \tr X -2\ov{\tr X} \right)\nab_3Q-\left(\frac 3 2 \tr \Xb+\ov{\tr \Xb}\right)  \nab_4Q+2 \etab \c \nab Q\\
&&+\left( 4H+2\Hb +2\ov{\Hb} \right)\c \nab Q+\tilde{C}_0(  Q)\\
&=&a\, L_Q(A) + Q( \err[\LL(A)])+\err[ [Q, \LL]A]
 \eeaa
 where we wrote $4 \etab \c \nab Q= 2 \etab \c \nab Q+ (\Hb +\ov{\Hb}) \c \nab Q$ using \eqref{relation0angular=der}. 
 
  Recall \eqref{wave-equation-Psi} applied to $Q$: 
 \beaa
\square_2 Q&=&-\nabc_4 \nabc_3 Q + \frac 1 2 \DD\hot( \DDb \c Q) -\frac 1 2 \tr X \nabc_3Q- \frac 1 2 \tr \Xb \nabc_4Q+2 \etab \c \nab  Q\\
&& + \left( - \frac 1 4 \tr X \ov{\tr \Xb}- \frac 1 4 \tr \Xb  \ov{\tr X} - 2P\right)Q+ r^{-1} \Ga_g Q.
\eeaa
We can therefore write the equation for $Q$ as 
 \bea\label{wave-eq-Q}
 \begin{split}
\square_2 Q&=2\ov{\tr X} \nab_3Q+\left( \tr \Xb+\ov{\tr \Xb}\right)  \nab_4Q-\left( 4H+2\Hb +2\ov{\Hb} \right)\c \nab Q+\tilde{V}Q\\
&+a\c L_Q(A) + \err[\square_2 Q]
\end{split}
 \eea
 where $\tilde{V}Q$ collects the linear zero-th order terms in $Q$, and the error terms are given by $\err[\square_2 Q]=  Q( \err[\LL(A)])+\err[ [Q, \LL]A]$. 
 
  We now want to rescale $Q$ in order to absorb the first order terms in \eqref{wave-eq-Q} into the wave operator. Observe that for a scalar function $f$, we have 
 \bea\label{square-f-Q-1}
\square_2( f Q)&=& \square(f) Q+f \square_2(Q)- \nab_3 f \nab_4Q- \nab_4f \nab_3 Q +2\nab f \c \nab Q.
\eea
 Let $f$ be given by 
 \beaa
 f&=& q \ov{q}^{3}.
 \eeaa
Recalling that
\beaa
\nab_3 q&=& \frac 1 2 \ov{\tr \Xb} \,  q +r \Ga_b,\\
\nab_4 q&=& \frac 1 2 \tr X q+r\Ga_g,
\eeaa
we deduce
 \beaa
 \nab_3(f)&=& \left(\frac 1 2 \ov{\tr \Xb} +\frac 3 2   \tr \Xb \right) f+r^4 \Ga_b, \\ 
  \nab_4(f) &=& \left(\frac 1 2 \tr X +\frac 3 2   \ov{\tr X} \right) f+r^4 \Ga_g.
 \eeaa
 Recalling that 
 \beaa
   \DDco q &=&  q\ov{H}+ r \Ga_g, \\
    \DD q&=&  q  \Hb + r \Ga_g,
    \eeaa
    we deduce
    \beaa
    2\nab q &=& (\nab + i \dual \nab) q + (\nab- i \dual \nab) q =\DD q + \DDco q= q(  \Hb+\ov{H})+ r \Ga_g, \\
    2 \nab \ov{q}&=& \ov{q}(  \ov{\Hb}+H) + r \Ga_g,
    \eeaa
   and therefore 
    \beaa
    \nab f&=&  \frac 1 2 (  \Hb+\ov{H})  f+\frac 3 2 (  \ov{\Hb}+H)f+r^4 \Ga_g,\\
    &=&  \left(\frac{3}{2} H+\frac 12\ov{H}+\frac 12   \Hb+ \frac{3 }{2}  \ov{\Hb}  \right)f+r^4 \Ga_g.
    \eeaa
     Defining
 \beaa
 \frak{q}&=&fQ= q \ov{q}^{3} Q
 \eeaa
 we combine \eqref{wave-eq-Q} and \eqref{square-f-Q-1} and obtain
 \beaa
 \square_2 \qf &=&2\ov{\tr X} \nab_3\qf +\left( \tr \Xb+\ov{\tr \Xb}\right)  \nab_4\qf -\left( 4H+2\Hb +2\ov{\Hb} \right)\c \nab \qf+ V\qf \\
 &&- \left(\frac 1 2 \ov{\tr \Xb} +\frac 3 2   \tr \Xb \right) \nab_4\qf- \left(\frac 1 2 \tr \Xb +\frac 3 2   \ov{\tr \Xb} \right) \nab_3 \qf +2\left(\frac{3}{2} H+\frac 12\ov{H}+\frac 12   \Hb+ \frac{3 }{2}  \ov{\Hb}  \right)\c \nab \qf\\
&&+f a\, L_Q(A)+f \err[\square_2 Q]\\
&=&\frac 1 2   \left( \ov{\tr \Xb} -   \tr \Xb \right) \nab_4\qf+\frac 1 2  \left(  \ov{\tr X}- \tr X  \right) \nab_3 \qf +\left(- H+\ov{H}-   \Hb+   \ov{\Hb}  \right)\c \nab \qf+ V\qf\\
&&+a\, L_{\qf}(A)+ \err[\square_2 \qf]
 \eeaa
 where 
 \beaa
 L_{\qf}(A)=f  \left( L_Q(A)\right) = c_1 \nabc ( \nabc_3 A) + c_2 \nabc_3 A + c_3 \nabc(A) + c_4 A.
 \eeaa
with $c_1,\ldots, c_4$  smooth  functions  of $(r, \th)$ which have the following fall-off in $r$ 
\beaa
c_1=O\left(r \right), \qquad c_2=O\left(1\right), \qquad c_3=O\left(1\right), \qquad c_4=O\left(\frac{1}{r}\right),
\eeaa
  and $V\qf$ collects the zero-th order terms in $\qf$. 
    
     We are left to analyze the first order terms in the above equation for $\qf$. We have
    \beaa
  &&  \frac 1 2   \left( \ov{\tr \Xb} -   \tr \Xb \right) \nab_4+\frac 1 2  \left(  \ov{\tr X}- \tr X  \right) \nab_3 =i \left(  \atrch \nab_3+    \atrchb   \nab_4\right).
    \eeaa
 Observe that,    in view of our definition of $\T$ and $\Z$, see \eqref{definition-T-vectorfield} \eqref{definition-Z-vectorfield}, 
 \bea
\frac{|q|^2}{2}\left(e_3+\frac{\Delta}{|q|^2}e_4\right) &=& (r^2+a^2)\T+a\Z 
 \eea
  so that
 \beaa
   \atrch \nab_3+    \atrchb   \nab_4&=&  \frac{2a\cos\th}{|q|^2}e_3+   \frac{2a\Delta\cos\th}{|q|^4}e_4+\Ga_g\dk\\
  &=&  \frac{4a\cos\th(r^2+ a^2)}{|q|^4} \T +   \frac{4a^2\cos\th}{|q|^4}  \Z +\Ga_g\dk. 
  \eeaa
  This gives
  \beaa
  \frac 1 2   \left( \ov{\tr \Xb} -   \tr \Xb \right) \nab_4\qf+\frac 1 2  \left(  \ov{\tr X}- \tr X  \right) \nab_3\qf   &=&  \frac{4ia\cos\th(r^2+ a^2)}{|q|^4} \T(\qf) +   \frac{4ia^2\cos\th}{|q|^4} \Z(\qf)+\Ga_g\dk.
 \eeaa

 In view of our definition of $\T$ and $\Z$, see \eqref{definition-T-vectorfield} \eqref{definition-Z-vectorfield}, we have
 \bea
 e_2 &=& \frac{a\sin\th}{|q|}\T+\frac{1}{|q|\sin\th}\Z.
 \eea
 Also,   observe that in Kerr $H_1=\ov{\Hb_1}, \quad H_2=-\ov{\Hb_2}$. Therefore
 \beaa
 \left(- H+\ov{H}-   \Hb+    \ov{\Hb}  \right) \c \nab&=& \left(- H_1+\ov{H}_1-   \Hb_1+    \ov{\Hb}_1  \right)  \nab_1+ \left(- H_2+\ov{H}_2-   \Hb_2+    \ov{\Hb}_2  \right)  \nab_2\\
 &=&  2\left(\ov{\Hb}_2-   \Hb_2  \right)  \nab_2\\
  &=&  2\left(-\frac{a\sin\th(r+ia\cos\th)}{|q|^3}+\frac{a\sin\th(r-ia\cos\th)}{|q|^3}  \right)  \nab_2\\
   &=&  -4i\frac{a^2\sin\th \cos\th}{|q|^3} \left(\frac{a\sin\th}{|q|}\T+\frac{1}{|q|\sin\th}\Z\right)\\
    &=&  -4i\frac{a^3\sin^2\th \cos\th}{|q|^4} \T -4i\frac{a^2\cos\th}{|q|^4} \Z.
 \eeaa
 This gives
 \beaa
 \left(- H+\ov{H}-   \Hb+   \ov{\Hb}  \right)\c \nab \qf&=&  -4i\frac{a^3\sin^2\th \cos\th}{|q|^4} \T(\qf) -4i\frac{a^2\cos\th}{|q|^4} \Z(\qf)+\Ga_g\dk. 
 \eeaa

 We finally obtain
 \beaa
 \square_2 \qf &=&\frac{4ia\cos\th(r^2+ a^2)}{|q|^4} \T(\qf) +   \frac{4ia^2\cos\th}{|q|^4} \Z(\qf) -4i\frac{a^3\sin^2\th \cos\th}{|q|^4} \T(\qf) -4i\frac{a^2\cos\th}{|q|^4} \Z(\qf)+ V\qf\\
&&+aL_{\qf}(A)+ \err[\square_2 \qf]\\
&=&\frac{4ia\cos\th(r^2+ a^2-a^2 \sin^2\th)}{|q|^4} \T(\qf)   + V\qf+aL_{\qf}(A)+ \err[\square_2 \qf]
 \eeaa
 which proves Theorem \ref{main-theorem}.


\subsection{The projection of the Regge-Wheeler type equation}
\label{section-projection-equation}


Using formula \eqref{formula-projection-wave} for the projection of the wave operator for a $2$-tensor to its first components, and neglecting error terms, we obtain
\beaa
(\square_2 \qf)_{11}&=&\square_{\g} (\qf_{11}) +i  \frac{4}{|q|^2}\frac{\cos\th}{\sin^2\th}  \Z (\mathfrak{q}_{11}) - \frac{4}{|q|^2}\cot^2\th \qf_{11} + a \tilde{V}\qf_{11}.
\eeaa

By projecting \eqref{wave-equation-qf} to the first component we then obtain for $\psi=\qf_{11}$:
\beaa
\square_{\g} \psi+\frac{4 i}{|q|^2}   \left(\frac{\cos\th}{\sin^2\th} \Z - a\cos\th \T \right)\psi - \frac{4}{|q|^2}\left( \cot^2\th+1-\frac{2m}{r} \right) \psi&=&  F_{11}+ a \tilde{V}\psi
\eeaa
where $F_{11}$ is the projection of the right hand side of \eqref{wave-equation-qf}. In the particular case of Kerr, the above equation can be written as
\bea\label{scalar-equation-Ma}
\square_{\g_{m,a}} \psi+\frac{4 i}{|q|^2}   \left(\frac{\cos\th}{\sin^2\th} \partial_{\vphi} - a\cos\th \partial_{t} \right)\psi - \frac{4}{|q|^2}\left( \cot^2\th+1-\frac{2m}{r} \right) \psi=  \tilde{F}.
\eea

\begin{remark}
Observe that equation \eqref{scalar-equation-Ma} is exactly the one obtained by Ma in \cite{Ma}, see equation (1.27b) for $s=2$ in that paper. The estimates in \cite{Ma} are derived through energy estimates, Morawetz estimates obtained by decomposing in modes as in \cite{D-H-R-Kerr}, and transport estimates for the lower order terms. Nevertheless, the estimates obtained in the particular case of Kerr cannot be easily generalized to the general case of perturbations of Kerr, since the difference between the wave operators gives rise to error terms which are not of the acceptable form in the sense of $\err[\square_2\qf]$  in Theorem \ref{main-theorem}. A physical space analysis which makes use of approximated hidden symmetries for the perturbed metric as in \cite{A-Blue1} may be a better way to approach the analysis of the above equation. 
\end{remark}



\section{Additional useful identities involving $Q(A)$}


We collect here some relations involving $Q(A)$ and its derivative. We summarize them in the following propositions. 

\begin{proposition}
The symmetric traceless $2$ tensor $Q(A)$ with $\C=2\tr\Xb$ and $\D=\frac 1 2 (\tr\Xb)^2$ satisfies 
\bea\label{alternative-formula-Q}
\begin{split}
Q(A) &=\DDc\hot \DDc\ov{P} +\frac{3}{2}\ov{P} \left(\tr\Xb\, \widehat{X} +\ov{\tr X} \widehat{\Xb} \right)\\
&+ \left(-\DDc \tr\Xb+  4 \nabc_3H   -\frac{1}{2}\tr\Xb   H \right)\hot B+\left(8\DDc \ov{P}+12 \ov{P} H\right)\hot H+\err[Q(A)]
\end{split}
\eea
where
\beaa
\err[Q(A)]&=& \err_{3\DDc\hot}[B]+\DDc (\Ga_g \Bb+\Ga_b A)+H (\Ga_g \Bb+\Ga_b A)+r^{-2} \Ga_g\Ga_g.
\eeaa

In particular, we can write 
\beaa
Q(A) &=&\DDc\hot \DDc\ov{P} +\frac{3}{2}\ov{P} \left(\tr\Xb\, \widehat{X} +\ov{\tr X} \widehat{\Xb} \right)+O_r+\err[Q(A)]
\eeaa
where $O_r$ are terms overshooting in $r$ with respect to the other ones, according to the expected bootstrap assumptions. 
\end{proposition}

\begin{proof} 
Recall that by Proposition \ref{prop:bianchi:complex}, we have,
 \beaa
 \nabc_3A +\frac{1}{2}\tr\Xb A &=&\DDc\hot B+  4 H   \hot B -3\ov{P}\Xh.
\eeaa
We infer
\beaa
\nabc_3 \left(\nabc_3A +\frac{1}{2}\tr\Xb A \right)&=&\DDc\hot \nabc_3B+[ \nabc_3, \DDc\hot ]B+  4 H   \hot \nabc_3B\\
&&+  4 \nabc_3H   \hot B -3\ov{P} \nabc_3\Xh -3\nabc_3\ov{P}\Xh.
\eeaa
By Lemma \ref{commutator-nab-c-3-DD-c-hot}, we have
 \beaa
 \, [\nabc_3, \DDc \hot ]B &=&- \frac 1 2 \tr \Xb \left( \DDc \hot B + 3H \hot B \right)  + H \hot \nabc_3 B+ \err_{3\DDc\hot}[B]
 \eeaa
 and hence
 \beaa
\nabc_3 \left(\nabc_3A +\frac{1}{2}\tr\Xb A \right)&=&\DDc\hot \nabc_3B+  5 H   \hot \nabc_3B- \frac 1 2 \tr \Xb \left( \DDc \hot B + 3H \hot B \right) \\
&&+  4 \nabc_3H   \hot B -3\ov{P} \nabc_3\Xh -3\nabc_3\ov{P}\Xh+ \err_{3\DDc\hot}[B].
\eeaa
Next, using Propositions \ref{prop-nullstr:complex-conf} and  \ref{prop:bianchi:complex},
 \beaa
\nabc_3B-\DDc\ov{P} &=& -\tr\Xb B+3\ov{P}H+\Ga_g \Bb+\Ga_b A,\\
\nabc_3P +\frac{1}{2}\DDbc \c\Bb &=& -\frac{3}{2}\ov{\tr\Xb} P - \ov{H} \c\Bb +\Ga_b B+\Ga_g \Ab,  \\
\nabc_3\widehat{X} +\frac{1}{2}\tr\Xb\, \widehat{X} &=& \DDc\hot H  +H\hot H -\frac{1}{2}\ov{\tr X} \widehat{\Xb}+\Ga_g \Ga_b,
\eeaa
we deduce
 \beaa
\nabc_3 \left(\nabc_3A +\frac{1}{2}\tr\Xb A \right)&=&\DDc\hot \left(\DDc\ov{P} -\tr\Xb B+3\ov{P}H\right)\\
&&+  5 H   \hot \left(\DDc\ov{P} -\tr\Xb B+3\ov{P}H\right)- \frac 1 2 \tr \Xb \left( \DDc \hot B + 3H \hot B \right) \\
&&+  4 \nabc_3H   \hot B -3\ov{P} \left(-\frac{1}{2}\tr\Xb\, \widehat{X} + \DDc\hot H  +H\hot H -\frac{1}{2}\ov{\tr X} \widehat{\Xb} \right)\\
&& -3\left(-\frac{3}{2}\tr\Xb \ov{P} \right)\Xh+\err
\eeaa
where
\beaa
\err&=& \err_{3\DDc\hot}[B]+\DDc (\Ga_g \Bb+\Ga_b A)+H (\Ga_g \Bb+\Ga_b A)+r^{-2} \Ga_g\Ga_g.
\eeaa
Hence, using \eqref{DD-hot-hF}, we have
 \beaa
\nabc_3 \left(\nabc_3A +\frac{1}{2}\tr\Xb A \right)&=&\DDc\hot \DDc\ov{P} +\frac{3}{2}\ov{P} \left(\tr\Xb\, \widehat{X} +\ov{\tr X} \widehat{\Xb} \right)\\
&& -\frac 3 2 \tr\Xb \left( \DDc\hot B -3 \ov{P}\Xh \right)\\
&&+ \left(-\DDc \tr\Xb+  4 \nabc_3H   -\frac{13}{2}\tr\Xb   H \right)\hot B\\
&&+\left(8\DDc \ov{P}+12 \ov{P} H\right)\hot H+\err.
\eeaa
Since 
 \beaa
\DDc\hot B -3\ov{P}\Xh&=& \nabc_3A +\frac{1}{2}\tr\Xb A- 4 H   \hot B 
\eeaa
this infer
 \beaa
&&\nabc_3 \left(\nabc_3A +\frac{1}{2}\tr\Xb A \right)+\frac 3 2 \tr\Xb \left( \nabc_3A +\frac{1}{2}\tr\Xb A\right)\\
&=&\DDc\hot \DDc\ov{P} +\frac{3}{2}\ov{P} \left(\tr\Xb\, \widehat{X} +\ov{\tr X} \widehat{\Xb} \right)+ \left(-\DDc \tr\Xb+  4 \nabc_3H   -\frac{1}{2}\tr\Xb   H \right)\hot B\\
&&+\left(8\DDc \ov{P}+12 \ov{P} H\right)\hot H+\err.
\eeaa
Using 
\beaa
\nabc_3\tr\Xb +\frac{1}{2}(\tr\Xb)^2 &=&\frak{d}^{\leq 1}\Ga_g 
\eeaa
the left hand side of the above is given by 
\beaa
&&\nabc_3 \nabc_3A +\frac{1}{2}\tr\Xb \nabc_3A+\frac{1}{2}\nabc_3\tr\Xb A +\frac 3 2 \tr\Xb \left( \nabc_3A +\frac{1}{2}\tr\Xb A\right)\\
&=&\nabc_3 \nabc_3A +2\tr\Xb \nabc_3A+\frac 1 2  (\tr\Xb)^2 A+\frak{d}^{\leq 1}\Ga_g A
\eeaa
which coincides with $Q(A)$. This concludes the proof. 
\end{proof}

\begin{proposition} 
The symmetric traceless $2$-tensor $Q(A)$ with $\C=2\tr\Xb$ and $\D=\frac 1 2 (\tr\Xb)^2$ satisfies
\beaa
 \nabc_3 \left(\ov{q}^5 Q(A) \right)&=& \ov{q}^5 \Bigg\{ -\frac{1}{2} \DDc \hot \DDc (\DDc \c \ov{\Bb} ) -\frac{3}{2} \ov{ P}( \DDc\hot\DDc \tr\Xb)\\
&& -\frac{3}{2}\ov{P} \ov{\tr X} \Ab +\frac{3}{2}\ov{P}\left(- \ov{\tr X}\ov{\tr\Xb} + \DDbc\c H+2\ov{P} \right)\widehat{\Xb} \Bigg\}\\
&&+\tilde{O}_r+\err[\nabc_3( \ov{q}^5 Q(A))]
\eeaa
where $\tilde{O}_r$ is overshooting in powers of $r$ and 
\beaa
\err[\nabc_3( \ov{q}^5 Q(A))]&=& r^{2}\dk^{\leq 2}( \Ga_g \Ga_b)+ \frac {5}{ 2} \tr \Xb \ov{  q}^5 \{\err[Q(A)] \}+ r^5\Ga_b Q(A)+\ov{q}^5\nabc_3\err[Q(A)].
\eeaa
\end{proposition}

\begin{proof}  
We start by formula \eqref{alternative-formula-Q}:
\beaa
\ov{q}^5 Q(A) &=&\ov{q}^5 \left\{\DDc\hot \DDc\ov{P} +\frac{3}{2}\ov{P} \left(\tr\Xb\, \widehat{X} +\ov{\tr X} \widehat{\Xb} \right)+O_r+\err[Q(A)]\right\}.
\eeaa
 Taking $\nabc_3$ derivative we deduce
 \beaa
 \nabc_3 \left(\ov{q}^5 Q(A) \right)&=& 5\ov{q}^4 \nabc_3 (\ov{q}) Q(A) + \ov{q}^5 L+\ov{q}^5\nabc_3 O_r+\ov{q}^5\nabc_3\err[Q(A)],\\
L:&=& \nabc_3 \left\{ \DDc\hot \DDc\ov{P} +\frac{3}{2}\ov{P} \left(\tr\Xb\, \widehat{X} +\ov{\tr X} \widehat{\Xb} \right)\right\}.
\eeaa
We calculate $L$ as follows
\beaa
L&=&   \DDc\hot \DDc  \nabc_3\ov{P}+[ \nabc_3,   \DDc\hot \DDc] \ov{P} +\frac{3}{2}\nabc_3(\ov{P} \tr\Xb\, \widehat{X}) +\frac{3}{2} \nabc_3(\ov{P} \ov{\tr X} \widehat{\Xb} ).
\eeaa
Ignoring cubic and higher order terms, we have
\beaa
\nabc_3(\ov{P} \tr\Xb\, \widehat{X})&=& \ov{P} \tr\Xb\, \nabc_3(\widehat{X})+\ov{P} \nabc_3(\tr\Xb)\, \widehat{X}+\nabc_3(\ov{P}) \tr\Xb\, \widehat{X}\\
&=& \ov{P} \tr\Xb\,\left(-\frac{1}{2}\tr\Xb\, \widehat{X} +\DDc\hot H  +H\hot H -\frac{1}{2}\ov{\tr X} \widehat{\Xb} \right)\\
&&+\ov{P}\left(-\frac{1}{2}(\tr\Xb)^2  \right)\, \widehat{X}+\left( -\frac{3}{2}\tr\Xb \ov{ P}  \right) \tr\Xb\, \widehat{X}+ r^{-3} \dk^{\leq 1}(\Ga_g \Ga_b)\\
&=& \ov{P} \tr\Xb\,\left(-\frac{5}{2}\tr\Xb\, \widehat{X}+\DDc\hot H   -\frac{1}{2}\ov{\tr X} \widehat{\Xb} \right)+O_1+ r^{-3} \dk^{\leq 1}(\Ga_g \Ga_b)
\eeaa
where
\beaa
O_1&=& \ov{P} \tr\Xb \ H\hot H 
\eeaa
is overshooting in powers of $r$. 
Also,
\beaa
\nabc_3(\ov{P} \ov{\tr X} \widehat{\Xb} )&=& \ov{P} \ov{\tr X} \nabc_3(\widehat{\Xb})+\ov{P} \nabc_3(\ov{\tr X} )\widehat{\Xb}+\nabc_3(\ov{P}) \ov{\tr X} \widehat{\Xb}\\
&=& \ov{P} \ov{\tr X} \left(-\Re(\tr\Xb) \Xbh+ \DDc\hot \Xib+   \Xib\hot(H+\Hb)-\Ab \right)\\
&&+\ov{P}\left(-\frac{1}{2}\ov{\tr\Xb} \ \ov{\tr X} + \DDbc\c H+H\c\ov{H}+2\ov{P} \right)\widehat{\Xb}\\
&&+\left( -\frac{3}{2}\tr\Xb \ov{ P}  \right) \ov{\tr X} \widehat{\Xb}+ r^{-3} \dk^{\leq 1}(\Ga_g \Ga_b)\\
&=& \ov{P} \ov{\tr X} \left(-\ov{\tr\Xb} \Xbh -2\tr\Xb  \widehat{\Xb}+ \DDc\hot \Xib-\Ab \right)+\ov{P}\left( \DDbc\c H+2\ov{P} \right)\widehat{\Xb}\\
&&+O_2+ r^{-3} \dk^{\leq 1}(\Ga_g \Ga_b)
\eeaa
where
\beaa
O_2&=& \ov{P} \ov{\tr X} \left( \Xib\hot(H+\Hb) \right)+\ov{P}\left(H\c\ov{H}\right)\widehat{\Xb}
\eeaa
is overshooting in powers of $r$. 
Also we have
\beaa
 \DDc\hot \DDc  \nabc_3\ov{P}&=& \DDc\hot \DDc  \left( -\frac{3}{2}\tr\Xb \ov{ P}-\frac{1}{2}\DDc \c \ov{\Bb} - H \c \ov{\Bb} \right)\\
 &=& \DDc\hot \left(  -\frac{3}{2} \DDc \tr\Xb \ov{ P}-\frac{3}{2}\tr\Xb  \DDc \ov{ P}-\frac{1}{2} \DDc (\DDc \c \ov{\Bb} )-  \DDc (H \c \ov{\Bb} )\right)\\
  &=&  -\frac{1}{2} \DDc \hot \DDc (\DDc \c \ov{\Bb} ) -\frac{3}{2} \ov{ P}( \DDc\hot\DDc \tr\Xb)-\frac{3}{2}\tr\Xb (\DDc \hot \DDc \ov{ P} )+O_3
  \eeaa
  where
  \beaa
O_3  &=& -3 \DDc \tr\Xb \hot \DDc \ov{ P}- \DDc \hot \DDc (H \c \ov{\Bb} )
\eeaa
is overshooting in powers of $r$. 
Now in view of Lemma \ref{commutator-nab-c-3-DD-c-hot} applied to $F=\DDc \ov{P}$ which is of conformal type $0$, 
\beaa
 \, [\nabc_3, \DDc \hot ]\DDc \ov{P} &=&- \frac 1 2 \tr \Xb \left( \DDc \hot \DDc \ov{P} + H \hot \DDc \ov{P} \right)  + H \hot \nabc_3 \DDc \ov{P}+ \err_{3\DDc\hot}[\DDc \ov{P}]\\
 &=&- \frac 1 2 \tr \Xb  \DDc \hot \DDc \ov{P} +O_4 +\err_{3\DDc\hot}[\DDc \ov{P}]
\eeaa
where
\beaa
O_4&=&- \frac 1 2 \tr \Xb \left( H \hot \DDc \ov{P} \right)  + H \hot \nabc_3 \DDc \ov{P}
\eeaa
is overshooting in powers of $r$. 

From Lemma \ref{lemma:comm}, we deduce 
       \beaa
        \,[\nab_3, \nab_a] P &=&-\frac 1 2 \left(\trchb \nab_a P+\atrchb \dual \nab P\right)+(\eta_a-\ze_a) \nab_3 P-\chibh_{ab}\nab_b P  +\xib_a \nab_4 P.
        \eeaa
        Taking the dual 
            \beaa
        \,[\nab_3, \dual \nab_a] P &=&-\frac 1 2 \left(\trchb \dual \nab_a P-\atrchb  \nab P\right)+\dual (\eta_a-\ze_a) \nab_3 P-\dual \chibh_{ab}  \nab_b P  +\dual \xib_a \nab_4 P.
        \eeaa
        Summing them we derive
             \beaa
        \,[\nab_3, \DD] P &=&-\frac 1 2 \tr \Xb \  \DD P +(H- Z) \nab_3 P-\Xbh \c \nab P  +\Xib  \nab_4 P
        \eeaa
        and therefore
             \beaa
           \,[\nabc_3, \DDc] \ov{P} &=&-\frac 1 2 \tr \Xb \  \DDc \ov{P} +H \nabc_3 \ov{P}-\frac 1 2 \Xbh \c \DDc \ov{P}  +\Xib  \nabc_4 \ov{P}\\
            &=&-\frac 1 2 \tr \Xb \  \DDc \ov{P} +H \left(-\frac{1}{2}\DDc \c\ov{\Bb}  -\frac{3}{2}\tr\Xb  \ov{P} - H \c\ov{\Bb} +\ov{\Xib}\c B -\frac{1}{4}\Xh\c\ov{\Ab} \right)\\
            &&-\frac 1 2 \Xbh \c \DDc \ov{P}  +\Xib  \left(\frac{1}{2}\DDbc\c B -\frac{3}{2}\ov{\tr X} \ov{P} + \ov{\Hb} \c B  \right)\\
              &=&-\frac 1 2 \tr \Xb \  \DDc \ov{P} -\frac{3}{2}\ov{P}  \left( \tr\Xb H + \ov{\tr X}  \Xib   \right) +O_5  +r^{-3}\dk^{\leq 1}( \Ga_g \Ga_b)
        \eeaa
        where
        \beaa
        O_5&=&H \left(-\frac{1}{2}\DDc \c \ov{\Bb}  - H \c \ov{\Bb}  \right)-\frac 1 2 \Xbh \c \DDc \ov{P} 
        \eeaa
        is overshooting in powers of $r$. We deduce 
        \beaa
          \DDc \hot    \,[\nabc_3, \DDc] \ov{P}               &=&  \DDc \hot    \left(-\frac 1 2 \tr \Xb \  \DDc \ov{P}  -\frac{3}{2}\ov{P}  \left( \tr\Xb H + \ov{\tr X}  \Xib   \right)+O_5  +r^{-3}\dk^{\leq 1}( \Ga_g \Ga_b)\right)\\
         &=&-\frac 1 2 \tr \Xb \ \DDc \hot \DDc \ov{P}  -\frac{3}{2}\ov{P}  \left( \tr\Xb \DDc \hot H + \ov{\tr X}  \DDc \hot \Xib   \right)\\
           &&-\frac 1 2\DDc \tr \Xb \hot  \DDc \ov{P}  -\frac{3}{2}\DDc \ov{P} \hot \left( \tr\Xb H + \ov{\tr X}  \Xib   \right)\\
           &&-\frac{3}{2} \ov{P}  \left( \DDc \tr\Xb \hot H + \DDc\ov{\tr X}  \hot \Xib   \right)\\
       &&+\DDc \hot O_5  +r^{-3}\dk^{\leq 2}( \Ga_g \Ga_b)      
        \eeaa
        which gives
        \beaa
              \DDc \hot    \,[\nabc_3, \DDc] \ov{P}                        &=&-\frac 1 2 \tr \Xb \ \DDc \hot \DDc \ov{P}  -\frac{3}{2}\ov{P}  \left( \tr\Xb \DDc \hot H + \ov{\tr X}  \DDc \hot \Xib   \right)+O_6  +r^{-3}\dk^{\leq 2}( \Ga_g \Ga_b)
        \eeaa
        where
        \beaa
  O_6&=&      -\frac 1 2\DDc \tr \Xb \hot  \DDc \ov{P}  -\frac{3}{2}\DDc \ov{P} \hot \left( \tr\Xb H + \ov{\tr X}  \Xib   \right)\\
           &&-\frac{3}{2} \ov{P}  \left( \DDc \tr\Xb \hot H + \DDc\ov{\tr X}  \hot \Xib   \right)+\DDc \hot O_5
        \eeaa
        is overshooting in powers of $r$. 

Hence, since $[ \nabc_3,   \DDc\hot \DDc] \ov{P}=[ \nabc_3,   \DDc\hot] \DDc \ov{P}+\DDc \hot [ \nabc_3,  \DDc] \ov{P}$, we obtain
\beaa
[ \nabc_3,   \DDc\hot \DDc] \ov{P}&=&-  \tr \Xb  \DDc \hot \DDc \ov{P}-\frac{3}{2}\ov{P}  \left( \tr\Xb \DDc \hot H + \ov{\tr X}  \DDc \hot \Xib   \right)  \\
&&+O_7 +r^{-3}\dk^{\leq 2}( \Ga_g \Ga_b)
\eeaa
with $O_7=O_4+O_6$. 
We deduce 
\beaa
L&=&  \DDc\hot \DDc  \nabc_3\ov{P}+[ \nabc_3,   \DDc\hot \DDc] \ov{P} +\frac{3}{2}\nabc_3(\ov{P} \tr\Xb\, \widehat{X}) +\frac{3}{2} \nabc_3(\ov{P} \ov{\tr X} \widehat{\Xb} )\\
&=& -\frac{1}{2} \DDc \hot \DDc (\DDc \c \ov{\Bb} ) -\frac{3}{2} \ov{ P}( \DDc\hot\DDc \tr\Xb)-\frac{5}{2}\tr\Xb (\DDc \hot \DDc \ov{ P} )\\
&&+\frac{3}{2}\ov{P} \tr\Xb\,\left(-\frac{5}{2}\tr\Xb\, \widehat{X}  -\frac{1}{2}\ov{\tr X} \widehat{\Xb} \right)  +\frac{3}{2}\ov{P} \ov{\tr X} \left(-\ov{\tr\Xb} \Xbh -2\tr\Xb  \widehat{\Xb}-\Ab \right)+\frac{3}{2}\ov{P}\left( \DDbc\c H+2\ov{P} \right)\widehat{\Xb} \\
&&+O_L+r^{-3}\dk^{\leq 2}( \Ga_g \Ga_b)
\eeaa
where $O_L=O_3+O_7+\frac 3 2 O_1+\frac 3 2 O_2$.
On the other hand, writing $ \nab_3 \ov{q}=\frac {1}{ 2} \tr \Xb \ov{  q}+r\Ga_b$, we have
\beaa
5\ov{q}^4 \nabc_3 (\ov{q}) Q(A)&=& \frac {5}{ 2} \tr \Xb \ov{  q}^5 Q(A)+ r^5\Ga_b Q(A)\\
&=& \frac {5}{ 2} \tr \Xb \ov{  q}^5 \left\{\DDc\hot \DDc\ov{P} +\frac{3}{2}\ov{P} \left(\tr\Xb\, \widehat{X} +\ov{\tr X} \widehat{\Xb} \right)+O_r+\err[Q(A)] \right\}+ r^5\Ga_b Q(A).
\eeaa
Hence, we finally obtain
\beaa
 \nabc_3 \left(\ov{q}^5 Q(A) \right)&=& \ov{q}^5 \Bigg\{ -\frac{1}{2} \DDc \hot \DDc (\DDc \c \ov{\Bb} ) -\frac{3}{2} \ov{ P}( \DDc\hot\DDc \tr\Xb)\\
&& -\frac{3}{2}\ov{P} \ov{\tr X} \Ab +\frac{3}{2}\ov{P}\left(- \ov{\tr X}\ov{\tr\Xb} + \DDbc\c H+2\ov{P} \right)\widehat{\Xb} \Bigg\}\\
&&+\tilde{O}_r+\err[\nabc_3( \ov{q}^5 Q(A))]
\eeaa
where
\beaa
\tilde{O}_r&=& \ov{q}^5 O_L+ \frac {5}{ 2} \tr \Xb \ov{  q}^5 O_r+\ov{q}^5\nabc_3 O_r
\eeaa
is overshooting in powers of $r$, and the error term is given by 
\beaa
\err[\nabc_3( \ov{q}^5 Q(A))]&=& r^{2}\dk^{\leq 2}( \Ga_g \Ga_b)+ \frac {5}{ 2} \tr \Xb \ov{  q}^5 \{\err[Q(A)] \}+ r^5\Ga_b Q(A)+\ov{q}^5\nabc_3\err[Q(A)]
\eeaa
as desired. 
\end{proof}


\appendix



\section{Relation with the Teukolsky equation in the literature}
\label{Teuk-literature}
\lab{appendix:comparisonofTeuk}


The goal of this appendix is to relate the Teukolsky equation obtained here to the Teukolsky equation as appears in the literature for the scalar curvature component $\alpha^{[+2]}$ in Newman-Penrose formalism:
\bea\label{teukolsky-dhr-square}
\begin{split}
|q|^2\square_g \a^{[+2]}&=-4(r-m) \pr_r \a^{[+2]}-4 \left(\frac{m(r^2-a^2)}{\Delta} - r -i a \cos\th \right)\pr_t \a^{[+2]} \\
&- 4\left(\frac{a(r-m)}{\Delta} + i \frac{\cos\th}{\sin^2\th} \right) \pr_\vphi \a^{[+2]}+ (4 \cot^2\th-2)\a^{[+2]}.
\end{split}
\eea

We will show how the above Teukolsky equation \eqref{teukolsky-dhr-square} and the Teukolsky equation \eqref{Teukolsky-equation-tens} are consistent one with another.


\subsection{Relation between $\a^{[+2]}$ and projected $A$}


In the literature, the curvature component $\a^{[+2]}$ is defined in Newman-Penrose formalism as 
\beaa
\a^{[+2]}&=& -W(l, m, l, m)
\eeaa
where 
\beaa
l = e_4, \qquad m= \frac{|q|}{\sqrt{2} q}\left(e_1 + ie_2\right),
\eeaa
where the vectors $e_4$, $e_3$, $e_1$ and $e_2$ are given in Boyer-Lindquist coordinates by \eqref{null-frames}.
We therefore deduce
\beaa
\a^{[+2]}&=& -\frac{|q|^2}{2 q^2}W(e_4, \left(e_1 + ie_2\right), e_4, \left(e_1 + ie_2\right))\\
&=& -\frac{\ov{q}}{2 q}\left(W_{4141}+i W_{4142}+i W_{4241}-W_{4242}\right)\\
&=& -\frac{\ov{q}}{2 q}\left(W_{4141}-W_{4242}+2i W_{4142}\right)
\eeaa
which gives, using that $W_{4242}=-W_{4141}$, 
\beaa
\a^{[+2]}&=& -\frac{\ov{q}}{ q}\left(W_{4141}+i W_{4142}\right).
\eeaa
Observe that
\beaa
A_{11}&=& W_{4141}+ i W_{4142}
\eeaa
and therefore denoting by $\mathfrak{a}$ the projection of $A$ to the $11$ component:
\beaa
\mathfrak{a}&=& A_{11}
\eeaa
we obtain the following relation:
\bea\label{relation-a-A}
\a^{[+2]}&=& -\frac{\ov{q}}{ q}\mathfrak{a}.
\eea
In particular $\a^{[+2]}$ and $\mathfrak{a}$ differ by a scalar factor given by $ -\frac{\ov{q}}{ q}= -\frac{r-i a \cos\th }{ r+ i a \cos \th}$, which is constant (equal to -1) in the case of Schwarzschild.


\subsection{The equation for $\mathfrak{a}$ from the Teukolsky in the literature}


We deduce the equation for $\mathfrak{a}$ from the standard Teukolsky equation in the literature \eqref{teukolsky-dhr-square} and the relation between $\a^{[+2]}$ and $\mathfrak{a}$ given by \eqref{relation-a-A}. We have
\beaa
|q|^2\square_g\mathfrak{a}&=& -|q|^2\square_g\left(\frac{q}{ \ov{q}}\a^{[+2]}\right)\\
&=& -|q|^2\square_g\left(\frac{q}{ \ov{q}}\right) \a^{[+2]}-\frac{q}{ \ov{q}} |q|^2\square_g(\a^{[+2]})-2\Delta\pr_r\left(\frac{q}{ \ov{q}}\right) \pr_r\a^{[+2]}-2\pr_\th \left(\frac{q}{ \ov{q}}\right)\pr_\th \a^{[+2]}.
\eeaa
Using that $\partial_r q=1$ and $\partial_\th q=-i a \sin \th$, we compute
\beaa
\pr_r\left(\frac{q}{ \ov{q}}\right)&=&\frac{\ov{q}-q}{ \ov{q}^2}=\frac{-2ia \cos\th}{ \ov{q}^2}, \\
\pr_\th\left(\frac{q}{ \ov{q}}\right)&=&\frac{-i a\sin\th \ov{q}- q ia \sin\th}{ \ov{q}^2}=\frac{-2ir a\sin\th }{ \ov{q}^2}.
\eeaa
Using \eqref{teukolsky-dhr-square}, we obtain
\beaa
|q|^2\square_g\mathfrak{a}&=&\left( -|q|^2\square_g\left(\frac{q}{ \ov{q}}\right)\right) \a^{[+2]}+\frac{q}{ \ov{q}}4(r-m) \pr_r \a^{[+2]}+4\frac{q}{ \ov{q}} \left(\frac{m(r^2-a^2)}{\Delta} - r -i a \cos\th \right)\pr_t \a^{[+2]} \\
&&+ 4\frac{q}{ \ov{q}}\left(\frac{a(r-m)}{\Delta} + i \frac{\cos\th}{\sin^2\th} \right) \pr_\vphi \a^{[+2]}-\frac{q}{ \ov{q}} (4 \cot^2\th-2)\a^{[+2]}\\
&&-2\Delta\frac{-2ia \cos\th}{ \ov{q}^2}\pr_r\a^{[+2]}-2\frac{-2ir a\sin\th }{ \ov{q}^2}\pr_\th \a^{[+2]}
\eeaa
which gives
\beaa
|q|^2\square_g\mathfrak{a}&=& -4\left(r-m+\Delta\frac{ia \cos\th }{|q|^2} \right) \left(-\frac{q}{ \ov{q}}\pr_r \a^{[+2]}\right)-4 \left(\frac{m(r^2-a^2)}{\Delta} - r -i a \cos\th \right)\left(-\frac{q}{ \ov{q}}\pr_t \a^{[+2]} \right)\\
&&- 4\left(\frac{a(r-m)}{\Delta} + i \frac{\cos\th}{\sin^2\th} \right) \left(-\frac{q}{ \ov{q}}\pr_\vphi \a^{[+2]}\right)-4\frac{ir a\sin\th }{ |q|^2}\left(-\frac{q}{ \ov{q}}\pr_\th \a^{[+2]}\right)\\
&&\left(4 \cot^2\th-2+ \ov{q}^2\square_g\left(\frac{q}{ \ov{q}}\right)\right) \left(-\frac{q}{ \ov{q}}\a^{[+2]} \right).
\eeaa
Since
\beaa
\pr_r\mathfrak{a}&=& \pr_r\left(-\frac{q}{\ov{q}}\a^{[+2]}\right)= -\frac{2ia \cos\th }{ |q|^2}\mathfrak{a}-\frac{q}{\ov{q}}\pr_r\a^{[+2]},\\
\pr_\th \mathfrak{a}&=&  \pr_\th\left(-\frac{q}{\ov{q}}\a^{[+2]}\right)= -\frac{2ir a\sin\th }{ |q|^2}\mathfrak{a}-\frac{q}{\ov{q}}\pr_\th\a^{[+2]},
\eeaa
we can write
\beaa
-\frac{q}{\ov{q}}\pr_r\a^{[+2]}&=& \pr_r\mathfrak{a}+\frac{2ia \cos\th }{ |q|^2}\mathfrak{a}, \qquad -\frac{q}{\ov{q}}\pr_\th\a^{[+2]}=\pr_\th \mathfrak{a}+\frac{2ir a\sin\th }{ |q|^2}\mathfrak{a}, \\
 -\frac{q}{\ov{q}}\pr_t \a^{[+2]}&=&\pr_t \mathfrak{a}, \qquad -\frac{q}{\ov{q}}\pr_\vphi \a^{[+2]}=\pr_\vphi \mathfrak{a}.
\eeaa
We therefore obtain 
\beaa
|q|^2\square_g\mathfrak{a}&=& -4\left(r-m+\Delta\frac{ia \cos\th }{|q|^2} \right) \left( \pr_r\mathfrak{a}+\frac{2ia \cos\th }{ |q|^2}\mathfrak{a}\right)-4 \left(\frac{m(r^2-a^2)}{\Delta} - r -i a \cos\th \right)\pr_t \mathfrak{a}\\
&&- 4\left(\frac{a(r-m)}{\Delta} + i \frac{\cos\th}{\sin^2\th} \right) \pr_\vphi \mathfrak{a}-4\frac{ir a\sin\th }{ |q|^2}\left(\pr_\th \mathfrak{a}+\frac{2ir a\sin\th }{ |q|^2}\mathfrak{a}\right)\\
&&\left(4 \cot^2\th-2+\ov{q}^2\square_g\left(\frac{q}{ \ov{q}}\right)\right) \mathfrak{a} 
\eeaa
which gives the following wave equation for the scalar component $\mathfrak{a}$:
\bea\label{wave-a-partial}
\begin{split}
|q|^2\square_g\mathfrak{a}&= -4\left(r-m+\Delta\frac{ia \cos\th }{|q|^2} \right) \pr_r\mathfrak{a} -4 \left(\frac{m(r^2-a^2)}{\Delta} - r -i a \cos\th \right)\pr_t \mathfrak{a}\\
&- 4\left(\frac{a(r-m)}{\Delta} + i \frac{\cos\th}{\sin^2\th} \right) \pr_\vphi \mathfrak{a}-4\frac{ir a\sin\th }{ |q|^2} \pr_\th \mathfrak{a} \\
&+\left(4 \cot^2\th-2+ \ov{q}^2\square_g\left(\frac{q}{ \ov{q}}\right)-4\left(r-m+\Delta\frac{ia \cos\th }{|q|^2} \right) \frac{2ia \cos\th }{ |q|^2}+\frac{8r^2 a^2\sin^2\th }{ |q|^4}\right) \mathfrak{a}. 
\end{split}
\eea

We will now make use of the null frames to rewrite the above equation in a more convenient form. 

We substitute the first order terms given as partial derivatives $\pr_r$, $\pr_t$, $\pr_\th$, $\pr_\varphi$ with the derivatives expressed in terms of null frames. 
From \eqref{null-frames}, we can write
\beaa
\pr_r&=& -\frac{|q|^2}{2\Delta} e_3+\frac 1 2 e_4,  \\
\pr_t&=& \frac 1 2 e_3+\frac{\Delta}{2|q|^2} e_4-\frac{a \sin\th}{|q|} e_2, \\
\pr_\th&=&|q| e_1,\\
\pr_\varphi&=& \left(|q|\sin\th +\frac{a^2 \sin^3\th}{|q|}\right) e_2-\frac 1 2a\sin^2\th e_3-\frac{a\sin^2\th\Delta}{2|q|^2} e_4.
\eeaa 
This implies from \eqref{wave-a-partial}:
\beaa
|q|^2\square_g\mathfrak{a}&=& -4\left(r-m+\Delta\frac{ia \cos\th }{|q|^2} \right) \left(-\frac{|q|^2}{2\Delta} e_3\mathfrak{a}+\frac 1 2 e_4\mathfrak{a}\right)\\
&& -4 \left(\frac{m(r^2-a^2)}{\Delta} - r -i a \cos\th \right)\left(\frac 1 2 e_3 \mathfrak{a}+\frac{\Delta}{2|q|^2} e_4 \mathfrak{a}-\frac{a \sin\th}{|q|} e_2 \mathfrak{a}\right)\\
&&- 4\left(\frac{a(r-m)}{\Delta} + i \frac{\cos\th}{\sin^2\th} \right) \left( \left(|q|\sin\th +\frac{a^2 \sin^3\th}{|q|}\right) e_2\mathfrak{a}-\frac 1 2a\sin^2\th e_3\mathfrak{a}-\frac{a\sin^2\th\Delta}{2|q|^2} e_4\mathfrak{a}\right) \\
&&-4\frac{ir a\sin\th }{ |q|^2} |q| e_1 \mathfrak{a} + \tilde{W} \mathfrak{a} 
\eeaa
where the potential is given by 
\bea
\tilde{W}&=& 4 \cot^2\th-2+ \ov{q}^2\square_g\left(\frac{q}{ \ov{q}}\right)-4\left(r-m+\Delta\frac{ia \cos\th }{|q|^2} \right) \frac{2ia \cos\th }{ |q|^2}+\frac{8r^2 a^2\sin^2\th }{ |q|^4}.
\eea
The above equation simplifies to
\beaa
|q|^2\square_g\mathfrak{a}&=& 2\left(\frac{|q|^2}{\Delta}(r-m)-\frac{m(r^2-a^2)}{\Delta}+a\sin^2\th\frac{a(r-m)}{\Delta} + r+3ia \cos\th  \right)  e_3\mathfrak{a}\\
&&+2 \left(-r+m -\frac{m(r^2-a^2)}{|q|^2} + \frac{\Delta}{|q|^2}r +\frac{a\sin^2\th}{|q|^2}a(r-m) + i \frac{a\Delta}{|q|^2}\cos\th \right) e_4\mathfrak{a}\\
&&+\left[4 \frac{a \sin\th}{|q|}\left(\frac{m(r^2-a^2)}{\Delta} - r -i a \cos\th \right) - 4\left(\frac{a(r-m)}{\Delta} + i \frac{\cos\th}{\sin^2\th} \right)\left(|q|\sin\th +\frac{a^2 \sin^3\th}{|q|}\right)\right]e_2\mathfrak{a}\\
&&-4\frac{ir a\sin\th }{ |q|} e_1(\mathfrak{a})+ \tilde{W} \mathfrak{a}.
\eeaa

We now use the following values in Kerr to rewrite the coefficients of the first order terms in the equation. Recall: 
\beaa
 \trch=\frac{2r}{|q|^2},\quad \atrch&=&\frac{2a\cos\th}{|q|^2}, \quad \trchb=-\frac{2r\Delta}{|q|^4}, \quad \atrchb=\frac{2a\Delta\cos\th}{|q|^4},\\
 \omb&=&-\frac{-a^2\cos^2\th (r-m)-mr^2+a^2r}{|q|^4}, \\
 \etab_1&=& -\frac{a^2\sin\th \cos\th }{|q|^3}, \qquad \etab_2 =-\frac{a\sin\th r}{|q|^3}.
\eeaa
From $P=-\frac{2m}{q^3}$, we deduce
\beaa
P=-\frac{2m}{q^3}=-\frac{2m}{|q|^6} (r-i a \cos\th)^3=\frac{2m}{|q|^6} (-r^3+3r  a^2 \cos^2\th+ia\cos\th(3r^2- a^2 \cos^2\th))
\eeaa
and therefore
\beaa
\rho&=& \frac{2m}{|q|^6} (-r^3+3r  a^2 \cos^2\th), \qquad \dual \rho= \frac{2ma\cos\th }{|q|^6} (3r^2- a^2 \cos^2\th).
\eeaa
Define also
\beaa
\La&=& \frac{r^2+a^2}{|q|^3}\cot\th.
\eeaa

We rewrite the coefficients in the following way.
\begin{itemize}
\item The real part of the coefficient of $e_3 \mathfrak{a}$ is given by
\beaa
&&\frac{|q|^2}{\Delta}(r-m)-\frac{m(r^2-a^2)}{\Delta}+a\sin^2\th\frac{a(r-m)}{\Delta} + r\\
&=&\frac{1}{\Delta}\left((r^2+a^2\cos^2\th)(r-m)-m(r^2-a^2)+a^2(1-\cos^2\th) (r-m) + r(r^2-2mr+a^2)\right)\\
&=&\frac{1}{\Delta}\left(r^3-mr^2-mr^2+a^2r + r(r^2-2mr+a^2)\right)\\
&=&\frac{1}{\Delta}\left(2 r\Delta\right)=2r.
\eeaa
Therefore the coefficient of $e_3 \mathfrak{a}$ is given by
\beaa
4r+6 i a \cos\th=|q|^2\left(2\trch+3 i \atrch\right).
\eeaa

\item The real part of the coefficient of $e_4 \mathfrak{a}$ is given by
\beaa
&&-r+m -\frac{m(r^2-a^2)}{|q|^2} + \frac{\Delta}{|q|^2}r +\frac{a\sin^2\th}{|q|^2}a(r-m) \\
&=&\frac{1}{|q|^2}\left((r^2+a^2\cos^2\th)(-r+m) -m(r^2-a^2) + \Delta r +a^2(1-\cos^2\th)(r-m)\right) \\
&=&\frac{1}{|q|^2}\left(-r^3-2a^2\cos^2\th r+(2a^2\cos^2\th)(m)  + (r^2-2mr+a^2) r +a^2(r)\right) \\
&=&\frac{1}{|q|^2}\left(-2a^2\cos^2\th (r-m) -2m r^2 +2a^2r\right). 
\eeaa
 Therefore the coefficient of $e_4\mathfrak{a}$ is given by
 \beaa
\frac{2}{|q|^2}\left(-2a^2\cos^2\th (r-m) -2m r^2 +2a^2r\right) + 2i \frac{a\Delta}{|q|^2}\cos\th =|q|^2\left( -4\omb+ i \atrchb\right).
 \eeaa
 
 \item The coefficient of $e_2\mathfrak{a}$ is given by
 \beaa
&&4 \frac{a \sin\th}{|q|}\left(\frac{m(r^2-a^2)}{\Delta} - r -i a \cos\th \right) - 4\left(\frac{a(r-m)}{\Delta} + i \frac{\cos\th}{\sin^2\th} \right)\left(|q|\sin\th +\frac{a^2 \sin^3\th}{|q|}\right)\\
&=&4 \frac{a \sin\th}{|q|}\left(\frac{m(r^2-a^2)}{\Delta} - r  \right)- 4\left(\frac{a(r-m)}{\Delta} \right)\left(|q|\sin\th +\frac{a^2 \sin^3\th}{|q|}\right)\\
&&-i 4 \frac{a \sin\th}{|q|}\left(  a \cos\th \right) - 4i\left(  \frac{\cos\th}{\sin^2\th} \right)\left(|q|\sin\th +\frac{a^2 \sin^3\th}{|q|}\right).
\eeaa
Its real part simplifies to
\beaa
&&4 \frac{a \sin\th}{|q|}\left(\frac{m(r^2-a^2)}{\Delta} - r  \right)- 4\left(\frac{a(r-m)}{\Delta} \right)\left(|q|\sin\th +\frac{a^2 \sin^3\th}{|q|}\right)\\
&=&4 \frac{a \sin\th}{|q|}\left(\frac{m(r^2-a^2)-r(r^2-2mr+a^2)}{\Delta}   \right)- 4\left(\frac{a(r-m)}{\Delta} \right)\left(\frac{|q|^2\sin\th +a^2 \sin^3\th}{|q|}\right)\\
&=&4 \frac{a \sin\th}{|q|\Delta }\left(m(r^2-a^2)-r(r^2-2mr+a^2)  - (r-m) (|q|^2 +a^2 \sin^2\th) \right)\\
&=&4 \frac{a \sin\th}{|q|\Delta }\left(mr^2-a^2m-r^3+2mr^2-a^2r  - (r-m) (r^2+ a^2) \right)\\
&=&4 \frac{a \sin\th}{|q|\Delta }\left(-2r^3+4mr^2-2a^2r    \right)=4 \frac{a \sin\th}{|q|\Delta }\left(-2r \Delta     \right)=- \frac{8ra \sin\th}{|q|}.
\eeaa
 Therefore the coefficient of $e_2\mathfrak{a}$ is given by
\beaa
- \frac{8ra \sin\th}{|q|}-4i  \left(\frac{2a^2 \sin\th\cos\th}{|q|}+  \frac{\cos\th}{\sin\th} |q| \right) =|q|^2\left(8 \etab_2-4i(\La- \etab_1) \right).
\eeaa
Indeed,
\beaa
\La- \etab_1&=& \frac{r^2+a^2}{|q|^3}\cot\th+\frac{a^2\sin\th \cos\th }{|q|^3}=\left( \frac{r^2+a^2+a^2\sin^2\th  }{|q|^3\sin\th}\right)\cos\th\\
&=&\left( \frac{r^2+a^2\cos^2\th+2a^2\sin^2\th  }{|q|^3\sin\th}\right)\cos\th=\left( \frac{1  }{|q|\sin\th}+\frac{2a^2\sin\th  }{|q|^3}\right)\cos\th.
\eeaa

\item The coefficient of $e_1\mathfrak{a}$ is given by 
\beaa
-4\frac{ir a\sin\th }{ |q|}=|q|^2 \left( i4 \etab_2\right).
\eeaa
\end{itemize}
 
We can therefore write
 \bea\label{wave-as-implied}
 \begin{split}
\square_{g_{Kerr}}\mathfrak{a}&= \left(2\trch+3 i \atrch\right) e_3\mathfrak{a}+\left( -4\omb+ i \atrchb\right) e_4\mathfrak{a}\\
&+ i4 \etab_2 e_1(\mathfrak{a})+\left(8 \etab_2-4i(\La- \etab_1) \right)e_2\mathfrak{a}+W \mathfrak{a} 
\end{split}
\eea
where the potential is given by
\beaa
W&=& \frac{1}{|q|^2}\left(4 \cot^2\th-2+ \ov{q}^2\square_g\left(\frac{q}{ \ov{q}}\right)-4\left(r-m+\Delta\frac{ia \cos\th }{|q|^2} \right) \frac{2ia \cos\th }{ |q|^2}+\frac{8r^2 a^2\sin^2\th }{ |q|^4}\right).
\eeaa
This gives
\beaa
W&=& \frac{1}{|q|^2}\left(4 \cot^2\th-2+8\left(-2mr+a^2 \right) \frac{a^2 \cos^2\th }{ |q|^4}+\frac{8r^2 a^2 }{ |q|^4}-8\left(r-m\right) \frac{ia \cos\th }{ |q|^2}+ \ov{q}^2\square_g\left(\frac{q}{ \ov{q}}\right)\right).
\eeaa
Putting at common denominator we obtain
\beaa
W&=& \frac{1}{|q|^2}\Bigg(\frac{4 \cot^2\th-2}{|q|^4}(r^4+2r^2a^2\cos^2\th+a^4\cos^4\th)+8\left(-2mr+a^2 \right) \frac{a^2 \cos^2\th }{ |q|^4}+\frac{8r^2 a^2 }{ |q|^4}\\
&&-8\left(r-m\right) \frac{ia \cos\th }{ |q|^4}(r^2+a^2\cos^2\th)+ \ov{q}^2\square_g\left(\frac{q}{ \ov{q}}\right)\Bigg)
\eeaa
which gives
\beaa
W&=& \frac{1}{|q|^6}\Bigg((4 \cot^2\th-2) r^4+(8 \cot^2\th a^2\cos^2\th-4a^2\cos^2\th+8a^2) r^2-16m a^2 \cos^2\th r\\
&&+4  a^4\cos^4\th\cot^2\th-2a^4\cos^4\th+8 a^4 \cos^2\th  \\
&&-8i a \cos\th  (r^3-mr^2+a^2\cos^2\th r -m a^2\cos^2\th) +|q|^4 \ov{q}^2\square_g\left(\frac{q}{ \ov{q}}\right)\Bigg).
\eeaa
We now compute the term $|q|^4 \ov{q}^2\square_g\left(\frac{q}{ \ov{q}}\right)$. Recall that in Kerr:
\beaa
 |q|^2\square_g &=&\Delta \pr^2_r+\left(-\frac{(r^2+a^2)^2}{\Delta}+a^2 \sin^2 \th \right)\pr_t^2+ \pr^2_\th+\left(\frac{1}{\sin^2\th}-\frac{a^2}{\Delta} \right)\pr_\vphi^2\\
 &&-\frac{4m a r}{\Delta} \pr_t\pr_\vphi  +2(r-m)\pr_r +\frac{\cos\th}{\sin \th} \pr_\th.
\eeaa
We therefore have
\beaa
 |q|^2\square_g\left(\frac{q}{ \ov{q}}\right) &=&\Delta \pr^2_r\left(\frac{q}{ \ov{q}}\right)+ \pr^2_\th\left(\frac{q}{ \ov{q}}\right)  +2(r-m)\pr_r\left(\frac{q}{ \ov{q}}\right) +\frac{\cos\th}{\sin \th} \pr_\th\left(\frac{q}{ \ov{q}}\right)\\
 &=& \Delta \pr_r\left(\frac{-2ia \cos\th}{ \ov{q}^2}\right)+\pr_\th\left(\frac{-2ir a\sin\th }{ \ov{q}^2}\right)+2(r-m)\frac{-2ia \cos\th}{ \ov{q}^2}+\frac{\cos\th}{\sin\th}(\frac{-2ir a\sin\th }{ \ov{q}^2})\\
&=& (r^2-2mr+a^2) \frac{4ia \cos\th}{ \ov{q}^3}+\frac{-2r a(i\cos\th \ov{q}+2a\sin^2\th) }{ \ov{q}^3}\\
&&+2(r-m)\frac{-2ia \cos\th}{ \ov{q}^2}+\cos\th\left(\frac{-2ir a}{ \ov{q}^2}\right)\\
&=&  \frac{4ia r^2\cos\th}{ \ov{q}^3}-\frac{8mia r\cos\th}{ \ov{q}^3}+\frac{4ia^3 \cos\th}{ \ov{q}^3}-\frac{4r a^2\sin^2\th }{ \ov{q}^3}+\frac{4mia \cos\th}{ \ov{q}^2}+\frac{-8ir a\cos\th}{ \ov{q}^2}\\
&=&  \frac{4ia r^2\cos\th}{ \ov{q}^3}-\frac{8mia r\cos\th}{ \ov{q}^3}+\frac{4ia^3 \cos\th}{ \ov{q}^3}-\frac{4r a^2\sin^2\th }{ \ov{q}^3}+\frac{4mia \cos\th(r-ia\cos\th)}{ \ov{q}^3}\\
&&+\frac{-8ir a\cos\th(r-ia\cos\th)}{ \ov{q}^3}\\
&=&\frac{1}{\ov{q}^3}\Big(-4r a^2\cos^2\th-4r a^2 +4m a^2 \cos^2\th +4ia\cos\th(- r^2-m r+a^2)\Big).
\eeaa
Therefore
\beaa
|q|^4 \ov{q}^2\square_g\left(\frac{q}{ \ov{q}}\right)&=& q \Big(-4r a^2\cos^2\th-4r a^2 +4m a^2 \cos^2\th +4ia\cos\th(- r^2-m r+a^2)\Big)\\
&=& (r+i a \cos\th ) \Big(-4r a^2\cos^2\th-4r a^2 +4m a^2 \cos^2\th +4ia\cos\th(- r^2-m r+a^2)\Big)\\
&=&  -4r^2 a^2\cos^2\th-4r^2 a^2 +4m a^2 \cos^2\th r  +4ia\cos\th(- r^3-m r^2+a^2 r)\\
&&+ (i a \cos\th ) \Big(-4r a^2\cos^2\th-4r a^2 +4m a^2 \cos^2\th\Big)-  \Big(4a^2\cos^2\th(- r^2-m r+a^2)\Big)\\
&=& ( -4 a^2 )r^2+8m a^2 \cos^2\th r -4a^4\cos^2\th +4ia\cos\th(- r^3-m r^2-r a^2\cos^2\th +m a^2 \cos^2\th).
\eeaa
We therefore obtain
\beaa
W&=& \frac{1}{|q|^6}\Big((4 \cot^2\th-2) r^4+(8 \cot^2\th a^2\cos^2\th-4a^2\cos^2\th+8a^2) r^2-16m a^2 \cos^2\th r\\
&&+4  a^4\cos^4\th\cot^2\th-2a^4\cos^4\th+8 a^4 \cos^2\th  \\
&&-8i a \cos\th  (r^3-mr^2+a^2\cos^2\th r -m a^2\cos^2\th)\\
&& + ( -4 a^2 )r^2+8m a^2 \cos^2\th r -4a^4\cos^2\th +4ia\cos\th(- r^3-m r^2-r a^2\cos^2\th +m a^2 \cos^2\th)\Big)
\eeaa
which gives
\beaa
|q|^6W&=&(4 \cot^2\th-2) r^4+(8 \cot^2\th a^2\cos^2\th-4a^2\cos^2\th+4a^2) r^2-8m a^2 \cos^2\th r\\
&&+4  a^4\cos^4\th\cot^2\th-2a^4\cos^4\th+4 a^4 \cos^2\th  \\
&&+4i a \cos\th  (-3r^3+mr^2-3a^2\cos^2\th r +3m a^2\cos^2\th).
\eeaa
In particular writing 
\beaa
\cot^2\th \cos^2\th&=& \cot^2\th(1- \sin^2\th)= \cot^2\th- \cos^2\th
\eeaa
we obtain
\beaa
|q|^6\Re(W)&=&( 4 \cot^2\th-2)r^4+(8 \cot^2\th+4-12\cos^2\th)a^2r^2-8mr a^2 \cos^2\th\\
&&+4a^4 \cot^2\th-6a^4\cos^4\th, \\
|q|^6\Im(W)&=& a\cos\th \left(-12r^3+4m  r^2-12ra^2\cos^2\th+12m a^2\cos^2\th  \right).
\eeaa

From considering the case of Schwarzschild, we know that the real part of the potential should contain the following terms:
\beaa
&& 4 \La^2+\frac 1 2 \trch \trchb-10\omb \trch-8\rho\\
&=& 4  \frac{(r^2+a^2)^2}{|q|^6}\cot^2\th-\frac 1 2 \frac{2r}{|q|^2} \frac{2r\Delta}{|q|^4}+10 \frac{2r}{|q|^2}\frac{-a^2\cos^2\th (r-m)-mr^2+a^2r}{|q|^4} -8 \frac{2m}{|q|^6} (-r^3+3r  a^2 \cos^2\th)\\
&=&\frac{1}{|q|^6} \Big( 4 (r^4+2r^2a^2+a^4)\cot^2\th- 2r^2(r^2-2mr+a^2)+20r(-mr^2+a^2r-a^2\cos^2\th r+a^2\cos^2\th m)\\
&& -16m (-r^3+3r  a^2 \cos^2\th)\Big)\\
&=&\frac{1}{|q|^6} \Big( (4 \cot^2\th-2) r^4+(8 \cot^2\th+18-20\cos^2\th) a^2r^2-28m r a^2\cos^2\th+4 a^4\cot^2\th \Big).
\eeaa
It could also contain terms like
\beaa
 \atrch \atrchb, \quad  \etab_1 \etab_1, \quad \etab_2\etab_2, \quad  e_1(\etab_1), \quad \La  \etab_1,
\eeaa
which are given by
\beaa
 \atrch \atrchb&=& \frac{2a\cos\th}{|q|^2}\frac{2a\Delta\cos\th}{|q|^4}= \frac{4a^2\cos^2\th}{|q|^6} (r^2-2mr+a^2)\\
 &=&\frac{1}{|q|^6} ((4\cos^2\th) a^2r^2-8mr a^2\cos^2\th +4a^4\cos^2\th ),
\eeaa
\beaa
\etab_1 \etab_1&=&  \frac{a^4\sin^2\th \cos^2\th }{|q|^6}= \frac{1}{|q|^6} a^4\cos^2\th(1-\cos^2\th),  \\
 \etab_2\etab_2&=& \frac{a^2\sin^2\th r^2}{|q|^6}= \frac{1}{|q|^6}(1- \cos^2\th) a^2r^2, \\
 e_1(\etab_1)&=&-a^2 \frac{1}{|q|}\pr_\th\left(\frac{\sin\th \cos\th }{|q|^3}\right)= \frac{1}{|q|^6}((-2\cos^2\th+1)a^2r^2+a^4\cos^2\th(-2+\cos^2\th)), \\
 \La  \etab_1&=&-  \frac{r^2+a^2}{|q|^3}\cot\th   \frac{a^2\sin\th \cos\th }{|q|^3}=\frac{1}{|q|^6}((- \cos^2\th )a^2r^2-a^4 \cos^2\th ).  
\eeaa
From the above it is clear that we need to have a term $-\frac 5 2 \atrch \atrchb$. We have
\beaa
&& 4 \La^2+\frac 1 2 \trch \trchb-10\omb \trch-8\rho-\frac 5 2 \atrch \atrchb\\
&=&\frac{1}{|q|^6} \Big( (4 \cot^2\th-2) r^4+(8 \cot^2\th+18-20\cos^2\th) a^2r^2-28m r a^2\cos^2\th+4 a^4\cot^2\th \Big)\\
&& +   \frac{1}{|q|^6} (-(10\cos^2\th) a^2r^2+20mr a^2\cos^2\th -10a^4\cos^2\th )\\
&=&\frac{1}{|q|^6} \Big( (4 \cot^2\th-2) r^4+(8 \cot^2\th+18-30\cos^2\th) a^2r^2-8m r a^2\cos^2\th+4 a^4\cot^2\th -10a^4\cos^2\th\Big).
\eeaa

Observe that 
\beaa
&&|q|^6\left(\Re(W)-(4 \La^2+\frac 1 2 \trch \trchb-10\omb \trch-8\rho-\frac 5 2 \atrch \atrchb)\right)\\
&=&|q|^6\Big((-14+18\cos^2\th)a^2r^2+10a^4\cos^2\th-6a^4\cos^4\th\Big).
\eeaa
We find a combination of $\etab_1\etab_1$,  $\etab_2\etab_2$, $e_1(\etab_1)$, $\La \etab_1$ which gives the above:
\beaa
&&|q|^6( n \etab_1\etab_1+ m \etab_2\etab_2+ p e_1(\etab_1)+ q\La \etab_1)\\
&=&|q|^6( n a^4\cos^2\th(1-\cos^2\th)+ m (1- \cos^2\th) a^2r^2+ p ((-2\cos^2\th+1)a^2r^2+a^4\cos^2\th(-2+\cos^2\th))\\
&& + q((- \cos^2\th )a^2r^2-a^4 \cos^2\th ) )\\
&=&|q|^6( (m+p+(- m-2p-q)\cos^2\th) a^2r^2+ (n-2p-q) a^4\cos^2\th+(-n+p) a^4\cos^4\th) ).
\eeaa
For every $n$, if $m=-n-8$, $p=n-6$, $q=2-n$ gives the desired expression, i.e.
\beaa
&&|q|^6( n \etab_1\etab_1+ (-n-8) \etab_2\etab_2+ (n-6) e_1(\etab_1)+ (2-n)\La \etab_1)\\
&=&|q|^6\Big((-14+18\cos^2\th)a^2r^2+10a^4\cos^2\th-6a^4\cos^4\th\Big)
\eeaa
which finally gives for every $n$
\beaa
|q|^6\Re(W)&=&|q|^6\Big(4 \La^2+\frac 1 2 \trch \trchb-10\omb \trch-8\rho-\frac 5 2 \atrch \atrchb\\
&&+ n \etab_1\etab_1+ (-n-8) \etab_2\etab_2+ (n-6) e_1(\etab_1)+ (2-n)\La \etab_1\Big).
\eeaa

In particular taking $n=8$ we can write
\bea\label{Re-W}
\begin{split}
\Re(W)&=4 \La^2+\frac 1 2 \trch \trchb-10\omb \trch-8\rho-\frac 5 2 \atrch \atrchb\\
&+ 2 e_1(\etab_1)-6\La \etab_1+ 8 \etab_1\etab_1-16 \etab_2\etab_2.
\end{split}
\eea

The imaginary part of the potential could contain terms like
\beaa
\atrch \trchb, \quad \atrchb \trch, \quad \omb \atrch, \quad \dual \rho, \quad \etab_1 \etab_2, \quad  e_1(\etab_2), \quad  \La \etab_2,
\eeaa
which are given by
\beaa
 \atrch\trchb&=&  \frac{a\cos\th}{|q|^6} (-4r^3+8mr^2-4a^2r), \\
 \atrchb \trch&=& \frac{a\cos\th}{|q|^6} (4r^3-8mr^2+4a^2r), \\
 \omb \atrch&=&\frac{a\cos\th}{|q|^6} (2mr^2-a^2(2-2\cos^2\th)r-2a^2\cos^2\th m),\\
 \dual \rho&=& \frac{a\cos\th }{|q|^6} (6mr^2- 2ma^2 \cos^2\th),\\
 \etab_1 \etab_2&=& \frac{a \cos\th}{|q|^6}(a^2(1-\cos^2\th) r), \\
 \Lambda \etab_2&=&\frac{a\cos\th }{|q|^6}( - r^3-a^2r), \\
 e_1(\etab_2)&=& \frac{1}{|q|} \pr_\th\left(-\frac{a\sin\th r}{|q|^3}\right)= -\frac{ar}{|q|} \left(\frac{\cos\th |q|^3+3a^2|q| \cos\th\sin^2\th }{|q|^6}\right)\\
 &=& \frac{a\cos\th}{|q|^6}(- r^3+2a^2\cos^2\th r-3a^2 r ).
\eeaa
We find a combination of the above which gives $\Im(W)$. We have
\beaa
&&n \atrch\trchb+w  \atrchb \trch+c \omb \atrch+b \dual \rho +p  \etab_1 \etab_2+q \Lambda \etab_2+ t  e_1(\etab_2)\\
&=& \frac{a\cos\th}{|q|^6}  \Big(n (-4r^3+8mr^2-4a^2r)+w((4r^3-8mr^2+4a^2r)\\
&&+ c(2mr^2-a^2(2-2\cos^2\th)r-2a^2\cos^2\th m)+b (6mr^2- 2ma^2 \cos^2\th)\\
&&+p a^2(1-\cos^2\th) r+q ( - r^3-a^2r)+ t  (- r^3+2a^2\cos^2\th r-3a^2 r )\Big)\\
&=& \frac{a\cos\th}{|q|^6}  \Big((-4n+4w-q-t)r^3+ (8n-8w+2c+6b)mr^2+(-4 n +4w+p-q-3t-2c)a^2r \\
&&+(-p+2t+2c)a^2\cos^2\th r+(-2c-2b)a^2\cos^2\th m\Big).
\eeaa
From the last term we obtain $c=-b-6$. From the $mr^2$ term we deduce $2n-2w=4-b$. By comparing each term we obtain for every $w$, $b$ and $t$:
\beaa
\Im(W)&=& (w-\frac 1 2 b +2) \atrch\trchb+w  \atrchb \trch+(-b-6) \omb \atrch+b \dual \rho +(2t-2b)  \etab_1 \etab_2\\
&&+(2b-t+4) \Lambda \etab_2+ t  e_1(\etab_2).
\eeaa
In particular for $b=4$, $t=0$ and $w=1$, we obtain
\bea\label{Im-W}
\Im(W)&=&  \atrch\trchb+  \atrchb \trch-10 \omb \atrch+4 \dual \rho +12 \Lambda \etab_2-8  \etab_1 \etab_2.
\eea


\subsection{The projection of the Teukolsky equation}


The linear Teukolsky equation as obtained in our formalism is given by
\beaa
 &&-\nabc_4\nabc_3A+ \frac{1}{2}\DDc\hot (\DDbc \c A)+\left(- \frac 1 2 \tr X -2\ov{\tr X} \right)\nabc_3A-\frac{1}{2}\tr\Xb \nabc_4A\\
&&+\frac{1}{2}  \ov{\Hb}\c \DDc A+ \left( 2H  +\frac 1 2 \underline{H} \right)\c \DDbc  A + \left(-\ov{\tr X} \tr \Xb +2\ov{P}\right) A+  2H  \hot (\ov{\Hb} \c A)=0.
\eeaa
Recall that in Kerr 
 \beaa
&& \om=0, \qquad  Z=-\Hb.
\eeaa 
We can therefore write out the conformal derivatives as
\beaa
\nabc_3A&=& \nab_3 A-4\omb A, \qquad \nabc_4 A=\nab_4A, \\
-\nabc_4 \nabc_3 A&=& -\nab_4\nab_3 A +4\omb \nab_4A+\left(4\nab_4\omb\right) A,\\
\DDbc \c A&=& \DDb \c A+ 2 \ov{Z} \c A= \DDb \c A- 2 \ov{\Hb} \c A, \\
\frac{1}{2}\DDc\hot (\DDbc \c A)&=& \frac{1}{2}(\DD+2 Z)\hot (\DDb \c A- 2 \ov{\Hb} \c A)=\frac{1}{2}(\DD-2 \Hb)\hot (\DDb \c A- 2 \ov{\Hb} \c A)\\
&=&\frac{1}{2}\DD\hot (\DDb \c A)- \Hb\c  \DDb  A-  \ov{\Hb} \c \DD A+\left(-\DD\c  \ov{\Hb} +2  \Hb \c  \ov{\Hb} \right)   A, \\
\frac{1}{2}  \ov{\Hb}\c \DDc A&=& \frac{1}{2}  \ov{\Hb}\c (\DD+2Z) A=\frac{1}{2}  \ov{\Hb}\c  \DD A- ( \ov{\Hb}\c \Hb ) A,\\
\left( 2H  +\frac 1 2 \underline{H} \right)\c \DDbc  A &=& \left( 2H  +\frac 1 2 \underline{H} \right)\c  (\DDb+2\ov{Z})   A = \left( 2H  +\frac 1 2 \underline{H} \right)\c  \DDb   A - \left( 4H \c \ov{\Hb} + \underline{H} \c \ov{\Hb}\right)   A .
\eeaa
This implies
\beaa
 && -\nab_4\nab_3 A+\frac{1}{2}\DD\hot (\DDb \c A)+\left(- \frac 1 2 \tr X -2\ov{\tr X} \right)\nab_3 A+\left(-\frac{1}{2}\tr\Xb+4\omb\right) \nab_4A\\
&&-\frac{1}{2}  \ov{\Hb}\c  \DD A+\left( 2H  -\frac 1 2 \underline{H} \right)\c  \DDb   A  \\
&&+ \left(-\ov{\tr X} \tr \Xb +2\ov{P}+2\omb \tr X  +8 \omb \ov{\tr X} + 4\nab_4\omb-\DD\c  \ov{\Hb}-4H \c \ov{\Hb} \right) A+  2H  \hot (\ov{\Hb} \c A)=0.
\eeaa
Also writing
\beaa
-\frac{1}{2}\tr\Xb +4\omb&=& -\frac{1}{2}(\trchb - i \atrchb ) +4\omb=\left(-\frac{1}{2}\trchb +4\omb\right)+ i \left(\frac 1 2 \atrchb \right),\\
- \frac 1 2 \tr X -2\ov{\tr X} &=& - \frac 1 2 (\trch - i \atrch) -2(\trch + i \atrch)=\left(-\frac 5 2 \trch\right) +i \left(- \frac 3 2 \atrch\right),
\eeaa
and writing $-\frac{1}{2}  \ov{\Hb}\c  \DD A-\frac{1}{2}  \Hb\c  \DDb A=-2 \etab \c \nab$, we obtain
\bea\label{Teuk-intermediate}
\begin{split}
 & -\nab_4\nab_3 A+\frac{1}{2}\DD\hot (\DDb \c A)+\left(-\frac 5 2 \trch- i \frac 3 2 \atrch \right)\nab_3 A+\left(-\frac{1}{2}\trchb +4\omb+ i \frac 1 2 \atrchb \right) \nab_4A\\
&-2\etab \c \nab A + 2H \c  \DDb   A  \\
&+ \left(-\ov{\tr X} \tr \Xb +2\ov{P}+2\omb \tr X  +8 \omb \ov{\tr X} + 4\nab_4\omb-\DD\c  \ov{\Hb}-4H \c \ov{\Hb} \right) A+  2H  \hot (\ov{\Hb} \c A)=0.
\end{split}
\eea

We collect in the following Lemma the calculations for the projection to the $11$ component. 
\begin{lemma}\label{lemma-projection} 
The following projections to the $11$ component hold true:
\beaa
(\nab_3 A)_{11} &=&  e_3( \mathfrak{a})+i  \atrchb \mathfrak{a},\\
(\nab_4 A)_{11} &=&  e_4( \mathfrak{a})+i  \atrch \mathfrak{a}, \\
(\nab_4\nab_3  A)_{11}&=&e_4 e_3 ( \mathfrak{a})+i  \atrchb  e_4 ( \mathfrak{a})+ i  \atrch  e_3 ( \mathfrak{a}) +\left( -  \atrch   \atrchb+i e_4\atrchb\right)  \mathfrak{a},\\
\frac 1 2 \DD\hot( \DDb \c A)_{11}&= & \lap \mathfrak{a}- \frac 1 2  i \atrch e_3( \mathfrak{a}) -\frac 1 2 i\atrchb e_4( \mathfrak{a}) +4 i \La  e_2(\mathfrak{a}) \\
&&+\left( -4 \La^2 +\atrch \atrchb +G_0\right) \mathfrak{a}, \\
(2 \etab \c \nab A)_{11}&=& 2 \etab_1 e_1\mathfrak{a}+2 \etab_2 e_2(\mathfrak{a} )+4i \La \etab_2  \mathfrak{a}, \\
(2H \c \DDb A)_{11}&=&4 \left(  \etab_1 - i \etab_2\right)e_1(\mathfrak{a})-4\left( \etab_2+i \etab_1\right)  e_2(\mathfrak{a} )+8 (\La  \etab_1-i \La \etab_2)    \mathfrak{a},
\eeaa
where $G_0=  \frac 1 2 \trch\trchb+  \frac 1 2 \atrch\atrchb+2\rho$ is given by the Gauss equation.
\end{lemma}

\begin{proof} 
Recall that in Kerr
\beaa
\g(\D_4e_1, e_2) &=& \chi_{12}=\frac 1 2 \atrch , \qquad \g(\D_3e_1, e_2) = \chib_{12}=\frac 1 2 \atrchb,
\eeaa
which gives
\beaa
\nab_4 e_1&=&\frac 1 2 \atrch e_2, \qquad \nab_3 e_1=\frac 1 2 \atrchb e_2.
\eeaa
We therefore have 
\beaa
(\nab_3 A)_{11} &=&  e_3( A_{11})- 2A_{\nab_31 1}=e_3( A_{11})- \atrchb\in_{12}  A_{21} =  e_3( A_{11})+i  \atrchb A_{11},\\
(\nab_4 A)_{11} &=&  e_4( A_{11})- 2 A_{\nab_41 1}=e_3( A_{11})- \atrch A_{21}=e_4( A_{11})+i  \atrch A_{11},
\eeaa
since $A_{12}=-i A_{11}$. This proves the first two identities.

We also have
\beaa
(\nab_4\nab_3  A)_{11} &=& e_4( \nab_3 A_{11})- 2\nab_3  A_{\nab_41 1}\\
&=&e_4\big( e_3(  \mathfrak{a})+i  \atrchb  \mathfrak{a}\big)-  \in_{12} \atrch  \nab_3 A_{21} \\
&=&  e_4 e_3 ( \mathfrak{a})+i (e_4\atrchb)  \mathfrak{a}+i  \atrchb  e_4 ( \mathfrak{a})+i   \atrch  \nab_3 A_{11}\\
 &=&  e_4 e_3 ( \mathfrak{a})+i (e_4\atrchb)  \mathfrak{a}+i  \atrchb  e_4 ( \mathfrak{a})+i   \atrch  ( e_3(  \mathfrak{a})+i  \atrchb  \mathfrak{a})
 \\
  &=&  e_4 e_3 ( \mathfrak{a})+i  \atrchb  e_4 ( \mathfrak{a})+ i  \atrch  e_3 ( \mathfrak{a}) +\left( -  \atrch   \atrchb+i e_4\atrchb\right)  \mathfrak{a}
\eeaa
which gives the third expression. 

Recall the relation:
\beaa
\frac 1 2 \DD\hot( \DDb \c A)&=& \lap_2 A - i (\nab_1\nab_2-\nab_2\nab_1) A
\eeaa
as derived in Proposition 4.10. 
Using the Gauss formula \eqref{Gauss-eq-first-com}:
\beaa
\big( \nab_1 \nab_2- \nab_2 \nab_1\big)  A &=&\frac 1 2 \left(\atrch\nab_3+\atrchb \nab_4\right) A+ i\left( \frac 1 2  \trch\trchb+ \frac 1 2 \atrch\atrchb+2\rho  \right) A 
\eeaa
one obtains
\beaa
\frac 1 2 \DD\hot( \DDb \c A)&=& \lap_2 A - \frac 1 2 i \atrch\nab_3A -\frac 1 2 i \atrchb \nab_4 A+  \left(  \frac 1 2 \trch\trchb+  \frac 1 2 \atrch\atrchb+2\rho \right) A 
\eeaa
i.e.
\bea\label{expression-DD-laplacian}
\frac 1 2 \DD\hot( \DDb \c A)&= & \lap_2 A- \frac 1 2  i \atrch \nab_3A -\frac 1 2i \atrchb \nab_4A+ G_0 A
\eea
with $G_0=  \frac 1 2 \trch\trchb+  \frac 1 2 \atrch\atrchb+2\rho$, given by the Gauss equation. 

Projecting to the $11$ component, and using the previous relations we obtain
\beaa
\frac 1 2 \DD\hot( \DDb \c A)_{11}&= & (\lap_2 A)_{11}- \frac 1 2  i \atrch \nab_3A_{11} -\frac 1 2i \atrchb \nab_4A_{11}+ (G_0+i G_1) A_{11}\\
&= & (\lap_2 A)_{11}- \frac 1 2  i \atrch \left(e_3( \mathfrak{a})+i  \atrchb \mathfrak{a}\right) -\frac 1 2 i\atrchb \left(e_4( \mathfrak{a})+i  \atrch \mathfrak{a}\right)+\left(G_0+i G_1\right) \mathfrak{a}\\
&= & (\lap_2 A)_{11}- \frac 1 2  i \atrch e_3( \mathfrak{a}) -\frac 1 2 i\atrchb e_4( \mathfrak{a})+\left(\atrch \atrchb + G_0+i G_1\right) \mathfrak{a}.
\eeaa

We now compute $(\lap_2 A)_{11}$. 
Recall that in Kerr
\beaa
(\La_1)_{21}:=\g(\D_1e_1, e_2) &=& 0,\\
(\La_2)_{21}:=\g(\D_2e_1, e_2) &=& \frac{r^2+a^2}{|q|^3}\cot\th:=\La ,\\
(\La_1)_{12}:=\g(\D_1e_2, e_1) &=& 0,\\
(\La_2)_{12}:=\g(\D_2e_2, e_1) &=& - \frac{r^2+a^2}{|q|^3}\cot\th=-\La,
\eeaa
i.e.,
\bea
\bsplit
\nab_{e_1} e_1&= \nab_{e_1} e_2=0, \qquad \nab_{e_2} e_1 =\La e_2, \qquad \nab_{e_2} e_2=-\La e_1.
\end{split}
\eea
Therefore
\beaa
\nab_1 A_{11}&=& e_1 A_{11} -2 A_{\nab_1 1 1} = e_1(\mathfrak{a}),\\
\nab_2 A_{11}&=& e_2(A_{11} ) - 2 A_{\nab_2 1 1} = e_2(A_{11} ) - 2 \La A_{21} = e_2(\mathfrak{a} )+2 i \La  \mathfrak{a}.
\eeaa
Moreover,
\beaa
\nab_1 \nab_1 A_{11}&=& e_1\big(\nab_1  A_{11}\big) - \nab_{\nab_1 1} A_{11}-2\nab_1 A_{\nab_1 1 1 }= e_1( e_1\mathfrak{a}),\\
\nab_2 \nab_2 A_{11}&=& e_2\big( \nab_2 A_{11}\big)-\nab_{\nab_2 2}A_{11} - 2 \nab_2A_{\nab_2 1 1}\\
&=& e_2\Big( e_2(\mathfrak{a} )+2 i \La  \mathfrak{a}\Big)+\La \nab_1A_{11} -2 \La \nab_2 A_{21}\\
&=& e_2 e_2(\mathfrak{a} )+2 i \La  e_2(\mathfrak{a})+\La e_1(\mathfrak{a}) +2i  \La \nab_2 A_{11}\\
&=& e_2 e_2(\mathfrak{a} )+2 i \La  e_2(\mathfrak{a})+\La e_1(\mathfrak{a}) +2i  \La ( e_2(\mathfrak{a} )+2 i \La  \mathfrak{a})\\
&=& e_2 e_2(\mathfrak{a} )+4 i \La  e_2(\mathfrak{a})+\La e_1(\mathfrak{a}) -4 \La^2  \mathfrak{a}.
\eeaa

We define the laplacian of a scalar $\mathfrak{a}$ as
\beaa
\lap\mathfrak{a}&=& \g^{ab} \nab_a\nab_b\mathfrak{a}= \nab_1\nab_1\mathfrak{a}+\nab_2\nab_2\mathfrak{a}\\
&=& e_1e_1\mathfrak{a}-\D_{e_1} e_1 \mathfrak{a}+e_2e_2\mathfrak{a}-\D_{e_2}e_2\mathfrak{a}\\
&=& e_1e_1\mathfrak{a}+e_2e_2\mathfrak{a}+\La e_1\mathfrak{a}.
\eeaa
We deduce,
\beaa
(\lap_2 A)_{11}&=& \nab_1 \nab_1 A_{11}+\nab_2 \nab_2 A_{11}\\
&=&e_1( e_1\mathfrak{a})+ e_2 e_2(\mathfrak{a} )+4 i \La  e_2(\mathfrak{a})+\La e_1(\mathfrak{a}) -4 \La^2  \mathfrak{a}\\
&=&\lap \mathfrak{a} +4 i \La  e_2(\mathfrak{a}) -4 \La^2  \mathfrak{a}.
\eeaa
From \eqref{expression-DD-laplacian}, we obtain
\beaa
\frac 1 2 \DD\hot( \DDb \c A)_{11}&= & \lap \mathfrak{a}- \frac 1 2  i \atrch e_3( \mathfrak{a}) -\frac 1 2 i\atrchb e_4( \mathfrak{a}) +4 i \La  e_2(\mathfrak{a}) \\
&&+\left( -4 \La^2 +\atrch \atrchb + G_0\right) \mathfrak{a}
\eeaa
as desired.

We have
\beaa
(2 \etab \c \nab A)_{11}&=& 2 \etab_1 \nab_1 A_{11}+2 \etab_2  \nab_2 A_{11}= 2 \etab_1 e_1\mathfrak{a}+2 \etab_2 e_2(\mathfrak{a} )+4i \La \etab_2  \mathfrak{a}.
\eeaa
Finally, we compute 
\beaa
(2H \c \DDb A)_{11}&=& 2H_1 \DDb_1 A_{11}+2H_2  \DDb_2 A_{11}\\
&=& 2H_1 (\nab_1- i \dual \nab_1) A_{11}+2H_2  (\nab_2- i \dual \nab_2) A_{11}\\
&=& 2H_1 (\nab_1- i \nab_2) A_{11}+2H_2  (\nab_2+i  \nab_1) A_{11}\\
&=& 2(H_1+ i H_2) \nab_1A_{11}+2(H_2-i H_1)  \nab_2 A_{11}.
\eeaa
We have
\beaa
H_1+ i H_2&=& \eta_1+ i \dual \eta_1+ i (\eta_2+ i \dual \eta_2)\\
&=& \eta_1+ i  \eta_2+ i (\eta_2- i  \eta_1)\\
&=& 2\eta_1+ 2i  \eta_2
\eeaa
and 
\beaa
H_2-i H_1&=& -i (H_1+ i H_2)=2  \eta_2-2i\eta_1.
\eeaa
This gives
\beaa
(2H \c \DDb A)_{11}&=& 4(\eta_1+ i  \eta_2) \nab_1A_{11}+4(  \eta_2-i\eta_1)  \nab_2 A_{11}\\
&=& 4(\eta_1+ i  \eta_2) e_1(\mathfrak{a})+4(  \eta_2-i\eta_1) (e_2(\mathfrak{a} )+2 i \La  \mathfrak{a})\\
&=& 4(\eta_1+ i  \eta_2) e_1(\mathfrak{a})+4(  \eta_2-i\eta_1) e_2(\mathfrak{a} )+8(\La \eta_1+ i \La\eta_2)   \mathfrak{a}.
\eeaa
Using that $\eta_1=\etab_1$, and $\eta_2=-\etab_2$, we have
\beaa
(2H \c \DDb A)_{11}&=& 4(\etab_1- i  \etab_2) e_1(\mathfrak{a})+4( - \etab_2-i\etab_1) e_2(\mathfrak{a} )+8(\La \etab_1- i \La\etab_2)   \mathfrak{a}
\eeaa
which gives the final identity.
\end{proof}

We can now use the above Lemma to project the linear Teukolsky equation \eqref{Teuk-intermediate} to the component $11$, to obtain a wave equation for $\mathfrak{a}$. We summarize the computation in the following proposition.

\begin{proposition} 
The complex scalar $\mathfrak{a}$ verifies the following wave equation:
\bea\label{wave-a}
\begin{split}
\square_{g_{Kerr}} \mathfrak{a}&=\left( -4\omb+ i  \atrchb \right) e_4( \mathfrak{a})+\left(2 \trch +3i  \atrch \right)e_3( \mathfrak{a}) \\
  &+   i4 \etab_2 e_1(\mathfrak{a})+\left(8 \etab_2-4i(\La- \etab_1)\right)  e_2(\mathfrak{a} )+ V \mathfrak{a}
  \end{split}
\eea
where 
\beaa
V&=& \left(4 \La^2+\frac 1 2 \trch \trchb-\frac 5 2 \atrch \atrchb-10\omb \trch-8\rho+2 e_1(\etab_1)-6\La \etab_1+8\etab_1\etab_1-16\etab_2 \etab_2\right)\\
&& + i\left(\atrch\trchb + \atrchb \trch-10\omb \atrch+  3e_1(\etab_2)+15\La\etab_2+8 \eta_1 \eta_2 \right) 
\eeaa
is the potential. 
\end{proposition}

\begin{proof} 
From projecting to the $11$ component the linear Teukolsky equation \eqref{Teuk-intermediate}, we obtain
\beaa
 &&-\nab_4 \nab_3A_{11}- 2\etab \c \nab A_{11}+\left(\left(-\frac{1}{2}\trchb +4\omb\right)+ i \left(\frac 1 2 \atrchb \right)\right) \nab_4A_{11}\\
 && +\left(\left(-\frac 5 2 \trch\right) +i \left(- \frac 3 2 \atrch\right) \right)\nab_3A_{11}+\frac{1}{2} \mathcal{D}\hot(\ov{\DD}\c A  )_{11} + 2H \c \DDb A_{11} \\
 & &+ \left(-\ov{\tr X} \tr \Xb +2\ov{P}+2\omb \tr X  +8 \omb \ov{\tr X} + 4\nab_4\omb-\DD\c  \ov{\Hb}-4H \c \ov{\Hb} \right) A_{11}+  2H  \hot (\ov{\Hb} \c A)_{11}=0.
\eeaa
Using Lemma \ref{lemma-projection}, the first line of the above reduces to
\beaa
&&-\nab_4 \nab_3A_{11}- 2\etab \c \nab A_{11}+\left(\left(-\frac{1}{2}\trchb +4\omb\right)+ i \left(\frac 1 2 \atrchb \right)\right) \nab_4A_{11}\\
&=&-e_4 e_3 ( \mathfrak{a})-i  \atrchb  e_4 ( \mathfrak{a})- i  \atrch  e_3 ( \mathfrak{a}) +\left(  \atrch   \atrchb-i e_4\atrchb\right)  \mathfrak{a}\\
&&+\left(\left(-\frac{1}{2}\trchb +4\omb\right)+ i \left(\frac 1 2 \atrchb \right)\right) ( e_4( \mathfrak{a})+i  \atrch \mathfrak{a} )-2 \etab_1 e_1\mathfrak{a}-2 \etab_2 e_2(\mathfrak{a} )-4i \La \etab_2  \mathfrak{a}\\
&=&-e_4 e_3 ( \mathfrak{a})- i  \atrch  e_3 ( \mathfrak{a}) +\left(\left(-\frac{1}{2}\trchb +4\omb\right)- i \left(\frac 1 2 \atrchb \right)\right) e_4( \mathfrak{a})-2 \etab_1 e_1\mathfrak{a}-2 \etab_2 e_2(\mathfrak{a} ) \\
&&+\left( \frac 1 2 \atrchb\atrch+i\left(-e_4\atrchb-\frac{1}{2}\atrch\trchb +4\omb \atrch-4 \Lambda \etab_2\right) \right)   \mathfrak{a}. 
\eeaa
The second line reduces to
\beaa
 && \left(\left(-\frac 5 2 \trch\right) +i \left(- \frac 3 2 \atrch\right) \right)\nab_3A_{11}+\frac{1}{2} \mathcal{D}\hot(\ov{\DD}\c A  )_{11} + 2H \c \DDb A_{11}\\
  &=& \left(\left(-\frac 5 2 \trch\right) +i \left(- \frac 3 2 \atrch\right) \right)(e_3( \mathfrak{a})+i  \atrchb \mathfrak{a})\\
  &&+ \lap \mathfrak{a}- \frac 1 2  i \atrch e_3( \mathfrak{a}) -\frac 1 2 i\atrchb e_4( \mathfrak{a}) +4 i \La  e_2(\mathfrak{a}) +\left( -4 \La^2 +\atrch \atrchb + G_0+i G_1\right) \mathfrak{a} \\
&&+ 4 \left(  \etab_1 - i \etab_2\right)e_1(\mathfrak{a})-4\left( \etab_2+i \etab_1\right)  e_2(\mathfrak{a} )+8 (\La  \etab_1-i \La \etab_2)    \mathfrak{a}\\
  &=&  \lap \mathfrak{a}+\left(\left(-\frac 5 2 \trch\right) +i \left(- 2 \atrch\right) \right)e_3( \mathfrak{a})-\frac 1 2 i\atrchb e_4( \mathfrak{a})\\
  &&+  \left( 4 \etab_1 - i4 \etab_2\right)e_1(\mathfrak{a})+4\left(- \etab_2+i(\La- \etab_1)\right)  e_2(\mathfrak{a} )\\
  &&+\left(\left( \frac 5 2  \atrchb \atrch-4 \La^2+8\La  \etab_1+G_0\right)+i \left(-\frac 5 2  \atrchb\trch- 8\La \etab_2\right)  \right) \mathfrak{a}.
\eeaa
We compute
\beaa
(2H   \hot (\ov{\Hb} \c A))_{11}&=& 2H_1 (\ov{\Hb} \c A)_1 - H \c (\ov{\Hb} \c A)\\
&=& 2H_1 (\ov{\Hb} \c A)_1 - H_1 (\ov{\Hb} \c A)_1 - H_2 (\ov{\Hb} \c A)_2\\
&=& H_1 (\ov{\Hb} \c A)_1  - H_2 (\ov{\Hb} \c A)_2.
\eeaa
We have, using that $A_{12}=-i A_{11}$ and $A_{22}=-A_{11}$:
\beaa
(\ov{\Hb} \c A)_1&=& \ov{\Hb}_1 A_{11}+ \ov{\Hb}_2 A_{12}=(\ov{\Hb}_1 -i \ov{\Hb}_2) A_{11},\\
(\ov{\Hb} \c A)_2&=& \ov{\Hb}_1 A_{12}+ \ov{\Hb}_2 A_{22}=-(i\ov{\Hb}_1 + \ov{\Hb}_2) A_{11}.
\eeaa
This gives
\beaa
(2H   \hot (\ov{\Hb} \c A))_{11}&=& H_1 (\ov{\Hb}_1 -i \ov{\Hb}_2) A_{11}  + H_2 (i\ov{\Hb}_1 + \ov{\Hb}_2) A_{11}\\
&=&\left( H_1 \ov{\Hb}_1+ H_2  \ov{\Hb}_2 - i (H_1  \ov{\Hb}_2 - H_2\ov{\Hb}_1) \right)A_{11}.
\eeaa
We also have 
\beaa
H \c \ov{\Hb} &=& H_1 \ov{\Hb}_1+H_2  \ov{\Hb}_2.
\eeaa
Therefore 
\beaa
-4H \c \ov{\Hb} A_{11}+  2H  \hot (\ov{\Hb} \c A)_{11}&=&\left( -3H_1 \ov{\Hb}_1-3 H_2  \ov{\Hb}_2 - i (H_1  \ov{\Hb}_2 - H_2\ov{\Hb}_1) \right)\mathfrak{a}.
\eeaa
Observe that in Kerr $H_1=\ov{\Hb_1}, \quad H_2=-\ov{\Hb_2}$, and therefore
\beaa
-4H \c \ov{\Hb} A_{11}+  2H  \hot (\ov{\Hb} \c A)_{11}&=&\left( -3\ov{\Hb_1} \ov{\Hb}_1+3 \ov{\Hb_2}  \ov{\Hb}_2 -2 i (\ov{\Hb_1}  \ov{\Hb}_2 ) \right)\mathfrak{a}.
\eeaa
Using that $\ov{\Hb}_1=-i \ov{\Hb}_2$, we have
\beaa
-4H \c \ov{\Hb} A_{11}+  2H  \hot (\ov{\Hb} \c A)_{11}&=& 4 \ov{\Hb}_2  \ov{\Hb}_2 \mathfrak{a}\\
&=& 4 (\etab_2+ i \etab_1) (\etab_2+ i \etab_1)  \mathfrak{a}\\
&=& 4(-\etab_1\etab_1+\etab_2 \etab_2 +2i(\etab_1 \etab_2)) \mathfrak{a}.
\eeaa 
We also compute
\beaa
 \DD\c\ov{\Hb}&=&  \DD_1\ov{\Hb}_1+ \DD_2\ov{\Hb}_2\\
 &=&  (\nab_1 + i \dual \nab_1)\ov{\Hb}_1+ (\nab_2+i \dual \nab_2)\ov{\Hb}_2\\
 &=&  (\nab_1 + i  \nab_2)\ov{\Hb}_1+ (\nab_2-i \nab_1)\ov{\Hb}_2\\
 &=&  \nab_1\ov{\Hb}_1 + i  \nab_2\ov{\Hb}_1+ \nab_2\ov{\Hb}_2-i \nab_1\ov{\Hb}_2.
\eeaa
We therefore obtain
\beaa
 \DD\c\ov{\Hb} &=&  e_1(\ov{\Hb}_1) + e_2(\ov{\Hb}_2)-\ov{\Hb}_{\nab_2 2} + i  (e_2(\ov{\Hb}_1)-\ov{\Hb}_{\nab_2 1}- e_1(\ov{\Hb}_2))\\
 &=&  e_1(\etab_1- i \etab_2) +\La \ov{\Hb}_{1} + i  (-\La \ov{\Hb}_{2}- e_1(\etab_2+ i \etab_1))\\
 &=& 2 e_1(\etab_1)- 2i e_1( \etab_2) +\La(\etab_1- i \etab_2)- i  \La (\etab_2 + i \etab_1)\\
  &=& 2 e_1(\etab_1)- 2i e_1( \etab_2) +2\La(\etab_1- i \etab_2).
\eeaa

The third line reduces to
\beaa
&& \left(-\ov{\tr X} \tr \Xb +2\ov{P}+2\omb \tr X  +8 \omb \ov{\tr X} + 4\nab_4\omb-\DD\c  \ov{\Hb}-4H \c \ov{\Hb} \right) A+  2H  \hot (\ov{\Hb} \c A)\\
&=& \Big(-(\trch + i \atrch) (\trchb - i \atrchb)+2(\rho -i \dual \rho)+2\omb (\trch - i \atrch)+8\omb (\trch + i \atrch)\\
&&+ 4\nab_4\omb-(2 e_1(\etab_1)- 2i e_1( \etab_2) +2\La(\etab_1- i \etab_2))+ 4(-\etab_1\etab_1+\etab_2 \etab_2 +2i(\etab_1 \etab_2))  \Big) \mathfrak{a}\\
&=& \Big(-\trch \trchb-\atrch \atrchb+10\omb \trch+2\rho+ 4\nab_4 \omb-2 e_1(\etab_1)-2\La \etab_1-4\etab_1\etab_1+4\etab_2 \etab_2\\
&& + i\left(- \atrch\trchb+\atrchb \trch+6\omb \atrch-2\dual \rho +2 e_1( \etab_2) +2\La  \etab_2+8 \etab_1 \etab_2\right)  \Big) \mathfrak{a}.
\eeaa

Putting the three lines together we obtain
\beaa
&&-e_4 e_3 ( \mathfrak{a})+ \lap \mathfrak{a}+\left(\left(-\frac{1}{2}\trchb +4\omb\right)- i \left( \atrchb \right)\right) e_4( \mathfrak{a})+\left(\left(-\frac 5 2 \trch\right) +i \left(- 3 \atrch\right) \right)e_3( \mathfrak{a}) \\
  &&+  \left( 2 \etab_1 - i4 \etab_2\right)e_1(\mathfrak{a})+\left(- 6\etab_2+4i(\La- \etab_1)\right)  e_2(\mathfrak{a} )\\
 &&  \Big(-4 \La^2-\trch \trchb+2\atrch \atrchb+10\omb \trch+2\rho+G_0+ 4\nab_4 \omb-2 e_1(\etab_1)+6\La \etab_1-4\etab_1\etab_1+4\etab_2 \etab_2\\
&& + i\left(-e_4\atrchb-\frac 32  \atrch\trchb-\frac 3 2 \atrchb \trch+10\omb \atrch-2\dual \rho +2 e_1( \etab_2) -10\La  \etab_2+8 \etab_1 \etab_2\right)  \Big) \mathfrak{a}=0.
\eeaa
Recall that for a scalar, we can write
\bea\label{wave-GKS}
\bsplit
\square_{g_{Kerr}} \mathfrak{a} &=- e_4e_3  \mathfrak{a}  -\frac 1 2 \trchb e_4 \mathfrak{a} -\frac 1 2 \trch e_3 \mathfrak{a} +\lap  \mathfrak{a}+ 2\etab_1  e_1(\mathfrak{a})+ 2\etab_2 e_2( \mathfrak{a}).
\end{split}
\eea
We finally obtain
\beaa
\square_{g_{Kerr}} \mathfrak{a}&=&\left( -4\omb+ i  \atrchb \right) e_4( \mathfrak{a})+\left(2 \trch +3i  \atrch \right)e_3( \mathfrak{a}) \\
  &&+   i4 \etab_2 e_1(\mathfrak{a})+\left(8 \etab_2-4i(\La- \etab_1)\right)  e_2(\mathfrak{a} )+ V \mathfrak{a}
\eeaa
where 
\beaa
V&=& 4 \La^2+\trch \trchb-2\atrch \atrchb-10\omb \trch-2\rho-G_0- 4\nab_4 \omb+2 e_1(\etab_1)-6\La \etab_1+4\etab_1\etab_1-4\etab_2 \etab_2\\
&& + i\left(e_4\atrchb+\frac 32  \atrch\trchb+\frac 3 2 \atrchb \trch-10\omb \atrch+2\dual \rho -2 e_1( \etab_2) +10\La  \etab_2-8 \etab_1 \etab_2\right)
\eeaa
is the potential.

Writing 
\beaa
\nab_4 \omb&=&\rho+(\eta-\etab) \cdot \ze-\eta \c \etab=\rho - (\eta-\etab) \cdot \etab-\eta\c\etab=\rho+\etab\c \etab-2\eta \c \etab \\
&=&\rho+\etab_1 \etab_1+\etab_2\etab_2-2\eta_1 \etab_1-2\eta_2\etab_2\\
&=&\rho+\etab_1 \etab_1+\etab_2\etab_2-2\etab_1 \etab_1+2\etab_2\etab_2\\
&=&\rho-\etab_1 \etab_1+3\etab_2\etab_2
\eeaa
and $G_0= \frac 1 2 \trch\trchb+  \frac 1 2 \atrch\atrchb+2\rho$, we obtain for the real part of the potential:
\beaa
\Re(V)&=& 4 \La^2+\frac 1 2 \trch \trchb-\frac 5 2 \atrch \atrchb-10\omb \trch-8\rho+2 e_1(\etab_1)-6\La \etab_1+8\etab_1\etab_1-16\etab_2 \etab_2
\eeaa
which coincides with \eqref{Re-W}.

For the imaginary part of the potential, we write
\beaa
e_4\atrchb&=&-\frac 1 2(\atrch \trchb+\trch\atrchb)+ 2 \curl \etab  +2 \dual \rho.
\eeaa
Observe that
\beaa
\curl \etab&=& \nab_1 \etab_2 - \nab_2 \etab_1=e_1 (\etab_2)-\etab_{\nab_1 2} -( e_2 \etab_1-\etab_{\nab_2 1})\\
&=&e_1 (\etab_2) +\Lambda \etab_2
\eeaa
which gives
\beaa
e_4\atrchb&=&-\frac 1 2(\atrch \trchb+\trch\atrchb)+ 2 e_1 (\etab_2) +2\Lambda \etab_2 +2 \dual \rho.
\eeaa
We therefore obtain
\beaa
\Im(V)&=& \atrch \trchb+ \trch\atrchb  -10\omb \atrch+4 \dual \rho+12\La  \etab_2-8 \etab_1 \etab_2
\eeaa
which coincides with \eqref{Im-W}. 
\end{proof}


\section{Proof of Proposition \ref{first-intermediate-step-main-theorem}}
\label{proof-appendix}


We write here the proof of Proposition \ref{first-intermediate-step-main-theorem}, which consists in computing the commutator $[Q,\LL ]A$.


\subsection{The commutators for $Q$}


We first compute the following commutators for $Q$, where
\beaa
 Q(U)&=& \nabc_3\nabc_3 U + \C \  \nabc_3U +\D \  U.
 \eeaa

\begin{proposition}\label{commutation-general-Q} 
Let $U$ be a symmetric traceless two tensor of conformal type $s$. Then we have:
\begin{itemize}
\item the following commutator with $\nabc_3$ and $\nabc_4$:
\beaa
\, [Q, \nabc_3] U&=& (- \nab^{(c)}_3 \C) \  \nab^{(c)}_3U+( - \nab^{(c)}_3\D ) \  U,
\eeaa
 \beaa
 [Q, \nabc_4] U  &=&4(\eta-\etab ) \c \nabc\nabc_3U \\
  &&+2\left(\nabc_3(\eta-\etab ) +\left(-\frac 1 2 \trchb +\C\right)(\eta-\etab)\right)\c \nabc U - \atrchb(\eta-\etab ) \c  \dual \nabc  U\\
  &&+\nabc_3(\CC^0_{3, 4} (U))+\CC^0_{3, 4} (\nabc_3U)+\left(2(\eta-\etab ) \c \eta -\nabc_4( \C )\right)\nabc_3 U \\
    &&+2(\eta-\etab ) \c \CC^0_{3,a}(U)+ \C \ \CC^0_{3, 4} (U)  -\nabc_4(\D) \  U +\frac{a}{r^2} \Ga_g \frak{d}^{\leq 1} U,
 \eeaa
 
 \item the following commutator with $\nabc_a$:
\beaa
\, [Q, \nabc_a] U_{bc}&=&2\eta_a \nabc_3\nabc_3 U_{bc} -   \trchb\, \nabc_a \nabc_3 U_{bc}- \atrchb\, \dual \nabc_a \nabc_3 U_{bc}\\
&&+\left(-\nabc_a ( \C )+\nabc_3\eta_a+(\C-\frac  1 2   \trchb)\,\eta_a-\frac 1 2 \atrchb\, \dual \eta_a\right) \nabc_3 U_{bc}\\
&&+\nabc_3C^0_{3,a}(U)+C^0_{3,a}(\nabc_3U)\\
&&+  \left(-\frac  \C 2   \trchb+\frac 1 2 \trchb^2-\frac 1 2 \atrchb^2\right)\, \nabc_a U_{bc}\\
&&+\left(-\frac \C 2 \atrchb+ \trchb\atrchb \right)\, \dual \nabc_a  U_{bc}\\
&&-\frac  1 2   \trchb\,C^0_{3,a}(U)-\frac 1 2 \atrchb\, \dual C^0_{3,a}(U)  + \C \  C^0_{3,a}(U)  -\nabc_a (\D)  U_{bc}\\
&&+\nabc_3 \left( \frac{a}{r} \Ga_g U+ \Ga_g \frak{d}^{\leq 1} U\right),
  \eeaa
  
 \item the following commutator with $\DDbc \c$:
   \beaa
\, [Q, \DDbc \c] U  &=&2\ov{H} \c \nabc_3\nabc_3 U- \ov{\tr\Xb}\,\ov{\DDc} \c  \nabc_3U+\frac 1 2\ov{\tr\Xb} (\ov{\tr\Xb}-\C)\,  \ov{\DDc} \c U \\
&&+\left(\nabc_3\ov{H} +(- (s-2)\ov{\tr\Xb}+\C)\,  \ov{H}-\DDbc ( \C )\right)\c \nabc_3U\\
&&+\left(\frac 1 2 (s-2) \ov{\tr\Xb}\left(-  \nabc_3\ov{H}+ (\ov{\tr\Xb}-\C)\,  \ov{H}\right)-\DDbc (\D) \right)\c U\\
&&+\nabc_3 \left( \frac{a}{r} \Ga_g U+ \Ga_g \frak{d}^{\leq 1} U\right).
 \eeaa
 \end{itemize}

Let $F$ be a one form of conformal type $s$. Then we have
\begin{itemize}
\item the following commutator with $\DDc \hot$:
\beaa
 \, [Q, \DDc \hot] F  &=&2H \hot \nabc_3\nabc_3 F-  \tr \Xb  \DDc\hot \nabc_3F+\frac 1 2 (\tr \Xb)(\tr\Xb-\C)  \DDc\hot F \\
 &&+\left(\nabc_3H+(- (s+1) \tr \Xb +\C )H -\DDc ( \C )\right) \hot \nabc_3 F\\
 &&+\left( \frac 1 2 (s+1) \tr \Xb \left(-\nabc_3H+(\tr\Xb-\C)H\right)  -\DDc (\D) \right)\hot F \\
 &&+\nabc_3 \left( \frac{a}{r} \Ga_g F+ \Ga_g \frak{d}^{\leq 1} F\right).
 \eeaa
 \end{itemize}
\end{proposition}

\begin{proof} 
We compute 
\beaa
\, [Q, \nabc_3] U&=&(\nabc_3\nabc_3  + \C \  \nabc_3 +\D \  )\nabc_3U - \nabc_3(\nabc_3\nabc_3U  + \C \  \nabc_3U +\D \  U)\\
&=& (- \nabc_3 \C) \  \nabc_3U+( - \nabc_3\D ) \  U
\eeaa
as stated.
 Similarly, we compute
 \beaa
 [Q, \nabc_4] U  &=&\nabc_3([\nabc_3, \nabc_4]U)+ [\nabc_3,\nabc_4]\nabc_3 U \\
 &&+ \C \  [\nabc_3, \nabc_4]U  -\nabc_4( \C )\  \nabc_3U  -\nabc_4(\D) \  U.
 \eeaa
Recall from Lemma \ref{commutator-nab-3-nab-4}, 
\beaa
 [\nabc_3, \nabc_4] U   &=& 2((\eta-\etab ) \c \nabc) U +\CC^0_{3, 4} (U)
 \eeaa   
 where 
 \bea\label{expression-C-0-3-4}
 \CC^0_{3, 4} (U)=((s-2)P+(s+2)\ov{P}  -2s\eta\c\etab)U - 4  \eta \hot (\etab \c U)+4 \etab \hot (\eta \c U)+\lot
 \eea
  is a zero-th order term in $U$. 
In particular, since $\nabc_3U$ is conformal of type $s-1$, we have 
\beaa
[\nab^{(c)}_3, \nab^{(c)}_4] \nabc_3U  &=& 2((\eta-\etab ) \c \nabc) \nabc_3U +\CC^0_{3, 4} (\nabc_3U).
\eeaa
 On the other hand we compute
 \beaa
  \nab^{(c)}_3([\nab^{(c)}_3, \nab^{(c)}_4]U) &=&2\left((\eta-\etab ) \c \nabc_3\nabc\right) U+2\left(\nabc_3(\eta-\etab ) \c \nabc\right) U +\nabc_3(\CC^0_{3, 4} (U))+\lot
 \eeaa
 Using Lemma \ref{commtator-3-a}, we have 
 \beaa
 [\nabc_3, \nabc_a] U_{bc} &=& -\frac  1 2   \trchb\, \nabc_a U_{bc}-\frac 1 2 \atrchb\, \dual \nabc_a  U_{bc}+\eta_a \nabc_3 U_{bc}+C^0_{3,a}(U)
\eeaa
where
\bea\label{expression-C-0-3-a}
\begin{split}
C^0_{3,a}(U)&= -\frac  1 2   \trchb\, \Big(s(\eta_a) U_{bc}+\eta_bU_{ac}+\eta_c U_{ab}-\de_{a b}(\eta \c U)_c-\de_{a c}(\eta \c U)_b \Big)\\
&-\frac 1 2 \atrchb\, \Big(s (\dual\eta_a) U_{bc} +\eta_b \dual U_{ac}+\eta_c \dual U_{ab}- \in_{a b}(\eta \c  U)_c- \in_{a c}(\eta \c  U)_b \Big)\\
&+ \frac{a}{r} \Ga_g U+ \Ga_g \frak{d}^{\leq 1} U
\end{split}
\eea
We can therefore write
 \beaa
  \nab^{(c)}_3([\nab^{(c)}_3, \nab^{(c)}_4]U) &=&2(\eta-\etab ) \c \left(\nabc\nabc_3U-\frac  1 2   \trchb\, \nabc U-\frac 1 2 \atrchb\, \dual \nabc  U+\eta \nabc_3 U+C^0_{3,a}(U)\right) \\
  &&+2\left(\nabc_3(\eta-\etab ) \c \nabc\right) U +\nabc_3(\CC^0_{3, 4} (U))+ \frac{a}{r^2}\left(\Bb+\frac{a}{r} \Ga_g\right) U+\frac{a}{r^2} \Ga_g \frak{d}^{\leq 1} U.
 \eeaa
 Putting the above together, we obtain the desired formula.

We compute 
\beaa
\, [Q, \nabc_a] U_{bc}&=& \nabc_3([\nabc_3, \nabc_a]U_{bc})+ [\nabc_3,\nabc_a]\nabc_3 U_{bc} + \C \  [\nabc_3, \nabc_a]U_{bc} \\
  && -\nabc_a ( \C )  \nabc_3U_{bc}  -\nabc_a (\D)  U_{bc}.
  \eeaa
  Using 
  \beaa
 [\nabc_3, \nabc_a] U_{bc} &=& -\frac  1 2   \trchb\, \nabc_a U_{bc}-\frac 1 2 \atrchb\, \dual \nabc_a  U_{bc}+\eta_a \nabc_3 U_{bc}+C^0_{3,a}(U)+ \frac{a}{r} \Ga_g U+ \Ga_g \frak{d}^{\leq 1} U
\eeaa 
we have 
\beaa
[\nabc_3,\nabc_a]\nabc_3 U_{bc} &=&-\frac  1 2   \trchb\, \nabc_a \nabc_3U_{bc}-\frac 1 2 \atrchb\, \dual \nabc_a  \nabc_3U_{bc}+\eta_a \nabc_3 \nabc_3 U_{bc}\\
&&+C^0_{3,a}(\nabc_3U) +\frac{a}{r} \Ga_g \nabc_3U+ \Ga_g \frak{d}^{\leq 1} \nabc_3U
\eeaa
and 
\beaa
&&\nabc_3([\nabc_3, \nabc_a]U_{bc})\\
&=& -\frac  1 2   \trchb\, \nabc_3\nabc_a U_{bc}-\frac  1 2   \nabc_3\trchb\, \nabc_a U_{bc}\\
&&-\frac 1 2 \atrchb\, \nabc_3\dual \nabc_a  U_{bc}-\frac 1 2 \nabc_3\atrchb\, \dual \nabc_a  U_{bc}\\
&&+\eta_a \nabc_3\nabc_3 U_{bc}+\nabc_3\eta_a \nabc_3 U_{bc}+\nabc_3C^0_{3,a}(U)+\nabc_3 \left( \frac{a}{r} \Ga_g U+ \Ga_g \frak{d}^{\leq 1} U\right)\\
&=& -\frac  1 2   \trchb\, (\nabc_a \nabc_3 U_{bc}-\frac  1 2   \trchb\, \nabc_a U_{bc}-\frac 1 2 \atrchb\, \dual \nabc_a  U_{bc}+\eta_a \nabc_3 U_{bc}+C^0_{3,a}(U))\\
&&-\frac  1 2   \left(-\frac 1 2 \trchb^2+\frac 1 2 \atrchb^2\right)\, \nabc_a U_{bc}\\
&&-\frac 1 2 \atrchb\, \dual (\nabc_a \nabc_3 U_{bc}-\frac  1 2   \trchb\, \nabc_a U_{bc}-\frac 1 2 \atrchb\, \dual \nabc_a  U_{bc}+\eta_a \nabc_3 U_{bc}+C^0_{3,a}(U))\\
&&-\frac 1 2 (-\trchb\atrchb)\, \dual \nabc_a  U_{bc}\\
&&+\eta_a \nabc_3\nabc_3 U_{bc}+\nabc_3\eta_a \nabc_3 U_{bc}+\nabc_3C^0_{3,a}(U)+\nabc_3 \left( \frac{a}{r} \Ga_g U+ \Ga_g \frak{d}^{\leq 1} U\right).
\eeaa
Putting the above together, we obtain the stated expression.

We compute  
\beaa
\, [Q, \DDbc \c] U  &=&\nabc_3([\nabc_3, \DDbc \c]U)+ [\nabc_3,\DDbc \c]\nabc_3 U + \C \  [\nabc_3, \DDbc \c]U \\
  && -\DDbc ( \C )\c  \nabc_3U  -\DDbc (\D) \c  U.
 \eeaa
 Recall that for a two tensor $U$ of conformal type $s$:
  \beaa
 [\nabc_3 , \ov{\DDc} \c] U &=&- \frac 1 2\ov{\tr\Xb}\, \left( \ov{\DDc} \c U + (s-2) \ov{H}\c U\right) +\ov{H} \c \nabc_3 U+ \frac{a}{r} \Ga_g U+ \Ga_g \frak{d}^{\leq 1} U.
  \eeaa
  In particular, since $\nabc_3U$ is conformal of type $s-1$, we have 
  \beaa
 [\nabc_3 , \ov{\DDc} \c] \nabc_3 U &=&- \frac 1 2\ov{\tr\Xb}\, \left( \ov{\DDc} \c \nabc_3 U + (s-3) \ov{H}\c \nabc_3 U\right) +\ov{H} \c \nabc_3 \nabc_3 U\\
 &&+ \frac{a}{r} \Ga_g \nabc_3U+ \Ga_g \frak{d}^{\leq 1} \nabc_3U.
  \eeaa
  On the other hand we compute 
  \beaa
 \nabc_3( [\nabc_3 , \ov{\DDc} \c] U )&=&- \frac 1 2\nabc_3\ov{\tr\Xb}\, \left( \ov{\DDc} \c U + (s-2) \ov{H}\c U\right)\\
 &&- \frac 1 2\ov{\tr\Xb}\, \left( \nabc_3\ov{\DDc} \c U + (s-2) \nabc_3\ov{H}\c U+ (s-2) \ov{H}\c \nabc_3U\right)\\
 && +\nabc_3\ov{H} \c \nabc_3 U +\ov{H} \c \nabc_3\nabc_3 U+\nabc_3( \frac{a}{r} \Ga_g U+ \Ga_g \frak{d}^{\leq 1} U)\\
 &=& \frac 1 2(\ov{\tr\Xb})^2\, \left( \ov{\DDc} \c U + (s-2) \ov{H}\c U\right)\\
 &&- \frac 1 2\ov{\tr\Xb}\, \left(\ov{\DDc} \c  \nabc_3U  + (s-2) \nabc_3\ov{H}\c U+ (s-1) \ov{H}\c \nabc_3U\right)\\
 && +\nabc_3\ov{H} \c \nabc_3 U +\ov{H} \c \nabc_3\nabc_3 U+\nabc_3\left( \frac{a}{r} \Ga_g U+ \Ga_g \frak{d}^{\leq 1} U\right).
  \eeaa
  Putting the above together we obtain the desired formula.

  We compute  
\beaa
\, [Q, \DDc \hot] F  &=&\nabc_3([\nabc_3, \DDc \hot]F)+ [\nabc_3,\DDc \hot]\nabc_3 F + \C \  [\nabc_3, \DDc \hot]F \\
  && -\DDc ( \C )\hot  \nabc_3F  -\DDc (\D) \hot  F.
 \eeaa
 Recall that for a one form of conformal type $s$:
  \beaa
 \, [\nabc_3 , \DDc \hot ]F &=&- \frac 1 2 \tr \Xb \left( \DDc\hot F + (s+1)H \hot F \right)  + H \hot \nabc_3 F.
 \eeaa
In particular, since $\nabc_3F$ is conformal of type $s-1$, we have 
  \beaa
 \, [\nabc_3 , \DDc \hot ]\nabc_3F &=&- \frac 1 2 \tr \Xb \left( \DDc\hot \nabc_3F + (s)H \hot \nabc_3F \right)  + H \hot \nabc_3 \nabc_3F.
 \eeaa
 On the other hand we compute 
 \beaa
 \nabc_3(\, [\nabc_3 , \DDc \hot ]F )&=&- \frac 1 2 \nabc_3\tr \Xb \left( \DDc\hot F + (s+1)H \hot F \right)\\
 &&- \frac 1 2 \tr \Xb \left( \nabc_3\DDc\hot F + (s+1)\nabc_3H \hot F+ (s+1)H \hot \nabc_3F \right)  \\
 &&+ \nabc_3H \hot \nabc_3 F+ H \hot \nabc_3\nabc_3 F\\
 &=& \frac 1 2 (\tr \Xb)^2 \left( \DDc\hot F + (s+1)H \hot F \right)\\
 &&- \frac 1 2 \tr \Xb \left( \DDc\hot \nabc_3F   + (s+1)\nabc_3H \hot F+ (s+2)H \hot \nabc_3F \right)  \\
 &&+ \nabc_3H \hot \nabc_3 F+ H \hot \nabc_3\nabc_3 F.
 \eeaa
 Putting the above together we obtain the desired formula. 
\end{proof}

 We now want to compute the commutator between $Q$ and $\LL$, and prove equation \eqref{final-commutator}.  
Using \eqref{Teukolsky-operator}, we separate the commutator into 
\beaa
[Q, \LL]A&=& I + J+K+L+M+N
\eeaa
where 
\beaa
I&=& - [Q, \nabc_4\nabc_3]A, \qquad\qquad\qquad\quad\,\,\, J= \frac{1}{2}[Q, \DDc\hot\DDbc \c ]A,\\
K&=&  \left[Q, \left(- \frac 1 2 \tr X -2\ov{\tr X} \right)\nabc_3\right]A, \qquad L= [Q,-\frac{1}{2}\tr\Xb \nabc_4]A,\\
M&=&[Q, \left( 4H+\Hb +\ov{\Hb} \right)\c \nabc] A, \qquad\quad\,\, N= [Q, \left(-\ov{\tr X} \tr \Xb +2\ov{P}\right)] A+ 2[Q, H   \hot \ov{\Hb} \c ]A.
\eeaa


\subsubsection{Expression for $I$}


We have
\beaa
 I&=& -[Q, \nabc_4] \nabc_3A-\nabc_4( [Q,  \nabc_3]A).
\eeaa
From Proposition \ref{commutation-general-Q}, we deduce
 \beaa
 [Q, \nabc_4] \nabc_3A  &=&4(\eta-\etab ) \c \nabc\nabc_3\nabc_3A \\
  &&+2\left(\nabc_3(\eta-\etab ) +\left(\frac 1 2 \trchb +\C\right)(\eta-\etab)\right)\c \nabc \nabc_3A \\
  &&- \atrchb(\eta-\etab ) \c  \dual \nabc  \nabc_3A\\
  &&+\nabc_3(\CC^0_{3, 4} (\nabc_3A))+\CC^0_{3, 4} (\nabc_3\nabc_3A)\\
  &&+\left(2(\eta-\etab ) \c \eta -\nabc_4( \C )\right)\nabc_3 \nabc_3A \\
    &&+2(\eta-\etab ) \c \CC^0_{3,a}(\nabc_3A)+ \C \ \CC^0_{3, 4} (\nabc_3A)  -\nabc_4(\D) \  \nabc_3A\\
    &&+ \frac{a}{r^2} \Ga_g \frak{d}^{\leq 1} \nabc_3 A.
 \eeaa
We also deduce
 \beaa
 \nabc_4( [Q,  \nabc_3]A)&=& \nabc_4( (- \nabc_3 \C) \  \nabc_3A+( - \nabc_3\D ) \  A)\\
 &=& (- \nabc_3 \C) \  \nabc_4\nabc_3A+ (- \nabc_4\nabc_3 \C) \  \nabc_3A+( - \nabc_3\D ) \  \nabc_4A\\
 &&+( - \nabc_4\nabc_3\D ) \  A.
 \eeaa
 
 We therefore obtain
\bea
\begin{split}
 I&= -4(\eta-\etab ) \c \nabc\nabc_3\nabc_3A +I_{43} \  \nabc_4\nabc_3A\\
 &+\tilde{I}_{33}(A) +I_{a3}\c \nabc \nabc_3A +I_{\dual a 3} \c  \dual \nabc  \nabc_3A+I_4 \  \nabc_4A+\tilde{I}_3(A)+I_0 \  A\\
 &+ \frac{a}{r^2} \Ga_g \frak{d}^{\leq 1} \nabc_3 A
 \end{split}
\eea
where 
\beaa
I_{43}&=& \nabc_3 \C, \\
\tilde{I}_{33}(A)&=& -\nabc_3(\CC^0_{3, 4} (\nabc_3A))-\CC^0_{3, 4} (\nabc_3\nabc_3A)-\left(2(\eta-\etab ) \c \eta -\nabc_4( \C )\right)\nabc_3 \nabc_3A, \\
I_{a3}&=& -2\left(\nabc_3(\eta-\etab ) +\left(\frac 1 2 \trchb +\C\right)(\eta-\etab)\right), \\
I_{\dual a 3}&=&  \atrchb(\eta-\etab ), \\
I_4&=&  \nabc_3\D, \\
\tilde{I}_3(A)&=& -2(\eta-\etab ) \c \CC^0_{3,a}(\nabc_3A)- \C \ \CC^0_{3, 4} (\nabc_3A)  +\left(\nabc_4(\D)+ \nabc_4\nabc_3 \C \right)\  \nabc_3A, \\
I_0&=&  \nabc_4\nabc_3\D.
\eeaa

We now compute $\tilde{I}_{33}(A)$ and $\tilde{I}_{3}(A)$. Using \eqref{expression-C-0-3-4} we write
\beaa
 \CC^0_{3, 4} (\nabc_3A)&=&(-P+3\ov{P}  -2\eta\c\etab)\nabc_3A - 4  \eta \hot (\etab \c \nabc_3A)+4 \etab \hot (\eta \c \nabc_3A)\\
  \CC^0_{3, 4} (\nabc_3\nabc_3A)&=&(-2P+2\ov{P} )\nabc_3\nabc_3A - 4  \eta \hot (\etab \c \nabc_3\nabc_3A)+4 \etab \hot (\eta \c \nabc_3\nabc_3A).
 \eeaa
We therefore have
\beaa
\tilde{I}_{33}(A)&=& -\nabc_3(-P+3\ov{P}  -2\eta\c\etab)\nabc_3A-(-P+3\ov{P}  -2\eta\c\etab)\nabc_3\nabc_3A + 4  \nabc_3\eta \hot (\etab \c \nabc_3A)\\
&&+ 4  \eta \hot (\nabc_3\etab \c \nabc_3A)+ 4  \eta \hot (\etab \c \nabc_3\nabc_3A)\\
&&-4 \nabc_3\etab \hot (\eta \c \nabc_3A)-4 \etab \hot (\nabc_3\eta \c \nabc_3A)-4 \etab \hot (\eta \c \nabc_3\nabc_3A)\\
&&+(2P-2\ov{P} )\nabc_3\nabc_3A + 4  \eta \hot (\etab \c \nabc_3\nabc_3A)-4 \etab \hot (\eta \c \nabc_3\nabc_3A)\\
&&-\left(2(\eta-\etab ) \c \eta -\nabc_4( \C )\right)\nabc_3 \nabc_3A. 
\eeaa

To simplify $\tilde{I}_3(A)$ we make use of the following lemma.
\begin{lemma} 
For a one form $\xi$ and a two tensor $U$, we have 
\beaa
\xi \c C^0_{3,a}(U)&=&-\frac  s 2   \left((\trchb \xi -\atrchb \dual \xi )\c \eta\right)U  -\frac  1 2   \trchb\, \Big(\eta \hot (\xi \c U)-\xi \hot (\eta \c U) \Big)\\
&&-\frac 1 2 \atrchb\, \Big(-\eta \hot (\dual \xi \c U)+\dual  \xi \hot(\eta \c  U) \Big)+r^{-1} (\Ga_g)^2 \frak{d}^{\leq 1} U.
\eeaa
\end{lemma}
\begin{proof} 
Using \eqref{expression-C-0-3-a}, we have
\beaa
(\xi \c C^0_{3,a}(U))_{bc}&=&\xi_a\Bigg[ -\frac  1 2   \trchb\, \Big(s(\eta_a) U_{bc}+\eta_bU_{ac}+\eta_c U_{ab}-\de_{a b}(\eta \c U)_c-\de_{a c}(\eta \c U)_b \Big)\\
&&-\frac 1 2 \atrchb\, \Big(s (\dual\eta_a) U_{bc} +\eta_b \dual U_{ac}+\eta_c \dual U_{ab}- \in_{a b}(\eta \c  U)_c- \in_{a c}(\eta \c  U)_b \Big)\\
&&+ \frac{a}{r} \Ga_g U+ \Ga_g \frak{d}^{\leq 1} U\Bigg]\\
&=&-\frac  s 2   \left(\trchb (\xi \c \eta)+\atrchb (\xi \c \dual \eta) \right)U_{bc} \\
&& -\frac  1 2   \trchb\, \Big(\xi_a\eta_bU_{ac}+\xi_a\eta_c U_{ab}-\xi_b(\eta \c U)_c-\xi_c(\eta \c U)_b \Big)\\
&&-\frac 1 2 \atrchb\, \Big(\xi_a\eta_b \dual U_{ac}+\xi_a\eta_c \dual U_{ab}+\dual  \xi_b(\eta \c  U)_c+\dual \xi_c(\eta \c  U)_b \Big)+r^{-1} (\Ga_g)^2 \frak{d}^{\leq 1} U.
\eeaa
Observe that
\beaa
\xi_a\eta_bU_{ac}+\xi_a\eta_c U_{ab}-\xi_b(\eta \c U)_c-\xi_c(\eta \c U)_b&=& \eta_b(\xi \c U)_{c}+\eta_c (\xi \c U)_{b}-\xi_b(\eta \c U)_c-\xi_c(\eta \c U)_b\\
&=& \eta \hot (\xi \c U)_{bc}-\xi \hot (\eta \c U)_{bc}
\eeaa
and
\beaa
&&\xi_a\eta_b \dual U_{ac}+\xi_a\eta_c \dual U_{ab}+\dual  \xi_b(\eta \c  U)_c+\dual \xi_c(\eta \c  U)_b\\
&=& \eta_b(\xi \c \dual U)_{c}+\eta_c (\xi \c \dual U)_{b}+\dual  \xi_b(\eta \c  U)_c+\dual \xi_c(\eta \c  U)_b\\
&=& -\eta_b(\dual \xi \c U)_{c}-\eta_c (\dual \xi \c U)_{b}+\dual  \xi_b(\eta \c  U)_c+\dual \xi_c(\eta \c  U)_b\\
&=& -\eta \hot (\dual \xi \c U)_{bc}+\dual  \xi \hot(\eta \c  U)_{bc}.
\eeaa
The lemma easily follows. 
\end{proof}

Using the lemma, we have
\beaa
\tilde{I}_3(A)&=&    \left((\trchb (\eta-\etab ) -\atrchb \dual (\eta-\etab ) )\c \eta\right)\nabc_3A \\
&& +   \trchb\, \Big(\eta \hot ((\eta-\etab ) \c \nabc_3A)-(\eta-\etab ) \hot (\eta \c \nabc_3A) \Big)\\
&&+ \atrchb\, \Big(-\eta \hot (\dual (\eta-\etab ) \c \nabc_3A)+\dual  (\eta-\etab ) \hot(\eta \c  \nabc_3A) \Big)\\
&&- \C \ ((-P+3\ov{P}  -2\eta\c\etab)\nabc_3A - 4  \eta \hot (\etab \c \nabc_3A)+4 \etab \hot (\eta \c \nabc_3A)) \\
&& +\left(\nabc_4(\D)+ \nabc_4\nabc_3 \C \right)\  \nabc_3A.
\eeaa

Therefore we have
\beaa
\tilde{I}_{33}(A)+\tilde{I}_3(A)&=& I_{33} \nabc_3 \nabc_3A+ I^a_{33}(A)+ I_3 \nabc_3A+I^a_3(A)
\eeaa
where
\beaa
I_{33}&=& 3P-5\ov{P} +4\eta\c\etab-2|\eta|^2  +\nabc_4( \C ),\\
I_{3}&=&-\nabc_3(-P+3\ov{P}  -2\eta\c\etab)+(\trchb (\eta-\etab ) -\atrchb \dual (\eta-\etab ) )\c \eta- \C \ (-P+3\ov{P}  -2\eta\c\etab)\\
&&+\nabc_4(\D)+ \nabc_4\nabc_3 \C,
\eeaa
and
\beaa
I^a_{33}(A)&=& 8  \eta \hot (\etab \c \nabc_3\nabc_3A)-8 \etab \hot (\eta \c \nabc_3\nabc_3A),\\
I^a_3(A)&=& 4  \nabc_3\eta \hot (\etab \c \nabc_3A)+ 4  \eta \hot (\nabc_3\etab \c \nabc_3A)\\
&&-4 \nabc_3\etab \hot (\eta \c \nabc_3A)-4 \etab \hot (\nabc_3\eta \c \nabc_3A)\\
&&  +\trchb\, \Big(\eta \hot ((\eta-\etab ) \c \nabc_3A)-(\eta-\etab ) \hot (\eta \c \nabc_3A) \Big)\\
&&+ \atrchb\, \Big(-\eta \hot (\dual (\eta-\etab ) \c \nabc_3A)+\dual  (\eta-\etab ) \hot(\eta \c  \nabc_3A) \Big)\\
&& - \C \ ( - 4  \eta \hot (\etab \c \nabc_3A)+4 \etab \hot (\eta \c \nabc_3A)).
\eeaa
We finally write 
\bea
\begin{split}
 I&= -4(\eta-\etab ) \c \nabc\nabc_3\nabc_3A +I_{43} \  \nabc_4\nabc_3A\\
 &+I_{33} \nabc_3 \nabc_3A+ I^a_{33}(A) +I_{a3}\c \nabc \nabc_3A +I_{\dual a 3} \c  \dual \nabc  \nabc_3A\\
 &+I_4 \  \nabc_4A+ I_3 \nabc_3A+I^a_3(A)+I_0 \  A + \frac{a}{r^2} \Ga_g \frak{d}^{\leq 1} \nabc_3 A+\lot
 \end{split}
\eea
with the above expressions.


\subsubsection{Expression for $J$}


We have
\beaa
J&=& \frac 1 2 [Q, \DDc\hot] (\DDbc \c A)+\frac 1 2 \DDc\hot ([ Q, \DDbc \c ]A).
\eeaa
From Proposition \ref{commutation-general-Q}, we deduce, recalling that $\DDbc \c A$ is a one form of conformal type 2:
\beaa
[Q, \DDc\hot] (\DDbc \c A)&=&2H \hot \nabc_3\nabc_3 \DDbc \c A-  \tr \Xb  \DDc\hot \nabc_3\DDbc \c A\\
&&+\frac 1 2 (\tr \Xb)(\tr\Xb-\C)  \DDc\hot \DDbc \c A \\
 &&+\left(\nabc_3H+(- 3 \tr \Xb +\C )H -\DDc ( \C )\right) \hot \nabc_3 \DDbc \c A\\
 &&+\left( \frac 3 2  \tr \Xb \left(-\nabc_3H+(\tr\Xb-\C)H\right)  -\DDc (\D) \right)\hot \DDbc \c A.
\eeaa
We write
  \beaa
 \nabc_3  \ov{\DDc} \c A &=&\ov{\DDc} \c \nabc_3 A- \frac 1 2\ov{\tr\Xb}\, \left( \ov{\DDc} \c A \right) +\ov{H} \c \nabc_3 A, 
 \eeaa
 \beaa
 \DDc\hot \nabc_3\DDbc \c A&=& \nabc_3\DDc\hot \DDbc \c A-[\nabc_3, \DDc\hot]\DDbc \c A\\
 &=& \nabc_3\DDc\hot \DDbc \c A+ \frac 1 2 \tr \Xb \DDc\hot \DDbc \c A \\
 &&- H \hot \nabc_3 \DDbc \c A+  \frac 3 2 \tr \Xb H \hot \DDbc \c A  \\
  &=& \nabc_3\DDc\hot \DDbc \c A+ \frac 1 2 \tr \Xb \DDc\hot \DDbc \c A - H \hot (\ov{\DDc} \c \nabc_3 A)\\
  &&- H \hot (\ov{H} \c \nabc_3 A)+  \left(\frac 3 2 \tr \Xb +\frac 1 2\ov{\tr\Xb}\right)H \hot \DDbc \c A,   
  \eeaa
  and 
 \beaa
 \nabc_3\nabc_3 \DDbc \c A&=& \nabc_3 \DDbc \c\nabc_3 A+\nabc_3([\nabc_3,  \DDbc \c] A)\\
 &=& \DDbc \c\nabc_3 \nabc_3 A+[\nabc_3 ,\DDbc \c]\nabc_3 A+\nabc_3([\nabc_3,  \DDbc \c] A)\\
  &=& \DDbc \c\nabc_3 \nabc_3 A\\
  &&- \frac 1 2\ov{\tr\Xb}\, \left( \ov{\DDc} \c \nabc_3 A - \ov{H}\c \nabc_3 A\right) +\ov{H} \c \nabc_3 \nabc_3 A\\
  &&+\frac 1 2(\ov{\tr\Xb})^2\, \left( \ov{\DDc} \c A \right)- \frac 1 2\ov{\tr\Xb}\, \left(\ov{\DDc} \c  \nabc_3A  +  \ov{H}\c \nabc_3A\right)\\
 && +\nabc_3\ov{H} \c \nabc_3 A +\ov{H} \c \nabc_3\nabc_3 A\\
 &=& \DDbc \c\nabc_3 \nabc_3 A +2\ov{H} \c \nabc_3 \nabc_3 A- \ov{\tr\Xb}\,  \ov{\DDc} \c \nabc_3 A \\
  &&+\frac 1 2(\ov{\tr\Xb})^2\,  \ov{\DDc} \c A  +\nabc_3\ov{H} \c \nabc_3 A, 
 \eeaa
where we used the intermediate computations in Proposition \ref{commutation-general-Q}. Putting the above together we obtain 
 \beaa
[Q, \DDc\hot] (\DDbc \c A)&=&-  \tr \Xb  \nabc_3\DDc\hot \DDbc \c A+2H \hot \DDbc \c\nabc_3 \nabc_3 A \\
&&-\frac \C 2 (\tr \Xb)  \DDc\hot \DDbc \c A+4H \hot(\ov{H} \c \nabc_3 \nabc_3 A)\\
&&+\left(\nabc_3H+(-2\tr\Xb- 2\ov{\tr\Xb}+\C)\, H-\DDc ( \C )\right) \hot (\ov{\DDc} \c \nabc_3 A) \\
  &&+2H \hot (\nabc_3\ov{H} \c \nabc_3 A) \\
 &&+\left(\nabc_3H+(- 2 \tr \Xb +\C )H -\DDc ( \C )\right) \hot (\ov{H} \c \nabc_3 A )\\
 &&+\Bigg( \left(\frac 3 2  \tr \Xb +\frac 1 2 \ov{\tr\Xb}\right)\left(-\nabc_3H-\C H\right)+\left(\tr\Xb \ov{\tr\Xb}+(\ov{\tr\Xb})^2\right)\,  H\\
 &&+\frac 1 2\ov{\tr\Xb}\DDc ( \C )  -\DDc (\D) \Bigg)\hot \DDbc \c A.
\eeaa
Using that
\beaa
F \hot (\DDbc \c U)&=& (F\c \DDbc) U=2 F  \c\nabc U, \\
E \hot ( \ov{E} \c U)&=& ( E \c \ov{E}) \ U,
\eeaa
the above becomes
 \beaa
[Q, \DDc\hot] (\DDbc \c A)&=&-  \tr \Xb  \nabc_3\DDc\hot \DDbc \c A+4H \c \nabc \nabc_3 \nabc_3 A \\
&&-\frac \C 2 (\tr \Xb)  \DDc\hot \DDbc \c A+4 (H \c \ov{H} ) \nabc_3 \nabc_3 A\\
&&+2\left(\nabc_3H+(-2\tr\Xb- 2\ov{\tr\Xb}+\C)\, H-\DDc ( \C )\right) \c \nabc  \nabc_3 A \\
  &&+2H \hot (\nabc_3\ov{H} \c \nabc_3 A) \\
 &&+\left(\nabc_3H+(- 2 \tr \Xb +\C )H -\DDc ( \C )\right) \hot (\ov{H} \c \nabc_3 A )\\
 &&+2\Bigg( \left(\frac 3 2  \tr \Xb +\frac 1 2 \ov{\tr\Xb}\right)\left(-\nabc_3H-\C H\right)+\left(\tr\Xb \ov{\tr\Xb}+(\ov{\tr\Xb})^2\right)\,  H\\
 &&+\frac 1 2\ov{\tr\Xb}\DDc ( \C )  -\DDc (\D) \Bigg)\c \nabc A.
\eeaa

From Proposition \ref{commutation-general-Q}, we deduce
\beaa
\DDc\hot ([ Q, \DDbc \c ]A)&=& \DDc\hot \Bigg(2\ov{H} \c \nabc_3\nabc_3 A- \ov{\tr\Xb}\,\ov{\DDc} \c  \nabc_3A+\frac 1 2\ov{\tr\Xb} (\ov{\tr\Xb}-\C)\,  \ov{\DDc} \c A \\
&&+\left(\nabc_3\ov{H} +\C\,  \ov{H}-\DDbc ( \C )\right)\c \nabc_3A+\left(-\DDbc (\D) \right)\c A\Bigg)
\eeaa
which gives
\beaa
\DDc\hot ([ Q, \DDbc \c ]A)&=& 2\DDc\hot (\ov{H} \c \nabc_3\nabc_3 A)- \ov{\tr\Xb}\,\DDc\hot\ov{\DDc} \c  \nabc_3A- \DDc\ov{\tr\Xb}\,\hot \ov{\DDc} \c  \nabc_3A\\
&&+\frac 1 2\ov{\tr\Xb} (\ov{\tr\Xb}-\C)\, \DDc\hot \ov{\DDc} \c A+\DDc(\frac 1 2\ov{\tr\Xb} (\ov{\tr\Xb}-\C))\hot\,  \ov{\DDc} \c A \\
&&+\DDc\hot(\left(\nabc_3\ov{H} +\C\,  \ov{H}-\DDbc ( \C )\right)\c \nabc_3A)+\DDc\hot\left(\left(-\DDbc (\D) \right)\c A\right).
\eeaa
Writing
\beaa
\DDc\hot\ov{\DDc} \c  \nabc_3A&=&\DDc\hot\nabc_3\ov{\DDc} \c  A-\DDc\hot[\nabc_3, \ov{\DDc} \c ]A \\
&=&\nabc_3\DDc\hot \DDbc \c A+ \frac 1 2 (\tr \Xb +\ov{\tr\Xb})\DDc\hot \DDbc \c A - H \hot (\ov{\DDc} \c \nabc_3 A)\\
  &&- H \hot (\ov{H} \c \nabc_3 A)+  \left(\frac 3 2 \tr \Xb +\frac 1 2\ov{\tr\Xb}\right)H \hot \DDbc \c A \\
  && +\frac 1 2\DDc (\ov{\tr\Xb})\hot(  \ov{\DDc} \c A ) -\DDc \hot (\ov{H} \c \nabc_3 A)
\eeaa
we have 
\beaa
\DDc\hot ([ Q, \DDbc \c ]A)&=&- \ov{\tr\Xb}\,\nabc_3\DDc\hot \DDbc \c A +\frac 1 2\ov{\tr\Xb} (-\tr\Xb-\C)\, \DDc\hot \ov{\DDc} \c A\\
&& + 2\DDc\hot (\ov{H} \c \nabc_3\nabc_3 A)+\ov{\tr\Xb}\DDc \hot (\ov{H} \c \nabc_3 A)+\ov{\tr\Xb} H \hot (\ov{H} \c \nabc_3 A)\\
&&+\left(\ov{\tr\Xb} H - \DDc\ov{\tr\Xb}\right)\,\hot (\ov{\DDc} \c  \nabc_3A)\\
&+ & \left(\left(-\frac 3 2 \tr \Xb \ov{\tr\Xb}-\frac 1 2\ov{\tr\Xb}^2\right)H +\frac 1 2\ov{\tr\Xb}\DDc (\ov{\tr\Xb})-\frac 1 2 \DDc(\ov{\tr\Xb} \C)\right)\hot \DDbc \c A \\
&&+\DDc\hot((\nabc_3\ov{H} +\C\,  \ov{H}-\DDbc ( \C ))\c \nabc_3A)-\DDc\hot(\DDbc (\D) \c A).
\eeaa
Using Leibniz rules, the above simplifies to
\beaa
&&\DDc\hot ([ Q, \DDbc \c ]A)\\
&=&- \ov{\tr\Xb}\,\nabc_3\DDc\hot \DDbc \c A +  4 \ov{H} \c \nabc \nabc_3\nabc_3 A\\
&&+\frac 1 2\ov{\tr\Xb} (-\tr\Xb-\C)\, \DDc\hot \ov{\DDc} \c A+ 2(\DDc\c\ov{H} )\nabc_3\nabc_3 A\\
&&+\left(4\ov{\tr\Xb}\eta  - 2\DDc\ov{\tr\Xb}+2(\nabc_3\ov{H} +\C\,  \ov{H}-\DDbc ( \C ))\right)\c \nabc\nabc_3A\\
&&+\left(\ov{\tr\Xb}(\DDc \c\ov{H}+ H \c\ov{H}+\DDc\c(\nabc_3\ov{H} +\C\,  \ov{H}-\DDbc ( \C )) \right)\nabc_3 A\\
&&+ 2\left(\left(-\frac 3 2 \tr \Xb \ov{\tr\Xb}-\frac 1 2\ov{\tr\Xb}^2\right)H +\frac 1 2\ov{\tr\Xb}\DDc (\ov{\tr\Xb})-\frac 1 2 \DDc(\ov{\tr\Xb} \C)-\DDbc (\D) \right)\c \nabc  A\\
&&-(\DDc \c\DDbc (\D))  A.
\eeaa
Putting the above together we finally obtain 
\beaa
J&=&-  \left(\tr \Xb+\ov{\tr \Xb}\right)  \nabc_3\left(\frac 1 2\DDc\hot \DDbc \c A\right)\\
&&+\left(- \frac 1 2\ov{\tr\Xb} \tr\Xb-\frac \C 2 (\tr \Xb+\ov{\tr\Xb})\right)\, \left(\frac 1 2 \DDc\hot \ov{\DDc} \c A\right)\\
&&+4 \eta \c \nabc \nabc_3 \nabc_3 A +\tilde{J}_{33} \nabc_3 \nabc_3 A+ \tilde{J}_{a3} \nabc \nabc_3 A\\
&&+\tilde{J}_{3} \nabc_3 A + \tilde{J}^a_3(A)+\tilde{J}_a \nabc A+\tilde{J}_0  A\\
&& + \nabc_3\left(\frac{a}{r^2} \Ga_g \frak{d}^{\leq 1}A + \frac 1 r \Ga_g \frak{d}^{\leq 2} A\right)
\eeaa
where
\beaa
\tilde{J}_{33}&=&\DDc\c\ov{H}+2H \c\ov{H}, \\
\tilde{J}_{a3}&=&\nabc_3H+(-2\tr\Xb- 2\ov{\tr\Xb}+\C)\, H-\DDc ( \C )+2\ov{\tr\Xb}\eta  - \DDc\ov{\tr\Xb}+\nabc_3\ov{H} +\C\,  \ov{H}-\DDbc ( \C ),\\
\tilde{J}_3&=&\frac 1 2 \ov{\tr\Xb}(\DDc \c\ov{H}+ H \c\ov{H})+\frac 1 2 \DDc\c(\nabc_3\ov{H} +\C\,  \ov{H}-\DDbc ( \C )),\\
\tilde{J}^a_3(A)&=&H \hot (\nabc_3\ov{H} \c \nabc_3 A) \\
 &&+\frac 1 2 \left(\nabc_3H+(- 2 \tr \Xb +\C )H -\DDc ( \C )\right) \hot (\ov{H} \c \nabc_3 A ),\\
 \tilde{J}_a&=& \left(\frac 3 2  \tr \Xb +\frac 1 2 \ov{\tr\Xb}\right)\left(-\nabc_3H-\C H\right)+\left(\tr\Xb \ov{\tr\Xb}+(\ov{\tr\Xb})^2\right)\,  H\\
 &&+\frac 1 2\ov{\tr\Xb}\DDc ( \C )  -\DDc (\D)+\left(-\frac 3 2 \tr \Xb \ov{\tr\Xb}-\frac 1 2\ov{\tr\Xb}^2\right)H \\
 &&+\frac 1 2\ov{\tr\Xb}\DDc (\ov{\tr\Xb})-\frac 1 2 \DDc(\ov{\tr\Xb} \C)-\DDbc (\D), \\
\tilde{J}_0&=&-\frac 1 2 (\DDc \c\DDbc (\D)).
  \eeaa
Recalling the definition \eqref{Teukolsky-operator} of $\LL(A)$, we write 
\beaa
 \frac{1}{2}\DDc\hot (\DDbc \c A) &=&\nabc_4\nabc_3A+\left( \frac 1 2 \tr X +2\ov{\tr X} \right)\nabc_3A+\frac{1}{2}\tr\Xb \nabc_4A\\
&&-\left( 4H+\Hb +\ov{\Hb} \right)\c \nabc A+ \left(\ov{\tr X} \tr \Xb -2\ov{P}\right) A-  2H   \hot (\ov{\Hb} \c A)
\eeaa
which gives
\beaa
&&\nabc_3\left( \frac{1}{2}\DDc\hot (\DDbc \c A) \right)\\
&=&\nabc_3\nabc_4\nabc_3A+\left( \frac 1 2 \tr X +2\ov{\tr X} \right)\nabc_3\nabc_3A\\
&&+\nabc_3\left( \frac 1 2 \tr X +2\ov{\tr X} \right)\nabc_3A+\frac{1}{2}\tr\Xb \nabc_3\nabc_4A+\frac{1}{2}\nabc_3\tr\Xb \nabc_4A\\
&&-\left( 4H+\Hb +\ov{\Hb} \right)\c \nabc_3\nabc A-\nabc_3\left( 4H+\Hb +\ov{\Hb} \right)\c \nabc A\\
&&+ \left(\ov{\tr X} \tr \Xb -2\ov{P}\right) \nabc_3A+ \nabc_3\left(\ov{\tr X} \tr \Xb -2\ov{P}\right) A\\
&&-  2\nabc_3H   \hot (\ov{\Hb} \c A)-  2H   \hot (\nabc_3 \ov{\Hb} \c A)-  2 H   \hot (\ov{\Hb} \c \nabc_3A).
\eeaa
Writing 
\beaa
\nabc_3 \nabc_4 A&=& \nabc_4 \nabc_3 A+2(\eta-\etab ) \c \nabc A +\CC^0_{3, 4} (A),\\
\nabc_3 \nabc_4 \nabc_3 A&=& \nabc_4 \nabc_3 \nabc_3 A+2(\eta-\etab ) \c \nabc \nabc_3A +\CC^0_{3, 4} (\nabc_3 A),\\
\,  \nabc_3\nabc A    &=&\nabc \nabc_3 A- \frac  1 2   \trchb\, \nabc A-\frac 1 2 \atrchb\, \dual \nabc  A+\eta \nabc_3 A+C^0_{3,a}(A),
\eeaa
we have 
\beaa
\nabc_3\left( \frac{1}{2}\DDc\hot (\DDbc \c A) \right)&=&\nabc_4 \nabc_3 \nabc_3 A+\frac{1}{2}\tr\Xb \nabc_4 \nabc_3 A-\frac{1}{4}\tr\Xb^2 \nabc_4A\\
&&+\hat{J}_{a3}\c \nabc \nabc_3A+\hat{J}_{33}\nabc_3\nabc_3A+\hat{J}_3 \nabc_3A+\hat{J}^a_3(A)\\
&&+\hat{J}_a\c \nabc A  +\hat{J}_{\dual a }\c  \dual \nabc  A+ \hat{J}_0A+\hat{J}^a_0(A)
\eeaa
where
\beaa
\hat{J}_{a3}&=& 2(\eta-\etab ) -\left( 4H+\Hb +\ov{\Hb} \right),\\
\hat{J}_{33}&=& \frac 1 2 \tr X +2\ov{\tr X}, \\
\hat{J}_3&=&\nabc_3\left( \frac 1 2 \tr X +2\ov{\tr X} \right)+\ov{\tr X} \tr \Xb -P+\ov{P}  -2\eta\c\etab-\left( 4H+\Hb +\ov{\Hb} \right)\c \eta, \\
\hat{J}^a_3(A)&=&  - 4  \eta \hot (\etab \c \nabc_3A)+4 \etab \hot (\eta \c \nabc_3A)-  2 H   \hot (\ov{\Hb} \c \nabc_3A),\\
\hat{J}_a&=&\tr\Xb (\eta-\etab )-\nabc_3\left( 4H+\Hb +\ov{\Hb} \right) +\frac  1 2   \trchb\left( 4H+\Hb +\ov{\Hb} \right),\\
\hat{J}_{\dual a}&=&\frac 1 2 \atrchb\left( 4H+\Hb +\ov{\Hb} \right),\\
\hat{J}_0&=&\nabc_3\left(\ov{\tr X} \tr \Xb -2\ov{P}\right) +\frac{1}{2}\tr\Xb(4\ov{P}  -4\eta\c\etab)\\
&&+(\trchb \left( 4H+\Hb +\ov{\Hb} \right) -\atrchb \dual \left( 4H+\Hb +\ov{\Hb} \right) )\c \eta,\\
\hat{J}^a_0(A)&=&\frac  1 2   \trchb\, \Big(\eta \hot (\left( 4H+\Hb +\ov{\Hb} \right) \c A)-\left( 4H+\Hb +\ov{\Hb} \right) \hot (\eta \c A) \Big)\\
&&+\frac 1 2 \atrchb\, \Big(-\eta \hot (\dual \left( 4H+\Hb +\ov{\Hb} \right) \c A)+\dual  \left( 4H+\Hb +\ov{\Hb} \right) \hot(\eta \c  A) \Big)\\
&&-  2\nabc_3H   \hot (\ov{\Hb} \c A)-  2H   \hot (\nabc_3 \ov{\Hb} \c A)+\frac{1}{2}\tr\Xb ( - 4  \eta \hot (\etab \c A)+4 \etab \hot (\eta \c A)).
\eeaa
We therefore obtain
\bea
\begin{split}
J&=-  \left(\tr \Xb+\ov{\tr \Xb}\right)  \nabc_4 \nabc_3 \nabc_3 A+4 \eta \c \nabc \nabc_3 \nabc_3 A\\
&+J_{43} \nabc_4 \nabc_3 A +J_4 \nabc_4 A +J_{a3}\c \nabc \nabc_3A+J_{33} \nabc_3 \nabc_3 A\\
&+J_3 \nabc_3A+J^a_3(A)+J_a\c \nabc A +J_{\dual a }\c  \dual \nabc  A+J_0 A+J^a_0(A)\\
&+ \nabc_3\left(\frac{a}{r^2} \Ga_g \frak{d}^{\leq 1} + \frac 1 r \Ga_g \frak{d}^{\leq 2} A\right)
\end{split}
\eea
where
\beaa
J_{43}&=& - \frac 1 2\ov{\tr\Xb} \tr\Xb-\frac {\C +\tr\Xb}{2} (\tr \Xb+\ov{\tr\Xb}),\\
J_4&=&\frac{1}{4}(\tr\Xb)^2 \left(\tr \Xb\right) -\frac \C 4\tr\Xb(\tr \Xb+\ov{\tr\Xb}),
\eeaa
and
\beaa
J_{a3}&=&-  \left(\tr \Xb+\ov{\tr \Xb}\right)  \hat{J}_{a3} +\tilde{J}_{a3},\\
J_{33}&=& -  \left(\tr \Xb+\ov{\tr \Xb}\right)  \hat{J}_{33}+\tilde{J}_{33},\\
J_3&=& -  \left(\tr \Xb+\ov{\tr \Xb}\right)  \hat{J}_3 +\left(- \frac 1 2\ov{\tr\Xb} \tr\Xb-\frac \C 2 (\tr \Xb+\ov{\tr\Xb})\right)\, \left( \frac 1 2 \tr X +2\ov{\tr X} \right)+\tilde{J}_{3},\\
J^a_3(A)&=&-  \left(\tr \Xb+\ov{\tr \Xb}\right)  (\hat{J}^a_3(A)+ \tilde{J}^a_3(A),\\
J_a&=&-  \left(\tr \Xb+\ov{\tr \Xb}\right)  \hat{J}_a -\left(- \frac 1 2\ov{\tr\Xb} \tr\Xb-\frac \C 2 (\tr \Xb+\ov{\tr\Xb})\right)\, \left( 4H+\Hb +\ov{\Hb} \right)+\tilde{J}_a, \\
J_{\dual a}&=& -  \left(\tr \Xb+\ov{\tr \Xb}\right)  \hat{J}_{\dual a },\\
J_0&=& -  \left(\tr \Xb+\ov{\tr \Xb}\right)  \hat{J}_0+\left(- \frac 1 2\ov{\tr\Xb} \tr\Xb-\frac \C 2 (\tr \Xb+\ov{\tr\Xb})\right)\,  \left(\ov{\tr X} \tr \Xb -2\ov{P}\right)+\tilde{J}_0,\\
J^a_0(A)&=&-  \left(\tr \Xb+\ov{\tr \Xb}\right)  \hat{J}^a_0(A)-\left(- \frac 1 2\ov{\tr\Xb} \tr\Xb-\frac \C 2 (\tr \Xb+\ov{\tr\Xb})\right)\, (   2H   \hot (\ov{\Hb} \c A)).
\eeaa


\subsubsection{Expression for $K$}

Observe that
 \bea\label{Q-f-g}
 Q(fg) &=& Q(f) g+ fQ(g)+2\nabc_3f \nabc_3g  -\D \ fg.
 \eea 
We therefore obtain, for a scalar $\FF$:
 \beaa
[Q, \FF \ \nabc_3]A&=& (Q(\FF)-\D \ \FF ) \nabc_3A+ \FF [Q, \nabc_3]A+2\nabc_3\FF  \nabc_3\nabc_3A \\
&=& \left(2\nabc_3\FF  \right)\nabc_3\nabc_3A +(Q(\FF)-\D \ \FF - \FF \nabc_3 \C) \nabc_3A+  ( - \FF \nabc_3\D ) \  A. 
\eeaa
In particular, 
\beaa
K&=& [Q, \FF \ \nab_3^{(c)}]A \qquad \text{for} \qquad \FF= - \frac 1 2 \tr X -2\ov{\tr X}.
\eeaa
We therefore obtain\footnote{The expression for $K$ does not have error terms. }
\beaa
K&=& K_{33} \nabc_3 \nabc_3 A+K_3 \nabc_3 A +K_0 A
\eeaa
where 
\beaa
K_{33}&=& 2\nabc_3\left(- \frac 1 2 \tr X -2\ov{\tr X} \right), \\
K_3&=& Q\left(- \frac 1 2 \tr X -2\ov{\tr X} \right)-\D \ \left(- \frac 1 2 \tr X -2\ov{\tr X} \right) - \left(- \frac 1 2 \tr X -2\ov{\tr X} \right) \nabc_3 \C, \\
K_0&=& -\left(- \frac 1 2 \tr X -2\ov{\tr X} \right)\nabc_3\D.
\eeaa


\subsubsection{Expression for $L$}


Using \eqref{Q-f-g}, we obtain for a scalar $\EE$, 
 \beaa
[Q, \EE \ \nabc_4]A&=& ( Q(\EE)-\D \ \EE )\nabc_4A+ \EE [Q, \nabc_4]A+2\nabc_3\EE  \nabc_3\nabc_4A\\
&=& ( Q(\EE)-\D \ \EE )\nabc_4A+ \EE [Q, \nabc_4]A+2\nabc_3\EE  \nabc_4\nabc_3A\\
&&+2\nabc_3\EE[\nabc_3, \nabc_4]A
\eeaa
which gives
 \beaa
[Q, \EE \ \nabc_4]A&=&\left(2\nabc_3\EE  \right)\nabc_4\nabc_3A+4 \EE(\eta-\etab ) \c \nabc\nabc_3A\\
&&+ \left( Q(\EE)-\D \ \EE \right)\nab_4^{(c)}A\\
  &&+\left[2 \EE\left(\nabc_3(\eta-\etab ) +\left(\frac 1 2 \trchb +\C\right)(\eta-\etab)\right)+4\nab^{(c)}_3\EE (\eta-\etab ) \right]\c \nabc A \\
  &&-  \EE\atrchb(\eta-\etab ) \c  \dual \nabc  A\\
  &&+ \EE\nabc_3(\CC^0_{3, 4} (A))+ \EE\CC^0_{3, 4} (\nabc_3A)+ \EE\left(2(\eta-\etab ) \c \eta -\nabc_4( \C )\right)\nabc_3 A \\
    &&+2 \EE(\eta-\etab ) \c \CC^0_{3,a}(A)+ \left(2\nab^{(c)}_3\EE+ \EE\C \right)\ \CC^0_{3, 4} (A)  - \EE\nabc_4(\D) \  A + \EE  \frac{a}{r^2} \Ga_g \frak{d}^{\leq 1} A.
\eeaa
In particular, 
\beaa
L&=& [Q, \EE \ \nabc_4]A \qquad \text{for} \qquad \EE= -\frac{1}{2}\tr\Xb.
\eeaa
We therefore obtain
\bea
\begin{split}
L&= L_{43} \nabc_4\nabc_3A+L_{a3} \c \nabc\nabc_3A+ L_4 \nab_4^{(c)}A+L_a \c \nabc A+L_{\dual a} \c \dual \nabc A\\
&+ L_3 \nabc_3 A + L_3^a(A) +L_0 A +L^a_0(A) + \frac{a}{r^3} \Ga_g \frak{d}^{\leq 1} A
\end{split}
\eea
where
\beaa
L_{43}&=& 2\nabc_3\EE=\frac 1 2 (\tr\Xb)^2, \\
L_{a3}&=& 4 \EE(\eta-\etab )=-2\tr\Xb(\eta-\etab ),\\
L_4&=& Q(\EE)-\D \ \EE =-\frac 1 4 (\tr\Xb)^3+\frac \C 4 (\tr\Xb)^2,
\eeaa
and 
\beaa
L_a&=&  -\tr\Xb\left(\nabc_3(\eta-\etab ) +\left(\frac 1 2 \trchb +\C\right)(\eta-\etab)\right)+\tr\Xb^2 (\eta-\etab ),  \\
  L_{\dual a}&=&\frac{1}{2}\tr\Xb\atrchb(\eta-\etab ),  \\
  L_3&=& -\frac{1}{2}\tr\Xb\left(-P+7\ov{P}  -6\eta\c\etab  +2(\eta-\etab ) \c \eta -\nabc_4( \C )\right),\\
  L_3^a(A)&=&-\tr\Xb(- 4  \eta \hot (\etab \c \nabc_3A)+4 \etab \hot (\eta \c \nabc_3A)), \\
  L_0&=&  \tr\Xb \left((\trchb (\eta-\etab ) -\atrchb \dual (\eta-\etab ) )\c \eta\right)+ \left(\frac 1 2 \tr\Xb^2-\frac{1}{2}\tr\Xb\C \right)\ (4\ov{P}  -4\eta\c\etab)\\
  && +\frac{1}{2}\tr\Xb\nabc_4(\D) \   -\frac{1}{2}\tr\Xb\nabc_3(4\ov{P}  -4\eta\c\etab),\\
L^a_0(A)&=& -\tr\Xb\Big[ -\frac  1 2   \trchb\, \Big(\eta \hot ((\eta-\etab ) \c A)-(\eta-\etab ) \hot (\eta \c A) \Big)\\
&&-\frac 1 2 \atrchb\, \Big(-\eta \hot (\dual (\eta-\etab ) \c A)+\dual  (\eta-\etab )\hot(\eta \c  A) \Big) \Big]\\
  &&+ \left(\frac 1 2 \tr\Xb^2-\frac{1}{2}\tr\Xb\C \right)\ (- 4  \eta \hot (\etab \c A)+4 \etab \hot (\eta \c A))  \\
  && -\frac{1}{2}\tr\Xb(- 4  \nabc_3\eta \hot (\etab \c A)- 4  \eta \hot (\nabc_3\etab \c A)+4 \nabc_3\etab \hot (\eta \c A))+4 \etab \hot (\nabc_3\eta \c A)).
  \eeaa


\subsubsection{Expression for $M$}


Observe that
\bea\label{Q-F-G}
 Q(F \c G) &=& Q(F) \c G+ F \c Q(G)+2\nabc_3F \c \nabc_3G  -\D \ F \c G.
 \eea 
 We therefore obtain
 \beaa
 M&=& Q \left( 4H+\Hb +\ov{\Hb} \right)\c \nabc A+\left( 4H+\Hb +\ov{\Hb} \right) \c [Q, \nabc ]A\\
 &&+ 2 \nabc_3 \left( 4H+\Hb +\ov{\Hb} \right) \c \nabc_3 \nabc A - \D \left( 4H+\Hb +\ov{\Hb} \right) \c \nabc A 
 \eeaa
 which gives
 \bea
 \begin{split}
 M&= M_{33} \nabc_3 \nabc_3 A +M_{a3} \c \nabc \nabc_3 A+ M_{\dual a 3} \c \dual \nabc \nabc_3 A+M^a_3(A)\\
 & +M_a\c \nabc A+M_{\dual a} \c \dual \nabc A+M^a_0(A)+ r^{-2}\nabc_3 \left( \frac{a}{r} \Ga_g A+ \Ga_g \frak{d}^{\leq 1} A\right)
 \end{split}
 \eea
  where
  \beaa
  M_{33}&=& 2\left( 4H+\Hb +\ov{\Hb} \right) \c \eta, \\
  M_{a3}&=& 2 \nabc_3 \left( 4H+\Hb +\ov{\Hb} \right)-   \trchb \left( 4H+\Hb +\ov{\Hb} \right), \\
  M_{\dual a 3}&=& - \atrchb\left( 4H+\Hb +\ov{\Hb} \right), \\
  M^a_{3}(A)&=& \left( 4H+\Hb +\ov{\Hb} \right) \c \left(-\nabc ( \C )+\nabc_3\eta+(\C-\frac  1 2   \trchb)\,\eta-\frac 1 2 \atrchb\, \dual \eta\right) \nabc_3 A\\
  &&+\left( 4H+\Hb +\ov{\Hb} \right) \c \Big(\nabc_3C^0_{3,a}(A)+C^0_{3,a}(\nabc_3A)\Big)+ 2 \nabc_3 \left( 4H+\Hb +\ov{\Hb} \right) \c \eta \nabc_3 A, \\
  M_a&=& Q \left( 4H+\Hb +\ov{\Hb} \right) +\left(-\frac  \C 2   \trchb+\frac 1 2 \trchb^2-\frac 1 2 \atrchb^2\right)\left( 4H+\Hb +\ov{\Hb} \right)\\
  &&-  \trchb \nabc_3 \left( 4H+\Hb +\ov{\Hb} \right)- \D \left( 4H+\Hb +\ov{\Hb} \right),\\
  M_{\dual a}&=& \left(-\frac \C 2 \atrchb+ \trchb\atrchb \right)\left( 4H+\Hb +\ov{\Hb} \right)- \atrchb \nabc_3 \left( 4H+\Hb +\ov{\Hb} \right),\\
  M^a_0(A)&=& \left( 4H+\Hb +\ov{\Hb} \right) \c \left( -\frac  1 2   \trchb\,C^0_{3,a}(A)-\frac 1 2 \atrchb\, \dual C^0_{3,a}(A)  + \C \  C^0_{3,a}(A)  -\nabc_a (\D)  A_{bc} \right)\\
  &&+ 2 \nabc_3 \left( 4H+\Hb +\ov{\Hb} \right) \c C^0_{3,a}(A).
  \eeaa

  
 \subsubsection{Expression  for $N$}
 
 
 Recall that 
 \beaa
 N&=& [Q, \left(-\ov{\tr X} \tr \Xb +2\ov{P}\right)] A+ 2[Q, H   \hot \ov{\Hb} \c ]A.
 \eeaa
 Using \eqref{Q-f-g} and the fact that
 \beaa
 Q(E \hot (F \c U))&=& Q(E) \hot (F \c U)+ E \hot (Q(F) \c U)+E \hot (F \c Q(U))\\
 &&+ 2\nabc_3 E \hot (\nabc_3 F \c U) + 2 \nabc_3 E \hot (F \c \nabc_3 U)\\
 &&+2 E \hot (\nabc_3 F \c \nabc_3 U) -2 \D E \hot (F \c U)
 \eeaa
  we obtain
 \beaa
  N&=& Q\left(-\ov{\tr X} \tr \Xb +2\ov{P}\right) A + 2 \nabc_3 \left(-\ov{\tr X} \tr \Xb +2\ov{P}\right) \nabc_3 A -\D \left(-\ov{\tr X} \tr \Xb +2\ov{P}\right) A \\
  &&+ 2Q(H) \hot (\ov{\Hb} \c A)+ 2H \hot (Q(\ov{\Hb}) \c A)+ 4\nabc_3 H \hot (\nabc_3 \ov{\Hb} \c A) + 4 \nabc_3 H \hot (\ov{\Hb} \c \nabc_3 A)\\
  &&+4 H \hot (\nabc_3 \ov{\Hb} \c \nabc_3 A) -4 \D H \hot (\ov{\Hb} \c A).
 \eeaa
 Therefore\footnote{The expression for $N$ does not have error terms.}
 \beaa
   N&=& N_3 \nabc_3 A + N_3^a(A)+N_0 A +N_0^a(A) 
 \eeaa
 where
 \beaa
 N_3&=& 2 \nabc_3 \left(-\ov{\tr X} \tr \Xb +2\ov{P}\right),  \\
 N_3^a(A)&=& 4 \nabc_3 H \hot (\ov{\Hb} \c \nabc_3 A)+4 H \hot (\nabc_3 \ov{\Hb} \c \nabc_3 A), \\
 N_0&=& Q\left(-\ov{\tr X} \tr \Xb +2\ov{P}\right)-\D \left(-\ov{\tr X} \tr \Xb +2\ov{P}\right),\\
 N_0^a(A)&=& 2Q(H) \hot (\ov{\Hb} \c A)+ 2H \hot (Q(\ov{\Hb}) \c A)\\
 &&+ 4\nabc_3 H \hot (\nabc_3 \ov{\Hb} \c A)-4 \D H \hot (\ov{\Hb} \c A).
 \eeaa

  
\subsubsection{The commutator}


 Putting the above expressions together we obtain
 \beaa
[Q, \LL]A&=& I + J+K+L+M+N\\
&=&  -4(\eta-\etab ) \c \nabc\nabc_3\nabc_3A+4 \eta \c \nabc \nabc_3 \nabc_3 A -  \left(\tr \Xb+\ov{\tr \Xb}\right)  \nabc_4 \nabc_3 \nabc_3 A\\
&&+\left( I_{43}+J_{43}+L_{43}\right) \  \nabc_4\nabc_3A+\left(I_4+J_4+L_4 \right) \  \nabc_4A\\
 &&+\left(I_{33}+J_{33}+K_{33}+M_{33}\right) \nabc_3 \nabc_3A+ I^a_{33}(A)\\
 && +\left(I_{a3}+J_{a3}+L_{a3}+M_{a3}\right)\c \nabc \nabc_3A +\left(I_{\dual a 3}+ M_{\dual a 3}\right) \c  \dual \nabc  \nabc_3A\\
 &&+ \left(I_3+J_3 +K_3+L_3+N_3\right)\nabc_3A+\left(I^a_3(A)+J^a_3(A)+ L_3^a(A)+M^a_3(A)+N_3^a(A)\right)\\
 &&+\left(J_a+L_a+M_a\right)\c \nabc A +\left(J_{\dual a }+L_{\dual a}+M_{\dual a} \right)\c  \dual \nabc  A\\
 &&+\left(I_0+J_0+K_0+L_0+N_0\right) \  A+J^a_0(A)+L^a_0(A) +M^a_0(A) +N_0^a(A) \\
 &&+ \nabc_3\left( \frac 1 r \Ga_g \frak{d}^{\leq 2} A\right)+\lot
 \eeaa
 which gives
 \beaa
 [Q, \LL]A&=&  4\etab  \c \nabc\nabc_3\nabc_3A -  \left(\tr \Xb+\ov{\tr \Xb}\right)  \nabc_4 \nabc_3 \nabc_3 A\\
&&+\left( I_{43}+J_{43}+L_{43}\right) \  \nabc_4\nabc_3A+\left(I_4+J_4+L_4 \right) \  \nabc_4A\\
 &&+\left(I_{33}+J_{33}+K_{33}+M_{33}\right) \nabc_3 \nabc_3A+ I^a_{33}(A)\\
 && +\left(I_{a3}+J_{a3}+L_{a3}+M_{a3}\right)\c \nabc \nabc_3A +\left(I_{\dual a 3}+ M_{\dual a 3}\right) \c  \dual \nabc  \nabc_3A\\
 &&+ \left(I_3+J_3 +K_3+L_3+N_3\right)\nabc_3A+\left(I^a_3(A)+J^a_3(A)+ L_3^a(A)+M^a_3(A)+N_3^a(A)\right)\\
 &&+\left(J_a+L_a+M_a\right)\c \nabc A +\left(J_{\dual a }+L_{\dual a}+M_{\dual a} \right)\c  \dual \nabc  A\\
 &&+\left(I_0+J_0+K_0+L_0+N_0\right) \  A+J^a_0(A)+L^a_0(A) +M^a_0(A) +N_0^a(A)+\err[ [Q, \LL]A].
\eeaa
Recalling the definition of $Q(A)$:
  \beaa
\nabc_3\nabc_3 A &=& Q(A)- \C \  \nabc_3A -\D \  A
 \eeaa
 we write
 \beaa
 \nabc \nabc_3 \nabc_3 A &=& \nabc Q(A)- \C \  \nabc \nabc_3A- \nabc\C \  \nabc_3A -\D \ \nabc A-\nabc \D \  A,\\
 \nabc_4\nabc_3\nabc_3A&=& \nabc_4Q(A)- \C \  \nabc_4\nabc_3A -\D \  \nabc_4A- \nabc_4\C \  \nabc_3A -\nabc_4\D \  A.
 \eeaa
 Hence,
  \beaa
[Q, \LL]A &=&  4\etab  \c \nabc Q(A)  -  \left(\tr \Xb+\ov{\tr \Xb}\right) \nabc_4Q(A)\\
&&+\left( I_{43}+J_{43}+L_{43}+\C \left(\tr \Xb+\ov{\tr \Xb}\right)\right) \  \nabc_4\nabc_3A\\
&&+\left(I_4+J_4+L_4 +\D \left(\tr \Xb+\ov{\tr \Xb}\right)\right) \  \nabc_4A\\
 &&+C_Q \ Q(A)+ 8  \eta \hot (\etab \c Q(A))-8 \etab \hot (\eta \c Q(A))\\
 && +C_{a3} \c \nabc \nabc_3A +C_{\dual a 3} \c  \dual \nabc  \nabc_3A+ C_3\nabc_3A+C_3^a(A)\\
 &&+C_a\c \nabc A +C_{\dual a } \c  \dual \nabc  A+C_0 \  A+C_0^a(A)+\err[ [Q, \LL]A]
  \eeaa
  where
  \beaa
  C_{Q}&=&I_{33}+J_{33}+K_{33}+M_{33}, \\
  C_{a3}&=&I_{a3}+J_{a3}+L_{a3}+M_{a3}-4\C \etab, \\
  C_{\dual a 3}&=& I_{\dual a 3}+ M_{\dual a 3},\\
  C_{3}&=&I_3+J_3 +K_3+L_3+N_3-4\etab \nabc \C+\nabc_4\C \left(\tr \Xb+\ov{\tr \Xb}\right)-\C \left(I_{33}+J_{33}+K_{33}+M_{33}\right),\\
  C_3^a(A)&=&I^a_3(A)+J^a_3(A)+ L_3^a(A)+M^a_3(A)+N_3^a(A)- 8\C  \eta \hot (\etab \c  \  \nabc_3A)+8\C \etab \hot (\eta \c  \nabc_3A ), \\
  C_a&=& J_a+L_a+M_a-4\D \etab, \\
  C_{\dual a}&=&J_{\dual a }+L_{\dual a}+M_{\dual a},\\
   C_0&=& I_0+J_0+K_0+L_0+N_0-4\etab \c \nabc \D+\nabc_4\D \left(\tr \Xb+\ov{\tr \Xb}\right)-\D \left(I_{33}+J_{33}+K_{33}+M_{33}\right),\\
   C_0^a(A)&=& J^a_0(A)+L^a_0(A) +M^a_0(A) +N_0^a(A)-8\D  \eta \hot (\etab \c   A)+ 8  \eta \hot (\etab \c Q(A))+8\D \etab \hot (\eta \c  A).
               \eeaa

Observe that the coefficients of $\nabc_4\nabc_3A$ and $\nabc_4A$ are respectively given by 
\beaa
I_{43} +J_{43}+L_{43}+\C \left(\tr \Xb+\ov{\tr \Xb}\right)&=& \nabc_3 \C+\frac {\C }{2} (\tr \Xb+\ov{\tr\Xb}) - \ov{\tr\Xb} \tr\Xb, \\
I_4 +J_4  + L_4 +\D \left(\tr \Xb+\ov{\tr \Xb}\right)&=& \nabc_3 \D +\D \left(\tr \Xb+\ov{\tr \Xb}\right)-\frac \C 4\tr\Xb(\ov{\tr\Xb}).
\eeaa
Therefore if the above are $\Ga_g$ and $r^{-1} \Ga_g$ respectively, as in the assumption of Proposition \ref{first-intermediate-step-main-theorem}, i.e. in \eqref{equation-nab-3-c-1} and \eqref{equation-nab-3-d-1}, the commutator becomes
 \beaa
[Q, \LL]A &=&  4\etab  \c \nabc Q(A)  -  \left(\tr \Xb+\ov{\tr \Xb}\right) \nabc_4Q(A)+C_0( Q(A))\\
 && +C_{a3} \c \nabc \nabc_3A +C_{\dual a 3} \c  \dual \nabc  \nabc_3A+ C_3\nabc_3A+C_3^a(A)\\
 &&+C_a\c \nabc A +C_{\dual a } \c  \dual \nabc  A+C_0 \  A+C_0^a(A)+ \err[[Q, \LL]A]
  \eeaa
  which proves the highest order terms of \eqref{final-commutator}.

  
  \subsubsection{The linear lower order terms}
  
  
  We now show that the coefficients of $\nabc_3A$ and $A$ are $O(|a|)$. We denote by $O\left(\frac{|a|}{r^c}\right)$ any function which vanishes in Schwarzschild, as multiples of $\eta$, $\etab$, $\atrchb$, $\dual \rho$, and has a $r^{-c}$ fall-off in $r$. In particular observe that, according to \eqref{definition-c-general} and \eqref{definition-d-general}, we can write
  \beaa
  \C&=& 2 \trchb +O\left(\frac{|a|}{r^2}\right), \qquad \D=\frac 1 2 \trchb^2 + O\left(\frac{|a|}{r^3}\right).
  \eeaa
  We therefore have
  \beaa
  \nabc_4 \C&=& - \trch\trchb+4\rho +O\left(\frac{|a|}{r^3}\right), \\
  \nabc_3 \C&=& -\trchb^2 +O\left(\frac{|a|}{r^3}\right), \\
  \nabc_4 \nabc_3 \C&=&  \trch\trchb^2 -4\trchb \rho+O\left(\frac{|a|}{r^4}\right), \\
  \nabc_4 \D&=& -\frac 1 2 \trch\trchb^2+2\trchb \rho +O\left(\frac{|a|}{r^4}\right), \\
  \nabc_3 \D&=& -\frac 1 2 \trchb^3 +O\left(\frac{|a|}{r^4}\right),\\
  \nabc_4 \nabc_3 \D&=& \frac 3 4  \trch\trchb^3- 3  \trchb^2 \rho  +O\left(\frac{|a|}{r^5}\right).
  \eeaa
  In what follows we omit to write the error terms since they can all be included in the above expression for $ \err[[Q, \LL]A]$. 
  
  We compute $C_3$. We compute
\beaa
I_{3}&=&-\nabc_3(-P+3\ov{P}  -2\eta\c\etab)+(\trchb (\eta-\etab ) -\atrchb \dual (\eta-\etab ) )\c \eta- \C \ (-P+3\ov{P}  -2\eta\c\etab)\\
&&+\nabc_4(\D)+ \nabc_4\nabc_3 \C \\
&=&-\nabc_3(2\rho)- 2\trchb \ (2\rho) -\frac 1 2 \trch\trchb^2+2\trchb \rho+  \trch\trchb^2 -4\trchb \rho +O\left(\frac{|a|}{r^4}\right)\\
&=&\frac 1 2 \trch\trchb^2-3\trchb \rho+O\left(\frac{|a|}{r^4}\right), 
\eeaa
\beaa
J_3&=& -  \left(\tr \Xb+\ov{\tr \Xb}\right)  \hat{J}_3 +\left(- \frac 1 2\ov{\tr\Xb} \tr\Xb-\frac \C 2 (\tr \Xb+\ov{\tr\Xb})\right)\, \left( \frac 1 2 \tr X +2\ov{\tr X} \right)+\tilde{J}_{3}\\
&=& -  2\trchb \left(\nabc_3\left(\frac 5 2 \trch\right)+(\trch) (\trchb ) \right) +\left(- \frac 52\trchb^2\right)\, \left( \frac 5 2 \trch\right)+O\left(\frac{|a|}{r^4}\right)\\
&=&-\frac{23}{4} \trch\trchb^2 -  10\trchb  \rho+O\left(\frac{|a|}{r^4}\right),
\eeaa
\beaa
K_3&=& Q\left(- \frac 1 2 \tr X -2\ov{\tr X} \right)-\D \ \left(- \frac 1 2 \tr X -2\ov{\tr X} \right) - \left(- \frac 1 2 \tr X -2\ov{\tr X} \right) \nabc_3 \C\\
&=& \nabc_3\nabc_3\left(-\frac 5 2 \trch \right)+2\trchb \nabc_3\left(-\frac 5 2 \trch \right)+\frac 5 2 \trch(-\trchb^2) +O\left(\frac{|a|}{r^4}\right)\\
&=& -\frac 5 2\nabc_3\left(-\frac 1 2 \trchb\trch+2\rho\right)-5\trchb \left(-\frac 1 2 \trchb\trch+2\rho\right)+\frac 5 2 \trch(-\trchb^2) +O\left(\frac{|a|}{r^4}\right)\\
&=& -\frac 5 2\left(\frac 1 4 \trchb^2\trch-\frac 1 2 \trchb\left(-\frac 1 2 \trchb\trch+2\rho\right)-3\trchb\rho\right)-10\trchb \rho+O\left(\frac{|a|}{r^4}\right)\\
&=& -\frac 5 4 \trchb^2\trch +O\left(\frac{|a|}{r^4}\right),
\eeaa
\beaa
 L_3&=& -\frac{1}{2}\tr\Xb\left(-P+7\ov{P}  -6\eta\c\etab  +2(\eta-\etab ) \c \eta -\nabc_4( \C )\right)\\
 &=& -\frac{1}{2}\trchb\left(6\rho  -(- \trch\trchb+4\rho)\right)+O\left(\frac{|a|}{r^4}\right)\\
  &=& -\frac{1}{2}\trch\trchb^2-\trchb \rho+O\left(\frac{|a|}{r^4}\right),
\eeaa
\beaa
N_3&=& 2 \nabc_3 \left(-\ov{\tr X} \tr \Xb +2\ov{P}\right) \\
&=& 2\left(-\nabc_3\trch\trchb-\trch\nabc_3\trchb+2\nabc_3\rho\right) +O\left(\frac{|a|}{r^4}\right)\\
&=& 2 \trch\trchb^2-10\rho\trchb+O\left(\frac{|a|}{r^4}\right).
\eeaa
We compute 
\beaa
I_{33}&=& 3P-5\ov{P} +4\eta\c\etab-2|\eta|^2  +\nabc_4( \C )=  - \trch\trchb+2\rho +O\left(\frac{|a|}{r^3}\right),\\
J_{33}&=& -  \left(\tr \Xb+\ov{\tr \Xb}\right) \left(\frac 1 2 \tr X +2\ov{\tr X}\right)+\DDc\c\ov{H}+2H \c\ov{H}= - 5\trchb  \trch+O\left(\frac{|a|}{r^3}\right),\\
K_{33}&=& 2\nabc_3\left(- \frac 1 2 \tr X -2\ov{\tr X} \right)= \frac 5 2 \trch\trch-10\rho +O\left(\frac{|a|}{r^3}\right),\\
 M_{33}&=& 2\left( 4H+\Hb +\ov{\Hb} \right) \c \eta=O\left(\frac{|a|}{r^4}\right).
\eeaa
We finally obtain
\beaa
 C_{3}&=&I_3+J_3 +K_3+L_3+N_3-4\etab \c \nabc \C+\nabc_4\C \left(\tr \Xb+\ov{\tr \Xb}\right)-\C \left(I_{33}+J_{33}+K_{33}+M_{33}\right)\\
 &=&\frac 1 2 \trch\trchb^2-3\trchb \rho-\frac{23}{4} \trch\trchb^2 -  10\trchb  \rho -\frac 5 4 \trchb^2\trch-\frac{1}{2}\trch\trchb^2-\trchb \rho\\
 &&+2 \trch\trchb^2-10\rho\trchb+ 2\trchb(- \trch\trchb+4\rho) \\
 &&-2\trchb \left(- \trch\trchb+2\rho - 5\trchb  \trch+\frac 5 2 \trch\trch-10\rho\right)+O\left(\frac{|a|}{r^4}\right)\\
 &=&O\left(\frac{|a|}{r^4}\right).
 \eeaa

We now compute $C_0$. We compute
\beaa
I_0&=&  \nabc_4\nabc_3\D= \frac 3 4  \trch\trchb^3- 3  \trchb^2 \rho  +O\left(\frac{|a|}{r^5}\right),
\eeaa
\beaa
J_0&=& -  \left(\tr \Xb+\ov{\tr \Xb}\right) \Big(\nabc_3\left(\ov{\tr X} \tr \Xb -2\ov{P}\right) +\frac{1}{2}\tr\Xb(4\ov{P}  -4\eta\c\etab)\\
&&+(\trchb \left( 4H+\Hb +\ov{\Hb} \right) -\atrchb \dual \left( 4H+\Hb +\ov{\Hb} \right) )\c \eta \Big)\\
&&+\left(- \frac 1 2\ov{\tr\Xb} \tr\Xb-\frac \C 2 (\tr \Xb+\ov{\tr\Xb})\right)\,  \left(\ov{\tr X} \tr \Xb -2\ov{P}\right)-\frac 1 2 (\DDc \c\DDbc (\D))\\
&=& -  2\trchb \Big(\nabc_3\left(\trch\trchb -2\rho\right) +\frac{1}{2}\trchb(4\rho  ) \Big)+\left(-\frac 5 2\trchb^2  \right)\,  \left(\trch\trchb -2\rho\right)+O\left(\frac{|a|}{r^5}\right)\\
&=& -\frac 1 2  \trch\trchb^3-9\rho\trchb^2+O\left(\frac{|a|}{r^5}\right),
\eeaa
\beaa
K_0&=& -\left(- \frac 1 2 \tr X -2\ov{\tr X} \right)\nabc_3\D= -\frac 5 4 \trch  \trchb^3+O\left(\frac{|a|}{r^5}\right),
\eeaa
\beaa
 L_0&=&  \tr\Xb \left((\trchb (\eta-\etab ) -\atrchb \dual (\eta-\etab ) )\c \eta\right)+ \left(\frac 1 2 \tr\Xb^2-\frac{1}{2}\tr\Xb\C \right)\ (4\ov{P}  -4\eta\c\etab)\\
  && +\frac{1}{2}\tr\Xb\nabc_4(\D) \   -\frac{1}{2}\tr\Xb\nabc_3(4\ov{P}  -4\eta\c\etab)\\
  &=& -2 \trchb^2  \rho +\frac{1}{2}\trchb\left(-\frac 1 2 \trch\trchb^2+2\trchb \rho\right)\   -2\trchb \nabc_3(\rho)+O\left(\frac{|a|}{r^5}\right)\\
   &=&-\frac{1}{4} \trch\trchb^3+ 2 \trchb^2  \rho \  +O\left(\frac{|a|}{r^5}\right),
\eeaa
\beaa
N_0&=& Q\left(-\ov{\tr X} \tr \Xb +2\ov{P}\right)-\D \left(-\ov{\tr X} \tr \Xb +2\ov{P}\right)\\
&=& \nabc_3 \nabc_3\left(-\trch\trchb + 2\rho \right)+2\trchb  \nabc_3\left(-\trch\trchb + 2\rho\right)\\
&=& \nabc_3 ( \trch\trchb^2-5\rho\trchb)+2\trchb  \left( \trch\trchb^2-5\rho\trchb\right)+O\left(\frac{|a|}{r^5}\right)\\
&=& \left( -\frac 1 2 \trch\trchb+2\rho\right)\trchb^2+ 2\trch\trchb \left(-\frac 1 2 \trchb^2\right)-5\left(-\frac 3 2 \trchb\rho\right)\trchb-5\rho\left(-\frac 1 2 \trchb^2\right)\\
&&+2\trchb  \left( \trch\trchb^2-5\rho\trchb\right)+O\left(\frac{|a|}{r^5}\right)\\
&=& \frac 1 2 \trch \trchb^3+ 2\rho\trchb^2+O\left(\frac{|a|}{r^5}\right).
\eeaa
We finally obtain
\beaa
 C_0&=& I_0+J_0+K_0+L_0+N_0-4\etab \c \nabc \D+\nabc_4\D \left(\tr \Xb+\ov{\tr \Xb}\right)-\D \left(I_{33}+J_{33}+K_{33}+M_{33}\right)\\
 &=&  \frac 3 4  \trch\trchb^3- 3  \trchb^2 \rho -\frac 1 2  \trch\trchb^3-9\rho\trchb^2-\frac 5 4 \trch  \trchb^3+2 \trchb^2  \rho -\frac{1}{4} \trch\trchb^3\\
 &&+2\rho\trchb^2+ \frac 1 2 \trch \trchb^3+2\trchb \left(-\frac 1 2 \trch\trchb^2+2\trchb \rho\right) \\
 &&-\frac 1 2 \trchb^2 \left(- \trch\trchb+2\rho - 5\trchb  \trch+\frac 5 2 \trch\trch-10\rho\right)+O\left(\frac{|a|}{r^5}\right)\\
 &=& O\left(\frac{|a|}{r^5}\right).
\eeaa
This implies that all the terms involving $\nabc_3A$ and $A$ are $O(|a|)$. We can write
\beaa
C_3\nabc_3A+C_3^a(A)+C_0 \  A+C_0^a(A)&=&a \left( d_2 \nabc_3 A +d_4 A\right)
\eeaa
with 
\beaa
d_2= O\left(\frac{1}{r^4}\right), \qquad d_4=O\left(\frac{1}{r^5}\right). 
\eeaa

  
  \subsection{The structure of the linear lower order terms: proof of Lemma \ref{lemma-lot}}\label{appendiz-s}


 Observe that in Kerr
\beaa
\eta_1&=& \Re(H_1)=-\Re(\ov{Z_1})=-\Re(Z_1)=\etab_1,\\
\eta_2&=& \Re(H_2)=\Re(\ov{Z_2})=\Re(Z_2)=-\etab_2,
\eeaa
therefore
\beaa
\left(\eta - \etab \right) \c \nabc&=& \left(\eta_1 - \etab_1 \right)  \nabc_1+\left(\eta_2 - \etab_2 \right)  \nabc_2=2\eta_2   \nabc_2=a\, d(r, \th) \nabc_2,
\eeaa
with $d(r, \th)$ is $O\left(\frac{1}{r^2}\right)$. Similarly, 
\beaa
(\dual \eta + \dual \etab) \c \nabc&=&a \a\, d(r, \th) \nabc_2.
\eeaa

We hereby show that there is a choice of $\C$, which satisfies the assumptions of Proposition \ref{first-intermediate-step-main-theorem}, such that the coefficient of the term $\dual \nabc \nabc_3 A$ are multiple of $\left(\eta - \etab \right)$ or $(\dual \eta + \dual \etab)$. This implies that 
\beaa
&&C_{a3} \c \nabc \nabc_3A +C_{\dual a 3} \c  \dual \nabc  \nabc_3A\\
&=&a\, d(r, \th) \left(\eta - \etab \right) \c \nabc\nabc_3 A +a\,  d(r, \th) \left(\dual \eta +\dual  \etab \right) \c \nabc\nabc_3 A\\
&=&a\, d(r, \th)   \nabc_2\nabc_3 A 
\eeaa
where $d(r, \th)$ are generic functions of $r$ and $\th$. In particular, there is no term of the form $\nabc_1 \nabc_3 A$.   In what follows we omit to write the error terms since they can all be included in the above expression for $ \err[[Q, \LL]A]$.


We denote $O(\eta+\etab)$ and $O(\dual \eta - \dual \etab)$ any expression which is linear combination of $\eta+\etab$ and $\dual \eta - \dual \etab$. 

We have 
\beaa
I_{\dual a 3}&=&  \atrchb(\eta-\etab )= \atrchb(\eta+\etab-2\etab )=-2\atrchb \  \etab+ O(\eta+\etab), \\
M_{\dual a 3}&=&  - \atrchb\left( 4H+\Hb +\ov{\Hb} \right)\\
&=&- \atrchb\left( 4\eta+4 i \dual \eta+2\etab  \right)\\
&=&- \atrchb\left( 4\eta+ 4\etab +4 i \dual \eta-4 i \dual \etab-2\etab +4i \dual \etab \right)\\
&=& \atrchb\left( 2\etab -4i \dual \etab \right)+O(\eta+ \etab) +O(\dual \eta - \dual \etab).
\eeaa
This gives
\beaa
C_{\dual a 3}&=& I_{\dual a 3}+ M_{\dual a 3}\\
&=&-2\atrchb \  \etab+ \atrchb\left( 2\etab -4i \dual \etab \right)+O(\eta+ \etab) +O(\dual \eta - \dual \etab)\\
&=&-4i \atrchb  \dual \etab +O(\eta+ \etab) +O(\dual \eta - \dual \etab).
\eeaa
Therefore
  \beaa
  C_{\dual a 3} \c  \dual \nabc  \nabc_3A  &=&\left(-4i \atrchb  \dual \etab +O(\eta+ \etab) +O(\dual \eta - \dual \etab)\right) \c \dual \nabc \nabc_3 A\\
  &=&\left(4i \atrchb  \ \etab +O(\dual\eta+ \dual\etab) +O(\eta -  \etab)\right) \c \nabc \nabc_3 A.
  \eeaa

To compute $C_{a3}$, we recall
\beaa
I_{a3}   &=& -2\left(\nabc_3(\eta-\etab ) +\left(\frac 1 2 \trchb +2\tr\Xb\right)(\eta-\etab)\right) =O(\eta- \etab),
\eeaa
\beaa
J_{a3}&=&-  \left(\tr \Xb+\ov{\tr \Xb}\right) (2(\eta-\etab ) -\left( 4H+\Hb +\ov{\Hb} \right)) +\nabc_3H+(-2\tr\Xb- 2\ov{\tr\Xb}+\C)\, H\\
&&-\DDc ( \C )+2\ov{\tr\Xb}\eta  - \DDc\ov{\tr\Xb}+\nabc_3\ov{H} +\C\,  \ov{H}-\DDbc ( \C )\\
&=&  2\trchb \left( 4\eta+ 4i \dual \eta+2\etab  \right) +2\nabc_3\eta+(-4 \trchb )\, (\eta+i \dual \eta)\\
&&+2(\trchb+ i \atrchb)\eta  - \DDc\ov{\tr\Xb}+2\C\,  \eta-2\nabc ( \C )+O(\eta-\etab).
\eeaa
Using Codazzi equation 
\beaa
 \DDc \ov{\tr\Xb}&=& (\tr\Xb-\ov{\tr\Xb})\Hb+O(\ep)= -2 i \atrchb (\etab + i \dual \etab )
\eeaa
and writing 
\beaa
\nabc_3 \eta&=& \nabc_3(\eta-\etab)+\nabc_3 \etab\\
&=&\frac{1}{2}\atrchb(\dual\etab-\dual\eta)+O(\eta-\etab)\\
&=&\atrchb \dual \etab +O(\eta-\etab) +O( \dual\eta+ \dual\etab)
\eeaa
we have
\beaa
J_{a3}&=&  2\trchb \left( 3\eta+ 2i \dual \eta+2\etab  \right) +2 i \atrchb\left( \eta+\etab  \right) +2\C\,  \eta-2\nabc ( \C )+O(\eta-\etab)+O( \dual\eta+ \dual\etab)\\
&=&  2\trchb \left( 5\etab  -2i \dual\etab   \right) +4 i \atrchb\, \etab  +2\C\,  \etab -2\nabc ( \C )+O(\eta-\etab) +O( \dual\eta+ \dual\etab).
\eeaa
We also have
\beaa
L_{a3}&=& -2\tr\Xb(\eta-\etab )+O(\eta-\etab)
\eeaa
and since $4H+\Hb +\ov{\Hb}=6\etab-4i\dual\etab+O(\eta-\etab)+O( \dual\eta+ \dual\etab)$
\beaa
 M_{a3}&=& 2 \nabc_3 \left( 4H+\Hb +\ov{\Hb} \right)-   \trchb \left( 4H+\Hb +\ov{\Hb} \right)\\
&=& 2 \nabc_3 \left( 6\etab-4i\dual\etab\right)-   \trchb \left( 6\etab-4i\dual\etab\right)+O(\eta-\etab, \dual\eta+ \dual\etab). 
 \eeaa
 Using that
 \beaa
 \nabc_3 \dual \etab  &=& -\frac{1}{2}\trchb(\dual \etab-\dual \eta)+\frac{1}{2}\atrchb(-\etab+\eta)\\
 &=& -\trchb \dual \etab+O(\eta-\etab)+O( \dual\eta+ \dual\etab)
 \eeaa
 we have
 \beaa
 M_{a3}&=& 12\atrchb \dual \etab+8i (\trchb \dual \etab)-   \trchb \left( 6\etab-4i\dual\etab\right)+O(\eta-\etab, \dual\eta+ \dual\etab)\\
 &=& 12\atrchb \dual \etab+12i \trchb \dual \etab-   6\trchb\, \etab+O(\eta-\etab)+O( \dual\eta+ \dual\etab).
 \eeaa
 Therefore, the coefficient of $\nabc \nabc_3 A$ is given by 
\beaa
C_{a3}&=&I_{a3}+J_{a3}+L_{a3}+M_{a3}-4\C \etab\\
&=& 2\trchb \left( 2\etab  +4i \dual\etab   \right) + 12\atrchb \dual \etab+4 i \atrchb\, \etab  -2\C\,  \etab -2\nabc ( \C )+O(\eta-\etab)+O( \dual\eta+ \dual\etab).
\eeaa
Putting together with the $C_{\dual a 3}$ we obtain
\beaa
C_{a3}+C_{\dual a 3}&=& 2\trchb \left( 2\etab  +4i \dual\etab   \right) + 12\atrchb \dual \etab +8 i \atrchb\, \etab -2\C\,  \etab -2\nabc ( \C )\\
&&+O(\eta-\etab)+O( \dual\eta+ \dual\etab).
\eeaa

We now use Lemma \ref{derivatives-tr} and the form of $\C=2 \trchb + i \frak{c} \atrchb$ to compute the above coefficient. We have
\beaa
\nabc \trchb &=& \nab\trchb-\trchb\ze=\nab\trchb+\trchb\, \etab= -\frac 1 2\trchb\,   \etab-\frac 3 2\trchb \eta -\frac 1 2\atrchb  (\dual \eta-  \dual \etab ), \\
\nabc \atrchb&=& \nab\atrchb-\atrchb \ze= \nab\atrchb+\atrchb\, \etab= -\frac 1 2  \atrchb\, \etab-\frac 3 2 \atrchb  \eta +\frac 1 2\trchb  (\dual \eta-  \dual \etab ).
\eeaa
This gives
\beaa
\nabc \C &=&  2 \nabc \trchb + i \frak{c} \nabc \atrchb\\
&=&  2 \left(-\frac 1 2\trchb   \etab-\frac 3 2\trchb \eta -\frac 1 2\atrchb  (\dual \eta-  \dual \etab )\right) + i \frak{c} \left(-\frac 1 2  \atrchb\, \etab-\frac 3 2 \atrchb  \eta +\frac 1 2\trchb  (\dual \eta-  \dual \etab )\right)\\
&=&   (-\trchb   \etab-3\trchb \eta -\atrchb  (\dual \eta-  \dual \etab )) + i \left(-\frac {\frak{c}}{  2}  \atrchb\, \etab-\frac {3\frak{c} }{ 2} \atrchb  \eta +\frac{ \frak{c}}{  2}\trchb  (\dual \eta-  \dual \etab )\right).
\eeaa
Going back to 
\beaa
C_{a3}+C_{\dual a 3}&=& 2\trchb \left( 2\etab  +4i \dual\etab   \right) + 12\atrchb \dual \etab +8 i \atrchb\, \etab -2\C\,  \etab -2\nabc ( \C )\\
&&+O(\eta-\etab) +O( \dual\eta+ \dual\etab)\\
&=&2\trchb\, \etab+6\trchb \eta+2\atrchb  \dual \eta  + 10\atrchb \dual \etab  \\
&&  + i ((-\frak{c} +8)   \atrchb\, \etab+3\frak{c}  \atrchb  \eta - \frak{c}  \trchb  \dual \eta+(\frak{c} +8) \trchb \dual \etab )\\
&&+O(\eta-\etab) +O( \dual\eta+ \dual\etab)
\eeaa
which can again be simplified to
\beaa
C_{a3}+C_{\dual a 3}&=&8\trchb\, \etab  + 8\atrchb \dual \etab   + i ((2\frak{c}+8) \atrchb\, \etab +(2\frak{c}+8) \trchb \dual \etab )\\
&&+O(\eta-\etab) +O( \dual\eta+ \dual\etab).
\eeaa
From \eqref{relation=tr-eta}, we deduce 
\beaa
0= \frac 1 2\trchb  (  \eta +\etab)-\frac 1 2   \atrchb  ( \dual \eta-\dual\etab )= \trchb\,    \etab+  \atrchb   \dual \etab+O(\eta-\etab)+O( \dual\eta+ \dual\etab).
\eeaa
Therefore we can write $  \atrchb   \dual \etab+\trchb\,    \etab=O(\eta-\etab)+O( \dual\eta+ \dual\etab)$, and obtain
\beaa
C_{a3}+C_{\dual a 3}&=& i ((2\frak{c}+8) \atrchb\,  \etab +(2\frak{c}+8) \trchb \dual \etab )+O(\eta-\etab) +O(\dual\eta+ \dual\etab).
\eeaa
Choosing $\frak{c}=-4$, we obtain
\beaa
C_{a3}+C_{\dual a 3}&=&O(\eta-\etab) +O(\dual\eta+ \dual\etab)
\eeaa
and therefore 
\beaa
C_{a3} \c \nabc \nabc_3A +C_{\dual a 3} \c  \dual \nabc  \nabc_3A= a\c d_1   \nabc_2\nabc_3 A \qquad \text{for $d_1=O\left(\frac{1}{r^3}\right)$.}
\eeaa

By performing similar computations we obtain
\beaa
C_a\c \nabc A +C_{\dual a } \c  \dual \nabc  A=a\c  d_3 \nabc A  \qquad \text{for $d_3=O\left(\frac{1}{r^4}\right)$.}
\eeaa


\end{document}